\documentclass[11pt,letterpaper]{amsart}
\usepackage[foot]{amsaddr}

\usepackage{mathtools}
\usepackage[T1]{fontenc}
\usepackage[latin9]{inputenc}
\usepackage{geometry}
\usepackage{wrapfig}
\usepackage{multicol}
\usepackage{graphicx}
\usepackage{soul}
\usepackage{xcolor}
\usepackage{amssymb}
\usepackage{placeins}
\usepackage{cite}
\usepackage{caption}
\usepackage{enumerate}
\usepackage{afterpage}
\usepackage{enumitem}
\usepackage{bmpsize}
\usepackage{hyperref}
\usepackage{tabu}
\usepackage{enumitem}
\numberwithin{equation}{section}
\usepackage{stmaryrd}
\usepackage{tikz}
\usetikzlibrary{matrix,graphs,arrows,positioning,calc,decorations.markings,shapes.symbols}

\makeatletter
\renewcommand{\email}[2][]{%
  \ifx\emails\@empty\relax\else{\g@addto@macro\emails{,\space}}\fi%
  \@ifnotempty{#1}{\g@addto@macro\emails{\textrm{(#1)}\space}}%
  \g@addto@macro\emails{#2}%
}
\makeatother

\newtheorem{theorem}{Theorem}[section]
\newtheorem{lemma}[theorem]{Lemma}
\newtheorem{proposition}[theorem]{Proposition}

{ \theoremstyle{definition}
\newtheorem{definition}[theorem]{Definition}}
{ \theoremstyle{remark}
\newtheorem{remark}[theorem]{Remark}}

\def\Re{ \mathsf{Re}}
\def\Im{ \mathsf{Im}}

\newcommand{\weyl}{W^\circ}

\newcommand{\topc}{K}
\newcommand{\g}{G}
\newcommand{\tsigma}{d}

\usepackage{mathtools}

\def\Rt{R^{\mathsf{top}}}
\def\Rb{ R^{\mathsf{bot}}}

\def\Nt{ N^{\mathsf{top}}}
\def\Nb{ N^{\mathsf{bot}}}
\def\Na{ N^{\mathsf{acc}}}
\def\Ns{ N^{\mathsf{sep}}}

\title{Tightness of $(H, H^{RW})$-Gibbsian line ensembles}

\author[E. Dimitrov]{Evgeni Dimitrov}
\address{E. Dimitrov, Department of Mathematics, Rm. 517, Columbia University, New York, NY 10027} 
\email{esd2138@columbia.edu}
\author[X. Wu]{Xuan Wu}
\address{X. Wu,  Department of Mathematics, Eckhart 319, University of Chicago, Chicago, IL 60637}
\email{xuanw@uchicago.edu}
\begin{document}

\maketitle

\begin{abstract}
We develop a black-box theory, which can be used to show that a sequence of Gibbsian line ensembles is tight, provided that the one-point marginals of the lowest labeled curves of the ensembles are tight and globally approximate an inverted parabola. Our theory is applicable under certain technical assumptions on the nature of the Gibbs property and the underlying random walk measure. As a particular application of our general framework we show that a certain sequence of Gibbsian line ensembles, which naturally arises in the log-gamma polymer, is tight in the ubiquitous KPZ class $1/3: 2/3$ scaling, and also that all subsequential limits satisfy the Brownian Gibbs property, introduced by Corwin and Hammond in (Invent. Math. 195, 441-508, 2014).

One of the core results proved in the paper, which underlies many of our arguments, is the construction of a continuous grand monotone coupling of Gibbsian line ensembles with respect to their boundary data (entrance and exit values, and bounding curves). {\em Continuous} means that the Gibbsian line ensemble measure varies continuously as one varies the boundary data, {\em grand} means that all uncountably many measures (one for each boundary data) are coupled to the same probability space, and {\em monotone} means that raising the values of the boundary data likewise raises the associated measure. Our continuous grand monotone coupling generalizes an analogous construction, which was recently implemented by Barraquand, Corwin and Dimitrov in (arXiv:2101.03045), from line ensembles with a single curve to ones with an arbitrary finite number of curves.
\end{abstract}

\tableofcontents

%
\section{Introduction and main results}\label{Section1}

%
\subsection{Preface}\label{Section1.1} Over the last two decades there has been a considerable surge in the development and study of the theory of {\em Gibbsian line ensembles}. One way to think of a Gibbsian line ensemble is as the collection of trajectories of a finite or countably infinite number of independent random walkers, whose joint distribution is reweighed by a Radon-Nikodym derivative proportional to the exponential of the sum of local interaction energies between the different walkers. Here, {\em local} means that the interaction energy of each walker depends only on the location of nearby walkers both in terms of time and label. 

Line ensembles generally come in two flavors -- continuous and discrete, depending on whether the underlying path measure is given by independent Brownian motions or discrete time random walks. One can impose various different interactions between the different walkers, resulting in different Gibbs measures, and one of simplest examples of a Gibbsian line ensemble is a collection of random walks conditioned not to touch or cross each other. In this case the local interaction energy is infinity or zero depending on whether the touching or crossing occurs or does not. Such Gibbsian line ensembles arise in discrete settings (when the path measure is given by Bernoulli random walks) as level lines of random lozenge or domino tilings, and in continuous settings (when the path measures is given by Brownian motions) in {\em Dyson Brownian motion} \cite{Dys62} and its edge scaling limit -- the {\em parabolic Airy line ensemble} \cite{CorHamA}. Line ensembles, whose local distribution is described by non-crossing Brownian motions (or bridges), are referred to as {\em Brownian Gibbsian line ensembles} or equivalently are said to possess the {\em Brownian Gibbs property} -- a term coined in \cite{CorHamA}.

There are various discrete line ensembles that arise in statistical mechanics models in integrable probability and satisfy a Gibbs property that is a discrete analogue of the Brownian Gibbs property. Examples include random lozenge tilings \cite{Petrov14}, domino tilings \cite{Jo02}, the multi-layer polynuclear growth model \cite{Spohn}, geometric and exponential last passage percolation \cite{KJ}, Hall-Littlewood plane partitions \cite{CD}, and the log-gamma polymer \cite{Sep12}. Given their prominence in integrable probability, there is a general interest in understanding the scaling limits of Gibbsian line ensembles and it is believed that under general conditions they should converge to the parabolic Airy line ensemble. The convergence to this line ensemble is not coincidental and is predicted by the Kardar-Parisi-Zhang (KPZ) scaling theory, for which the parabolic Airy ensemble is of special significance, specifically through its role in the construction of the {\em Airy sheet} in \cite{DOV18}.

The convergence to the parabolic Airy ensemble for various discrete line ensembles, including geometric and exponential last passage percolation, has been recently established in \cite{DNV19}. A distinguished feature of the models covered in \cite{DNV19} is that they have the structure of a {\em determinantal point process}. The determinantal structure is what allows the authors in \cite{DNV19} to establish finite dimensional convergence of the discrete line ensembles they consider to the parabolic Airy ensemble and the Gibbs property is what allows them to improve this convergence to one that is uniform over compact sets. Beyond the works on determinantal point processes, there has been limited progress in demonstrating that Gibbsian line ensembles converge to the parabolic Airy line ensemble with the notable exception of the KPZ line ensemble \cite{QuaSar,virag2020heat, Wu21}.\\

Our present study is prompted by our interest in demonstrating that a certain sequence of Gibbsian line ensemble, which arises in the log-gamma polymer of \cite{Sep12}, converges to the parabolic Airy line ensemble. We refer to each such ensemble as a {\em log-gamma line ensemble} and we will give a proper definition of it in Section \ref{Section1.2.2}. Here we mention that the way this ensemble appears in the log-gamma polymer is through a connection of the latter with the Whittaker processes \cite{COSZ}, and that the structure of the local energy in this model is considerably more complicated than that of non-touching or crossing random walks. The proof of convergence of a sequence of log-gamma line ensembles to the parabolic Airy line ensemble goes through the following two steps:
\begin{enumerate}
\item Show that the lowest-indexed curves of the log-gamma line ensembles converge in the finite dimensional sense to the {\em parabolic Airy process} (the top curve in the parabolic Airy line ensemble).
\item Show that the sequence of log-gamma line ensembles is tight and all subsequential limits satisfy the Brownian Gibbs property.
\end{enumerate}
The fact that the above two statements imply the convergence of a sequence of log-gamma line ensembles to the parabolic Airy line ensemble follows from a recent result due to the first author and Matetski in\cite{DimMat}, which characterizes the parabolic Airy line ensemble as the unique line ensemble that has the parabolic Airy process as its top curve and satisfies the Brownian Gibbs property. See \cite[Section 1]{DimMat} for more details. 

In the present paper, we implement the second step of the above program and the first step will be completed in the upcoming work \cite{ED21+}. We mention that the above framework is the one that allowed the proof of the convergence of the KPZ line ensemble to the parabolic Airy line ensemble, with the first step established independently in \cite{QuaSar} and \cite{virag2020heat}, and the second step established by the second author in \cite{Wu21}.\\

The primary aim of our work is to develop a black-box theory (Theorem \ref{Thm1}) which proves tightness for sequences of discrete Gibbsian line ensembles and that all subsequential limits satisfy the Brownian Gibbs property. We develop this theory for a general class of line ensembles in which the underlying random walk measure has continuous jumps and scales diffusively to Brownian motion and in which the interaction energy is such that a key stochastic monotonicity property holds (Lemma \ref{MonCoup}). Section \ref{Section1.2.1} contains a statement of our black-box theory (various definitions and terminology are introduced in more detail in the main text). The secondary aim of our work is to apply our black-box theory to a sequence of log-gamma line ensembles, and show that under the KPZ $1/3: 2/3$ scaling this sequence of ensembles is tight. In Section \ref{Section1.2.2} we recall the definition of the log-gamma polymer as well as describe the nature of the Gibbsian line ensemble into which it embeds. Combining this embedding with the one-point convergence proved recently in \cite{BCDA}, we can apply our black-box theory and arrive at the advertised tightness result, see Theorem \ref{ThmLG}.

%
\subsection{Main results}\label{Section1.2} In this section we present our main results. Our black-box theory is described in Section \ref{Section1.2.1} and in Section \ref{Section1.2.2} we discuss our application to the log-gamma polymer.

%
\subsubsection{Tightness for general Gibbsian line ensembles}\label{Section1.2.1} Our discussion in this section will closely follow that of \cite[Section 1.1]{BCDB} and \cite[Section1.2]{DREU}.

In this paper we deal with discrete time, continuous valued Gibbsian line ensembles -- see Figure \ref{LineEnsembleFig} for an illustration and Section \ref{Section2.2} for a precise definition. Informally, these are measures on collections of curves $\mathfrak{L}= (L_1, \dots, L_K)$ so that each $L_i$ is the linear interpolation of a function from $\llbracket T_0, T_1 \rrbracket$ to $\mathbb{R}$, $\topc\geq 2$ and $\llbracket T_0, T_1 \rrbracket\subset \mathbb{Z}$ is an integer interval. Here and throughout the paper we write $\llbracket a,b \rrbracket = \{a, a+1, \dots, b\}$ for two integers $b \geq a$. The key property, which these measures enjoy, is a resampling invariance, which we refer to as the {\em (partial) $(H,H^{RW})$-Gibbs property}. The function $H^{RW}$ is the {\em random walk Hamiltonian}, and the function $H$ is the {\em interaction Hamiltonian}. We describe this Gibbs property informally here. For any
$k_1\leq k_2$ with $k_1,k_2 \in \llbracket 1,\topc-1 \rrbracket$ and any $a<b$ with $a, b\in\llbracket T_0, T_1 \rrbracket$, the law of the curves $L_{k_1},\ldots, L_{k_2}$ on the interval $\llbracket a, b \rrbracket$ is a reweighing of random walk bridges according to a specific Radon-Nikodym derivative. The random walk bridges have starting and ending values that match the values of $L_{k_1},\ldots, L_{k_2}$ at $a$ and $b$, respectively, and have jump increments with density, given by $G(x)=e^{-H^{RW}(x)}$. The Radon-Nikodym derivative, which reweighs this measure, is proportional to
\begin{equation}
\prod_{i=k_1-1}^{k_2}\prod_{m=a}^{b-1} e^{-H\left(L_{i+1}(m+1)-L_i(m)\right)}.
\label{RadonNikodymIntroeq}
\end{equation}
The fact that we assume that $k_2\leq \topc-1$ means that we are fixing the curve indexed by $\topc$ and not resampling it. This is why we use the term {\em partial} in describing this Gibbs property. 

The above definition implies that on a given domain, the law of the inside of the line ensemble is determined only by the boundary values of the domain, and is independent of what lies outside. In this sense, this is similar to a spatial version of the Markov property. Section \ref{Section2.2} contains a precise definition of the $(H,H^{RW})$-Gibbsian line ensembles that we have described above.

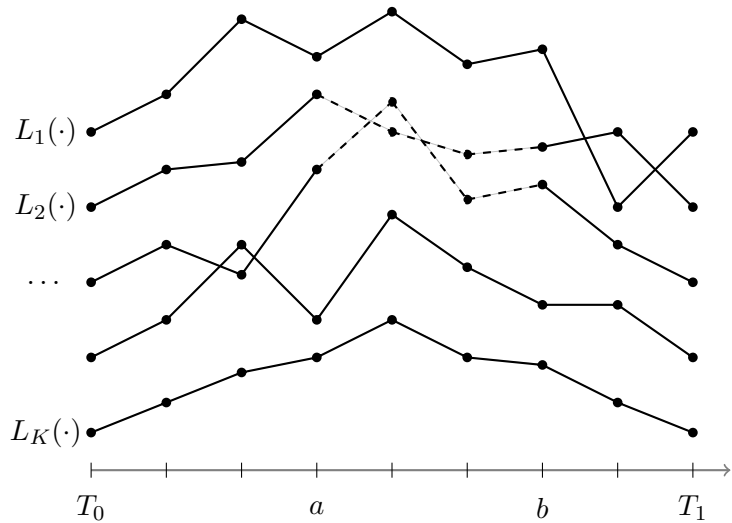
\begin{figure}
\begin{tikzpicture}[scale=1]
\filldraw[thick] (-1,4.5) circle(0.05) -- (0,5) circle(0.05) -- (1,6) circle(0.05)  -- (2,5.5) circle(0.05) -- (3,6.1) circle(0.05) -- (4,5.4) circle(0.05) --  (5,5.6) circle(0.05) -- (6,3.5) circle(0.05) -- (7,4.5) circle(0.05);
\filldraw[thick] (-1,3.5) circle(0.05) -- (0,4) circle(0.05) -- (1,4.1) circle(0.05)  -- (2,5) circle(0.05);
\filldraw[thick, dashed] (2,5) -- (3,4.5) circle(0.05) -- (4,4.2) circle(0.05) -- (5,4.3);
\filldraw[thick]  (5,4.3) circle(0.05) -- (6,4.5) circle(0.05)-- (7,3.5) circle(0.05);
\filldraw[thick]  (-1,2.5) circle(0.05) -- (0,3) circle(0.05) -- (1,2.6) circle(0.05)  -- (2,4) circle(0.05);
\filldraw[thick, dashed] (2,4)  -- (3,4.9) circle(0.05) -- (4,3.6) circle(0.05) -- (5,3.8) ;
\filldraw[thick] (5,3.8) circle(0.05)-- (6,3) circle(0.05)-- (7,2.5) circle(0.05);
\filldraw[thick] (-1,1.5) circle(0.05) -- (0,2) circle(0.05) -- (1,3) circle(0.05)  -- (2,2) circle(0.05) -- (3,3.4) circle(0.05) -- (4,2.7) circle(0.05) -- (5,2.2) circle(0.05) -- (6,2.2) circle(0.05)-- (7,1.5) circle(0.05);
\filldraw[thick] (-1,0.5) circle(0.05) -- (0,0.9) circle(0.05) -- (1,1.3) circle(0.05)  -- (2,1.5) circle(0.05) -- (3,2) circle(0.05) -- (4,1.5) circle(0.05) -- (5,1.4) circle(0.05) -- (6,0.9) circle(0.05)-- (7,0.5) circle(0.05);
\draw[gray, thick, -> ] (-1,0)  -- (7.5,0);
\foreach \x in {-1,...,7}
\draw (\x, 0.1) -- (\x,-0.1);
\draw (-1,-0.5) node{$T_0$};
\draw (7,-0.5) node{$T_1$};
\draw (2,-0.5) node{$a$};
\draw (5,-0.5) node{$b$};
\draw (-1.6,4.5) node{$ L_1(\cdot)$};
\draw (-1.6,3.5) node{$ L_2(\cdot)$};
\draw (-1.6,2.5) node{$\dots$};
\draw (-1.6,0.5) node{$ L_\topc(\cdot)$};
\end{tikzpicture}
\caption{A discrete line ensemble $\mathfrak{L}$. A special case of the partial $(H,H^{RW})$-Gibbs property is that the distribution of lines $L_i$, indexed by $i\in \llbracket 2,3 \rrbracket$, on the interval $\llbracket a,b\rrbracket$ (corresponding to the portion of lines that are dashed in the picture) is absolutely continuous with respect to the law of random walk bridges (with jump increment density $G(x)=e^{-H^{RW}(x)}$) joining $L_i(a)$ to $L_i(b)$, with Radon-Nikodym derivative proportional to \eqref{RadonNikodymIntroeq} with $k_1 = 2$ and $k_2 = 3$.}
\label{LineEnsembleFig}
\end{figure}

\medskip

In the remainder of this section we fix a sequence $\mathfrak{L}^N = (L_1^N, \dots, L_N^N)$ of $(H,H^{RW})$-Gibbsian line ensembles on $\llbracket a_N, b_N \rrbracket$ where $a_N \leq 0$ and $b_N \geq 0$ are integers. Our interest is in understanding the asymptotic behavior of $\mathfrak{L}^N$ as $N \rightarrow \infty$ (i.e. when the number of curves tends to infinity). Below we we list several assumptions on the sequence $\mathfrak{L}^N$, which rely on parameters $\alpha > 0$, $p \in \mathbb{R}$ and $\lambda >0$ (these are all fixed), as well as the Hamiltonians $H$ and $H^{RW}$. The parameter $\alpha$ is related to the fluctuation exponent of the line ensemble and the assumptions below will indicate that $L_1^N(0)$ fluctuates on order $N^{\alpha/2}$. The parameter $p$ is the global slope of the line ensemble and the parameter $\lambda$ is related to the global curvature of the line ensemble, and the assumptions below will indicate that once the slope is removed the line ensemble approximates the parabola $-\lambda x^2$. We now turn to formulating our assumptions precisely.

\vspace{2mm}

{\bf \raggedleft Assumption 1.} We assume that there is a function $\psi: \mathbb{N} \rightarrow (0, \infty)$ such that $\lim_{N \rightarrow \infty} \psi(N) = \infty$ and $a_N < -\psi(N) N^{\alpha}$ while $b_N > \psi(N) N^{\alpha}$.
\begin{remark}
The significance of Assumption 1 is that the sequence of intervals $[a_N, b_N]$ (on which the line ensemble $\mathfrak{L}^N$ is defined) on scale $N^{\alpha}$ asymptotically cover the entire real line. The nature of $\psi$ is not important and any function converging to infinity along the integers works for our purposes.
\end{remark}

\vspace{2mm}

{\bf \raggedleft Assumption 2.}  There is a function $\phi: (0, \infty) \rightarrow (0,\infty)$ such that for any $\epsilon > 0$ we have 
\begin{equation}\label{S1globalParabola}
 \sup_{n \in \mathbb{Z}} \limsup_{N \rightarrow \infty} \mathbb{P} \left( \left|N^{-\alpha/2}(L_1^N( \lfloor n N^{\alpha} \rfloor ) - p n N^{\alpha} + \lambda n^2 N^{\alpha/2}) \right| \geq \phi(\epsilon) \right) \leq \epsilon.
\end{equation}

\begin{remark}
The $n = 0$ case of (\ref{S1globalParabola}) indicates that $N^{-\alpha/2} L_1(0)$ is a tight sequence of random variables and so $\alpha/2$ is the fluctuation exponent of the ensemble. The transversal exponent is $\alpha$ and is reflected in the way time (the argument in $L_1^N$) is scaled -- it is twice $\alpha/2$ as expected by Brownian scaling. The essence of Assumption 2 is that if one removes a global line with slope $p$ from $L_1^N$ and rescales by $N^{\alpha/2}$ vertically and $N^{\alpha}$ horizontally the resulting curve asymptotically approximates the parabola $-\lambda x^2$. The way the statement is formulated, this approximation needs to happen uniformly over the integers. The choice of $\mathbb{Z}$ is not important and one can replace $\mathbb{Z}$ with any subset of $\mathbb{R}$ that has arbitrarily large positive and negative points. The choice of $\mathbb{Z}$ is made for convenience. Equation (\ref{S1globalParabola}) indicates that for each $n \in \mathbb{Z}$ the sequence of random variables $X_n^N = N^{-\alpha/2}(L_1^N( \lfloor n N^{\alpha}\rfloor ) - p n N^{\alpha} + \lambda n^2 N^{\alpha/2}) $ is tight, but it says a bit more. Namely, it states that if $\mathcal{M}_n$ is the family of all possible subsequential limits of $\{X_n^N\}_{N \geq 1}$ then $\cup_{n \in \mathbb{Z}} \mathcal{M}_n$ is itself a tight family of distributions on $\mathbb{R}$.  A simple case when Assumption 2 is satisfied is when $X_n^N$ converges to the same random variable for all $n$ as $N \rightarrow \infty$. In this case the family $\cup_{n \in \mathbb{Z}} \mathcal{M}_n$ contains only a single measure and so is naturally tight. 
\end{remark}

\vspace{2mm}
{\bf \raggedleft Assumption 3.} $H^{RW}$ and $H$ satisfy the conditions in Definitions \ref{AssHR} and \ref{AssH}, respectively.
\begin{remark} The conditions in Definitions \ref{AssHR} and \ref{AssH} are somewhat technical and required to apply two crucial results in our arguments. The first result, see Theorem \ref{KMT}, states that we can strongly couple random walk bridges with jump increment density $G(x)=e^{-H^{RW}(x)}$ to Brownian bridges of appropriate variance. In order to use this result $H^{RW}$ needs to satisfy the conditions of Definitions \ref{AssHR}, which for example require that $H^{RW}$ is convex on $\mathbb{R}$, and the random variable with density $G(x)$ has a finite moment generating function in a real neighborhood of the origin. The second result, see Lemma \ref{MonCoup}, states that we can monotonically couple $(H,H^{RW})$-Gibbsian line ensembles with respect to their boundary data. In order to use this result $H$ needs to satisfy the conditions of Definitions \ref{AssH}, which for example require that $H$ is convex and increasing on $\mathbb{R}$.
\end{remark}

The final thing we need to do is embed all of our line ensembles $\mathfrak{L}^N$ in the same space. The latter is necessary as we want to talk about tightness of these line ensembles that presently are defined on different state spaces (remember that the number of curves and interval of definition $[a_N, b_N]$ are changing with $N$). We consider $\mathbb{N} \times \mathbb{R}$ with the product topology coming from the discrete topology on $\mathbb{N}$ and the usual topology on $\mathbb{R}$. We let $C(\mathbb{N} \times \mathbb{R})$ be the space of continuous functions on $\mathbb{N} \times \mathbb{R}$ with the topology of uniform convergence over compacts and corresponding Borel $\sigma$-algebra. For each $N \in \mathbb{N}$ and $\sigma_p$ as in Definition \ref{AssHR}, we let 
$$f^N_i(s) =  \sigma_p^{-1} N^{-\alpha/2}(L^N_i(sN^{\alpha}) - p s N^{\alpha}), \mbox{ for $s\in [-\psi(N) ,\psi(N)]$ and $i = 1,\dots, N$,}$$
and extend $f^N_i$ to $\mathbb{R}$ by setting for $i = 1, \dots, N$
$$f^N_i(s) = f^N_i(-\psi(N)) \mbox{ for $s \leq -\psi(N)$ and } f^N_i(s) = f_i^N(\psi(N)) \mbox{ for $s \geq \psi(N)$}.$$
If $i \geq N+1$ we define $f^N_i(s) = 0$ for $s \in \mathbb{R}$. With the above we have that $\mathcal{L}^N$, defined by 
\begin{equation}\label{DefLineEns}
\mathcal{L}^N(i,s) = f^N_i(s),
\end{equation}
 is a random variable taking value in $C(\mathbb{N} \times \mathbb{R})$ and we let $\mathbb{P}_N$ denote its distribution. We remark that the particular extension we chose for $f^N_i$ outside of $[-\psi(N) ,\psi(N)]$ and for $i \geq N+1$ is immaterial since all of our tightness results are formulated for the topology of uniform convergence over compacts. Consequently, only the behavior of these functions on compact intervals and for finite indices matters and not what these functions do near infinity, which is where the modification happens as $\lim_{N \rightarrow \infty} \psi(N) = \infty$ by assumption. 

We are now ready to state our main result, whose proof can be found in Section \ref{Section3.1}.
\begin{theorem}\label{Thm1} Under Assumptions 1, 2 and 3 the sequence $\mathbb{P}_N$ is tight. Moreover, if $\mathcal{L}^\infty$ denotes any subsequential limit of $\mathcal{L}^N$ then $\mathcal{L}^\infty$ satisfies the Brownian Gibbs property of Section \ref{Section1.1} (see also Definition \ref{DefPBGP}).
\end{theorem}
\begin{remark} In simple words, Theorem \ref{Thm1} states that if one has a sequence of $(H,H^{RW})$-Gibbsian line ensembles with a mild but uniform control of the one-point marginals of the top curves $L_1^N$ then the sequence of line ensembles needs to be tight. The idea of utilizing the Gibbs property of a line ensemble to improve one-point tightness of the top curves to tightness of the entire line ensembles has appeared previously in several different contexts. For line ensembles whose underlying path structure is Brownian it first appeared in the seminal work of \cite{CorHamA} and more recently in \cite{CIW19b,CIW19a}. For discrete Gibbsian line ensembles partial results of this kind appear in \cite{CD} and \cite{BCDB}.
\end{remark}
\begin{remark}
Under assumptions that are analogous to Assumptions 1,2 and 3 above, \cite[Theorem 1.2]{BCDB} established tightness of the top curves $\mathcal{L}^N_1$ of the line ensembles $\mathcal{L}^N$ in (\ref{DefLineEns}). Theorem \ref{Thm1} generalizes \cite[Theorem 1.2]{BCDB} in that tightness is proved not just for the top curves $\mathcal{L}^N_1$, but for the full line ensembles $\mathcal{L}^N$. The essential new input that is required to improve the tightness statement is the uniform global parabolic behavior of the curves $\mathcal{L}^N_1$ in Assumption 2. The place this assumption is used in the proof is to show that it is impossible for the curves $\mathcal{L}^N_k$ of index $k \geq 2$ to dip to negative infinity along subsequences with positive probability. We believe that one can establish the tightness of $\mathcal{L}_1^N$ under the weaker (than Assumption 2) assumption that the sequence of random variables $N^{-\alpha/2}(L_1^N(\lfloor n N^{\alpha}\rfloor) - p n N^{\alpha} )$ is tight for each $n \in \mathbb{Z}$. However, we are not aware of a way to prove the tightness of the full ensembles $\mathcal{L}^N$ without assuming some uniform global parabolic shape for $\mathcal{L}_1^N$, such as the one assured by (\ref{S1globalParabola}).
\end{remark}

\begin{remark}
An important technical innovation in our current work is a new approach to proving stochastic monotonicity (see Lemma \ref{MonCoup}), which extends \cite[Lemma 2.10]{BCDB} from line ensembles with a single curve to ones with an arbitrary finite number of curves. This extension was predicted in \cite{BCDB}, and we refer the interested reader to \cite[Section 1.1.1]{BCDB} for an at length discussion of the various properties and advantages of this coupling result and a comparison to other similar results in the literature. 

Here we briefly mention that our construction, which is the content of Lemma \ref{MonCoup}, provides a continuous grand monotone coupling of Gibbsian line ensembles with respect to their boundary data (entrance and exit values, and bounding curves). {\em Continuous} means that the Gibbsian line ensemble measure varies continuously as one varies the boundary data, {\em grand} means that all uncountably many measures (one for each boundary data) are coupled to the same probability space, and {\em monotone} means that raising the values of the boundary data likewise raises the associated measure. This result applies to a general class of Gibbsian line ensembles where the underlying random walk measure is discrete time, continuous valued and log-convex, and the interaction Hamiltonian is nearest neighbor and convex. 
\end{remark}

 As mentioned, the proof of Theorem \ref{Thm1} is presented in Section \ref{Section3.1}, and follows from the more general (and technical) Theorem \ref{PropTightGood}. Theorem \ref{PropTightGood} is the main technical result of the paper and its proof is based on three key lemmas -- Lemma \ref{keyL1}, Lemma \ref{keyL2} and Lemma \ref{keyL3}. These three lemmas form the backbone of our argument and their proof is the content of Section \ref{Section4}. We mention that the proof of these lemmas is based on an involved inductive argument where various estimates are obtained for the lowest-indexed curve of a sequence of line ensembles, which are then used to deduce similar estimates for the second curve and so on. Compared to previous results in the literature our inductive argument most closely resembles the proof of \cite[Propositions 6.1, 6.2 and 6.3]{CorHamK}; however, there are several unique challenges in our setting that we need to overcome. 

Firstly, in \cite{CorHamK} the authors consider a positive temperature limit of a sequence of line ensembles, while in our case we are dealing with a zero temperature limit. This difference requires that we not only show that along subsequences our line ensembles do not go to high or too low, but also asymptotically become ordered {\em and} well-separated. An analogous challenge was overcome by the second author in \cite{Wu21}, where the zero temperature limit of the KPZ line ensemble was investigated. The disadvantage we face in our work is that we are not working with a nice object like the KPZ line ensemble, for which various estimates and properties (like stationarity) are readily available. Instead, we are dealing with {\em generic} sequences of line ensembles, with almost nothing but one-point tightness for their lowest-indexed curve to work with. The reason we assume so little is partially due to the fact that this is what is essentially known for the log-gamma line ensemble -- our integrable model of interest (defined in the next section). Given our limited input, we have to develop our inductive argument in a delicate way so as to interweave between estimates on the maxima and minima of the various curves, estimates that show that the curves also become ordered and separated from each other with high probability. This interweaving sits behind the complex structure of the inductive argument and the technical nature of Section \ref{Section4}.

An additional challenge we face, compared to \cite{CorHamK}, is that we are dealing with discrete (rather than continuous) line ensembles. The latter difference demands that we derive various estimates on random walk bridges, which are generally harder objects to work with than Brownian bridges. The way we obtain the estimates we require in this paper is by using a strong (KMT-type) coupling result, Proposition \ref{KMT}, which was established by the authors of this paper in \cite{DW19}. This strong coupling result works in tandem with our monotone coupling lemma, Lemma \ref{MonCoup}, to essentially translate any estimate we require for a collection of {\em interacting} random walk bridges, to one involving {\em independent} Brownian bridges. Most of the estimates we derive are presented in Section \ref{Section7}, and while the arguments in Section \ref{Section7.1} and \ref{Section7.2} have analogues in the literature, we mention that the arguments in Section \ref{Section7.3} are new and more closely tuned to our setting. Section \ref{Section7.3} is devoted to the proof of  Lemma \ref{LAccProb}, which roughly states that with high probability finitely many interacting random walk bridges arrange themselves in configurations that allow for a strong comparison to independent Brownian bridges. This statement is crucial for the proof of Lemma \ref{keyL3} (one of the three key lemmas mentioned above) and its proof, which occupies Section \ref{Section7.3}, is based on a delicate bootstrapping argument that involves various estimates established throughout Section \ref{Section7}.

%
\subsubsection{Application to the log-gamma polymer}\label{Section1.2.2} We begin by introducing the log-gamma polymer model, following \cite[Section 6]{BCDB}.

 Recall that a continuous random variable $X$ is said have the inverse-gamma distribution with parameter $\theta > 0 $ if its density is given by
\begin{equation}\label{invGammaDens}
f_\theta(x) = \frac{{\bf 1} \{ x > 0 \} }{\Gamma(\theta)} \cdot x^{-\theta - 1} \cdot \exp( - x^{-1}).
\end{equation}

Let us fix $N \in \mathbb{N}$ and $\theta > 0$. We let $d = \left( d_{i,j} : i \geq 1, 1\leq j \leq N \right)$ denote the semi-infinite random matrix such that $d_{i,j}$ are i.i.d. random variables with density $f_\theta$ as in (\ref{invGammaDens}). In addition, for $n \geq 1$ we denote by $d^{[1,n]}$ the $n \times N$ matrix $\left( d_{i,j} : 1\leq i \leq n, 1\leq j \leq N \right)$. A {\em directed lattice path} is a sequence of vertices $(x_1, y_1), \dots, (x_k, y_k) \in \mathbb{Z}^2$ such that $x_1 \leq x_2 \leq \cdots \leq x_k$, $y_1 \leq y_2 \leq \cdots \leq y_k$ and $\big(x_i - x_{i-1}\big) + \big(y_i - y_{i-1}\big) = 1$ for $i = 2, \dots, k$. In words, a directed lattice path is an up-right path on $\mathbb{Z}^2$, which makes unit steps in the coordinate directions. A collection of paths $\pi = (\pi_1, \dots, \pi_{\ell})$ is said to be {\em  non-intersecting} if the paths $\pi_1, \dots, \pi_\ell$ are pairwise vertex-disjoint. For $1 \leq \ell \leq k \leq N$ we let $\Pi_{n,k}^\ell$ denote the set of $\ell$-tuples $\pi = (\pi_1, \dots, \pi_{\ell})$ of non-intersecting directed lattice paths in $\mathbb{Z}^2$ such that for $1 \leq r \leq \ell$, $\pi_r$ is a lattice path from $(1,r)$ to $(n, k+ r - \ell)$.

Given an $\ell$-tuple $\pi = (\pi_1, \dots, \pi_{\ell})$ we define its {\em weight} to be
\begin{equation}\label{PathWeight}
w(\pi) = \prod_{r = 1}^\ell \prod_{(i,j) \in \pi_r} d_{i,j}.
\end{equation}
For $1 \leq \ell \leq k \leq N$ we define
\begin{equation}\label{PartitionFunct}
\tau_{k, \ell}(n) = \sum_{ \pi \in \Pi_{n,k}^\ell} w(\pi).
\end{equation}
Note that if $0 \leq n < \ell \leq k \leq N$ then $ \Pi_{n,k}^\ell = \varnothing$ and so, as by convention, we set $\tau_{k, \ell}(n)  = 0$. If $\ell = k$ then $\Pi_{n,k}^\ell $ consists of a unique element, and in fact we have
$$\tau_{k,\ell}(n) = \delta_{k,\ell} \cdot \tau_{k, n} (n) \mbox{ for } 0 \leq n < \ell \leq k \leq N,$$
where $\delta_{k,\ell}$ is the Kronecker delta.

Given $\tau_{k,\ell}(n)$ we define the array $z(n) = \{z_{k, \ell}(n): 1 \leq k \leq N \mbox{ and } 1 \leq \ell \leq \min( k, n) \}$ through
\begin{equation}\label{ZfromTau}
z_{k,1}(n) z_{k,2}(n) \cdots z_{k, \ell}(n) = \tau_{k, \ell }(n).
\end{equation}

We next explain how one can construct an $(H,H^{RW})$-Gibbsian line ensemble from $z_{N,i}$, with 
\begin{equation}\label{S1Hamiltonian}
\g_\theta(x) = e^{-H^{RW}(x)}, \hspace{3mm}H^{RW}(x) = \theta x + e^{-x} + \log \Gamma (\theta) \mbox{ and } H(x) = e^x. 
\end{equation}
The precise statement is contained in the following proposition.
\begin{proposition}\label{PropLOG} \cite[Proposition 6.4]{BCDB}
Let $H,H^{RW}$ be as in (\ref{S1Hamiltonian}). Fix $\topc, N \in \mathbb{N}$ with $N \geq \topc \geq 2$. Let $T_0, T_1 \in \mathbb{N}$ be such that $T_0 < T_1$ and $T_0 \geq \topc$. Then we can construct a probability space $\mathbb{P}$ that supports a $\llbracket 1, \topc \rrbracket \times \llbracket T_0, T_1 \rrbracket$-indexed line ensemble $\mathfrak{L} = (L_{1}, \dots, L_\topc)$ such that:
\begin{enumerate}
\item the $\mathbb{P}$-distribution of $(L_i(j): (i,j) \in \llbracket 1, \topc \rrbracket \times \llbracket T_0, T_1 \rrbracket )$ is the same as that of $ ( \log(z_{N,i}(j)): (i,j) \in \llbracket 1, \topc \rrbracket \times \llbracket T_0, T_1 \rrbracket )$ as in (\ref{ZfromTau});
\item $\mathfrak{L}$ satisfies the partial $(H,H^{RW})$-Gibbs property of Section \ref{Section1.2.1} (see also Definition \ref{DefLGGP}).
\end{enumerate}
\end{proposition}

In view of Proposition \ref{PropLOG} we know that  $ ( \log(z_{N,i}(j)): (i,j) \in \llbracket 1, \topc \rrbracket \times \llbracket T_0, T_1 \rrbracket )$ form a sequence (in $N$) of $(H,H^{RW})$-Gibbsian line ensembles and we seek to apply our black-box result, Theorem \ref{Thm1}, to deduce an appropriate tightness statement for it. In order to formulate a precise statement we need some additional notation.

\begin{definition}
Let $\Psi(x)$ denote the digamma function, defined by
\begin{equation}\label{digammaS1}
\Psi(z) = \frac{\Gamma'(z)}{\Gamma(z)} = - \gamma_{E} + \sum_{n = 0}^\infty \left[\frac{1}{n + 1} - \frac{1}{n+z} \right],
\end{equation}
with $\gamma_{E}$ denoting the Euler constant.
Define the function
\begin{equation}\label{DefLittleg}
g_\theta(z) = \frac{\sum_{n =0}^\infty \frac{1}{(n+\theta - z)^2}}{ \sum_{n = 0}^\infty \frac{1}{(n+z)^2}} = \frac{\Psi'(\theta -z)}{\Psi'(z)},
\end{equation}
and observe that it is a smooth, strictly increasing bijection from $(0, \theta)$ to $(0, \infty)$. The inverse function $g_\theta^{-1}: (0, \infty) \rightarrow (0,\theta)$ is also a strictly increasing smooth bijection. For $x \in (0,\infty)$, define the function
\begin{equation}\label{HDefLLN}
h_\theta(x) = x \cdot  \Psi(g_\theta^{-1}(x)) + \Psi( \theta - g_\theta^{-1}(x)),
\end{equation}
which is easily seen to be a smooth function on $(0, \infty)$.
\end{definition}

For each $N \in \mathbb{N}$, $i \in \llbracket 1, N \rrbracket$ and $j \in \llbracket -N, N \rrbracket$ we define 
\begin{equation}\label{LinesLG}
L_i^N(j) = \log(z_{2N,i}(2N + j)) +  2N h_{\theta}(1),
\end{equation}
and set $\mathfrak{L}^N = (L_1^N, \dots, L_N^N)$. 
\begin{remark}
We mention that we chose the subscript of $z$ in (\ref{LinesLG}) to be $2N$ rather than $N$ to ensure that $L_i^N(j)$ are well-defined for all $j \in \llbracket -N, N \rrbracket$ since from (\ref{ZfromTau}) the quantity $z_{k,\ell}(n)$ is only well-defined if $1 \leq \ell \leq \min(k, n)$. 
\end{remark}

We also define the $C(\mathbb{N}\times \mathbb{R})$-valued random variables $\mathcal{L}^N$ to be as in (\ref{DefLineEns}) for the line ensemble $\mathfrak{L}^N$ we just defined. Here, we take $\alpha = 2/3$, $p = -h'_{\theta}(1)$, $\psi(N) = (1/2)N^{1/3}$ and $\sigma_p =  \sqrt{\Psi'(\theta/2)}$. We denote by $\mathbb{P}^{LG}_N$ the distribution of $\mathcal{L}^N$.

With the above notation, we state the main result we prove about the log-gamma polymer.
\begin{theorem}\label{ThmLG}The sequence $\mathbb{P}^{LG}_N$ is tight. Moreover, if $\mathcal{L}^\infty$ denotes any subsequential limit of $\mathcal{L}^N$ then $\mathcal{L}^\infty$ satisfies the Brownian Gibbs property of Section \ref{Section1.1} (see also Definition \ref{DefPBGP}).
\end{theorem}
\begin{remark} The formulation of Theorem \ref{ThmLG} is for a line ensemble built from the values of $z_{k,\ell}(n)$ around the point $(k, n) = (2N, 2N)$. In particular, this point lies on the line of slope $1$, which is why the argument of $h_{\theta}$ in (\ref{LinesLG}) is equal to $1$. One could formulate an analogous statement for line ensembles built from the values of $z_{k,\ell}(n)$ around the point $(k, n) = (\alpha N,  \beta N)$ for any $\alpha, \beta > 0$, but we chose $\alpha = \beta = 2$ for simplicity.
\end{remark}
\begin{remark} The limit taken in Theorem \ref{ThmLG} is sometimes referred to as a {\em strong noise limit} of the polymer. There is also a {\em weak noise limit}, which involves simultaneously scaling the parameter $\theta$ of the log-gamma polymer together with $N$. Such a limit transition was studied by the second author in \cite{Wu19} where it was again shown that under the weak noise scaling the line ensembles built from the log-gamma polymer form a tight sequence; however, the subsequential limits satisfy the {\em $H$-Brownian Gibbs property} of \cite{CorHamK}, rather than the Brownian Gibbs property. We refer the interested reader to \cite{Wu19} for the details.
\end{remark}

\subsection*{Outline}
Section \ref{Section2} contains a number of foundational definitions and results about continuous and discrete Gibbsian line ensembles. Section \ref{Section2.1} sets up the notation for continuous line ensembles, and the Brownian Gibbs property, while Section \ref{Section2.2} defines the $(H,H^{RW})$-Gibbsian line ensembles we work with in this paper. Section \ref{Section2.3} contains the statement of our continuous grand monotone coupling result (Lemma \ref{MonCoup}), whose proof is deferred to Section \ref{Section6}. Finally, Section \ref{Section2.4} contains additional technical assumptions (see Definition \ref{AssHR}) that we make on the random walk Hamiltonian $H^{RW}$ to be able to strongly couple it to a Brownian bridge.

Section \ref{Section3} contains the main technical result we prove about sequences of $(H,H^{RW})$-Gibbsian line ensembles -- this is Theorem \ref{PropTightGood}. In Sections \ref{Section3.1} and \ref{Section3.2} we use Theorem \ref{PropTightGood} to prove Theorems \ref{Thm1} and \ref{ThmLG} from Section \ref{Section1.2} above. In Sections \ref{Section3.3} and \ref{Section3.4} we prove Theorem \ref{PropTightGood} by utilizing three key lemmas --  Lemmas \ref{keyL1} , \ref{keyL2}  and \ref{keyL3}. These lemmas are proved in Section \ref{Section4}. Section \ref{Section7} contains the proofs of various technical results we require throughout the paper.

\subsection*{Acknowledgments} The authors would like to thank Ivan Corwin for many useful comments on earlier drafts of the paper. ED was partially supported by the Minerva Foundation Fellowship and the NSF grant DMS:2054703.

%
\section{Gibbsian line ensembles}\label{Section2}
In this section we recall the definition of a (continuous) line ensemble and the (partial) Brownian Gibbs property. We also introduce the notion of a discrete $(H,H^{RW})$-Gibbsian line ensemble and establish some of its properties.

%
\subsection{Continuous line ensembles}\label{Section2.1}
In this section we introduce the notions of a {\em line ensemble} and the {\em (partial) Brownian Gibbs property}. Our exposition in this section closely follows those of \cite[Section 2]{DimMat} and  \cite[Section 2.1]{DREU}, which in turn go back to \cite[Section 2]{CorHamA}. \\

Given two integers $p \leq q$, we let $\llbracket p, q \rrbracket$ denote the set $\{p, p+1, \dots, q\}$. If $p > q$ then $\llbracket p,q\rrbracket$ is the empty set. Given an interval $\Lambda \subset \mathbb{R}$ we endow it with the subspace topology of the usual topology on $\mathbb{R}$. We let $(C(\Lambda), \mathcal{C})$ denote the space of continuous functions $f: \Lambda \rightarrow \mathbb{R}$ with the topology of uniform convergence over compacts, see \cite[Chapter 7, Section 46]{Munkres}, and Borel $\sigma$-algebra $\mathcal{C}$. Given a set $\Sigma \subset \mathbb{Z}$ we endow it with the discrete topology and denote by $\Sigma \times \Lambda$ the set of all pairs $(i,x)$ with $i \in \Sigma$ and $x \in \Lambda$ with the product topology. We also denote by $\left(C (\Sigma \times \Lambda), \mathcal{C}_{\Sigma}\right)$ the space of continuous functions on $\Sigma \times \Lambda$ with the topology of uniform convergence over compact sets and Borel $\sigma$-algebra $\mathcal{C}_{\Sigma}$. Typically, we will take $\Sigma = \llbracket 1, N \rrbracket$ (we use the convention $\Sigma = \mathbb{N}$ if $N = \infty$) and then we write  $\left(C (\Sigma \times \Lambda), \mathcal{C}_{|\Sigma|}\right)$ in place of $\left(C (\Sigma \times \Lambda), \mathcal{C}_{\Sigma}\right)$.

\begin{definition}\label{CLEDef}
Let $\Sigma \subset \mathbb{Z}$ and $\Lambda \subset \mathbb{R}$ be an interval. A {\em $\Sigma$-indexed line ensemble $\mathcal{L}$ on $\Lambda$} is a random variable defined on a probability space $(\Omega, \mathcal{F}, \mathbb{P})$ that takes values in $\left(C (\Sigma \times \Lambda), \mathcal{C}_{\Sigma}\right)$. Intuitively, $\mathcal{L}$ is a collection of random continuous curves (sometimes referred to as {\em lines}), indexed by $\Sigma$,  each of which maps $\Lambda$ in $\mathbb{R}$. We will often slightly abuse notation and write $\mathcal{L}: \Sigma \times \Lambda \rightarrow \mathbb{R}$, even though it is not $\mathcal{L}$ which is such a function, but $\mathcal{L}(\omega)$ for every $\omega \in \Omega$. For $i \in \Sigma$ we write $\mathcal{L}_i(\omega) = (\mathcal{L}(\omega))(i, \cdot)$ for the curve of index $i$ and note that the latter is a map $\mathcal{L}_i: \Omega \rightarrow C(\Lambda)$, which is $(\mathcal{C}, \mathcal{F})-$measurable. If $a,b \in \Lambda$ satisfy $a < b$ we let $\mathcal{L}_i[a,b]$ denote the restriction of $\mathcal{L}_i$ to $[a,b]$.
\end{definition}

In the paper we will require several basic statements about line ensembles and the space $C (\Sigma \times \Lambda)$, which we summarize in the lemmas below.
\begin{lemma}\label{Polish}\cite[Lemma 2.2]{DREU} There is a metric $d$ on $C (\Sigma \times \Lambda)$ such that $(C (\Sigma \times \Lambda), d)$ is complete and separable. Moreover, the metric topology of this space is the same as that of the topology of uniform convergence over compacts.
\end{lemma}

\begin{definition}
Given a sequence $\{ \mathcal{L}^n: n \in \mathbb{N} \}$ of random $\Sigma$-indexed line ensembles we say that $\mathcal{L}^n$ {\em converge weakly} to a line ensemble $\mathcal{L}$, and write $\mathcal{L}^n \implies \mathcal{L}$ if for any bounded continuous function $f: C (\Sigma \times \Lambda) \rightarrow \mathbb{R}$ we have that 
$$\lim_{n \rightarrow \infty} \mathbb{E} \left[ f(\mathcal{L}^n) \right] = \mathbb{E} \left[ f(\mathcal{L}) \right].$$

We also say that $\{ \mathcal{L}^n: n \in \mathbb{N} \}$ is {\em tight} if for any $\epsilon > 0$ there exists a compact set $K \subset C (\Sigma \times \Lambda)$ such that $\mathbb{P}(\mathcal{L}^n \in K) \geq 1- \epsilon$ for all $n \in \mathbb{N}$.

We call a line ensemble {\em non-intersecting} if $\mathbb{P}$-almost surely $\mathcal{L}_i(r) > \mathcal{L}_j(r)$  for all $i < j$ and $r \in \Lambda$.
\end{definition}

We will require in the paper the following necessary and sufficient condition for tightness of a sequence of line ensembles.
\begin{lemma}\label{2Tight}\cite[Lemma 2.4]{DREU}
Let $\Sigma \subset \mathbb{Z}$ and $\Lambda\subset\mathbb{R}$ be an interval. Suppose that $\{a_n\}_{n = 1}^\infty, \{b_n\}_{n = 1}^\infty$ are sequences of real numbers such that $a_n < b_n$, $[a_n, b_n] \subset \Lambda$, $a_{n+1} \leq a_n$, $b_{n+1} \geq b_n$ and $\cup_{n = 1}^\infty [a_n, b_n] = \Lambda$. Then $\{\mathcal{L}^n\}$ is tight if and only if for some $a_0 \in \Lambda$ and every $i\in\Sigma$ we have
\begin{enumerate}[label=(\roman*)]
\item $\lim_{a\to\infty} \limsup_{n\to\infty}\, \mathbb{P}(|\mathcal{L}^n_i(a_0)|\geq a) = 0$ 
\item For all $\epsilon>0$ and $k \in \mathbb{N}$,  $\lim_{\delta\to 0} \limsup_{n\to\infty}\, \mathbb{P}\bigg(\sup_{\substack{x,y\in [a_k,b_k], \\ |x-y|\leq\delta}} |\mathcal{L}^n_i(x) - \mathcal{L}^n_i(y)| \geq \epsilon\bigg) = 0.$
\end{enumerate}
\end{lemma}

We next turn to formulating the (partial) Brownian Gibbs property -- we do this in Definition \ref{DefPBGP} after introducing some relevant notation and results. If $W_t$ denotes a standard one-dimensional Brownian motion, then the process
$$\tilde{B}(t) =  W_t - t W_1, \hspace{5mm} 0 \leq t \leq 1,$$
is called a {\em Brownian bridge (from $\tilde{B}(0) = 0$ to $\tilde{B}(1) = 0 $) with diffusion parameter $1$.} For brevity we call the latter object a {\em standard Brownian bridge}.

Given $a, b,x,y \in \mathbb{R}$ with $a < b$ we define a random variable on $(C([a,b]), \mathcal{C})$ through
\begin{equation}\label{BBDef}
B(t) = (b-a)^{1/2} \cdot \tilde{B} \left( \frac{t - a}{b-a} \right) + \left(\frac{b-t}{b-a} \right) \cdot x + \left( \frac{t- a}{b-a}\right) \cdot y, 
\end{equation}
and refer to the law of this random variable as a {\em Brownian bridge (from $B(a) = x$ to $B(b) = y$) with diffusion parameter $1$.} Given $k \in \mathbb{N}$ and $\vec{x}, \vec{y} \in \mathbb{R}^k$ we let $\mathbb{P}^{a,b, \vec{x},\vec{y}}_{free}$ denote the law of $k$ independent Brownian bridges $\{B_i: [a,b] \rightarrow \mathbb{R} \}_{i = 1}^k$ from $B_i(a) = x_i$ to $B_i(b) = y_i$ all with diffusion parameter $1$.

The following definition introduces the notion of an $(f,g)$-avoiding Brownian line ensemble, which in simple terms is a collection of $k$ independent Brownian bridges, conditioned on not-crossing each other and staying above the graph of $g$ and below the graph of $f$, where $f,g$ are continuous.
\begin{definition}\label{DefAvoidingLaw}
Let $k \in \mathbb{N}$ and $\weyl_k$ denote the open Weyl chamber in $\mathbb{R}^{k}$, i.e.
$$\weyl_k = \{ \vec{x} = (x_1, \dots, x_k) \in \mathbb{R}^k: x_1 > x_2 > \cdots > x_k \}.$$
(In \cite{CorHamA} the notation $\mathbb{R}_{>}^k$ was used for this set.)
Let $\vec{x}, \vec{y} \in \weyl_k$, $a,b \in \mathbb{R}$ with $a < b$, and $f: [a,b] \rightarrow (-\infty, \infty]$ and $g: [a,b] \rightarrow [-\infty, \infty)$ be two continuous functions. The latter condition means that either $f: [a,b] \rightarrow \mathbb{R}$ is continuous or $f = \infty$ everywhere, and similarly for $g$. We also assume that $f(t) > g(t)$ for all $t \in[a,b]$, $f(a) > x_1, f(b) > y_1$ and $g(a) < x_k, g(b) < y_k.$

With the above data we define the {\em $(f,g)$-avoiding Brownian line ensemble on the interval $[a,b]$ with entrance data $\vec{x}$ and exit data $\vec{y}$} to be the $\Sigma$-indexed line ensemble $\mathcal{Q}$ with $\Sigma = \llbracket 1, k\rrbracket$ on $\Lambda = [a,b]$ and with the law of $\mathcal{Q}$ equal to $\mathbb{P}^{a,b, \vec{x},\vec{y}}_{free}$ (the law of $k$ independent Brownian bridges $\{B_i: [a,b] \rightarrow \mathbb{R} \}_{i = 1}^k$ from $B_i(a) = x_i$ to $B_i(b) = y_i$) conditioned on the event 
$$E  = \left\{ f(r) > B_1(r) > B_2(r) > \cdots > B_k(r) > g(r) \mbox{ for all $r \in[a,b]$} \right\}.$$ 
It is worth pointing out that $E$ is an open set of positive measure and so we can condition on it -- see \cite[Definition 2.7]{DREU}. In particular, given measurable subsets $A_1, \dots, A_k$ of $C([a,b])$ we have that 
$$\mathbb{P}(\mathcal{Q}_i \in A_i \mbox{ for $i = 1, \dots, k$} ) = \frac{\mathbb{P}^{a,b, \vec{x},\vec{y}}_{free} \left( \{ B_i \in A_i \mbox{ for $i = 1, \dots, k$}\} \cap E \right) }{\mathbb{P}^{a,b, \vec{x},\vec{y}}_{free}(E)}.$$
We denote the probability distribution of $\mathcal{Q}$ as $\mathbb{P}_{avoid}^{a,b, \vec{x}, \vec{y}, f, g}$ and write $\mathbb{E}_{avoid}^{a,b, \vec{x}, \vec{y}, f, g}$ for the expectation with respect to this measure. 
\end{definition}

The following definition introduces the partial Brownian Gibbs property from \cite{DimMat}.
\begin{definition}\label{DefPBGP}
Fix a set $\Sigma = \llbracket 1 , N \rrbracket$ with $N \in \mathbb{N}$ or $N  = \infty$ and an interval $\Lambda \subset \mathbb{R}$.  A $\Sigma$-indexed line ensemble $\mathcal{L}$ on $\Lambda$ is said to satisfy the {\em partial Brownian Gibbs property} if and only if it is non-intersecting and for any finite $K = \{k_1, k_1 + 1, \dots, k_2 \} \subset \Sigma$ with $k_2 \leq N - 1$ (if $\Sigma \neq \mathbb{N}$), $[a,b] \subset \Lambda$ and any bounded Borel-measurable function $F: C(K \times [a,b]) \rightarrow \mathbb{R}$ we have $\mathbb{P}$-almost surely
\begin{equation}\label{PBGPTower}
\mathbb{E} \left[ F(\mathcal{L}|_{K \times [a,b]}) {\big \vert} \mathcal{F}_{ext} (K \times (a,b))  \right] =\mathbb{E}_{avoid}^{a,b, \vec{x}, \vec{y}, f, g} \bigl[ F(\tilde{\mathcal{Q}}) \bigr],
\end{equation}
where we recall that $D_{K,a,b} = K \times (a,b)$ and $D_{K,a,b}^c = (\Sigma \times \Lambda) \setminus D_{K,a,b}$, and
$$\mathcal{F}_{ext} (K \times (a,b)) = \sigma \left \{ \mathcal{L}_i(s): (i,s) \in D_{K,a,b}^c \right\}$$
is the $\sigma$-algebra generated by the variables in the brackets above, $ \mathcal{L}|_{K \times [a,b]}$ denotes the restriction of $\mathcal{L}$ to the set $K \times [a,b]$, $\vec{x} = (\mathcal{L}_{k_1}(a), \dots, \mathcal{L}_{k_2}(a))$, $\vec{y} = (\mathcal{L}_{k_1}(b), \dots, \mathcal{L}_{k_2}(b))$, $f = \mathcal{L}_{k_1 - 1}[a,b]$ with the convention that $f = \infty$ if $k_1 - 1 \not \in \Sigma$, and $g = \mathcal{L}_{k_2 +1}[a,b]$.

If $N = \infty$ or if $N < \infty$ and (\ref{PBGPTower}) also holds for $k_2 = N$ (and then $g = -\infty$) we drop the term ``partial'' and say that $\mathbb{P}$ satisfies the {\em Brownian Gibbs property} (this term was coined in \cite{CorHamA}).
\end{definition}
\begin{remark}\label{RemMeas} We mention that by \cite[Lemma 3.4]{DimMat}, we have that the right side of (\ref{PBGPTower}) is measurable with respect to the $\sigma$-algebra 
$$ \sigma \left\{ \mathcal{L}_i(s) : \mbox{  $i \in K$ and $s \in \{a,b\}$, or $i \in \{k_1 - 1, k_2 +1 \}$ and $s \in [a,b]$} \right\}.$$
This implies the measurability with respect to $\mathcal{F}_{ext} (K \times (a,b))$, so that (\ref{PBGPTower}) makes sense.
\end{remark}

%
\subsection{Discrete $(H,H^{RW})$-Gibbsian line ensembles}\label{Section2.2}
In this section we introduce the notion of a discrete line ensemble and the $(H,H^{RW})$-Gibbs property. Our discussion will parallel that of \cite[Section 2]{Wu19} and \cite[Section 2]{BCDB}, which in turn goes back to \cite[Section 2.1]{CorHamK}.

\begin{definition}\label{YVec} For a finite set $J \subset \mathbb{Z}^2$ we let $Y(J)$ denote the space of functions $f: J \rightarrow \mathbb{R}$ with the Borel $\sigma$-algebra $\mathcal{D}$ coming from the natural identification of ${Y}(J)$ with $\mathbb{R}^{|J|}$. We think of an element of $Y(J)$ as a $|J|$-dimensional vector whose coordinates are indexed by $J$. In particular, if $f(j) = x_j \in \mathbb{R}$ for $j \in J$ we will denote this vector by $(x_j: j \in J)$. 

Analogously, we let $Y^+(J)$ denote the space of functions $f: J \rightarrow (-\infty, \infty]$ with the Borel $\sigma$-algebra $\mathcal{D}^+$ coming from the natural identification of ${Y}(J)$ with $(-\infty, \infty]^{|J|}$. Similarly, we let $Y^-(J)$ denote the space of functions $f: J \rightarrow [-\infty, \infty)$ with the Borel $\sigma$-algebra $\mathcal{D}^-$ coming from the natural identification of ${Y}^-(J)$ with $[-\infty, \infty)^{|J|}$. We think of an element of $Y^{\pm}(J)$ as a $|J|$-dimensional vector whose coordinates are indexed by $J$. 

When $J = \emptyset$ we let $Y(J) = Y^-(J) = Y^+(J)$ denote the set with a single element $\{\omega_0\}$ with the discrete topology and $\sigma$-algebra. 
\end{definition}

\begin{definition}\label{DefDLE}
Let $k_1, k_2, T_0, T_1 \in \mathbb{Z}$ with $k_1 \leq k_2$, $T_0 < T_1$  and denote $\Sigma = \llbracket k_1, k_2 \rrbracket$. A $\Sigma$-{\em indexed discrete line ensemble $\mathfrak{L}$ on $\llbracket T_0, T_1 \rrbracket$ }  is a random variable defined on a probability space $(\Omega, \mathcal{B}, \mathbb{P})$, taking values in $Y(\Sigma \times \llbracket T_0, T_1 \rrbracket)$ as in Definition \ref{YVec} such that $\mathfrak{L}$ is a $(\mathcal{B}, \mathcal{D})$-measurable function.
\end{definition}

The way we think of a $\Sigma$-indexed discrete line ensemble $\mathfrak{L}$ is as a random $(k_2 - k_1 +1) \times (T_1 - T_0 + 1)$ matrix, whose rows are indexed by $\Sigma$ and whose columns are indexed by $\llbracket T_0, T_1 \rrbracket$. For $i \in \Sigma$ we let $L_i(\omega)$ denote the $i$-th row of this random matrix, and then $L_i$ is a $Y(\llbracket T_0, T_1 \rrbracket)$-valued random variable on $(\Omega, \mathcal{B}, \mathbb{P})$. Conversely, if we are given $k_2 - k_1 +1$ random $Y(\llbracket T_0, T_1 \rrbracket)$-valued random variables $L_{k_1}, \dots, L_{k_2}$ defined on the same probability space, then we can define a $\Sigma$-indexed discrete line ensemble $\mathfrak{L}$ on $\llbracket T_0, T_1 \rrbracket$ through $\mathfrak{L}(\omega)(i,j) = L_i(\omega)(j)$. Consequently, a $\Sigma$-indexed discrete line ensemble $\mathfrak{L}$ on $\llbracket T_0, T_1\rrbracket$ is equivalent to having $k_2 - k_1 +1$ random $Y(\llbracket T_0, T_1 \rrbracket)$-valued random variables $L_{k_1}, \dots, L_{k_2}$ on the same probability space and depending on the context we will switch between these two formulations. For $i \in \Sigma$ and $j \in \llbracket T_0, T_1 \rrbracket$  we denote by $L_i(j): \Omega \rightarrow \mathbb{R}$ the function $L_i(j)(\omega) = L_i(\omega)(j)$ and observe that the latter are real random variables on $(\Omega, \mathcal{B}, \mathbb{P})$. If $A \subset \Sigma \times \llbracket T_0, T_1 \rrbracket$ we write $\mathfrak{L}\vert_{A} : \Omega \rightarrow Y(A)$ to denote the function $\mathfrak{L} \vert_A(\omega)(a) = \mathfrak{L}(\omega)(a)$ for $a \in A$. If $\llbracket a, b \rrbracket \subset \llbracket T_0, T_1 \rrbracket$ and $i \in \Sigma$ we denote the random vector $(L_i(a), \dots, L_i(b)) \in Y(\llbracket a, b\rrbracket)$ by $L_i\llbracket a, b \rrbracket$.

Observe that one can view an indexed set of real numbers $L(j)$ for $j \in \llbracket T_0, T_1 \rrbracket $ as a continuous curve by linearly interpolating the points $(j, L(j))$ -- see Figure \ref{LineEnsembleFig} for an illustration of such an interpolation for a discrete line ensemble. This allows us to define $ (\mathfrak{L}(\omega)) (i, s)$ for non-integer $s \in [T_0,T_1]$ by linear interpolation and to view discrete line ensembles as line ensembles in the sense of Definition \ref{CLEDef}. Specifically, by linear interpolation we can extend $L_i(\omega)$ to a continuous curve on $[T_0, T_1]$ and in this way we can view it as a random variable on $(\Omega, \mathcal{B}, \mathbb{P})$ taking values in $(C[T_0,T_1], \mathcal{C})$  -- the space of continuous functions on $[T_0,T_1]$ with the uniform topology and Borel $\sigma$-algebra $\mathcal{C}$ (see e.g. Chapter 7 in \cite{Bill}). We will denote this random continuous curve by $L_i[T_0,T_1]$. We will often slightly abuse notation and suppress the $\omega$ from the above notation as one does for usual random variables, writing for example  $\{L_i(j) \in A\}$ in place of either $\{\omega \in \Omega: L_i(j)(\omega) \in A\}$ or $\{\omega \in \Omega: L_i(\omega)(j) \in A\}$ (notice that these sets are the same and in general the definitions are consistent so that the suppression of $\omega$ does not lead to any ambiguity).

\begin{definition}\label{DefBridge} Let $H^{RW}: \mathbb{R} \rightarrow \mathbb{R}$ be a continuous function and $\g(x) = e^{-H^{RW}(x)}$. We assume that $\g(x)$ is bounded and $\int_{\mathbb{R}} \g(x) dx = 1$. Let $Y_1, Y_2, \dots$ be i.i.d. random variables with density $\g(\cdot)$ and let $S^x_n =x + Y_1 + \cdots + Y_n$ denote the random walk with jumps $Y_m$ started from $x$. We denote by $\g^x_n(\cdot)$ the density of $S^x_n$ and note that
\begin{equation}\label{RWN}
\g^x_n(y) = \g^0_n(y - x) = \int_\mathbb{R} \cdots \int_{\mathbb{R}} \g(y_1) \cdots \g(y_{n-1}) \cdot \g(y - x - y_1 - \cdots - y_{n-1}) dy_1 \cdots dy_{n-1}.
\end{equation}
Given $x, y \in \mathbb{R}$ and $a, b\in \mathbb{Z}$ with $a < b$ we let $S(x,y; a, b) = \{ S_m(x,y;a,b) \}_{m = a}^b$ denote the process with the law of $\{S_m^x\}_{m = 0}^{b-a}$, conditioned so that $S_{b-a}^x = y$. We call this process an $H^{RW}$ {\em random walk bridge} between the points $(a,x)$ and $(b,y)$. Explicitly, viewing $S(x,y; a, b)$ as a random vector taking values in $Y(\llbracket a,b \rrbracket)$ we have that its distribution is given by the density
\begin{equation}\label{RWB}
\g(y_a, \dots, y_b; x,y; a,b) = \frac{\delta_x(y_a) \cdot \delta_y(y_b) \cdot \prod_{m = a+1}^{b} \g(y_m - y_{m-1}) }{\g_{b-a}^x(y)},
\end{equation}
 where we recall that $\delta_z$ is the Dirac delta measure at $z$. As before we can also view $S(x,y; a, b)$ as a random continuous curve between the points $(a,x)$ and $(b,y)$ once we linearly interpolate the points $(m, S_m(x,y; a, b))$ for $m \in \llbracket a, b \rrbracket$.
\end{definition}

\begin{definition}\label{Pfree}
Let $H^{RW}$ be as in Definition \ref{DefBridge}. Fix $k_1 \leq k_2$, $a < b$ with $k_1, k_2, a,b \in \mathbb{Z}$ and two vectors $\vec{x}, \vec{y} \in \mathbb{R}^{k_2 - k_1 + 1}$. A $\llbracket k_1, k_2 \rrbracket$-indexed discrete line ensemble $\mathfrak{L} = (L_{k_1}, \dots, L_{k_2})$ on $ \llbracket a, b\rrbracket$ is called a {\em free $H^{RW}$ bridge line ensemble} with entrance data $\vec{x}$ and exit data $\vec{y}$ if its law $\mathbb{P}_{H^{RW}}^{k_1, k_2, a, b, \vec{x}, \vec{y}}$ is that of $k_2 - k_1 + 1$ independent $H^{RW}$ random walk bridges indexed by $\llbracket k_1, k_2 \rrbracket$ with the $i$-th bridge $L_{k_1 + i-1}$ being between the points $(a, x_i)$ and $(b, y_i)$ for $i \in \llbracket 1, k_2 -k_1 + 1\rrbracket$, see \eqref{RWB}. We write $\mathbb{E}_{H^{RW}}^{k_1, k_2, a, b, \vec{x}, \vec{y}}$ for the expectation with respect to this measure. When the parameters $k_1, k_2, a,b, \vec{x}, \vec{y}$ are clear from the context we will drop them from the notation and simply write $\mathbb{P}_{H^{RW}}$ and $\mathbb{E}_{H^{RW}}$. Observe that the measure remains unchanged upon replacing $(k_1, k_2)$ with $(k_1 + m, k_2+m)$ for some $m \in \mathbb{Z}$, except for a reindexing of the $L_i$'s.

An {\em interaction Hamiltonian} $H$ is defined to be any continuous function $H: [-\infty, \infty) \rightarrow [0, \infty)$ such that $H(- \infty) = 0$. Suppose we are given a sequence of interaction Hamiltonians $H_m$ for $m \in \llbracket a, b - 1 \rrbracket$ (abbreviated by $\vec{H} = (H_a, H_{a+1}, \dots, H_{b-1})$) and two functions $f: \llbracket a, b \rrbracket \rightarrow \mathbb{R} \cup \{ \infty \}$ and $g : \llbracket a,b \rrbracket \rightarrow \mathbb{R} \cup \{-\infty\}$. We define the $\llbracket k_1, k_2 \rrbracket$-indexed $( \vec{H} , H^{RW})$ line ensemble on $\llbracket a, b \rrbracket$ with entrance data $\vec{x}$ and exit data $\vec{y}$ and boundary data $(f,g)$ to be the law $\mathbb{P}_{\vec{H},H^{RW}}^{k_1, k_2, a,b, \vec{x}, \vec{y}, f, g}$ on $\mathfrak{L} = (L_{k_1}, \dots, L_{k_2})$ with $L_{k_1}, \dots, L_{k_2} : \llbracket a, b \rrbracket \rightarrow \mathbb{R}$ given in terms of the following Radon-Nikodym derivative (with respect to the free $H^{RW}$ bridge line ensemble $\mathbb{P}_{H^{RW}}^{k_1, k_2, a, b, \vec{x}, \vec{y}}$):
\begin{equation}\label{RND}
 \frac{d \mathbb{P}_{\vec{H},H^{RW}}^{k_1, k_2, a ,b, \vec{x}, \vec{y},f,g}}{d\mathbb{P}_{H^{RW}}^{k_1, k_2, a, b, \vec{x}, \vec{y}}} (L_{k_1}, \dots, L_{k_2}) = \frac{ W_{\vec{H}}^{k_1, k_2, a ,b,f,g} (L_{k_1}, \dots, L_{k_2}) }{Z_{\vec{H},H^{RW}}^{k_1, k_2, a ,b, \vec{x}, \vec{y},f,g}}.
\end{equation}
Here we call $L_{k_1 - 1} = f$ and $L_{k_2 + 1} = g$ and define the {\em Boltzmann weight}
\begin{equation}\label{WH}
W_{\vec{H}}^{k_1, k_2, a ,b,f,g} (L_{k_1}, \dots, L_{k_2}) : = \exp \left( - \sum_{i = k_1 - 1}^{k_2}  \sum_{ m = a}^{b-1} H_m (L_{i + 1}(m + 1) - L_{i}(m)) \right),
\end{equation}
and the {\em normalizing constant} (called the {\em acceptance probability})
\begin{equation}\label{AccProb}
Z_{\vec{H},H^{RW}}^{k_1, k_2, a ,b, \vec{x}, \vec{y},f,g} := \mathbb{E}_{H^{RW}}^{k_1, k_2, a, b, \vec{x}, \vec{y}} \Big[ W_{\vec{H}}^{k_1, k_2, a ,b,f,g} (L_{k_1}, \dots, L_{k_2})  \Big],
\end{equation}
where we recall that on the right side in \eqref{AccProb} the vectors $L_{k_1}, \dots, L_{k_2}$ are distributed according to the measure $\mathbb{P}_{H^{RW}}^{k_1, k_2, a, b, \vec{x}, \vec{y}}$. Notice that by our assumption on $f$ and $g$ we have that the argument of $H_m$ in \eqref{WH} is always in $[-\infty, \infty)$ and so $W_{ \vec{H}}^{k_1, k_2, a ,b,f,g}  \in (0,1]$ almost surely, which implies that $Z_{\vec{H},H^{RW}}^{k_1, k_2, a ,b, \vec{x}, \vec{y},f,g} \in (0,1]$ and we can indeed divide by this quantity in (\ref{RND}). We write the expectation with respect to $\mathbb{P}_{\vec{H},H^{RW}}^{k_1, k_2, a ,b, \vec{x}, \vec{y},f,g}$ as $\mathbb{E}_{\vec{H},H^{RW}}^{k_1, k_2, a ,b, \vec{x}, \vec{y},f,g}$.

For the most part, we will consider the case when $\vec{H}$ is such that $H_m(x) = H(x)$ for $m \in S$ and $H_m(x) = 0$ for $m \in \llbracket a, b-1 \rrbracket \setminus S$, where $S \subset \llbracket a, b-1 \rrbracket$ and $H$ is some fixed interaction Hamiltonian. In this case we will write $ W_{H,S}^{k_1, k_2, a ,b,f,g} (L_{k_1}, \dots, L_{k_2})$, $Z_{H,H^{RW},S}^{k_1, k_2, a ,b, \vec{x}, \vec{y},f,g}$, $\mathbb{P}_{H,H^{RW},S}^{k_1, k_2, a ,b, \vec{x}, \vec{y},f,g}$ and $\mathbb{E}_{H,H^{RW}, S}^{k_1, k_2, a ,b, \vec{x}, \vec{y},f,g}$. If $S = \llbracket a, b-1 \rrbracket$ we will drop it from the notation and simply write  $ W_{H}^{k_1, k_2, a ,b,f,g} (L_{k_1}, \dots, L_{k_2})$, $Z_{H,H^{RW}}^{k_1, k_2, a ,b, \vec{x}, \vec{y},f,g}$, $\mathbb{P}_{H,H^{RW}}^{k_1, k_2, a ,b, \vec{x}, \vec{y},f,g}$ and $\mathbb{E}_{H,H^{RW}}^{k_1, k_2, a ,b, \vec{x}, \vec{y},f,g}$.
\end{definition}

We record the following result, which states that for any bounded measurable function $h$ we have that $\mathbb{E}_{\vec{H},H^{RW}}^{k_1, k_2, a ,b, \vec{x}, \vec{y},f,g} \left[ h(\mathfrak{L}) \right]$ is bounded and measurable in the boundary data $\vec{x}, \vec{y}, f,g$. This result is a straightforward generalization of \cite[Lemma 7.2]{BCDB} and so we omit its proof.
\begin{lemma}\label{ContinuousGibbsCond} Let $H^{RW}$ be as in Definition \ref{Pfree}. Suppose that $a,b,k_1, k_2 \in \mathbb{Z}$ with $a < b $, $ k_1 \leq k_2 $ and that $\vec{H} = (H_{a}, \dots, H_{b-1})$ is a sequence of interaction Hamiltonians as in Definition \ref{Pfree}. In addition, suppose that $h: Y( \llbracket k_1 ,k_2 \rrbracket \times \llbracket a,b \rrbracket) \rightarrow \mathbb{R}$ is a bounded Borel-measurable function (recall that $Y(J)$ was defined in Definition \ref{YVec}). Let $V_L =  \llbracket k_1, k_2 \rrbracket  \times \{a \}$, $V_R = \llbracket k_1, k_2 \rrbracket  \times \{b \}$, $V_T =  \{k_1 - 1\} \times \llbracket a,b \rrbracket$ and $V_B = \{k_2 + 1 \} \times \llbracket a, b\rrbracket$ and define the set
\begin{equation*}
\begin{split}
 \mathcal{S} = \left\{ (\vec{x}, \vec{y}, \vec{u},\vec{v}) \in Y(V_L) \times Y(V_R) \times Y^+(V_T) \times Y^-(V_B) \right \},
\end{split}
\end{equation*}
where we endow $\mathcal{S}$ with the product topolgy and corresponding Borel $\sigma$-algebra. Then the function $G_h:  \mathcal{S} \rightarrow \mathbb{R}$, given by
\begin{equation}
G_h(\vec{x}, \vec{y}, \vec{u},\vec{v}) = \mathbb{E}_{\vec{H}, H^{RW}}^{a,b,\vec{x}, \vec{y}, \vec{u} ,\vec{v}} \left[ h(\mathfrak{L})\right],
\end{equation}
is bounded and measurable. In the above equations the random variable over which we are taking the expectation is denoted by $\mathfrak{L}$.
\end{lemma}

The key definition of this section is the following partial $(\vec{H}, H^{RW})$-Gibbs property.
\begin{definition}\label{DefLGGP}
Let $H^{RW}$ be as in Definition \ref{Pfree}. Fix $\topc \in \mathbb{N}$, two integers $T_0 < T_1$, a sequence $\vec{H} = (H_{T_0}, \dots, H_{T_1-1})$ of interaction Hamiltonians as in Definition \ref{Pfree} and set $\Sigma = \llbracket 1,\topc \rrbracket$. Suppose that $\mathbb{P}$ is the probability distribution of a $\Sigma$-indexed discrete line ensembles $\mathfrak{L} = (L_1, \dots, L_\topc)$ on $\llbracket T_0, T_1 \rrbracket$ and adopt the convention that $L_0 = \infty$. A $\Sigma$-indexed line ensemble $\mathfrak{L}$ on $\llbracket T_0, T_1 \rrbracket$ has the {\em the partial $(\vec{H}, H^{RW})$-Gibbs property} if and only if for any $\mathbf{k}  = \llbracket k_1, k_2 \rrbracket \subset \llbracket 1, \topc-1 \rrbracket$ and $\llbracket a, b\rrbracket \subset \llbracket T_0, T_1 \rrbracket$ and any bounded Borel-measurable function $F$ from $Y( \mathbf{k}  \times \llbracket a, b \rrbracket)$ (here $Y$ is as in Definition \ref{YVec}) to $\mathbb{R}$ we have $\mathbb{P}$-almost surely
\begin{equation}\label{GibbsEq}
\mathbb{E} \Big[ F \big( \mathfrak{L}\vert_{\mathbf{k}  \times \llbracket a, b \rrbracket} \big)  \big{\vert} \mathcal{F}_{ext} (\mathbf{k}  \times \llbracket a + 1, b - 1 \rrbracket) \Big]  = \mathbb{E}_{\vec{H} \llbracket a, b-1 \rrbracket ,H^{RW}}^{k_1, k_2, a ,b, \vec{x}, \vec{y},f,g}  \big[F( \tilde{\mathfrak{L}}) \big],
\end{equation}
where
\begin{equation}\label{GibbsCond}
 \mathcal{F}_{ext} (\mathbf{k}  \times \llbracket a + 1, b - 1 \rrbracket) := \sigma \left( L_i(s): (i,s) \in \Sigma \times \llbracket T_0, T_1 \rrbracket \setminus \mathbf{k}  \times  \llbracket a + 1, b - 1 \rrbracket \right).
\end{equation}
is the $\sigma$-algebra generated by the variables in the brackets above, $f = L_{k_1 - 1}$, $g = L_{k_2 + 1}$, $\vec{x} = (L_{k_1}(a), \dots L_{k_2}(a))$, $\vec{y} = (L_{k_1}(b), \dots L_{k_2}(b))$ and $\vec{H}\llbracket a,b -1 \rrbracket = (H_{a}, H_{a+1}, \dots, H_{b-1})$. On the right side of (\ref{GibbsEq}) the variable $\tilde{\mathfrak{L}}$ has law $\mathbb{P}_{\vec{H} \llbracket a, b-1 \rrbracket,H^{RW}}^{k_1, k_2, a ,b, \vec{x}, \vec{y},f,g} $.

If (\ref{GibbsEq}) also holds for $k_2 = K$ (and then $g = -\infty$) we drop the term ``partial'' and say that $\mathbb{P}$ satisfies the $(\vec{H}, H^{RW})$-Gibbs property. If $H_m(x) = H(x)$ for all $m \in \llbracket T_0, T_1 - 1\rrbracket$, where $H$ is a fixed interaction Hamiltonian, we say that $\mathbb{P}$ satisfies the (partial) $(H, H^{RW})$-Gibbs property.
\end{definition}
\begin{remark} \label{CondGibbs} From Lemma \ref{ContinuousGibbsCond} the right side of (\ref{GibbsEq}) is measurable with respect to  $\mathcal{F}_{ext} (\mathbf{k}  \times \llbracket a + 1, b - 1 \rrbracket)$ and thus equation (\ref{GibbsEq}) makes sense. Note that any $\{1 \}$-indexed line ensemble on $\llbracket T_0, T_1\rrbracket$ satisfies the partial $(\vec{H}, H^{RW})$-Gibbs property (bot not necessarily the $(\vec{H}, H^{RW})$-Gibbs property).
\end{remark}
\begin{remark} \label{restrict} From Definition \ref{DefLGGP} it is clear that for $\topc^\prime \leq \topc$ and $\llbracket a,b \rrbracket \subset \llbracket T_0, T_1\rrbracket$, we have that the induced law on $L_i(j)$ for $(i,j) \in \llbracket 1, \topc^\prime \rrbracket \times \llbracket a,b \rrbracket $ from $\mathbb{P}$ satisfies the partial $(\vec{H}\llbracket a, b-1 \rrbracket, H^{RW})$-Gibbs property as an $\llbracket 1, \topc^\prime \rrbracket$-indexed line ensemble on $ \llbracket a, b \rrbracket$. Note that if $K' < K$ then the above induced law satisfies necessarily only the partial $(\vec{H}\llbracket a, b-1 \rrbracket, H^{RW})$-Gibbs property even if $\mathbb{P}$ satisfies the $(H, H^{RW})$-Gibbs property. In plain words, the partial $(\vec{H}, H^{RW})$-Gibbs property is stable under projections, while the $(\vec{H}, H^{RW})$-Gibbs property is not. 
\end{remark}
We end this section with the following result, which is a special case of \cite[Lemma 2.8]{BCDB}.
\begin{lemma}\label{HHRWSG} Let $k \in \mathbb{N}$, $T_0, T_1 \in \mathbb{Z}$ and $T_0 < T_1$. Fix $\vec{x} \in \mathbb{R}^k$, $\vec{y} \in \mathbb{R}^k$ and $\vec{z} \in Y^-(\llbracket T_0, T_1 \rrbracket)$. Then $\mathbb{P}_{H,H^{RW}}^{1, k, T_0 ,T_1, \vec{x}, \vec{y},\infty, \vec{z}}$ from Definition \ref{Pfree} satisfies the partial $(H, H^{RW})$-Gibbs property.
\end{lemma}

%
\subsection{Continuous grand monotone coupling lemma}\label{Section2.3} The goal of this section is to state a continuous grand monotone coupling lemma for  $Y(\llbracket 1, k \rrbracket \times \llbracket T_0, T_1 \rrbracket )$-valued random variables whose laws are given by $\mathbb{P}_{\vec{H}, H^{RW}}^{1, k, T_0 ,T_1, x, y, \infty ,\vec{z}}$ as in Definition \ref{Pfree}. This result is given as Lemma \ref{MonCoup} and is one of the main results we prove about line ensembles satisfying the partial $(\vec{H},H^{RW})$-Gibbs property. The proof of Lemma \ref{MonCoup} is presented in Section \ref{Section6}.

\begin{lemma}\label{MonCoup} Let $T, k \in \mathbb{N}$ satisfy $T \geq 2$ and assume that $ H^{RW}$ is as in Definition \ref{Pfree}. We also let $\vec{H} = (H_0, \dots, H_{T-2})$ be a sequence of interaction Hamiltonians as in Definition \ref{Pfree}. Then the following statements hold.
\begin{enumerate}[label={\Roman*)}]
 \item (Grand coupling) There exists a probability space $( \Omega^{k,T}, \mathcal{F}^{k,T},\mathbb{P}^{k,T})$ that supports random vectors $\ell^{k,T,\vec{x},\vec{y}, \vec{z}} \in  Y(\llbracket 1, k \rrbracket \times \llbracket 0 , T - 1 \rrbracket )$ for all $\vec{x},\vec{y} \in \mathbb{R}^k$ and $\vec{z} \in Y^-(\llbracket 0, T-1\rrbracket)$ such that under $\mathbb{P}^{k,T}$ the random vector $\ell^{k,T,\vec{x},\vec{y}, \vec{z}}$ has law $\mathbb{P}_{\vec{H}, H^{RW}}^{ 1,k, 0, T - 1,\vec{x}, \vec{y}, \infty, \vec{z}}$ as in Definition \ref{Pfree}.
 \item (Monotone coupling) Moreover, if we further suppose that $H_m$ and $H^{RW}$ are convex and $H_m$ are increasing for $m \in \llbracket 0 , T-2 \rrbracket$, then for any fixed $\vec{x},\vec{y},\vec{x}',\vec{y}' \in \mathbb{R}^k$ with $x_i \leq x_i'$ and $y_i \leq y_i'$ for $i \in \llbracket 1, k \rrbracket$ and $\vec{z}, \vec{z}' \in Y^-(\llbracket 0, T-1\rrbracket)$ with $\vec{z}(i) \leq \vec{z}'(i)$ for $i \in \llbracket 0, T-1 \rrbracket$ we have $\mathbb{P}^{k,T}$-almost surely that $\ell^{k,T,\vec{x},\vec{x}, \vec{z}}(i,j) \leq \ell^{k,T,\vec{x}',\vec{y}', \vec{z}'}(i,j)$ for $(i,j) \in \llbracket 1, k \rrbracket \times \llbracket 0, T-1 \rrbracket$.
 \item (Continuous coupling) If $T \geq 3$ the probability space $( \Omega^{k,T}, \mathcal{F}^{k,T},\mathbb{P}^{k,T})$ in part I can be taken to be $(0,1)^{k(T-2)}$ with the Borel $\sigma$-algebra and Lebesgue measure. If $T = 2$ then $( \Omega^{k,T}, \mathcal{F}^{k,T},\mathbb{P}^{k,T})$ can be taken to be the space with a single point $\omega_0$, discrete $\sigma$-algebra and the measure that assigns unit mass to the point $\omega_0$. 

Furthermore, the construction in part I can be made so that the map $\Phi^{k,T}: \mathbb{R}^k \times \Omega^{k,T} \times \mathbb{R}^k \times Y^-(\llbracket 0, T-1\rrbracket) \rightarrow Y(\llbracket 1, k \rrbracket \times \llbracket 0, T-1 \rrbracket) \times Y^-(\llbracket 0, T-1\rrbracket)$ defined by
$$\Phi^{k,T}(\vec{x},\omega, \vec{y}, \vec{z}) = (\ell^{k,T,\vec{x},\vec{y}, \vec{z}}(\omega), \vec{z})$$ 
is a homeomorphism between the spaces $ \mathbb{R}^k \times \Omega^{k,T} \times \mathbb{R}^k \times Y^-(\llbracket 0, T-1\rrbracket) $ and $Y(\llbracket 1, k \rrbracket \times \llbracket 0, T-1 \rrbracket)  \times Y^-(\llbracket 0, T-1\rrbracket)$. In the last statement $\Omega^{k,T} = (0,1)^{k(T-2)}$ and we endow it with the subspace topology from $\mathbb{R}^{k(T-2)}$ (and discrete topology if $T = 2$) and the space $\mathbb{R}^k \times \Omega^{k,T} \times \mathbb{R}^k \times Y^-(\llbracket 0, T-1\rrbracket) $ has the product topology.
\end{enumerate}
\end{lemma}
\begin{remark}
 Lemma \ref{MonCoup} was recently proved in \cite[Lemma 2.10]{BCDB} in the special case $k = 1$ and $H_i = H$ for all $i \in \llbracket 0, T-2 \rrbracket$. Lemma \ref{MonCoup} expands upon \cite[Lemma 2.10]{BCDB} by constructing a grand continuous monotone coupling for line ensembles with arbitrarily many curves and general sequences of interaction Hamiltonians $\vec{H}$.
\end{remark}

\begin{remark}
We observe that when $\vec{z} = (- \infty)^T$ and $H_m = 0$ for $m \in \llbracket 0, T-2 \rrbracket$ we have that $\mathbb{P}_{\vec{H}, H^{RW}}^{1,k, 0,T-1,\vec{x}, \vec{y}, \infty,\vec{z}}$ is precisely $\mathbb{P}_{H^{RW}}^{1, k, 0, T-1, \vec{x}, \vec{y}}$ -- the law of $k$ independent $H^{RW}$ random walk bridges between the points $(0,x_i)$ and $(T-1,y_i)$ for $i \in \llbracket 1, k \rrbracket$.
\end{remark}

\begin{remark}
We mention here that part II of Lemma \ref{MonCoup} (upon reindexing) provides a monotone coupling for random variables distributed according to $\mathbb{P}_{\vec{H}, H^{RW}}^{1,k, T_0 ,T_1, x, y, \infty ,\vec{z}}$. In \cite[Section 6]{Wu19} the second author presents an argument that establishes an analogous coupling statement for the case when $H_m$ are all equal to the same interaction Hamiltonian $H$. The approach taken in \cite{Wu19} goes through approximating a $\mathbb{P}_{\vec{H}, H^{RW}}^{1, k, T_0 ,T_1, \vec{x}, \vec{y}, \infty ,\vec{z}}$-distributed discrete line ensemble $\mathfrak{L}$ by a sequence of discrete line ensembles $\mathfrak{L}^n$ that take values on compact lattices and showing the monotonicity statement for each of those using Markov chain Monte Carlo (MCMC) methods. One then deduces the monotonicity statement for the limiting line ensemble $\mathfrak{L}$ from the one for $\mathfrak{L}^n$.

Presently, the work in \cite[Section 6]{Wu19} shows that two line ensembles with fixed boundary data (that are ordered) can be monotonically coupled. Instead, part II of Lemma \ref{MonCoup} shows that {\em all } line ensembles with {\em all possible} boundary data can be simultaneously monotonically coupled. One might be able to improve the MCMC argument in the discrete setting to prove part II of Lemma \ref{MonCoup}; however, there is some delicacy in showing that the convergence of $\mathfrak{L}^n$ to $\mathfrak{L}$ happens simultaneously for all boundary data, since the latter form an uncountable set.

Rather than justifying the discrete approximation and MCMC approach in \cite{Wu19} we will directly construct a monotone coupling in the continuum. In particular, this approach allows us to establish the full-strength form of part II, as well as parts I and III of Lemma \ref{MonCoup}, with the latter two being completely new (apart from the $k = 1$ case considered in \cite{BCDB}). 
\end{remark}

%
\subsection{Strong coupling of $H^{RW}$ random walk bridges}\label{Section2.4} From Definition \ref{Pfree} $\mathbb{P}_{H^{RW}}^{1,1, T_0,T_1, x,y}$ is the law of an $H^{RW}$ random walk bridge between the points $(T_0, x)$ and $(T_1,y)$, see (\ref{RWB}). To simplify the notation, we denote such measures by $\mathbb{P}^{T_0,T_1,x,y}_{H^{RW}}$ and write $\mathbb{E}^{T_0,T_1,x,y}_{H^{RW}}$ for their expectation. In this section we state a strong coupling between random walk bridges and Brownian bridges from \cite{DW19} -- recalled here as Proposition \ref{KMT}. In order to apply this coupling result we need to make several assumptions on the function $H^{RW}$, summarized in the following definition.
\begin{definition}\label{AssHR} We make the following five assumptions on $H^{RW}$.

{\bf \raggedleft Assumption 1.} We assume that $H^{RW}: \mathbb{R} \rightarrow \mathbb{R}$ is a continuous convex function and $\g(x) = e^{-H^{RW}(x)}$. We assume that $\g(x)$ is bounded and $\int_{\mathbb{R}} \g(x) dx = 1$.\\

If $X$ is a random variable with density $\g$ we denote
\begin{equation}\label{S2S1E1}
M_X(t) := \mathbb{E} \big[ e^{t X} \big], \hspace{3mm}  \phi_X(t) := \mathbb{E} \big[ e^{itX} \big], \hspace{3mm} \Lambda(t) :=\log M_X(t), \hspace{3mm} \mathcal{D}_\Lambda := \{ x: \Lambda(x) < \infty\}.
\end{equation}

{\bf \raggedleft Assumption 2.} We assume that $\mathcal{D}_\Lambda$ contains an open neighborhood of $0$.

It is easy to see that $\mathcal{D}_{\Lambda}$ is an interval. We denote $(A_{\Lambda}, B_{\Lambda})$ the interior of $\mathcal{D}_\Lambda$ where $A_{\Lambda} < 0$ and $B_{\Lambda} > 0$ by Assumption 2. We write $M_X(u)$ for all $u \in D = \{u \in \mathbb{C}: A_\Lambda< \Re(u) < B_\Lambda \}$ to mean the (unique) analytic extension of $M_X(x)$ to $D$, afforded by \cite[Lemma 2.1]{DW19}.\\

{\bf \raggedleft Assumption 3.} We assume that the function $\Lambda(\cdot)$ is lower semi-continuous on $\mathbb{R}$.\\

Under Assumptions 1,2 and 3 for a given $p \in \mathbb{R}$ the quantity  $\Lambda''( (\Lambda')^{-1}(p))$ is well-defined -- see \cite[Section 2.1]{DW19}. For brevity we write $\sigma_p^2 := \Lambda''( (\Lambda')^{-1}(p))$. \\

{\raggedleft \bf Assumption 4.} We assume that for every $B_\Lambda > t> s > A_\Lambda$ there exist constants $K(s,t)>0$ and $p(s,t) > 0$ such that $\left|M_X(z)\right| \leq \frac{K(s,t)}{(1 + |\Im(z)|)^{p(s,t)}}$, provided $s\leq \Re(z) \leq t$.\\

{\bf \raggedleft Assumption 5.} We suppose that there are constants $ D, d > 0$ such that at least one of the following statements holds
\begin{equation}\label{S2S1E2}
\mbox{1. }\g(x) \leq De^{-dx^2} \mbox{ for all $x \geq 0$ \qquad or\qquad  2. }\g(x) \leq De^{-dx^2} \mbox{ for all $x \leq 0$}.
\end{equation}
\end{definition}
\begin{remark} The interested reader is referred to \cite[Remark 2.14]{BCDB} for a brief and to \cite[Section 2.3]{DW19} for a detailed discussion of the significance of the above five assumptions.
\end{remark}

If $W_t$ denotes a standard one-dimensional Brownian motion and $\sigma > 0$, then the process
$$B^{\sigma}_t = \sigma (W_t - t W_1), \hspace{5mm} 0 \leq t \leq 1,$$
is called a {\em Brownian bridge (conditioned on $B_0 = 0, B_1 = 0$)  with diffusion parameter $\sigma$.}  With the above notation we state the strong coupling result we use.
\begin{proposition}\label{KMT}
Suppose $H^{RW}$ satisfies the assumptions of Definition \ref{AssHR}. Let $p \in \mathbb{R}$ and $\sigma_p^2$ be as in that definition. There exist constants $0 < C, a, \alpha < \infty$ (depending on $p$ and $H^{RW}$) such that for every positive integer $T$, there is a probability space on which are defined a Brownian bridge $B^\sigma$ with diffusion parameter $\sigma = \sigma_p$ and a family of random curves $\ell^{(T,z)}$ on $[0,T]$, which is parameterized by $z \in \mathbb{R}$ such that $\ell^{(T,z)}$  has law $\mathbb{P}^{0,T,0,z}_{H^{RW}}$ and
\begin{equation}\label{KMTeq}
\mathbb{E}\big[ e^{a \Delta(T,z)} \big] \leq C e^{\alpha (\log T)}e^{|z- p T|^2/T}, \mbox{ where $\Delta(T,z)=  \sup_{0 \leq t \leq T} \big| \sqrt{T} B^\sigma_{t/T} + \frac{t}{T}z - \ell^{(T,z)}(t) \big|.$}
\end{equation}
Here we recall that $\ell^{(T,z)}(s)$ was defined for non-integer $s$ by linear interpolation.
\end{proposition}
\begin{proof}
This result is a special case of \cite[Theorem 2.3]{DW19}. Indeed, Assumptions 1-5 in Definition \ref{AssHR} imply Assumptions C1-C5 in \cite[Section 2.1]{DW19}. In addition, Assumption C6 in \cite[Section 2.1]{DW19} is satisfied in view of \cite[Lemma 7.2]{DW19} and the fact that $H^{RW}$ is convex.
\end{proof}

We end this section by recalling a useful result from \cite{BCDB}. For a function $f \in C[a,b]$ we define its {\em modulus of continuity} by
\begin{equation}\label{MOCS4}
w(f,\delta) = \sup_{\substack{x,y \in [a,b]\\ |x-y| \leq \delta}} |f(x) - f(y)|.
\end{equation}
\begin{lemma}\label{MOCLemmaS4}\cite[Lemma 2.25]{BCDB} Let $\ell$ have distribution $\mathbb{P}^{0,T,0,y}_{H^{RW}}$ with $H^{RW}$ satisfying the assumptions in Definition \ref{AssHR}. Let $M > 0$ and $p \in \mathbb{R}$ be given. For each positive $\epsilon$ and $\eta$, there exist a $\delta > 0$ and $W = W(M, p, \epsilon, \eta) \in \mathbb{N}$ such that  for $T \geq W$ and $|y - pT| \leq MT^{1/2}$ we have
\begin{equation}\label{MOCeqS4}
\mathbb{P}^{0,T,0,y}_{H^{RW}}\Big( w\big({f^\ell},\delta\big) \geq \epsilon \Big) \leq \eta,
\end{equation}
where $f^\ell(u) = T^{-1/2}\big(\ell(uT) - puT\big)$  for $u \in [0,1]$.
\end{lemma}
\begin{remark}
In words, Lemma \ref{MOCLemmaS4} states that if $\ell$ is a random walk bridge that is started from $(0,0)$ and terminates at $(T,y)$ with $y$ close to $pT$ (i.e. with well-behaved endpoints) then the modulus of continuity of $\ell$ is also well-behaved with high probability.
\end{remark}

%
\section{Tightness of $(H,H^{RW})$-Gibbsian line ensembles}\label{Section3}
In this section we describe a general framework that can be used to prove tightness for a sequence of line ensembles that satisfy the $(H,H^{RW})$-Gibbs property. We start by summarizing our assumptions on $H$ in the following definition.
\begin{definition}\label{AssH} We let $H: [-\infty, \infty) \rightarrow [0,\infty)$ be continuous, increasing and convex. We assume that $\lim_{x \rightarrow \infty} H(x) = \infty$ and $\lim_{x \rightarrow \infty} x^2 H(-x) = 0$.
\end{definition}

We next introduce the following useful definition.
\begin{definition}\label{Def1} Let $H^{RW}$ be as in Definition \ref{AssHR} and $H$ as in Definition \ref{AssH}. Fix $K \in \mathbb{N}$, $\alpha, \lambda > 0$ and $p \in \mathbb{R}$.
Suppose we are given a sequence $\{T_N\}_{N = 1}^\infty$ with $T_N \in \mathbb{N}$ and that $\{\mathfrak{L}^N\}_{N = 1}^\infty$, is a sequence of $\llbracket 1, K \rrbracket $-indexed discrete line ensembles $\mathfrak{L}^N = (L^N_1, L^N_2, \dots, L^N_K)$ on $\llbracket -T_N, T_N \rrbracket$. We say that the sequence $\big\{\mathfrak{L}^N\big\}_{N=1}^{\infty}$ is $(\alpha,p,\lambda)$--{\em good} if there exists $N_0(\alpha, p, \lambda) \in \mathbb{N}$ such that 
\begin{itemize}
\item $\mathfrak{L}^N$ satisfies the partial $(H, H^{RW})$-Gibbs property of Definition \ref{DefLGGP} for $N \geq N_0$;
\item there is a function $\psi: \mathbb{N} \rightarrow (0, \infty)$ such that $\lim_{n \rightarrow \infty} \psi(n) = \infty$ and for each $N \geq N_0$ we have $T_N > \psi(N) N^{\alpha}$,
\item there are functions $\phi_1 : \mathbb{Z} \times (0,\infty) \rightarrow (0, \infty)$ and $\phi_2 : (0, \infty) \rightarrow (0, \infty)$ such that for any $\epsilon > 0$, $n \in \mathbb{Z}$ and $N \geq \phi_1(n, \epsilon)$
\end{itemize}
\begin{equation}\label{UnifTight}
 \mathbb{P} \left( \left| N^{-\alpha/2} \left( L_1^N(\lfloor nN^{\alpha} \rfloor ) - p n N^{\alpha} \right) + \lambda n^2   \right| \geq \phi_2(\epsilon) \right) \leq \epsilon.
\end{equation}
\end{definition}
\begin{remark} Let us briefly explain the meaning of Definition \ref{Def1}. In order for a sequence $\mathfrak{L}^N$ of $\llbracket 1, K \rrbracket$-indexed discrete line ensembles on $\llbracket -T_N, T_N \rrbracket$ to be $(\alpha,p,\lambda)$-good, we want several conditions to be satisfied. Firstly, we want for all large enough $N$, the discrete line ensemble $\mathfrak{L}^N$ to satisfy the $(H, H^{RW})$-Gibbs property. The second condition ensures that the intervals $[-T_N, T_N]$ grow on scale $N^{\alpha}$ fast enough to cover the entire real line. The third condition requires that for each $n \in \mathbb{N}$ the sequence of random variables $N^{-\alpha/2} \left( L_1^N(\lfloor nN^{\alpha} \rfloor) - p n N^{\alpha} \right)$ is tight, but moreover we want the weak subsequential limits of these variables to globally look like the parabola $- \lambda n^2$. For example, the existence of $\phi_1, \phi_2$ satisfying (\ref{UnifTight}) would be assured if $ N^{-\alpha/2} \left( L_1^N(\lfloor nN^{\alpha} \rfloor ) - p n N^{\alpha} \right) + \lambda n^2$ for each $n \in \mathbb{N}$ converge to the same random variable (this will be the case in our applications).
\end{remark}

The main technical result of this section is as follows.
\begin{theorem}\label{PropTightGood}Fix $ K \in \mathbb{N}$ with $K \geq 2$, $\alpha, \lambda > 0$ and $p \in \mathbb{R}$ and let $\mathfrak{L}^N = (L^N_1, L^N_2, \dots, L^N_K)$ be an $(\alpha, p, \lambda)$-good sequence of $\llbracket 1, K \rrbracket$-indexed discrete line ensembles as in Definition \ref{Def1}.  Set
$$f^N_i(s) =  \sigma_p^{-1}N^{-\alpha/2}(L^N_i(sN^{\alpha}) - p s N^{\alpha}), \mbox{ for $s\in [-\psi(N) ,\psi(N)]$ and $i = 1,\dots, K -1$,}$$
where $\sigma_p$ is as in Definition \ref{AssHR}, and extend $f^N_i$ to $\mathbb{R}$ by setting for $i = 1, \dots, K - 1$
$$f^N_i(s) = f^N_i(-\psi(N)) \mbox{ for $s \leq -\psi(N)$ and } f^N_i(s) = f_i^N(\psi(N)) \mbox{ for $s \geq \psi(N)$}.$$
Let $\mathbb{P}_N$ denote the law of $\mathcal{L}^N = (f_1^N, \dots, f_{K-1}^N)$ as a $\llbracket 1, K-1 \rrbracket$-indexed line ensembles on $\mathbb{R}$ as in Definition \ref{CLEDef}. Then:
\begin{enumerate}[label=(\roman*)]
	\item The sequences $\mathbb{P}_N$ is tight.
	\item Any subsequential limit ${\mathcal{L}}^\infty = ({f}_1^{\infty}, \dots, {f}_{K-1}^{\infty})$ of ${\mathbb{P}}_N$ satisfies the partial Brownian Gibbs property of Definition \ref{DefPBGP}.
\end{enumerate}
\end{theorem}

%
\subsection{Proof of Theorem \ref{Thm1}}\label{Section3.1} The proof we present here is similar to the proof of \cite[Theorem 1.1]{DREU}. We use the same notation and assumptions as in the statement of the theorem. For clarity we split the proof into two steps.\\

{\bf \raggedleft Step 1.} In this step we prove that $\mathcal{L}^N$ is tight. In view of Lemma \ref{2Tight} to establish the tightness of $\mathcal{L}^N$ it suffices to show that for every $k \in \mathbb{N}$  
\begin{enumerate}[label=(\roman*)]
\item $\lim_{a\to\infty} \limsup_{N\to\infty}\, \mathbb{P}(|\mathcal{L}^N_k(0)|\geq a) = 0$.
\item For all $\epsilon>0$ and $m \in \mathbb{N}$,  $\lim_{\delta\to 0} \limsup_{N\to\infty}\, \mathbb{P}\bigg(\sup_{\substack{x,y\in [-m,m], \\ |x-y|\leq\delta}} |\mathcal{L}^N_k(x) - \mathcal{L}^N_k(y)| \geq \epsilon\bigg) = 0.$
\end{enumerate}

Let $T_N = \min (-a_N, b_N)$ and for $N \geq k +1$ let $\tilde{\mathfrak{L}}^N = (\tilde{L}_1^N , \tilde{L}_2^N, \dots, \tilde{L}_{k+1}^N)$ denote the $\llbracket 1, k+1\rrbracket$-indexed discrete line ensemble on $\llbracket - T_N, T_N \rrbracket$, obtained from $\mathfrak{L}^N$ by restriction to the top $k+1$ lines and the interval $\llbracket -T_N, T_N \rrbracket$. Since by assumption we have that $\mathfrak{L}^N$ satisfies the partial $(H,H^{RW})$-Gibbs property on $\llbracket a_N, b_N \rrbracket$, we conclude that $\tilde{\mathfrak{L}}^N$ satisfies the partial $(H,H^{RW})$-Gibbs property on $\llbracket -T_N, T_N \rrbracket$ for each $N \geq k+1$, cf. Remark \ref{restrict}.

We next observe that Assumptions 1,2 and 3 in Section \ref{Section1.2} imply that $\{\tilde{\mathfrak{L}}^N\}_{N \geq k+1}$ is an $(\alpha, p, \lambda)$-good in the sense of Definition \ref{Def1}. Specifically, we have that the conditions in Definition \ref{Def1} are satisfied with $N_0 = k+1$, $K = k$, $\alpha, p ,\lambda, H, H^{RW}$ as in the statement of the theorem, $T_N$ as above and $\psi$ as in Assumption 1 in Section \ref{Section1.2}. For the functions $\phi_1, \phi_2$ we may set $\phi_2(\epsilon) = \phi(\epsilon/2)$, where $\phi$ is as in Assumption 2 in Section \ref{Section1.2}, which we recall required that
$$ \sup_{n \in \mathbb{Z}} \limsup_{N \rightarrow \infty} \mathbb{P} \left( \left|N^{-\alpha/2}(L_1^N( \lfloor n N^{\alpha} \rfloor) - p n N^{\alpha} + \lambda n^2 N^{\alpha/2}) \right| \geq \phi(\epsilon) \right) \leq \epsilon.$$
The last equation and the fact that $\phi_2(\epsilon) = \phi(\epsilon/2)$ implies that for each $n \in \mathbb{Z}$ and $\epsilon > 0$ there exists $A(n, \epsilon) \in \mathbb{N}$ such that for $N \geq A(n, \epsilon) $ we have
$$\mathbb{P} \left( \left|N^{-\alpha/2}(L_1^N(\lfloor n N^{\alpha}\rfloor ) - p n N^{\alpha} + \lambda n^2 N^{\alpha/2}) \right| \geq \phi_2(\epsilon) \right) \leq \epsilon,$$
and then we can set $\phi_1(n, \epsilon) = A(n, \epsilon) $.

Since $\{\tilde{\mathfrak{L}}^N\}_{N \geq k+1}$ is an $(\alpha, p, \lambda)$-good sequence we have from Theorem \ref{PropTightGood} that the line ensembles $\{ \tilde{f}^N_i \}_{i = 1}^{k}$, as in the statement of that theorem for the line ensembles $ \tilde{\mathfrak{L}}^N$, form a tight sequence in $(C( \llbracket 1, k \rrbracket \times \mathbb{R}), \mathcal{C}_{k})$. We may now apply the ``only if'' part of Lemma \ref{2Tight} to $\{ \tilde{f}^N_i \}_{i = 1}^{k}$ and conclude that statements (i) and (ii) from the beginning of this step hold for $\tilde{f}^N_k,$ which in turn implies they hold for $\mathcal{L}^N_k$, since by construction $\mathcal{L}^N_k$ and $\tilde{f}^N_k$ have the same distribution.\\

{\bf \raggedleft Step 2.} We next suppose that $\mathcal{L}^\infty$ is any subsequential limit of $\mathcal{L}^N$ and that $n_m \uparrow \infty$ is a sequence such that $\mathcal{L}^{n_m}$ converges weakly to $\mathcal{L}^\infty$. We want to show that $\mathcal{L}^\infty$ satisfies the Brownian Gibbs property. Suppose that $a, b \in \mathbb{R}$ with $a < b$ and $K = \{k_1, k_1+1, \dots, k_2\} \subset \mathbb{N}$ are given. We wish to show that $\mathcal{L}^\infty$ is almost surely non-intersecting and for any bounded Borel-measurable function $F: C(K \times [a,b]) \rightarrow \mathbb{R}$ almost surely
\begin{equation}\label{BGPTowerV2S2}
\mathbb{E} \left[ F\left(\mathcal{L}^\infty|_{K \times [a,b]} \right)  {\big \vert} \mathcal{F}_{ext} (K \times (a,b))  \right] =\mathbb{E}_{avoid}^{a,b, \vec{x}, \vec{y}, f, g} \bigl[ F(\tilde{\mathcal{Q}}) \bigr],
\end{equation}
where we use the same notation as in Definition \ref{DefPBGP}. In particular, we recall that 
$$\mathcal{F}_{ext} (K \times (a,b)) = \sigma \left \{ \mathcal{L}^\infty_i(s): (i,s) \in D_{K,a,b}^c \right\} \mbox{, with $D_{K,a,b}^c = (\mathbb{N} \times \mathbb{R}) \setminus K \times (a,b).$ } $$

Let $k \geq k_2 +1$ and consider the map $\Pi_{k}: C( \mathbb{N} \times \mathbb{R}) \rightarrow C( \llbracket 1, k  \rrbracket \times \mathbb{R})$ given by $[\Pi_{k}(g)](i,t) = g(i,t)$, which is continuous, and so $\Pi_{k} [\mathcal{L}^{n_m}]$ converge weakly to $\Pi_{k}[\mathcal{L}^\infty]$ as random variables in $C( \llbracket 1, k  \rrbracket \times \mathbb{R})$. If $\{\tilde{f}^N_i \}_{i = 1}^{k}$ are as in Step 1, then we know by construction that the restriction of $\{\tilde{f}^N_i \}_{i = 1}^{k}$ to $[- \psi(N), \psi(N)]$ has the same distribution as the restriction of $\Pi_{k} [\mathcal{L}^{N}]$ to the same interval. Since $\psi(N) \rightarrow \infty$ by assumption and $\Pi_{k} [\mathcal{L}^{n_m}]$ converge weakly to $\Pi_{k}[\mathcal{L}^\infty]$ we conclude that $\{\tilde{f}^{n_m}_i \}_{i = 1}^{k}$ converge weakly to $\Pi_{k}[\mathcal{L}^\infty]$ (here we used that the topology is that of uniform convergence over compacts). In particular, by the second part of Theorem \ref{PropTightGood} we conclude that $\Pi_{k} [\mathcal{L}^{\infty}]$ satisfies the partial Brownian Gibbs property as a $\llbracket 1, k\rrbracket$-indexed line ensemble on $\mathbb{R}$. The latter implies that $\Pi_{k} [\mathcal{L}^{\infty}]$ is non-intersecting almost surely and almost surely 
\begin{equation}\label{BGPTowerV2S3}
\mathbb{E} \left[ F\left(\mathcal{L}^\infty|_{K \times [a,b]} \right)  {\big \vert} \tilde{\mathcal{F}}_{ext} (K \times (a,b))  \right] =\mathbb{E}_{avoid}^{a,b, \vec{x}, \vec{y}, f, g} \bigl[ F(\tilde{\mathcal{Q}}) \bigr],
\end{equation}
where
$$\tilde{\mathcal{F}}_{ext} (K \times (a,b)) = \sigma \left \{ \mathcal{L}^\infty_i(s): (i,s) \in \tilde{D}_{K,a,b}^c \right\} \mbox{, with $\tilde{D}_{K,a,b}^c = ( \llbracket 1, k \rrbracket \times \mathbb{R}) \setminus K \times (a,b).$ } $$

Since $\Pi_{k} [\mathcal{L}^{\infty}]$ is non-intersecting almost surely and $k \geq k_2 + 1$ was arbitrary we conclude that $\mathcal{L}^{\infty}$ is almost surely non-intersecting. Let $\mathcal{A}$ denote the collection of sets $A$ of the form
$$A = \{ \mathcal{L}^\infty(i_r, x_r) \in B_r \mbox{ for $r = 1, \dots, p$ } \},$$
where $p \in \mathbb{N}$, $B_1, \dots, B_p \in \mathcal{B}(\mathbb{R})$ (the Borel $\sigma$-algebra on $\mathbb{R}$) and $(i_1, x_1), \dots, (i_p, x_p) \in D_{K,a,b}^c$. Since in (\ref{BGPTowerV2S3}) we have that $k \geq k_2 +1$ was arbitrary we conclude that for all $A \in \mathcal{A}$ we have 
$$\mathbb{E} \left[  F\left(\mathcal{L}^\infty|_{K \times [a,b]} \right)   \cdot {\bf 1 }_A \right] = \mathbb{E} \left[  \mathbb{E}_{avoid}^{a,b, \vec{x}, \vec{y}, f, g} \bigl[ F(\tilde{\mathcal{Q}}) \bigr] \cdot {\bf 1 }_A \right].$$
In view of the bounded convergence theorem, we see that the collection of sets $A$ that satisfies the last equation is a $\lambda$-system and as it contains the $\pi$-system $\mathcal{A}$ we conclude by the $\pi-\lambda$ theorem that it contains $\sigma(\mathcal{A})$, which is precisely $  \mathcal{F}_{ext} (K \times (a,b)) $. We may thus conclude (\ref{BGPTowerV2S2}) from the defining properties of conditional expectation and the fact that the right side of (\ref{BGPTowerV2S2}) is $\mathcal{F}_{ext} (K \times (a,b))$-measurable as follows from \cite[Lemma 3.4]{DimMat}. This suffices for the proof.

%
\subsection{Proof of Theorem \ref{ThmLG}}\label{Section3.2} We use the same notation and assumptions as in the statement of the theorem. We proceed to show that $\mathcal{L}^N$ satisfy the conditions of Theorem \ref{Thm1}, which if true would imply our desired statement. 

In view of Proposition \ref{PropLOG} we know that $\mathfrak{L}^N$ is a $\llbracket 1, N \rrbracket$-indexed discrete line ensemble on $\llbracket -N, N \rrbracket$, which satisfies the partial $(H,H^{RW})$-Gibbs property with $H,H^{RW}$ as in (\ref{S1Hamiltonian}). What remains is to show that Assumptions 1,2 and 3 from Section \ref{Section1.2.1} all hold.

Since $a_N = -N$, $b_N = N$, $\alpha = 2/3$ and $\psi(N) = (1/2) N^{1/3}$ we see that Assumption 1 is satisfied. The fact that $H^{RW}$ and $H$ satisfy the conditions in Definitions \ref{AssHR} and \ref{AssH}, respectively, is a straightforward verification and carried out in \cite[Section 6.3]{BCDB}. This verifies Assumption 3. In addition, in  \cite[Section 6.3]{BCDB} it was shown that if $\sigma_p$ is as in Definition \ref{AssHR} we have $\sigma_p^2 = \Psi'(g_{\theta}^{-1}(1)) = \Psi'(\theta/2)$, which agrees with the $\sigma_p$ in Section \ref{Section1.2.2}. What remains is to show that Assumption 2 is satisfied with $p = - h_{\theta}'(1)$, $\alpha = 2/3$ for some $\lambda > 0$ and function $\phi: (0,\infty) \rightarrow (0,\infty)$.\\

From \cite[Theorem 1.2]{BCDA} we have for each $y \in \mathbb{R}$ and $n \in \mathbb{Z}$ that
\begin{equation*}
\lim_{N \rightarrow \infty} \mathbb{P} \left( \frac{L_1^N(\lfloor nN^{2/3} \rfloor) - 2N h_{\theta}(1) + 2N h_{\theta}((2N + nN^{2/3})/(2N))}{(2N)^{1/3} \tsigma_\theta(1) } \leq y \right) = F_{GUE}(y),
\end{equation*}
where $F_{GUE}$ denotes the GUE Tracy-Widom distribution \cite{TWPaper} and 
\begin{equation*}
\tsigma_\theta(x) = \left[ \sum_{n = 0}^\infty \frac{x}{\big(n+g_\theta^{-1}(x)\big)^3} +  \sum_{n = 0}^\infty \frac{1}{\big(n+\theta - g_\theta^{-1}(x)\big)^3} \right]^{1/3}.
\end{equation*}
Taylor expanding $h_{\theta}$ we conclude that 
\begin{equation}\label{S3TWConv}
\lim_{N \rightarrow \infty} \mathbb{P} \left( \frac{L_1^N(\lfloor nN^{2/3} \rfloor) + h_{\theta}'(1) nN^{2/3} + (1/4) h_{\theta}''(1)  n^2 N^{1/3}}{(2N)^{1/3} \tsigma_\theta(1) } \leq y \right) = F_{GUE}(y).
\end{equation}
We now set $\lambda = (1/4) h_{\theta}''(1) $ and note that $\lambda >0 $ since $h_{\theta}'(x) = \Psi(g^{-1}_{\theta}(x))$, $\Psi'(y) > 0$ for all $y \in \mathbb{R}$ and $(d/dx)\left(g_{\theta}^{-1}(x) \right)> 0$ for all $x \in (0,\infty)$. In addition, if $\epsilon \in (0,\infty)$ we let $\phi(\epsilon)$ be such that 
\begin{equation}
 \mathbb{P} \left(|X| \geq \phi(\epsilon) \right) \leq \epsilon,
\end{equation}
where $X /(2^{1/3} d_{\theta}(1))$ has the GUE Tracy-Widom distribution. From (\ref{S3TWConv}) we have that
$$N^{-1/3} \left( L_1^N(\lfloor nN^{2/3}\rfloor ) - p nN^{2/3} + \lambda n^2 N^{1/3} \right)$$
converges weakly to $X$ as $N \rightarrow \infty$ for each $n \in \mathbb{Z}$. Since $X$ has a continuous cumulative distribution function (as $F_{GUE}$ is continuous) we conclude that for each $n \in \mathbb{N}$ we have
$$  \lim_{N \rightarrow \infty} \mathbb{P} \left( \left|N^{-1/3}(L_1^N(\lfloor n N^{2/3} \rfloor ) - p n N^{2/3} + \lambda n^2 N^{1/3}) \right| \geq \phi(\epsilon) \right) =  \mathbb{P} \left(|X| \geq \phi(\epsilon) \right) \leq \epsilon.$$
This shows that Assumption 2  is satisfied with the above choice of $\lambda$ and $\phi$, which concludes the proof of the theorem.

%
\subsection{Proof of Theorem \ref{PropTightGood}$(i)$ }\label{Section3.3} In this section we prove Theorem \ref{PropTightGood}$(i)$ and for the remainder we assume the same notation as in the statement of the theorem. In particular, we assume that 
\begin{equation}\label{GoodEq1}
\{ \mathfrak{L}^N = (L_1^N, \cdots, L_K^N)\}_{N =1}^\infty \mbox{ is an $(\alpha, p, \lambda)$-good sequence of discrete line ensembles,}
\end{equation}
defined on a probability space with measure $\mathbb{P}$. The general strategy behind the proof is inspired by \cite{CorHamK} and we begin by stating three key lemmas. The proofs of these lemmas are postponed until Section \ref{Section4}. All constants in the statements below implicitly depend on 
\begin{equation}\label{Depend}
\mbox{ $K, \alpha, p, \lambda, N_0, H, H^{RW}$ and the functions $\psi, \phi_1, \phi_2$ in Definition \ref{Def1},}
\end{equation}
which are fixed. We will not list this dependence explicitly.

\begin{lemma}\label{keyL1} Let $\mathbb{P}$ be the measure from the beginning of this section. For any $\epsilon > 0$, $r\geq 2$ and $k \in \llbracket 1, K \rrbracket$ there exist $\Rt_k = \Rt_k(r, \epsilon) > 0$ and $\Nt_k = \Nt_k(r, \epsilon) \in \mathbb{N}$ such that for $N \geq \Nt_k$ 
\begin{equation}\label{NoBigMax}
\mathbb{P} \left( \sup_{x \in [- rN^{\alpha}, r N^{\alpha}] }\left[L_k^N(x) - p x \right] \geq \Rt_k N^{\alpha/2} \right) < \epsilon.
\end{equation}
\end{lemma}

\begin{lemma}\label{keyL2} Let $\mathbb{P}$ be the measure from the beginning of this section. For any $\epsilon > 0$, $r \geq 2$ and $k \in \llbracket 1, K -1 \rrbracket$ there exist $\Rb_k = \Rb_k(r, \epsilon) > 0$ and $\Nb_k = \Nb_k(r, \epsilon) \in \mathbb{N}$ such that for $N \geq \Nb_k$ 
\begin{equation}\label{NoLowMin}
\mathbb{P} \left( \inf_{x \in [- rN^{\alpha}, r N^{\alpha}] }\left[L_k^N(x) - p x \right] \leq  - \Rb_k N^{\alpha/2} \right) < \epsilon.
\end{equation}
\end{lemma}

\begin{lemma}\label{keyL3} Let $\mathbb{P}$ be the measure from the beginning of this section. For any $\epsilon > 0$, $r \geq 2$ and $k \in \llbracket 1, K -1 \rrbracket$ there exist $\delta^{\mathsf{acc}}_k = \delta^{\mathsf{acc}}_k(r, \epsilon) > 0$ and $\Na_k = \Na_k(r, \epsilon) \in \mathbb{N}$ such that for $N \geq \Na_k$ 
\begin{equation}\label{AccP}
\mathbb{P} \left(  Z_{H, H^{RW}}\left(s^-_N, s^+_N, \vec{x}, \vec{y}, \infty, L^N_{k+1}\llbracket s^-_N , s^+_N \rrbracket \right) \leq  \delta^{\mathsf{acc}}_k \right) < \epsilon,
\end{equation}
where $s^{\pm}_N = \lfloor \pm r N^{\alpha} \rfloor$, $\vec{x} = (L_1^N(s^-_N), \dots, L_{k}^N(s^-_N))$, $\vec{y} =  (L_1^N(s^+_N), \dots, L_{k}^N(s^+_N))$, $L_{k+1}^N\llbracket s^-_N, s^+_N \rrbracket$ is the restriction of $L_{k+1}^N$ to the set $\llbracket s^-_N, s^+_N \rrbracket$ and $Z_{H, H^{RW}}$ is as in Definition \ref{Pfree}. Furthermore, there exist $\delta^{\mathsf{sep}}_k = \delta^{\mathsf{sep}}_k(r, \epsilon) > 0$ and $\Ns_k = \Ns_k(r, \epsilon) \in \mathbb{N}$ such that for $N \geq \Ns_k$ 
\begin{equation}\label{Sep}
\mathbb{P} \left(  \cup_{\zeta \in \{\pm 1\}} \cup_{i = 1}^{k-1} \left\{ L_i^N(  s^{\zeta}_N) - L_{i+1}^N(s^{\zeta}_N)  \leq  \delta^{\mathsf{sep}}_k \right\}  \right) < \epsilon,
\end{equation}
\end{lemma}
\begin{remark} From Lemma \ref{ContinuousGibbsCond} we have that $Z_{H, H^{RW}}\left(a, b, \vec{x}, \vec{y}, \infty, g\right) $ is a bounded measurable function of $(\vec{x}, \vec{y}, g) \in Y(\llbracket 1,k \rrbracket) \times Y(\llbracket 1,k \rrbracket) \times Y^-(\llbracket a,b\rrbracket) $, which takes values in $[0,1]$. In particular, the event in (\ref{AccP}) is measurable and its probability is well-defined.
\end{remark}

We now turn to the proof of Theorem \ref{PropTightGood}$(i)$.

\begin{proof}[Proof of Theorem \ref{PropTightGood}$(i)$] For clarity we split the proof into three steps.\\

{\bf \raggedleft Step 1.} By Lemma \ref{2Tight} to establish that ${\mathbb{P}}_N$ is tight, it suffices to verify that for all $i \in \llbracket 1,K-1\rrbracket$
\begin{equation}\label{ThmCond1}
\lim_{a\to\infty} \limsup_{N\to\infty} {\mathbb{P}}(|{f}^N_i(0)|\geq a) = 0 
\end{equation}
and also for any $r, \epsilon, \eta > 0$ and $i \in \llbracket 1,K-1\rrbracket$
\begin{equation}\label{ThmCond2}
\lim_{\delta\to 0} \limsup_{N\to\infty} {\mathbb{P}}\left(\sup_{{x,y\in [-r,r], |x-y|\leq\delta}} |{f}^N_i(x) - {f}^N_i(y)| \geq \epsilon \right) \leq \eta.
\end{equation}
Equation (\ref{ThmCond1}) is an immediate consequence of Lemmas \ref{keyL1} and \ref{keyL2}, and so we focus on (\ref{ThmCond2}). In the sequel we fix $r, \epsilon, \eta > 0$ and $i \in \llbracket 1,K-1\rrbracket$.\\

Let $t_N^{\pm} = \lfloor \pm (r+1) N^{\alpha} \rfloor$. We claim we can find $ \delta > 0$ such that for all $N$ sufficiently large we have
\begin{equation}\label{PTG2.5}
\mathbb{P} \bigg(  \sup_{\substack{ x,y \in [ t^-_N, t_N^+] \\  |x - y| \leq   \delta (t_N^+ - t_N^-) }}  \big| L_i^N(x) - L_i^N(y) - p(x - y) \big| \geq \frac{\epsilon  \sigma_p (t_N^+ -t_N^-)^{1/2}}{2   (2r+2)^{1/2}}  \bigg) \leq \eta.
\end{equation}
As usual, we are treating $L_i^N$ as the continuous curve which linearly interpolates between its values on integers. We establish (\ref{PTG2.5}) in the steps below. Here we assume its validity and deduce (\ref{ThmCond2}).\\

Using that for all large enough $N$ we have $(t_N^+ - t_N^-) N^{-\alpha} \geq 1$ we get
\begin{equation*}
 \mathbb{P}\left(\sup_{{x,y\in [-r,r], |x-y|\leq\delta}} \hspace{-2mm} |{f}^N_i(x) - {f}^N_i(y)| \geq \epsilon \right) \leq \mathbb{P} \bigg(  \sup_{\substack{ x,y \in [ t_N^-, t_N^+]  \\  |x - y| \leq \delta(t_N^+ - t_N^-) }} \hspace{-3mm} \big| L_i^N(x) - L_i^N(y) - p(x - y) \big| \geq \sigma_p \epsilon N^{\alpha/2} \bigg).
\end{equation*}
Since $t^{\pm}_N = \lfloor \pm (r+1)N^{\alpha} \rfloor$ we see that $\frac{\sigma_p \epsilon (t_N^+ - t_N^-)^{1/2}}{2 (2r+2)^{1/2}} \sim (\sigma_p \epsilon /2) N^{\alpha/2}$ as $N$ becomes large and so we conclude that for all sufficiently large $N$ we have $\frac{\sigma_p \epsilon (t_N^+ -t_N^-)^{1/2}}{2 (2r+2)^{1/2}} < \sigma_p \epsilon N^{\alpha/2}$. The observations in this paragraph and (\ref{PTG2.5}) together imply (\ref{ThmCond2}). \\

{\raggedleft \bf Step 2.} In Step 1 we reduced the proof of the theorem to establishing (\ref{PTG2.5}). This step sets up notation, which we need in the this and the next step in order to prove (\ref{PTG2.5}).

From Lemmas \ref{NoBigMax} and \ref{NoLowMin} we can find $M_1 > 0$ sufficiently large so that for all large $N$ we have
$$\mathbb{P} (E_1) \geq 1 - \eta/4,\quad \mbox{where}\quad  E_1 = \bigg\{ \max_{\zeta \in \{ +, -\} , i \in \llbracket 1, K-1 \rrbracket } \big| L^N_i(t^{\zeta}_N) - pt^{\zeta}_N \big|  \leq M_1N^{\alpha/2} \bigg\}.$$
In addition, by Lemma \ref{keyL3} we can find $\delta_1 > 0$ such that for all large $N$ we have
$$\mathbb{P} (E_2) \geq 1 - \eta/4,\quad \mbox{where}\quad   E_2 = \Big\{ Z_{H, H^{RW}}\left(t^-_N, t^+_N, \vec{x}, \vec{y}, \infty, L_{K}\llbracket t^-_N , t^+_N \rrbracket \right) > \delta_1  \Big\},$$
where $\vec{x} = (L_1^N(t^-_N), \dots, L_{K}^N(t^-_N))$, $\vec{y} =  (L_1^N(t^+_N), \dots, L_{K}^N(t^+_N))$, and we recall that $L_{K}^N\llbracket t^-_N, t^+_N \rrbracket$ is the restriction of $L_{K}^N$ to the set $\llbracket t^-_N, t^+_N \rrbracket$.

For $\delta > 0$ and any continuous curve $\ell$ on $[t_N^-, t_N^+]$ we define
$$V(\delta, \ell) = \sup_{\substack{ x,y \in [t^-_N, t^+_N] \\  |x - y| \leq  \delta (t_N^+ -t_N^-)}}  \left| \ell(x) - \ell(y) - p(x - y) \right|.$$
We assert that we can find $\delta > 0$ such that for all large $N$ we have
\begin{equation}\label{PTG7.5}
\mathbb{P} \Big{(} V(\delta, L_i^N[t_N^-, t_N^+])\geq A \big\} \cap E_1\cap E_2 \Big{)} \leq \eta/2,\quad \mbox{where}\quad  A = \frac{\sigma_p \epsilon (t_N^+ -t_N^-)^{1/2}}{2 (2r+2)^{1/2}}.
\end{equation}
In the above $L_i^N[t_N^-, t_N^+]$ denotes the restriction of $L_1^N$ to the interval $[t_N^-, t_N^+].$

Let us assume the validity of (\ref{PTG7.5}) and deduce (\ref{PTG2.5}). We have
$$\mathbb{P}\Big{(} V(\delta, L_i^N[t_N^-, t_N^+] )\geq A \Big{)} \leq \mathbb{P} \Big{(}\big\{V(\delta, L_i^N[t_N^-, t_N^+] )\geq A \big\} \cap E_1\cap E_2\Big{)} + \eta/2 < \eta,$$
where we used that $\mathbb{P}( E^c_1 ) \leq \eta/4$ and $\mathbb{P}( E^c_2) \leq \eta/4$. Identifying  $\mathbb{P}\big(V(\delta, L_i^N[t_N^-, t_N^+])\geq A \big)$ with  the left-hand side of \eqref{PTG2.5} we see that the last inequality implies (\ref{PTG2.5}).\\

{\raggedleft \bf Step 3.} In this step we establish (\ref{PTG7.5}). Let us write $F_{\delta} = \{V(\delta, L_i^N[t_N^-, t_N^+] )\geq A \big\}$.
Using the $(H,H^{RW})$-Gibbs property (see (\ref{GibbsEq})) we know that
\begin{equation}\label{PTG10}
\begin{split}
&\mathbb{P} \Big{(}\big\{V(\delta, L_i^N[t_N^-, t_N^+] )\geq A \big\} \cap E_1\cap E_2\Big{)} = \mathbb{E} \left[ \mathbb{E} \left[ {\bf 1}_{F_\delta}  {\bf 1}_{E_1}  {\bf 1}_{E_2}  \big{\vert} \mathcal{F}_{ext} ( \{i\} \times \llbracket t_N^- \hspace{-1mm}+ 1, t_N^+ \hspace{-1mm} - 1 \rrbracket) \right] \right] = \\
&\mathbb{E} \left[  {\bf 1}_{E_1}  {\bf 1}_{E_2} \mathbb{E} \left[ {\bf 1}_{F_\delta}   \big{\vert} \mathcal{F}_{ext} ( \{i\} \times \llbracket t_N^- \hspace{-1mm}+ 1, t_N^+ \hspace{-1mm} - 1 \rrbracket) \right] \right] =\mathbb{E} \left[  {\bf 1}_{E_1} \cdot {\bf 1}_{E_2} \cdot \mathbb{E}_{H, H^{RW}}\left[ {\bf 1}\{ V(\delta, \ell_i) \geq A\} \right]  \right],
\end{split}
\end{equation}
where we have written $\mathbb{E}_{H, H^{RW}}$ to stand for $\mathbb{E}_{H, H^{RW}}^{1, K-1, t_N^-, t_N^+,\vec{x}, \vec{y}, \infty,  L_K^N \llbracket t_N^-, t_N^+ \rrbracket}$ to ease the notation, with $\vec{x}, \vec{y}$ as in Step 2. The $\llbracket 1, K- 1 \rrbracket$-indexed line ensemble on $\llbracket t_N^-, t_N^+ \rrbracket$ that has law $\mathbb{P}_{H, H^{RW}}$ is denoted by $(\ell_1, \dots, \ell_{K-1})$ in (\ref{PTG10}).

In addition, from (\ref{RND}) we have
\begin{equation}\label{PTG11}
\begin{split}
 \mathbb{E}_{H, H^{RW}}\left[ {\bf 1}\{ V(\delta, \ell_i) \geq A\} \right]  = \frac{\mathbb{E}_{H^{RW}}^{1, K-1, t_N^-, t_N^+,\vec{x}, \vec{y}} \left[ W_H \cdot {\bf 1}\{ V(\delta, \tilde{\ell}_i) \geq A\}  \right]}{Z_{H, H^{RW}}\left(t^-_N, t^+_N, \vec{x}, \vec{y}, \infty, L_{K}\llbracket t^-_N , t^+_N \rrbracket \right)},
\end{split}
\end{equation}
where we have written $W_H$ in place of $W_{H}^{1, K-1, t_N^- ,t_N^+,\infty, L_{K}\llbracket t^-_N , t^+_N \rrbracket} (\tilde{\ell}_{1}, \dots, \tilde{\ell}_{K-1})$ to ease the notation. The $\llbracket 1, K- 1 \rrbracket$-indexed line ensemble on $\llbracket t_N^-, t_N^+ \rrbracket$ that has law $\mathbb{P}_{H^{RW}}^{1, K-1, t_N^-, t_N^+,\vec{x}, \vec{y}} $ is denoted by $(\tilde{\ell}_1, \dots, \tilde{\ell}_{K-1})$ in (\ref{PTG11}).

We next use the fact that $W_H \in [0,1]$ and $Z_{H, H^{RW}}\left(t^-_N, t^+_N, \vec{x}, \vec{y}, \infty, L_{K}\llbracket t^-_N , t^+_N \rrbracket \right) > \delta_1$ on $E_2$ by definition to conclude that
\begin{equation}\label{PTG12}
\begin{split}
 {\bf 1}_{E_1} {\bf 1}_{E_2} \cdot  \mathbb{E}_{H, H^{RW}}\left[ W_H \cdot{\bf 1}\{ V(\delta, \ell_i) \geq A\} \right]  \leq  {\bf 1}_{E_1} {\bf 1}_{E_2} \cdot  \frac{\mathbb{P}_{H^{RW}}^{t_N^-, t_N^+,L_i^N(t_N^-), L_i^N(t_N^+)} \left( V(\delta, \ell) \geq A  \right)}{\delta_1},
\end{split}
\end{equation}
where we used that $\tilde{\ell}_i$ under $\mathbb{P}_{H^{RW}}^{1, K-1, t_N^-, t_N^+,\vec{x}, \vec{y}}$ has the same law as $\ell$ under $\mathbb{P}_{H^{RW}}^{t_N^-, t_N^+,L_i^N(t_N^-), L_i^N(t_N^+)}$ (recall that this notation was defined in Section \ref{Section2.4}).\\

We now observe that
\begin{equation}\label{PTG13}
\begin{split}
\mathbb{P}_{H^{RW}}^{t_N^-, t_N^+,L_i^N(t_N^-), L_i^N(t_N^+)} \left( V(\delta, \ell) \geq A  \right)= \mathbb{P}_{H^{RW}}^{0, t_N^+ - t_N^-,0, L_i^N(t_N^+) - L_i^N(t_N^-)} \left(w(f^{\ell},\delta ) \geq \frac{\sigma_p \epsilon}{2(2r+2)^{1/2}}  \right),
\end{split}
\end{equation}
where on the right side $\ell$ is $\mathbb{P}_{H^{RW}}^{0, t_N^+ - t_N^-,0, L_i^N(t_N^+) - L_i^N(t_N^-)} $-distributed and we used the notation $f^\ell$ from Lemma \ref{MOCLemmaS4}. In deriving the above equation we used the definition of $A$, as well as the fact that if two random curves $\ell$ and $\tilde{\ell}$ are distributed according to $\mathbb{P}_{H^{RW}}^{t_1, t_2,x,y}$ and $\mathbb{P}_{H^{RW}}^{0, t_2 - t_1,0,y - x}$ then they have the same distribution except for a re-indexing and a vertical shift by $x$ -- hence their modulus of continuity has the same distribution. Notice that on the event $E_1$ we have that
$$
| L_i^N(t_N^+) - L_i^N(t_N^-)  - p(t_N^+ - t_N^-)| \leq 2M_1 N^{\alpha/2} \leq 2M_1 (t_N^+- t_N^-)^{1/2}.
$$
The latter and Lemma \ref{MOCLemmaS4} (applied to $\eta = \delta_1 (\eta/2)$, $\epsilon =  \frac{\sigma_p \epsilon}{2(2r+2)^{1/2}}$, $M = 2M_1$, $p$ as in the statement of the theorem and $T = t_N^+ - t_N^-$) together imply that we can find $\delta > 0$ sufficiently small such that for all large enough $N$ we have
\begin{equation}\label{PTG14}
\begin{split}
{\bf 1}_{E_1} \cdot \mathbb{P}_{H^{RW}}^{0, t_N^+ - t_N^-,0, L_i^N(t_N^+) - L_i^N(t_N^-)} \left(w(f^{\ell},\delta ) \geq \frac{\sigma_p \epsilon}{2(2r+2)^{1/2}}  \right) \leq{\bf 1}_{E_1} \cdot \delta_1 \eta/2.
\end{split}
\end{equation}

Combining (\ref{PTG12}), (\ref{PTG13}), (\ref{PTG14}) we see that
\begin{equation*}
\begin{split}
 {\bf 1}_{E_1} \cdot {\bf 1}_{E_2} \cdot  \mathbb{E}_{H, H^{RW}}\left[ {\bf 1}\{ V(\delta, \ell_i) \geq A\} \right]  \leq  {\bf 1}_{E_1} \cdot {\bf 1}_{E_2} \cdot  \eta/2,
\end{split}
\end{equation*}
which together with (\ref{PTG10}) implies (\ref{PTG7.5}). This suffices for the proof.
\end{proof}

%
\subsection{Proof of Theorem \ref{PropTightGood}$(ii)$ }\label{Section3.4} In this section we prove Theorem \ref{PropTightGood}$(ii)$. We continue with the same notation as in Section \ref{Section3.3}. 

We begin by summarizing a bit of notation in the following definition.
\begin{definition}\label{scaledRW}
Let $N, k \in \mathbb{N}$, $p \in \mathbb{R}$, $\alpha > 0$, $a,b\in N^{-\alpha}\mathbb{Z}$ with $a<b$, $\vec{x}, \vec{y} \in \mathbb{R}^k$ and set $A = aN^{\alpha}, B = bN^{\alpha}$. Let $\mathfrak{X}^N = (X^N_1, \dots, X^N_k)$ be a $\llbracket 1, k \rrbracket$-indexed discrete line ensemble on $\llbracket A, B \rrbracket$  with law $\mathbb{P}^{1, k, A, B, \vec{x}, \vec{y}, f, g}_{H, H^{RW}}$ as in Definition \ref{Pfree}, where $f \in Y^+(\llbracket A, B \rrbracket)$ and $g \in Y^-(\llbracket A, B \rrbracket)$. Here we assume that $H^{RW}$ is as in Definition \ref{AssHR} and $H$ as in Definition \ref{AssH}.

Let $\mathcal{Y}^N = (\mathcal{Y}^N_1, \dots, \mathcal{Y}^N_k)$ be the $\llbracket 1, k\rrbracket$-indexed line ensemble on $[a,b]$, defined via
\begin{equation}\label{ContLE}
\mathcal{Y}^N_i(t) = N^{-\alpha/2} \sigma_p^{-1} \cdot \left( X^N_i(tN^{\alpha}) - p s N^{\alpha} \right), \quad t\in [a,b],
\end{equation}
where $\sigma_p$ is as in Definition \ref{AssHR}. We denote the law of $\mathcal{Y}^N$ by $\mathbb{P}^{\mathsf{scaled}}_N$ for brevity, and write $\mathbb{E}^{\mathsf{scaled}}_N$ for the expectation with respect to this measure.
\end{definition}

We next state a technical lemma, which will be required for the proof of Theorem \ref{PropTightGood}$(ii)$, and whose proof is postponed until Section \ref{Section7} -- see Lemma \ref{S7scaledavoidBB}. In plain words, the lemma states that if the boundary data of the measures $\mathbb{P}^{\mathsf{scaled}}_N$ from Definition \ref{scaledRW} converge as $N \rightarrow \infty$, then $\mathbb{P}^{\mathsf{scaled}}_N$ converge weakly to the law of avoiding Brownian bridges with the boundary limiting data, as in Definition \ref{DefAvoidingLaw}.
\begin{lemma}\label{scaledavoidBB}
Fix $k\in\mathbb{N}$, $p \in \mathbb{R}$, $\alpha > 0$, and $a,b\in\mathbb{R}$ with $a<b$. Let $f_{\infty}:[a-1,b+1]\to(-\infty,\infty]$, $g_{\infty}:[a-1,b+1]\to[-\infty,\infty)$ be continuous functions such that $f_{\infty}(t) > g_{\infty}(t)$ for all $t\in[a-1,b+1]$. Let $\vec{x},\vec{y}\in W_k^\circ$ be such that $f_{\infty}(a) > x_1$, $f_{\infty}(b) > y_1$, $g_{\infty}(a) < x_k$, and $g_{\infty}(b) < y_k$. Let $a_N = \lfloor aN^\alpha\rfloor N^{-\alpha}$ and $b_N = \lceil bN^\alpha\rceil N^{-\alpha}$, and let $f_N : [a-1,b+1]\to(-\infty,\infty]$ and $g_N : [a-1,b+1]\to[-\infty,\infty)$ be continuous functions such that $f_N\to f_{\infty}$ and $g_N\to g_{\infty}$ uniformly on $[a-1,b+1]$. If $f_{\infty} \equiv \infty$ the last statement means that $f_N \equiv \infty$ for all large enough $N$ and if $g_{\infty} \equiv - \infty$ the latter means that $g_N \equiv -\infty$ for all large enough $N$. 

Lastly, let $\vec{x}\,^N, \vec{y}\,^N \in \mathbb{R}^k$, write $\tilde{x}^N_i =  N^{-\alpha/2} \sigma_p^{-1} (x_i^N - pa_N N^{\alpha})$, $\tilde{y}^N_i = N^{-\alpha/2} \sigma_p^{-1} (y_i^N - pb_N N^{\alpha})$, and suppose that $\tilde{x}^N_i \to x_i$ and $\tilde{y}^N_i \to y_i$ as $N\to\infty$ for each $i\in\llbracket 1,k\rrbracket$. 

Let $\mathcal{Y}^N$ have laws $\mathbb{P}^{\mathsf{scaled}}_N$ as in Definition \ref{scaledRW} for $\vec{x} = \vec{x}^N, \vec{y} = \vec{y}^N$, $a = a_N$, $b = b_N$, $f,g$ given by
\begin{equation*}
\begin{split}
&N^{-\alpha/2} \sigma_p^{-1}( f(xN^{\alpha})  - p x N^{\alpha}) = f_N(x) \mbox{ and } N^{-\alpha/2} \sigma_p^{-1}( g(xN^{\alpha})  - p x N^{\alpha}) = g_N(x) \mbox{ for $x \in [a_N,  b_N]$}.
\end{split}
\end{equation*}
Let $\mathcal{Z}^N = \mathcal{Y}^N|_{\llbracket 1, k\rrbracket \times[a,b]}$, i.e. $\mathcal{Z}^N$ is a sequence of random variables on $C(\llbracket 1, k \rrbracket \times [a,b])$ obtained by projecting $\mathcal{Y}^N$ to $\llbracket 1, k\rrbracket \times[a,b]$, then the laws of $\mathcal{Z}^N$ converge weakly to $\mathbb{P}^{a,b,\vec{x},\vec{y},f_{\infty},g_{\infty}}_{avoid}$ as $N\to\infty$.
\end{lemma}

With the above result in place we are ready to give the proof of Theorem \ref{PropTightGood}$(ii)$.

\begin{proof}[Proof of Theorem \ref{PropTightGood}$(ii)$] As the proof we present is quite similar to the proof of \cite[Theorem 2.26$(ii)$]{DREU} we will omit some of the details. 

From our work in Section \ref{Section3.3} we know that the sequence of $\llbracket 1, K -1 \rrbracket$-indexed line ensembles $\mathcal{L}^N = ({f}^N_1,\dots,{f}^N_{K-1}) $ on $\mathbb{R}$ is tight. Since $\mathcal{L}^\infty$ is a weak subsequential limit of $\mathcal{L}^N$ by possibly passing to a subsequence we may assume that $\mathcal{L}^N \implies \mathcal{L}^\infty$. We will still call the subsequence $\mathcal{L}^N$. By the Skorokhod representation theorem \cite[Theorem 6.7]{Bill}, we can also assume that $\mathcal{L}^N$ and $\mathcal{L}^\infty$ are all defined on the same probability space with measure $\mathbb{P}$ and the convergence is $\mathbb{P}$-almost sure. Here we are implicitly using Lemma \ref{Polish}  from which we know that the random variables $\mathcal{L}^N$ and $\mathcal{L}^\infty$ take value in a Polish space so that the Skorokhod representation theorem is applicable. 

Fix a set $\Sigma_1 = \llbracket k_1,k_2\rrbracket \subseteq \llbracket 1, K-2\rrbracket$ and $a,b\in\mathbb{R}$ with $a<b$. We also fix a bounded Borel-measurable function $F:C(\Sigma_1 \times[a,b])\to\mathbb{R}$. To complete the proof we need to show that $\mathcal{L}^{\infty}$ is non-intersecing almost surely and also
\begin{equation}\label{BGPcondex}
\mathbb{E}[F(\mathcal{L}^\infty|_{\Sigma_1 \times[a,b]})\,|\,\mathcal{F}_{ext}(\Sigma_1 \times(a,b))] = \mathbb{E}^{a,b,\vec{x},\vec{y},f,g}_{avoid}[F(\mathcal{Q})] \hspace{3mm} \mbox{ $\mathbb{P}$-a.s., }
\end{equation}
where $\vec{x} = ({f}^\infty_{k_1}(a),\dots,{f}^\infty_{k_2}(a))$, $\vec{y} = ({f}^\infty_{k_1}(b),\dots,{f}^\infty_{k_2}(b))$, $f={f}^\infty_{k_1-1}$ (with ${f}^\infty_0 = +\infty$), $g={f}^\infty_{k_2+1}$, the $\sigma$-algebra $\mathcal{F}_{ext}(\Sigma_1 \times(a,b))$ is as in Definition \ref{DefPBGP}, and $\mathcal{Q}$ has law $\mathbb{P}^{a,b,\vec{x},\vec{y},f,g}_{avoid}$. We mention that by Lemma \ref{keyL3} we have $\mathbb{P}$-a.s. that $\mathcal{L}^\infty_i (x) > \mathcal{L}^\infty_{i+1} (x)$ for all $i \in \llbracket 1, K-2\rrbracket$ and $x \in \mathbb{R}$ so that the right side of (\ref{BGPcondex}) is well-defined. In addition, we mention that if (\ref{BGPcondex}) holds then by a simple monotone class argument we would have for any bounded Borel-measurable function $H: C(\llbracket 1, K -1 \rrbracket \times [a,b]) \rightarrow \mathbb{R}$ 
\begin{equation}\label{S4Fixed1}
\mathbb{E} \left[ H(\mathcal{L}^{\infty}|_{\llbracket 1, K- 1 \rrbracket \times [a,b]}) | \mathcal{F}_{ext}( \llbracket 1, K- 2 \rrbracket \times (a,b)) \right] =  \mathbb{E}^{a,b,\vec{x},\vec{y},f,g}_{avoid}[H(\mathcal{Q},g)] \hspace{3mm} \mbox{ $\mathbb{P}$-a.s., }
\end{equation}
where on the right side $(\mathcal{Q},g)$ is the line ensemble with $K-1$ curves, whose top $K-2$ curves agree with $\mathcal{Q}$ and the $(K-1)$-st one agrees with $g$. By letting $H(f_1, \dots, f_{K-1}) = {\bf 1} \{f_1(s) > \cdots > f_{K-1}(s) \mbox{ for all }s \in [a,b] \}$ in (\ref{S4Fixed1}) we would conclude that $\mathcal{L}^{\infty}$ is almost surely non-intersecting. Thus what remains is to prove (\ref{BGPcondex}). 
	
Fix $m\in\mathbb{N}$, $n_1,\dots,n_m\in\Sigma$, $t_1,\dots,t_m\in\mathbb{R}$, and $h_1,\dots,h_m : \mathbb{R}\to\mathbb{R}$ bounded continuous functions. Define $S = \{i\in\llbracket 1,m\rrbracket : n_i \in \Sigma_1, t_i \in [a,b]\}$. We will prove that
\begin{equation}\label{BBcondexsplit}
\mathbb{E}\left[\prod_{i=1}^m h_i({f}^\infty_{n_i}(t_i))\right] = \mathbb{E}\left[\prod_{s\notin S} h_s({f}^\infty_{n_s}(t_s))\cdot\mathbb{E}^{a,b,\vec{x},\vec{y},f,g}_{avoid}\left[\prod_{s\in S} h_s(Q_{n_s}(t_s))\right]\right],
\end{equation}
where $Q$ denotes a random variable with law $\mathbb{P}^{a,b,\vec{x},\vec{y},f,g}_{avoid}$. The fact that (\ref{BBcondexsplit}) implies (\ref{BGPcondex}) is a standard monotone class argument and so we omit it. The interested reader is referred to Step 2 in the proof of  \cite[Theorem 2.26$(ii)$]{DREU}, which can be repeated verbatim to show that (\ref{BBcondexsplit}) implies (\ref{BGPcondex}).\\

In the remainder we focus on proving (\ref{BBcondexsplit}). From the $\mathbb{P}$-a.s. convergence of $\mathcal{L}^N$ to $\mathcal{L}^{\infty}$ we have
\begin{equation}\label{BGPweak}
\mathbb{E}\left[\prod_{i=1}^m h_i({f}^\infty_{n_i}(t_i))\right] =  \lim_{N\to\infty}\mathbb{E}\left[\prod_{i=1}^m h_i(\mathcal{L}^N_{n_i}(t_i))\right].
\end{equation}
We define the sequences $a_N = \lfloor aN^\alpha\rfloor N^{-\alpha}$, $b_N = \lceil bN^\alpha\rceil N^{-\alpha}$, $\vec{x}^N = (L_{k_1}^N(a_N),\dots,L_{k_2}^N(a_N))$, $\vec{y}^N = (L_{k_1}^N(b_N),\dots,L_{k_2}^N(b_N))$, $f_N = {f}_{k_1-1}^N$ (where ${f}^N_0 = +\infty$), $g_N = {f}_{k_2+1}^N$. Since the line ensemble $(L_1^N,\dots,L_{K-1}^N)$ in the definition of $\mathcal{L}^N$ satisfies the $(H, H^{RW})$-Gibbs property of Definition \ref{DefLGGP}, we conclude that that the law of $\mathcal{L}^N|_{\Sigma_1 \times[a_N,b_N]}$, conditioned on the $\sigma$-algebra $ \mathcal{F} = \mathcal{F}_{ext} ( \Sigma_1  \times \llbracket \lfloor aN^\alpha\rfloor +1 ,\lceil bN^\alpha\rceil  -1 \rrbracket)$ as in (\ref{GibbsCond}) is (upto a reindexing of the curves) precisely $\mathbb{P}^{\mathsf{scaled}}_N$ as in Definition \ref{scaledRW} for $\vec{x} = \vec{x}^N, \vec{y} = \vec{y}^N$, $a = a_N$, $b = b_N$ and $f,g$ given by
\begin{equation*}
\begin{split}
&N^{-\alpha/2} \sigma_p^{-1}( f(xN^{\alpha})  - p x N^{\alpha}) = f_N(x) \mbox{ and } N^{-\alpha/2} \sigma_p^{-1}( g(xN^{\alpha})  - p x N^{\alpha}) = g_N(x) \mbox{ for $x \in [a_N,  b_N]$}.
\end{split}
\end{equation*}
Therefore, writing $Z^N$ for a random variable with this law, we have for large $N$
\begin{equation}\label{BBschur}
\begin{split}
&\mathbb{E}\left[\prod_{i=1}^m h_i(\mathcal{L}^N_{n_i}(t_i))\right]  = \mathbb{E}\left[ \prod_{s\notin S} h_s(\mathcal{L}^N_{n_s}(t_s)) \mathbb{E}\left[\prod_{s\in S}  h_s(\mathcal{L}^N_{n_s}(t_s))\Big{|} \mathcal{F} \right]  \right] =  \\
& \mathbb{E}\left[\prod_{s\notin S} h_s(\mathcal{L}^N_{n_s}(t_s))\cdot\mathbb{E}^{\mathsf{scaled}}_N \left[\prod_{s\in S} h_s(Z^N_{n_s-k_1+1}(t_s))\right]\right].
\end{split}
\end{equation}
We mention that in the deriving the first equality of (\ref{BBschur}) we used that $a_N \to a$, $b_N\to b$ so that $t_s < a_N$ or $t_s > b_N$ for all $s\notin S$ with $n_s \in \Sigma_1$, provided that $N$ is sufficiently large. We also used the tower property of conditional expectations.

From Lemma \ref{keyL3} we have $\mathbb{P}$-a.s. that ${f}^\infty_{k_1 -1}(a) > {f}^\infty_{k_1}(a) > \cdots > {f}^\infty_{k_2}(a) > {f}^\infty_{k_2+1}(a)$, ${f}^\infty_{k_1 -1}(b) > {f}^\infty_{k_1}(b) > \cdots > {f}^\infty_{k_2}(b) > {f}^\infty_{k_2+1}(b)$ and $f_{k_1-1}^{\infty}(x) >f_{k_2 + 1}^{\infty}(x)$ for all $x \in \mathbb{R}$.  In addition, we have by part (i) of Theorem \ref{PropTightGood} that $\mathbb{P}$-almost surely $f_N\to  {f}^\infty_{k_1 - 1}$ and $g_N\to  {f}^\infty_{k_2 + 1}$ uniformly on $[a-1,b+1]\supseteq [a_N,b_N]$, and $N^{-\alpha/2} \sigma_{p}^{-1} (x_i^N - pa_N N^{\alpha})\to {f}_i^{\infty}(a)  $, $N^{-\alpha/2} \sigma_{p}^{-1} (y_i^N - pb_N N^{\alpha})\to {f}_i^{\infty}(b) $ for $i\in\llbracket 1,K-1\rrbracket$. It follows from Lemma \ref{scaledavoidBB} that $\mathbb{P}$-almost surely
\begin{equation}\label{BGPNweak}
\lim_{N\to\infty}\mathbb{E}^{\mathsf{scaled}}_N \left[\prod_{s\in S} h_s(Z^N_{n_s-k_1+1}(t_s))\right] = \mathbb{E}^{a,b,\vec{x},\vec{y},f,g}_{avoid}\left[\prod_{s\in S} h_s(Q_{n_s}(t_s))\right].
\end{equation}
Lastly, the continuity of the $h_i$'s implies that
\begin{equation}\label{BGPuniform}
\lim_{N\to\infty}\prod_{s\notin S} h_s(\mathcal{L}_{n_s}^N(t_s)) = \prod_{s\notin S} h_s({f}^\infty_{n_s}(t_s)).
\end{equation}
Combining \eqref{BGPweak}, \eqref{BBschur}, \eqref{BGPNweak}, and \eqref{BGPuniform} with the bounded convergence theorem proves \eqref{BBcondexsplit}. This suffices for the proof.
	
\end{proof}

%
\section{Proof of three key lemmas}\label{Section4} In this section we give the proof of Lemmas \ref{keyL1}, \ref{keyL2} and \ref{keyL3}. We continue with the same notation as in Section \ref{Section3.3}. Observe that Lemma \ref{keyL1} consists of $K$ statements, one for each $k \in \llbracket 1 ,K \rrbracket$, while Lemmas \ref{keyL2} and \ref{keyL3} consist of $K-1$ statements, one for each $k \in \llbracket 1, K - 1\rrbracket$.

 For $m \in \llbracket 1, K \rrbracket$ we let Lemma \ref{keyL1}[$m$] denote the statement in Lemma \ref{keyL1} for $k \in \llbracket 1, m \rrbracket$. Analogously, for $m \in \llbracket 1 , K-1\rrbracket$ we let Lemma \ref{keyL2}[$m$] and Lemma \ref{keyL3}[$m$] denote the statement in Lemma \ref{keyL2} and Lemma \ref{keyL3}, respectively, for $k \in \llbracket 1, m \rrbracket$. We will establish the following three implications 
\begin{equation}\label{Imp1}
  \begin{aligned}
   &\mbox{Lemma \ref{keyL1}[$m +1$]} \\
    & \mbox{Lemma \ref{keyL2}[$m + 1$]} \\
  \end{aligned} 
\implies 
 \begin{aligned}
   &\mbox{Lemma \ref{keyL1}[$m +2$]} \\
  \end{aligned} \hspace{5mm} \mbox{ for $m \in \llbracket 0, K-2 \rrbracket$;}
\end{equation}
\begin{equation}\label{Imp2}
  \begin{aligned}
   \begin{aligned}
   &\mbox{Lemma \ref{keyL1}[$m +1$]}\\
    & \mbox{Lemma \ref{keyL2}[$m$]} \\
    &\mbox{Lemma \ref{keyL3}[$m$]}
  \end{aligned} 
  \end{aligned}
\implies 
 \begin{aligned}& \mbox{Lemma \ref{keyL2}[$m + 1$]} \\
  \end{aligned} \hspace{5mm}  \mbox{ for $m \in \llbracket 1, K-2 \rrbracket$;}
\end{equation}
\begin{equation}\label{Imp3}
  \begin{aligned}
   \begin{aligned}
   &\mbox{Lemma \ref{keyL1}[$m +2$]}\\
    & \mbox{Lemma \ref{keyL2}[$m + 1$]} \\
  \end{aligned} 
  \end{aligned}
\implies 
 \begin{aligned}
    &\mbox{Lemma \ref{keyL3}[$m + 1$]} 
  \end{aligned} \hspace{5mm} \mbox{ for $m \in \llbracket 0, K-2 \rrbracket$.}
\end{equation}
In addition, we will separately prove that 
\begin{equation}\label{Imp4}
\mbox{ Lemma \ref{keyL1}[$1$] and Lemma \ref{keyL2}[$1$] both hold}.
\end{equation}

Assuming that (\ref{Imp1}), (\ref{Imp2}), (\ref{Imp3}) and (\ref{Imp4}) all hold we would conclude the validity of Lemmas \ref{keyL1}, \ref{keyL2} and \ref{keyL3} as we explain here. For $m \in \llbracket 1, K-1 \rrbracket$ let $\mathsf{S}_m$ denote the statement
$$\mathsf{S}_m := ``\mbox{ Lemma \ref{keyL1}[$m+1$],  Lemma \ref{keyL2}[$m$] and Lemma \ref{keyL3}[$m$] all hold ''}.$$
In view of (\ref{Imp4}) we know that Lemma \ref{keyL1}[$1$] and Lemma \ref{keyL2}[$1$] both hold. This would imply from (\ref{Imp1}) that Lemma \ref{keyL1}[$2$] holds and then from (\ref{Imp3}) we have Lemma \ref{keyL3}[$1$] as well. This proves $\mathsf{S}_1$. Assuming the validity of $\mathsf{S}_m$  we establish $\mathsf{S}_{m+1}$ by showing that 
\begin{equation*}
\begin{aligned}
   &\mbox{Lemma \ref{keyL1}[$m +1$]} \\
    & \mbox{Lemma \ref{keyL2}[$m$]} \\
    &\mbox{Lemma \ref{keyL3}[$m$]}
  \end{aligned}\implies  \begin{aligned}
   &\mbox{Lemma \ref{keyL1}[$m +1$]} \\
    & \mbox{Lemma \ref{keyL2}[$m + 1$]} \\
    &\mbox{Lemma \ref{keyL3}[$m$]}
  \end{aligned}
\implies 
 \begin{aligned}
   &\mbox{Lemma \ref{keyL1}[$m +2$]} \\
    & \mbox{Lemma \ref{keyL2}[$m + 1$]} \\
    &\mbox{Lemma \ref{keyL3}[$m$]}
  \end{aligned}
\implies 
 \begin{aligned}
   &\mbox{Lemma \ref{keyL1}[$m +2$]}\\
    & \mbox{Lemma \ref{keyL2}[$m + 1$]} \\
    &\mbox{Lemma \ref{keyL3}[$m + 1$]}
  \end{aligned},
\end{equation*}
where the first implication used (\ref{Imp2}), the second (\ref{Imp3}) and the third (\ref{Imp4}). Since $\mathsf{S}_1$ is true we conclude by induction on $m$ that $\mathsf{S}_m$ all hold for $m \in \llbracket 1, K-1\rrbracket$, and $\mathsf{S}_{K-1}$ is precisely the statement that all three lemmas hold.

The above paragraph explains why it is enough to prove (\ref{Imp1}), (\ref{Imp2}), (\ref{Imp3}) and (\ref{Imp4}). We establish these four statement in the next four sections.

%

\subsection{Proof of (\ref{Imp4})}\label{Section4.2} In this section we present the proof of (\ref{Imp4}), for which we require the following two results, whose proofs are given in Section \ref{Section7} -- see Lemmas \ref{S7LStayInBand} and \ref{S7LNoDip}.

\begin{lemma}\label{LStayInBand} Let $\ell$ have distribution $\mathbb{P}_{H^{RW}}^{T_0,T_1, x,y}$(recall this was defined in Section \ref{Section2.4}) with $H^{RW}$ satisfying the assumptions in Definition \ref{AssHR}. For any $\epsilon \in (0,1)$, $p \in \mathbb{R}$ and $M > 0$ there exist $A = A(\epsilon,p,H^{RW})> 0$ and $W_1 = W_1(M, p, \epsilon,H^{RW}) \in \mathbb{N}$  such that the following holds. For $T_1 - T_0 \geq W_1$, $x, y \in \mathbb{R}$ with $|x - p T_0| \leq M (T_1-T_0)^{1/2}$ and $|y - p T_1| \leq M(T_1-T_0)^{1/2}$ we have
\begin{equation}\label{EStayInBand}
\mathbb{P}^{T_0,T_1,x,y}_{H^{RW}} \left( \sup_{ s \in [T_0, T_1] }  \left| \ell(s) - \frac{T_1 - s }{T_1 - T_0} \cdot x - \frac{s - T_0 }{T_1 - T_0} \cdot y  \right| \leq A (T_1 - T_0)^{1/2} \right) \geq 1 - \epsilon.
\end{equation}
\end{lemma} 
\begin{remark}
Lemma \ref{LStayInBand} states that a random walk bridge between well-behaved endpoints $(T_0,x)$ and $(T_1,y)$ is unlikely to deviate much from the straight segment connecting these two points.
\end{remark}

\begin{lemma}\label{LNoDip} Fix $k \in \mathbb{N}$ and let $\mathfrak{L} = (L_1, \dots, L_k)$ have law $\mathbb{P}_{H,H^{RW}}^{1, k, T_0 ,T_1, \vec{x}, \vec{y},\infty,g}$ as in Definition \ref{Pfree}, where we assume that $H^{RW}$ is as in Definition \ref{AssHR} and $H$ is as in Definition \ref{AssH}. Let $ p\in\mathbb{R}$, $r, \epsilon \in (0,1)$ and $M^{\mathsf{side}} > 0$ be given. Then we can find constants $W_2 = W_2(k,p,r,\epsilon, M^{\mathsf{side}}, H, H^{RW}) \in \mathbb{N}$ and $M^{\mathsf{dip}} =  M^{\mathsf{dip}}(k,p,r,\epsilon, M^{\mathsf{side}}, H, H^{RW})> 0$ so that the following holds. 

 For any $T_0, T_1 \in \mathbb{Z}$ with $ T_1 - T_0 \geq W_2$, $r (T_1 - T_0) \leq R \leq  r^{-1} (T_1 - T_0)$, $\vec{x}, \vec{y} \in \mathbb{R}^k$  that satisfy
$$ \left| x_i - pT_0 \right| \leq M^{\mathsf{side}} R^{1/2}, \left| y_i - pT_1 \right| \leq M^{\mathsf{side}} R^{1/2} \mbox{ for $i \in \llbracket 1, k \rrbracket$, }$$ 
and $g \in Y^{-}(\llbracket T_0, T_1 \rrbracket)$ we have that 
\begin{equation}\label{ENoDip}
\mathbb{P}_{H,H^{RW}}^{1, k, T_0 ,T_1, \vec{x}, \vec{y},\infty,g} \left( L_k(x) - p x  \geq - M^{\mathsf{dip}} R^{1/2}  \mbox{ for all $x \in [T_0, T_1]$}\right) \geq 1 - \epsilon.
\end{equation}
\end{lemma}
\begin{remark}
Lemma \ref{LNoDip} states that the highest-indexed curve of a discrete line ensemble with law $\mathbb{P}_{H,H^{RW}}^{1, k, T_0 ,T_1, \vec{x}, \vec{y},\infty,g}$, whose entrance and exit data $\vec{x}, \vec{y}$ are well-behaved is unlikely to dip too low.
\end{remark}

\begin{proof}[Proof of \eqref{Imp4}, Lemma \ref{keyL1}] We continue with the same notation as in the statement of the lemma and Section \ref{Section3.3}. All constants in the proof depend on the quantities in (\ref{Depend}) as well as $r \geq 2$ and $\epsilon > 0$ as in the statement of the lemma, which are fixed. For clarity we split the proof into two steps.\\

{\bf \raggedleft Step 1.} Let $r_0 = \lceil r \rceil +1$ and put $u_N = \lfloor r_0 N^{\alpha} \rfloor$, $s_N = \lfloor 2r_0 N^{\alpha} \rfloor$, $t_N = \lfloor 4r_0 N^{\alpha} \rfloor$. Since $\mathfrak{L}^N$ is an $(\alpha,p,\lambda)$-good sequence we can find $M_1 > 0$ and $N_1 \geq N_0$ such that for $N \geq N_1$ we have 
\begin{equation}\label{S4S1E1}
\begin{split}
&\mathbb{P}(\mathsf{D}_0^N) < \epsilon/4 \mbox{, }1 - \mathbb{P}(\mathsf{D}_1^N) < \epsilon/4, \mbox{ and $T_N \geq t_N$ where } \\
&\mathsf{D}_0^N= \{ L^N_1(s_N )- p s_N  > M_1 N^{ \alpha/2}\} \mbox{ and }\mathsf{D}_1^N= \{|L^N_1(t_N )-pt_N |\leq M_1  N^{ \alpha/2}\}.
\end{split}
\end{equation}
Let $A(\epsilon/2, p, H^{RW})$ be as in Lemma \ref{LStayInBand}, and let $\Rt_1$, $N_2 \geq N_1$ be sufficiently large so that for $N \geq N_2$, and all $q \in \llbracket - u_N, u_N \rrbracket$ we have 
\begin{equation}\label{S4S1RLarge}
\frac{T_1 - s_N }{T_1 - T_0} \cdot \Rt_1 N^{\alpha/2}  -  \frac{s_N - T_0 }{T_1 - T_0} \cdot M_1 N^{\alpha/2} > A (T_1 - T_0)^{1/2} + M_1 N^{\alpha/2},
\end{equation}
where $T_1 = t_N$ and $T_0 = q$. Observe that such a choice is possible since for large $N$, depending on $r$ and $\alpha$, we have $rN^{\alpha} \leq t_N - s_N = T_1 - s_N \leq  T_1 - T_0 \leq u_N + t_N \leq 5 r_0N^{\alpha}$. This specifies $\Rt_1$. 

From Lemma \ref{LStayInBand} applied to $\epsilon = \epsilon/2$, $p, H^{RW}$ as in the present lemma and $M = \max (M_1, \Rt_1)$ we can find $N_3 \geq N_2$ such that for $N \geq N_3$ we have $t_N - u_N \geq W_1$ as in that lemma. We set $\Nt_1 = N_3$ and this specifies our choice of $\Nt_1$. We proceed to prove Lemma \ref{keyL1}[$1$] with this choice of $\Rt_1$ and $\Nt_1$.\\

Define the event
\begin{align*}
\mathsf{High}^{N}\coloneqq \left\{\max_{x\in \llbracket - u_N, u_N \rrbracket }\left[ L^N_1(x)-p  x \right]\geq  \Rt_1 N^{\alpha/2} \right\}.
\end{align*} 
We claim that for $N \geq \Nt_1$ we have
\begin{equation}\label{S4S1E2}
\mathbb{P}\left( \mathsf{High}^{N} \cap \mathsf{D}_1^N \cap (\mathsf{D}_0^N)^c \right) \leq \epsilon/2.
\end{equation}
We prove (\ref{S4S1E2}) in the second step. Here we assume its validity and prove Lemma \ref{keyL1}[$1$].\\

Using that $u_N \geq rN^{\alpha}$ and that $L^N_1$ is a linear interpolation of its values on $\mathbb{Z}$ we get for $N \geq \Nt_1$
\begin{equation*}
\begin{split}
&\mathbb{P} \left( \sup_{x \in [- rN^{\alpha}, r N^{\alpha}] }\left[L_1^N(x) - p x \right] \geq \Rt_1 N^{\alpha/2} \right) \leq \mathbb{P} \left( \sup_{x \in [-u_N, u_N] }\left[L_1^N(x) - p x \right] \geq \Rt_1 N^{\alpha/2} \right) = \\
&\mathbb{P}\left( \mathsf{High}^{N} \right) \leq \mathbb{P}\left( \mathsf{High}^{N} \cap \mathsf{D}_1^N \cap (\mathsf{D}_0^N)^c \right) + \mathbb{P} \left( (\mathsf{D}_1^N)^c  \right) + \mathbb{P} \left( \mathsf{D}_0^N  \right) < \epsilon,
\end{split}
\end{equation*}
where in the last inequality we used (\ref{S4S1E1}) and (\ref{S4S1E2}). The last equation gives Lemma \ref{keyL1}[$1$].\\

{\bf \raggedleft Step 2.} Define for $q \in \llbracket -u_N, u_N \rrbracket$ the events 
\begin{align*}
\mathsf{High}^{N}_q\coloneqq \left\{  L^N_1(q)-  p q  \geq  R^{\textsf{top}}_1N^{\alpha/2} \right\}\cap \left\{\max_{x\in \llbracket -u_N , q-1 \rrbracket }\left[ L^N_1(x)-p  x \right]<  R^{\textsf{top}}_1 N^{\alpha/2} \right\},
\end{align*}
and note that $\mathsf{High}^{N} = \sqcup_{q \in  \llbracket -u_N, u_N \rrbracket} \mathsf{High}^{N}_q$. We will prove that for $N \geq \Nt_1$
\begin{equation}\label{S4S1E3}
\mathbb{P}\left( \mathsf{High}^{N}_q \cap \mathsf{D}_1^N \cap (\mathsf{D}_0^N)^c \right) \leq (\epsilon/2) \cdot \mathbb{P}\left( \mathsf{High}^{N}_q \cap \mathsf{D}_1^N \right),
\end{equation}
which if true would imply (\ref{S4S1E2}) upon summation over $q \in \llbracket -u_N, u_N \rrbracket$.

We observe from the $(H,H^{RW})$-Gibbs property that if $N \geq \Nt_1$ (and hence $T_N \geq t_N$ from (\ref{S4S1E1}))
\begin{equation}\label{S4S1E4}
\begin{split}
&\mathbb{P}\left( \mathsf{High}^{N}_q \cap \mathsf{D}_1^N \cap (\mathsf{D}_0^N)^c \right) = \mathbb{E}\left[ \mathbb{E} \left[ {\bf 1}_{\mathsf{High}^{N}_q}  {\bf 1}_{\mathsf{D}^N_1}  {\bf 1}_{(\mathsf{D}^N_0)^c} \Big{|} \mathcal{F}_ q\right] \right] = \mathbb{E}\left[{\bf 1}_{\mathsf{High}^{N}_q} {\bf 1}_{\mathsf{D}^N_1}  \mathbb{E} \left[  {\bf 1}_{(\mathsf{D}^N_0)^c} \Big{|} \mathcal{F}_ q\right] \right]  \\
& = \mathbb{E}\left[{\bf 1}_{\mathsf{High}^{N}_q} \cdot {\bf 1}_{\mathsf{D}^N_1} \cdot \mathbb{P}_{H,H^{RW}}^{1,1,q, t_N,x,y,\infty,g}(\tilde{L}_1(s_N) - p s_N \leq M_1 N^{ \alpha/2} )\right],
\end{split}
\end{equation}
where $\mathcal{F}_q =\mathcal{F}_{ext}(\{1\}\times \llbracket q+1, t_N - 1 \rrbracket) $ as in (\ref{GibbsCond}), $x=L^N_1(q)$, $y=L^N_1(t_N)$ and $g = L^N_2\llbracket q, t_N \rrbracket$. We mention that in the first equality we used the tower property for conditional expectations and in the second that $\mathsf{High}^{N}_q, \mathsf{D}^N_1 \in \mathcal{F}_q$. Also $\tilde{L}_1$ is distributed according to $ \mathbb{P}_{H,H^{RW}}^{1,1,q, t_N,x,y,\infty,g}$. 

From Lemma \ref{MonCoup} we have $\mathbb{P}$-almost surely for $N \geq \Nt_1$
\begin{equation}\label{S4S1E5}
\begin{split}
&{\bf 1}_{\mathsf{High}^{N}_q} \cdot {\bf 1}_{\mathsf{D}^N_1}  \cdot \mathbb{P}_{H,H^{RW}}^{1,1,q, t_N,x,y,\infty,g}(\tilde{L}_1(s_N) - p s_N \leq M_1 N^{ \alpha/2} ) \leq \\
& {\bf 1}_{\mathsf{High}^{N}_q} \cdot {\bf 1}_{\mathsf{D}^N_1}  \cdot \mathbb{P}_{H^{RW}}^{q, t_N,pq + \Rt_1 N^{\alpha/2} , p t_N - M_1 N^{\alpha/2}}(\ell(s_N) - p s_N \leq M_1 N^{ \alpha/2} ) ,
\end{split}
\end{equation}
where on the right side $\ell$ is distributed according to $\mathbb{P}_{H^{RW}}^{q, t_N,x,y}$. We mention that in deriving (\ref{S4S1E5}) we used that $\mathbb{P}_{H,H^{RW}}^{1,1,q, t_N,x,y,\infty,-\infty}$ is the same as $\mathbb{P}_{H^{RW}}^{q, t_N,x,y}$. We also used that on $\mathsf{High}^{N}_q \cap \mathsf{D}^N_1$ we have $x \geq pq + \Rt_1 N^{\alpha/2}$ and $y \geq p t_N - M_1 N^{\alpha/2}$.

Finally, from Lemma \ref{LStayInBand}, applied to the same constants as in Step 1, we have for $N \geq \Nt_1$ that $t_N - q \geq t_N - u_N \geq W_1$ and so 
\begin{equation}\label{S4S1E6}
\begin{split}
&\mathbb{P}_{H^{RW}}^{q, t_N, x', y'}(\ell(s_N) - p s_N \leq M_1 N^{ \alpha/2} ) \leq \\
&\mathbb{P}_{H^{RW}}^{q, t_N, x', y'} \left(\ell(s_N) - \frac{T_1 - s_N }{T_1 - T_0} \cdot x' - \frac{s_N - T_0 }{T_1 - T_0} \cdot y'  < -A(T_1-T_0)^{1/2} \right) \leq \epsilon/2.
\end{split}
\end{equation}
where $x' = pq + \Rt_1 N^{\alpha/2}$ and $y' = p t_N - M_1 N^{\alpha/2}$. We mention that in deriving the last inequality we used the definition of $\Rt_1 $ from (\ref{S4S1RLarge}).

Combining (\ref{S4S1E4}), (\ref{S4S1E5}) and (\ref{S4S1E6}) we obtain (\ref{S4S1E3}). This suffices for the proof.
\end{proof}

\begin{proof}[Proof of \eqref{Imp4}, Lemma \ref{keyL2}] We continue with the same notation as in the statement of the lemma and Section \ref{Section3.3}. All constants in the proof depend on the quantities in (\ref{Depend}) as well as $r \geq 2$ and $\epsilon > 0$ as in the statement of the lemma, which are fixed. 

Let $r_0 = \lceil r \rceil +1$ and put $u^{\pm}_N = \lfloor \pm r_0 N^{\alpha} \rfloor$. Since $\mathfrak{L}^N$ is an $(\alpha,p,\lambda)$-good sequence we can find $M_1 > 0$ and $N_1 \geq N_0$ such that for $N \geq N_1$ we have $T_N \geq | u_N^{\pm}|$ and
\begin{equation}\label{S4S1F1}
\begin{split}
&1 - \mathbb{P}(\mathsf{A}^N_{\pm}) < \epsilon/4 \mbox{ where } \mathsf{A}^N_{\pm} =  \{|L^N_1(u^{\pm}_N )-p u^{\pm}_N  |\leq M_1 N^{ \alpha/2}\}.
\end{split}
\end{equation}
Let $W_2, M^{\mathsf{dip}}$ be as in Lemma \ref{LNoDip} for $k = 1$, $p, H, H^{RW}$ as in the present lemma, $\epsilon = \epsilon/2$, $r = (2r_0 + 1)^{-1}$ and $M^{\mathsf{side}} = M_1$. We then fix $\Rb_1 > M^{\mathsf{dip}}$, which specifies our choice of $\Rb_1$. 

Let $\Nb_1 \geq N_1$ be sufficiently large so that for $N \geq \Nb_1 $
\begin{equation}\label{S4S1F2}
\begin{split}
u^+_N \geq rN^{\alpha}, \hspace{2mm} u_N^- \leq -rN^{\alpha}, \hspace{2mm} T_N \geq |u_N^{\pm}| \mbox{ and } (2r_0 + 1) N^{\alpha} \geq u^+_N - u_N^- \geq W_2.
\end{split}
\end{equation}
We proceed to prove  Lemma \ref{keyL2}[$1$] with the above choice of $\Rb_1$ and $\Nb_1$. \\

Define the event
\begin{align*}
\mathsf{Low}^{N}\coloneqq \left\{\inf_{x\in [ u_N^-, u_N^+ ]}\left[ L^N_1(x)-p  x \right]\leq - \Rb_1 N^{\alpha/2} \right\}.
\end{align*} 
Since $u^+_N \geq rN^{\alpha}$ and $u_N^- \leq -rN^{\alpha}$ we have for $N \geq \Nb_1$ that
\begin{equation}\label{S4S1F3}
\mathbb{P} \left( \inf_{x \in [- rN^{\alpha}, r N^{\alpha}] }\left[L_1^N(x) - p x \right] \leq  - \Rb_1 N^{\alpha/2} \right)  \leq \mathbb{P}\left( \mathsf{Low}^{N} \right).
\end{equation}
From the $(H,H^{RW})$-Gibbs property of $\mathfrak{L}^N$ we have for $N \geq \Nb_1$ (note $T_N \geq |u^{\pm}_N|$ from (\ref{S4S1F2}))
\begin{equation}\label{S4S1F4}
\begin{split}
& \mathbb{P}\left( \mathsf{Low}^{N} \cap \mathsf{A}^N_{+} \cap \mathsf{A}^N_{-}  \right) =  \mathbb{E}\left[ \mathbb{E} \left[ {\bf 1}_{\mathsf{Low}^{N}} \cdot  {\bf 1}_{\mathsf{A}^N_{+} } \cdot  {\bf 1}_{ \mathsf{A}^N_{-} } \Big{|} \mathcal{F} \right] \right] =  \mathbb{E}\left[ {\bf 1}_{\mathsf{A}^N_{+} } \cdot  {\bf 1}_{ \mathsf{A}^N_{-} } \cdot \mathbb{E} \left[  {\bf 1}_{\mathsf{Low}^{N}}   \Big{|} \mathcal{F}\right] \right]   \\
& = \mathbb{E}\left[{\bf 1}_{\mathsf{A}^N_{+} }  \cdot {\bf 1}_{ \mathsf{A}^N_{-} }  \cdot \mathbb{P}_{H,H^{RW}}^{1,1,u^{-}_N, u^+_N,x,y,\infty,g}\left(\inf_{x\in [ u_N^-, u_N^+ ]}\left[ \tilde{L}_1(x)-p  x \right]\leq - R^{\mathsf{bot}}_1 N^{\alpha/2}\right)\right],
\end{split}
\end{equation}
where $\mathcal{F} =\mathcal{F}_{ext}(\{1\}\times \llbracket u^-_N+1, u^+_N - 1 \rrbracket) $ as in (\ref{GibbsCond}), $x=L^N_1(u^-_N)$, $y=L^N_1(u^+_N)$ and $g = L^N_2\llbracket u_N^-, u_N^+ \rrbracket$. We mention that in the first equality we used the tower property for conditional expectations and in the second that $\mathsf{A}^N_{+}, \mathsf{A}^N_{-}  \in \mathcal{F}$. Also $\tilde{L}_1$ is distributed according to $\mathbb{P}_{H,H^{RW}}^{1,1,u^{-}_N, u^+_N,x,y,\infty,g}$. 

From Lemma \ref{LNoDip} with the choice of parameters as in the beginning of the proof, $T_0 = u_N^-$, $T_1 = u_N^+$ and $R = N^{\alpha}$ we have for $N \geq \Nb_1$ (note $u_N^+ - u_N^- \geq W_2$ from (\ref{S4S1F2})) that $\mathbb{P}$-almost surely 
\begin{equation}\label{S4S1F5}
\begin{split}
{\bf 1}_{\mathsf{A}^N_{+} } \cdot  {\bf 1}_{ \mathsf{A}^N_{-} }  \cdot \mathbb{P}_{H,H^{RW}}^{1,1,u^{-}_N, u^+_N,x,y,\infty,g}\left(\inf_{x\in [ u_N^-, u_N^+ ]}\left[ \tilde{L}_1(x)-p  x \right]\leq - R^{\mathsf{bot}}_1 N^{\alpha/2}\right) \leq  \epsilon/2.
\end{split}
\end{equation}
 We mention that in deriving the last equation we used that on $\mathsf{A}^N_{+} \cap \mathsf{A}^N_{-} $ we have $|x- p u_N^-|\leq M^{\mathsf{side}} R^{1/2}$ and $|y- p u_N^+|\leq M^{\mathsf{side}} R^{1/2}$, and also that $R^{\mathsf{bot}}_1 > M^{\mathsf{dip}}$ by definition.

Combining (\ref{S4S1F1}), (\ref{S4S1F3}), (\ref{S4S1F4}) and (\ref{S4S1F5}) we get for $N \geq \Nb_1$
\begin{equation*}
\begin{split}
\mathbb{P} \left( \inf_{x \in [- rN^{\alpha}, r N^{\alpha}] }\left[L_k^N(x) - p x \right] \leq  - \Rb_k N^{\alpha/2}\hspace{-0.5mm} \right)\hspace{-0.5mm}  \leq \hspace{-0.5mm} \mathbb{P}\left( \mathsf{Low}^{N} \cap \mathsf{A}^N_{+} \cap \mathsf{A}^N_{-} \right) + \mathbb{P}\left((\mathsf{A}^N_{+})^c \right) + \mathbb{P}\left((\mathsf{A}^N_{+})^c \right)  < \epsilon.
\end{split}
\end{equation*}
The last equation gives Lemma \ref{keyL2}[$1$].
\end{proof}

%
\subsection{Proof of (\ref{Imp1})}\label{Section4.2} In this section we present the proof of (\ref{Imp1}), for which we require the following result, whose proof is postponed until Section \ref{Section7} -- see Lemma \ref{S7LHighBottom}.

\begin{lemma}\label{LHighBottom} Fix $k \in \mathbb{N}$ and let $\mathfrak{L} = (L_1, \dots, L_k)$ have law $\mathbb{P}_{H,H^{RW}}^{1, k, T_0 ,T_1, \vec{x}, \vec{y},\infty,g}$ as in Definition \ref{Pfree}, where we assume that $H$ is as in Definition \ref{AssH}, while $H^{RW}$ is as in Definition \ref{AssHR}. Let $ p \in \mathbb{R}$, $r,\epsilon \in (0,1)$, $M^{\mathsf{bot}}, M^{\mathsf{side}} > 0$ and $t \in (0,1/3)$ be given. Then we can find a constant $W_3 = W_3(k,p,r,M^{\mathsf{side}}, M^{\mathsf{bot}},t, \epsilon, H, H^{RW}) \in \mathbb{N}$ so that the following holds. 

For any $T_0, T_1 \in \mathbb{Z}$ with $  T_1 - T_0 \geq W_3$ and $t_0, t_1 \in \llbracket T_0, T_1 \rrbracket$ with $\min(t_0 - T_0,T_1 - t_1, t_1 - t_0)  \geq t (T_1 - T_0)$, $r (T_1 - T_0) \leq R \leq  r^{-1} (T_1 - T_0)$, $\vec{x}, \vec{y} \in \mathbb{R}^k, g \in Y^{-}(\llbracket T_0, T_1 \rrbracket)$  that satisfy
$$ \left| x_i - pT_0 \right| \leq M^{\mathsf{side}} R^{1/2}, \left| y_i - pT_1 \right| \leq M^{\mathsf{side}} R^{1/2} \mbox{ for $i \in \llbracket 1, k \rrbracket$, and }$$ 
$$g(j) - pj \geq M^{\mathsf{bot}} R^{1/2} \mbox{ for some $j \in \llbracket t_0, t_1 - 1 \rrbracket$,}$$
we have that 
\begin{equation}\label{EHighBottom}
\mathbb{P}_{H,H^{RW}}^{1, k, T_0 ,T_1, \vec{x}, \vec{y},\infty,g} \left(L_k(x) - p x  \geq M^{\mathsf{bot}} R^{1/2} - R^{1/4} \mbox{ for some $x \in [t_0, t_1]$}\right) \geq 1 - \epsilon.
\end{equation}
\end{lemma}
\begin{remark} Lemma \ref{LHighBottom} states that the highest-indexed curve of a discrete line ensemble with law $\mathbb{P}_{H,H^{RW}}^{1, k, T_0 ,T_1, \vec{x}, \vec{y},\infty,g}$, whose entrance and exit data $\vec{x}, \vec{y}$ are well-behaved is unlikely to stay low if the bottom bounding curve $g$ is high at least at one point.
\end{remark}

\begin{proof}[Proof of \eqref{Imp1}] We continue with the same notation as in the statement of the lemma and Section \ref{Section3.3}. All constants in the proof depend on the quantities in (\ref{Depend}) as well as $r \geq 2$ and $\epsilon > 0$ as in the statement of the lemma, which are fixed.

 Let $r_0 = \lceil r \rceil +1$ and put $u^{\pm}_N = \lfloor \pm r_0 N^{\alpha} \rfloor$ and $t^{\pm}_N = \lfloor \pm 4r_0 N^{\alpha} \rfloor$. In addition, suppose that $N_1 \geq N_0$ is sufficiently large so that for $N \geq N_1$.
\begin{equation}\label{S4S2E1}
\begin{split}
&u_N^+ -1 \geq rN^{\alpha}, \hspace{2mm} u_N^- \leq - rN^{\alpha}, \hspace{2mm} \min(t^+_N - u^+_N, u_N^{+} - u_N^-, u_N^- - t_N^-) \geq (t_N^+ - t_N^-)/10 \mbox{, } \\
& \left| t^{\pm}_N \right| \leq 5r_0 N^{\alpha}, \mbox{ and } T_N \geq \left| t^{\pm}_N \right|.
\end{split}
\end{equation}

Let $R_1 > 0$ and $N_2 \in \mathbb{N}$, $N_2 \geq N_1$ be sufficiently large so that 
\begin{equation}\label{S4S2E2}
\begin{split}
&R_1 \geq \max \left( \Rt_{k}(5r_0, \epsilon/(4m+4)),  \Rb_{k}(5r_0, \epsilon/(4m+4)) \right) \mbox{ and } \\
& N_2 \geq \max \left( \Nt_{k}(5r_0, \epsilon/(4m+4)),  \Nb_{k}(5r_0, \epsilon/(4m+4))  \right) \mbox{ for $k \in \llbracket 1, m+1 \rrbracket$,}
\end{split}
\end{equation}
where $\Rt_{k}, \Nt_{k}$ are as in Lemma \ref{keyL1}[$m+1$] and $\Rb_{k}, \Nb_{k}$ are as in Lemma \ref{keyL2}[$m+1$], which we have assumed to hold as part of our input in  \eqref{Imp1}.

Let $W_3$ be as in Lemma \ref{LHighBottom} for $k = m+1$, $p, H, H^{RW}$ as in the statement of the present lemma, $\epsilon = \epsilon/4$, $t = 1/10$, $r = 1/(10r_0)$, $M^{\mathsf{side}} = R_1$ and $M^{\mathsf{bot}} = R_1 + 1$. We suppose that $\Nt_{m+2} $ is large enough so that $\Nt_{m+2} \geq N_2$ and $u_N^+ - u_N^- \geq W_3$ for $N \geq\Nt_{m+2}$. We also set $ \Rt_{m+2} = R_1 + 1$ and proceed to prove that for $N \geq \Nt_{m+2}$ we have
\begin{equation}\label{S4S2NoBigMax}
\mathbb{P} \left( \sup_{x \in [- rN^{\alpha}, r N^{\alpha}] }\left[L_{m+2}^N(x) - p x \right] \geq \Rt_{m+2} N^{\alpha/2} \right) < \epsilon.
\end{equation}
Our assumption that Lemma \ref{keyL1}[$m+1$] holds and (\ref{S4S2NoBigMax}) together give Lemma \ref{keyL1}[$m+2$] as desired.\\

In the remainder we prove (\ref{S4S2NoBigMax}). Define the events
\begin{equation*}
\begin{split}
&\mathsf{High}_{m+2}^{N}\coloneqq \left\{\max_{x\in \llbracket  u_N^-, u_N^+ - 1 \rrbracket}\left[ L^N_{m+2}(x)-p x \right]\geq  R^{\textsf{top}}_{m+2} N^{\alpha/2} \right\}, \\
&\mathsf{High}^{N}_{m+1}\coloneqq \left\{\max_{x\in \llbracket  u_N^-, u_N^+  \rrbracket}\left[ L^N_{m+1}(x)-p x \right]\geq  R_1N^{\alpha/2} \right\}, \mbox{ and } \\
&\mathsf{Side}^N \coloneqq\left\{ \left|L^N_{j}(t^{\pm}_N )- p t^{\pm}_N \right|\leq R_1 N^{\alpha/2} \mbox{ for all $j \in \llbracket 1, m +1 \rrbracket$} \right\}.
\end{split}
\end{equation*}
In view of (\ref{S4S2E2}) and a union bound we have for $N \geq  \Nt_{m+2}$
\begin{equation}\label{S4S2E3}
\begin{split}
1 - \mathbb{P} \left(\mathsf{Side}^N  \right) < \epsilon/2 \mbox{ and } \mathbb{P} \left(\mathsf{High}^{N}_{m+1}  \right) < \epsilon/4. 
\end{split}
\end{equation}
In deriving (\ref{S4S2E3}) we used that $ \left| u^{\pm}_N \right| \leq \left| t^{\pm}_N \right| \leq 5r_0 N^{\alpha}$ as assured by (\ref{S4S2E1}).

From the $(H,H^{RW})$-Gibbs property of $\mathfrak{L}^N$ we have for $N \geq \Nt_{m+2}$ (note $T_N \geq \left| t^{\pm}_N \right|$ from (\ref{S4S2E1}))
\begin{equation}\label{S4S2E4}
\begin{split}
& \mathbb{P}\left( \mathsf{High}_{m+2}^{N} \cap  \mathsf{Side}^N  \cap  (\mathsf{High}_{m+1}^{N})^c \right) =  \mathbb{E}\left[ \mathbb{E} \left[ {\bf 1}_{\mathsf{High}^{N}_{m+2}} \cdot  {\bf 1}_{\mathsf{Side}^N } \cdot  {\bf 1}_{(\mathsf{High}^{N}_{m+1})^c}  \Big{|} \mathcal{F} \right] \right] = \\
&  \mathbb{E}\left[  {\bf 1}_{\mathsf{High}^{N}_{m+2}} \cdot  {\bf 1}_{\mathsf{Side}^N }  \cdot \mathbb{E} \left[ {\bf 1}_{(\mathsf{High}^{N}_{m+1})^c}  \Big{|} \mathcal{F}\right] \right] =  \\
& \mathbb{E}\left[ {\bf 1}_{\mathsf{High}^{N}_{m+2}} \cdot  {\bf 1}_{\mathsf{Side}^N }     \cdot \mathbb{P}_{H,H^{RW}}^{1,m+1,t^{-}_N, t^+_N,\vec{x},\vec{y},\infty,g}\left(\max_{x\in \llbracket  t_N^-, t_N^+  \rrbracket}\left[ \tilde{L}_{m+1}(x)-p x \right]<  R_1N^{\alpha/2}  \right)\right],
\end{split}
\end{equation}
where $\mathcal{F} =\mathcal{F}_{ext}(\llbracket 1, m + 1 \rrbracket \times \llbracket t^-_N+1, t^+_N - 1 \rrbracket) $ as in (\ref{GibbsCond}), $\vec{x}= (L^N_1(t^-_N), \dots,L^N_{m+1}(t^-_N)) $, $\vec{y}= (L^N_1(t^+_N), \dots,L^N_{m+1}(t^+_N)) $ and $g = L^N_{m+2}\llbracket t_N^-, t_N^+ \rrbracket$. We mention that in the first equality we used the tower property for conditional expectations and in the second that $\mathsf{High}^{N}_{m+2}, \mathsf{Side}^N \in \mathcal{F}$. Also $(\tilde{L}_1, \dots,\tilde{L}_{m+1}) $ is distributed according to $\mathbb{P}_{H,H^{RW}}^{1,m+1,t^{-}_N, t^+_N,\vec{x},\vec{y},\infty,g}$.

From Lemma \ref{LHighBottom} with the same choice of parameters as in the beginning of the proof, $R = N^{\alpha}$, $T_1 = t_N^+$, $T_0 = t_N^-$, $t_0 = u_N^-$, $t_1 = u_N^+$ we have $\mathbb{P}$-almost surely for $N \geq \Nt_{m+2}$ 
\begin{equation}\label{S4S2E5}
\begin{split}
&{\bf 1}_{\mathsf{High}^{N}_{m+2}} \cdot  {\bf 1}_{\mathsf{Side}^N }   \cdot \mathbb{P}_{H,H^{RW}}^{1,m+1,t^{-}_N, t^+_N,\vec{x},\vec{y},\infty,g}\left(\max_{x\in \llbracket  t_N^-, t_N^+  \rrbracket}\left[ \tilde{L}_{m+1}(x)-p x \right]<  R_1N^{\alpha/2}  \right) = \\
&{\bf 1}_{\mathsf{High}^{N}_{m+2}} \cdot  {\bf 1}_{\mathsf{Side}^N }   \cdot \mathbb{P}_{H,H^{RW}}^{1,m+1,t^{-}_N, t^+_N,\vec{x},\vec{y},\infty,g}\left(\sup_{x\in [  t_N^-, t_N^+  ]}\left[ \tilde{L}_{m+1}(x)-p x \right]<  R_1N^{\alpha/2}  \right) \leq \epsilon/4,
\end{split}
\end{equation}
where in the first equality we used that $\tilde{L}_{m+1}$ linearly interpolates its values at the integers. We mention that in deriving the last inequality we used that $M^{\mathsf{bot}} R^{1/2} - R^{1/4} = (R_1 + 1)N^{\alpha/2} - N^{\alpha/4} \geq R_1N^{\alpha/2}$. The fact that $T_0, T_1, t_0, t_1$ satisfy the conditions of Lemma \ref{LHighBottom} follows from the inequalities in (\ref{S4S2E1}). We also mention that we used that on $\mathsf{Side}^N$ the random vectors $\vec{x}, \vec{y}$ as in (\ref{S4S2E5}) satisfy the conditions of Lemma \ref{LHighBottom} and on $\mathsf{High}^{N}_{m+2}$ we have that $g(j) - pj \geq M^{\mathsf{bot}} R^{1/2}$ for some $j \in \llbracket  u_N^-, u_N^+ - 1 \rrbracket$. 

Combining (\ref{S4S2E4}) and (\ref{S4S2E5}) we conclude that for $N \geq \Nt_{m+2}$ 
\begin{equation}\label{S4S2E6}
\begin{split}
& \mathbb{P}\left( \mathsf{High}_{m+2}^{N} \cap  \mathsf{Side}^N  \cap  (\mathsf{High}_{m+1}^{N})^c \right) \leq \epsilon/4.
\end{split}
\end{equation}
Finally, using that $u_N^+ -1 \geq rN^{\alpha}$ and $ u_N^- \leq - rN^{\alpha}$ for $N \geq \Nt_{m+2}$, see (\ref{S4S2E1}), the fact that $L^N_{m+2}$ is a linear interpolation of its values on the integers and equations (\ref{S4S2E3}) and (\ref{S4S2E6}) we conclude that 
 \begin{equation*}
\begin{split}
&\mathbb{P} \left( \sup_{x \in [- rN^{\alpha}, r N^{\alpha}] }\left[L_{m+2}^N(x) - p x \right] \geq \Rt_{m+2} N^{\alpha/2} \right) \leq\mathbb{P} \left( \mathsf{High}_{m+2}^{N} \right) \leq \\
& \mathbb{P}\left( \mathsf{High}_{m+2}^{N} \cap  \mathsf{Side}^N  \cap  (\mathsf{High}_{m+1}^{N})^c \right)  + \mathbb{P}\left( ( \mathsf{Side}^N)^c \right)  + \mathbb{P}\left(\mathsf{High}_{m+1}^{N} \right) < \epsilon.
\end{split}
\end{equation*}
The last expression gives (\ref{S4S2NoBigMax}).
\end{proof}

%
\subsection{Proof of (\ref{Imp2})}\label{Section4.3} In this section we present the proof of (\ref{Imp2}), for which we require the following result, whose proof is postponed until Section \ref{Section7} -- see Lemma \ref{S7LNoParDip}.

\begin{lemma}\label{LNoParDip} Fix $k \in \mathbb{N}$ and let $\mathfrak{L} = (L_1, \dots, L_k)$ have law $\mathbb{P}_{H,H^{RW}}^{1, k, T_0 ,T_1, \vec{x}, \vec{y},\infty,g}$ as in Definition \ref{Pfree}, where we assume that $H$ is as in Definition \ref{AssH}, while $H^{RW}$ is as in Definition \ref{AssHR}. Let $ p \in \mathbb{R}$, $\lambda> 0$, $L \geq 1$, $ \epsilon \in (0,1)$, $M^{\mathsf{top}}  > 0 $ be given. Then we can find constants $A = A(k,p, \epsilon, H^{RW}, H) > 0$, $W_4 = W_4(k,p, \lambda, M^{\mathsf{top}}, \epsilon, H, H^{RW},L) \in \mathbb{N}$ and $M^{\mathsf{base, top}} = M^{\mathsf{base, top}} (k,p, \lambda, M^{\mathsf{top}}, \epsilon, H, H^{RW},L) > 0$ such that the following holds. 

If $T \geq W_4$, $T_0, T_1 \in \mathbb{Z}$ with $\max(|T_0|, |T_1|) \leq T \cdot L$, $\Delta T = T_1 - T_0 \geq T$, $\vec{x}, \vec{y} \in \mathbb{R}^k$ and $g \in Y^-(\llbracket T_0, T_1 \rrbracket)$ satisfy 
$$  x_i - pT_0 + \lambda T^{1/2} (T_0/T)^2 \leq M^{\mathsf{top}} T^{1/2}, y_i - pT_1 + \lambda T^{1/2}  (T_1/T)^2   \leq M^{\mathsf{top}} T^{1/2} \mbox{ for $i \in \llbracket 1, k \rrbracket$, and }$$ 
$$g(j) - pj \leq - M^{\mathsf{base, top}} T^{1/2} \mbox{ for all $j \in \llbracket T_0, T_1 - 1 \rrbracket$,}$$
then we have that 
\begin{equation}\label{ENoParDip}
\begin{split}
1-\epsilon \leq \mathbb{P}_{H,H^{RW}}^{1, k, T_0 ,T_1, \vec{x}, \vec{y},\infty,g} &\Bigg{(} L_1(x) - p x +  \frac{x- T_0}{T_1 - T_0}\cdot \lambda T^{1/2} (T_1/T)^2 +  \\
&   \frac{T_1 - x}{T_1 - T_0} \cdot \lambda T^{1/2} (T_0/T)^2 \leq (A +M^{\mathsf{top}}) \Delta T^{1/2}  \mbox{ for all $x \in [T_0, T_1]$}\Bigg{)}.
\end{split}
\end{equation}
\end{lemma}
\begin{remark} Lemma \ref{LNoParDip} states that the lowest-indexed curve of a discrete line ensemble with law $\mathbb{P}_{H,H^{RW}}^{1, k, T_0 ,T_1, \vec{x}, \vec{y},\infty,g}$, whose entrance and exit data $\vec{x}, \vec{y}$ are not higher than a certain inverted parabola at times $T_0$ and $T_1$, is unlikely to rise higher than the straight segment connecting the points of this parabola at times $T_0$ and $T_1$. This is true provided that the bottom bounding curve $g$ is sufficiently low on the interval $[T_0,T_1]$.
\end{remark}

\begin{proof}[Proof of \eqref{Imp2}] We continue with the same notation as in the statement of the lemma and Section \ref{Section3.3}. All constants in the proof depend on the quantities in (\ref{Depend}) as well as $r \geq 2$ and $\epsilon > 0$ as in the statement of the lemma, which are fixed. For clarity we split the proof into five steps.\\

{\bf \raggedleft Step 1.} In this step we fix some notation and state a relevant claim that we will prove and use.

 Let $r_0 = \lceil r \rceil +1$ and put $u^{\pm}_N = \lfloor \pm r_0 N^{\alpha} \rfloor$ and $t^{\pm}_N = \lfloor \pm 4r_0 N^{\alpha} \rfloor$. For $\tau \in \mathbb{N}$ we let $v_N^{\pm,\tau} = \lfloor \pm \tau N^{\alpha} \rfloor$. We claim that we can find $\tau \geq 4r_0 + 2$, $N_1 \in \mathbb{N}$ and $R^{\mathsf{dip}} > 0$, such that $N_1 \geq N_0$ and $T_N \geq |v_N^{\pm,\tau}|$ for $N \geq N_1$, and for $N \geq N_1$ we have 
\begin{equation}\label{S4S3E1}
\begin{split}
&\mathbb{P} \left( L_{m+1}^N(x) - p x \leq -R^{\mathsf{dip}} N^{\alpha/2} \mbox{ for all } x \in \llbracket v_N^{-,\tau}, t_N^- \rrbracket \right) < \epsilon/8, \\
&\mathbb{P} \left( L_{m+1}^N(x) - p x \leq -R^{\mathsf{dip}} N^{\alpha/2} \mbox{ for all } x \in \llbracket t_N^+, v_N^{+,\tau} \rrbracket \right)  < \epsilon/8.
\end{split}
\end{equation}
We will prove (\ref{S4S3E1}) in Step 4 below. In this and the next two steps we assume its validity and conclude the proof of Lemma \ref{keyL2}[$m+1$].\\

We fix $\tau, N_1, R^{\mathsf{dip}}$ as in the previous paragraph. Suppose that $N_2 \geq N_1$ is sufficiently large so that for $N \geq N_2$.
\begin{equation}\label{S4S3E2}
\begin{split}
&u_N^+ \geq rN^{\alpha}, \hspace{2mm} u_N^- \leq - rN^{\alpha}, \hspace{2mm} \min(v^{+,\tau}_N - u^+_N, u_N^{+} - u_N^-, u_N^- - v^{-,\tau}_N) \geq ( v^{+,\tau}_N -  v^{-,\tau}_N)/\tau \mbox{, } \\
&\mbox{ and } \left| v^{\pm,\tau}_N \right| \leq 2 \tau N^{\alpha} .
\end{split}
\end{equation}

Let $R_1 \geq R^{\mathsf{dip}}$ and $N_3 \in \mathbb{N}$, $N_3 \geq N_2$ be sufficiently large so that 
\begin{equation}\label{S4S3E3}
\begin{split}
&R_1 \geq \max_{k \in \llbracket 1, m+1 \rrbracket}  \Rt_{k}(2\tau, \epsilon/(4m+4)), \hspace{2mm} R_1 \geq \max_{k \in \llbracket 1, m\rrbracket}   \Rb_{k}(2\tau, \epsilon/(4m+4)), \mbox{ and }  \\
&N_3 \geq \max_{k \in \llbracket 1, m+1 \rrbracket}  \Nt_{k}(2\tau, \epsilon/(4m+4)), \hspace{2mm} N_3 \geq \max_{k \in \llbracket 1, m\rrbracket}   \Nb_{k}(2\tau, \epsilon/(4m+4)),
\end{split}
\end{equation}
where $\Rt_{k}, \Nt_{k}$ are as in Lemma \ref{keyL1}[$m+1$] and $\Rb_{k}, \Nb_{k}$ are as in Lemma \ref{keyL2}[$m$], which we have assumed to hold as part of our input in  \eqref{Imp2}. 

Finally, we let $N_4 \geq N_3$ be sufficiently large so that $v_N^{+,\tau} - v_N^{-,\tau} \geq W_2$ as in Lemma \ref{LNoDip} with $k = m+1$, $p, H, H^{RW}$ as in the statement of this lemma, $r = (2\tau)^{-1}$, $\epsilon = \epsilon/4$ and $M^{\mathsf{side}} = R_1$. We set $\Nb_{m+1} = N_4$ and fix $\Rb_{m+1} \geq R_1$ large enough so that $\Rb_{m+1} \geq M^{\mathsf{dip}}$, where $ M^{\mathsf{dip}}$ is as in Lemma \ref{LNoDip}  with the choice of parameters we just listed. This specifies our choice of $\Nb_{m+1}$ and $\Rb_{m+1}$ and we prove below that for $N \geq \Nb_{m+1}$ we have 
\begin{equation}\label{S4S3E4}
\begin{split}
\mathbb{P} \left( \mathsf{Low}^N_{m+1} \right) < \epsilon , \mbox{ where }\mathsf{Low}^N_{m+1} =  \left\{ \min_{x \in \llbracket u_N^-, u_N^+ \rrbracket }\left[L_{m+1}^N(x) - p x \right] \leq  - \Rb_{m+1} N^{\alpha/2} \right\}.
\end{split}
\end{equation}
Observe that since $u_N^+ \geq rN^{\alpha}, \hspace{2mm} u_N^- \leq - rN^{\alpha}$ for $N \geq \Nb_{m+1}$, see (\ref{S4S3E2}), and $L_{m+1}^N$ is equal to the linear interpolation of its values on the integers we have that (\ref{S4S3E4}) implies for $N \geq \Nb_{m+1}$ 
$$\mathbb{P} \left( \inf_{x \in [- rN^{\alpha}, r N^{\alpha}] }\left[L_{m+1}^N(x) - p x \right] \leq  - \Rb_{m+1} N^{\alpha/2} \right) \leq \mathbb{P} \left( \mathsf{Low}^N_{m+1} \right)  < \epsilon.$$
Our assumption that Lemma \ref{keyL2}[$m$] holds and the last equation together give Lemma \ref{keyL2}[$m+1$].\\

{\bf \raggedleft Step 2.} In this step we prove (\ref{S4S3E4}). Define the events 
\begin{equation*}
\begin{split}
&\mathsf{High}_{m+1}^{+,N}\coloneqq \left\{\max_{x\in \llbracket  t_N^+, v_N^{+,\tau}  \rrbracket}\left[ L^N_{m+1}(x)-p x \right]\geq  - R_1 N^{\alpha/2} \right\},\\
&\mathsf{High}_{m+1}^{-,N}\coloneqq \left\{\max_{x\in \llbracket  v_N^{-,\tau}, t_N^-  \rrbracket}\left[ L^N_{m+1}(x)-p x \right]\geq  - R_1 N^{\alpha/2} \right\} \mbox{ and }\\
&\mathsf{Side}^N \coloneqq \left\{ \left|L^N_{j}(x )- p x \right|\leq R_1 N^{\alpha/2} \mbox{ for all $j \in \llbracket 1, m \rrbracket$ and $x \in \llbracket v_N^{-,\tau}, v_N^{+,\tau} \rrbracket $ }  \right\}.
\end{split}
\end{equation*}

In view of (\ref{S4S3E3}) and a union bound as well as (\ref{S4S3E1}) we have for $N \geq \Nb_{m+1}$
\begin{equation}\label{S4S3E5}
\begin{split}
\mathbb{P} \left( \mathsf{High}_{m+1}^{+,N} \right) > 1 - \epsilon/8, \hspace{2mm} \mathbb{P} \left( \mathsf{High}_{m+1}^{-,N} \right) > 1 - \epsilon/8, \hspace{2mm} \mathbb{P} \left( \mathsf{Side}^N \right) > 1 - \epsilon/2.
\end{split}
\end{equation}
Here we also used that $\left| v^{\pm,\tau}_N \right| \leq 2 \tau N^{\alpha}$ as follows from (\ref{S4S3E2}).

We also observe that 
\begin{equation*}
\begin{split}
&\mathsf{High}_{m+1}^{+,N} = \sqcup_{q^+ \in  \llbracket  t_N^+, v_N^{+,\tau}  \rrbracket} \mathsf{High}_{m+1}^{+,N}(q^+) \mbox{ and }\mathsf{High}_{m+1}^{-,N} = \sqcup_{q^- \in  \llbracket v_N^{-,\tau}, t_N^-  \rrbracket} \mathsf{High}_{m+1}^{-,N}(q^-) , \mbox{ where }\\
&\mathsf{High}_{m+1}^{+,N}(q^+)  =  \left\{ L^N_{m+1}(q^+)-p q^+ \geq  - R_1 N^{\alpha/2}  \right\} \cap \left\{ \max_{ x \in \llbracket q^+ + 1, v_N^{+,\tau} \rrbracket} \left[ L^N_{m+1}(x)-p x \right]<  - R_1 N^{\alpha/2} \right\}, \\
&\mathsf{High}_{m+1}^{-,N}(q^-)  =  \left\{\left[ L^N_{m+1}(q^-)-p q^- \right]\geq  - R_1 N^{\alpha/2} \right\} \cap \left\{ \max_{ x \in \llbracket v_N^{-,\tau}, q^- - 1\rrbracket }\left[ L^N_{m+1}(x)-p x \right]<  - R_1 N^{\alpha/2} \right\}.
\end{split}
\end{equation*}
Finally, we define for $q^+ \in   \llbracket  t_N^+, v_N^{+,\tau}  \rrbracket$ and $q^- \in \llbracket v_N^{-,\tau}, t_N^- \rrbracket$ the event 
$$\mathsf{Side}^N(q^-, q^+) = \left\{ \left|L^N_{j}(x )- p x \right|\leq R_1 N^{\alpha/2} \mbox{ for all $j \in \llbracket 1, m \rrbracket$ and $x \in \{q^-, q^+ \} $ }  \right\}.$$

We claim that for each $N \geq  \Nb_{m+1}$, $q^+ \in \llbracket  t_N^+, v_N^{+,\tau}  \rrbracket$ and $q^- \in \llbracket v_N^{-,\tau}, t_N^-  \rrbracket$  we have 
\begin{equation}\label{S4S3E6}
\begin{split}
&\mathbb{P} \left( \mathsf{High}_{m+1}^{+,N}(q^+) \cap \mathsf{High}_{m+1}^{-,N}(q^-)  \cap \mathsf{Side}^N(q^-, q^+) \cap \mathsf{Low}_{m+1}^N  \right) \leq  \\
&(\epsilon/4) \cdot \mathbb{P} \left( \mathsf{High}_{m+1}^{+,N}(q^+) \cap \mathsf{High}_{m+1}^{-,N}(q^-)    \right).
\end{split}
\end{equation}
We will prove (\ref{S4S3E6}) in Step 3. Here we assume its validity and conclude the proof of (\ref{S4S3E4}).\\

Since $\mathsf{Side}^N \subset \mathsf{Side}^N(q^-, q^+) $ for each $q^+ \in   \llbracket  t_N^+, v_N^{+,\tau}  \rrbracket$ and $q^- \in \llbracket v_N^{-,\tau}, t_N^-  \rrbracket$ we have from summing (\ref{S4S3E6}) over $q^+ \in   \llbracket  t_N^+, v_N^{+,\tau}  \rrbracket$ and $q^- \in \llbracket v_N^{-,\tau}, t_N^-  \rrbracket$ that 
\begin{equation*}
\begin{split}
&\mathbb{P} \left( \mathsf{High}_{m+1}^{+,N} \cap \mathsf{High}_{m+1}^{-,N} \cap \mathsf{Side}^N\cap \mathsf{Low}_{m+1}^N  \right) \leq (\epsilon/4) \cdot \mathbb{P} \left( \mathsf{High}_{m+1}^{+,N} \cap \mathsf{High}_{m+1}^{-,N}    \right) \leq \epsilon/4.
\end{split}
\end{equation*}
Combining the last equation with (\ref{S4S3E5}) gives 
\begin{equation*}
\begin{split}
\mathbb{P} \left( \mathsf{Low}_{m+1}^N  \right) \leq \hspace{2mm} & \mathbb{P} \left( \mathsf{High}_{m+1}^{+,N} \cap \mathsf{High}_{m+1}^{-,N} \cap \mathsf{Side}^N\cap \mathsf{Low}_{m+1}^N  \right) + \\
&\mathbb{P} \left( (\mathsf{High}_{m+1}^{+,N})^c \right) +  \mathbb{P} \left( (\mathsf{High}_{m+1}^{-,N})^c \right)  + \mathbb{P} \left(( \mathsf{Side}^N)^c \right) < \epsilon,
\end{split}
\end{equation*}
which proves (\ref{S4S3E4}).\\

{\bf \raggedleft Step 3.} In this step we prove (\ref{S4S3E6}). From the $(H,H^{RW})$-Gibbs property of $\mathfrak{L}^N$ we have for $N \geq \Nb_{m+1}$ (note $T_N \geq | v^{\pm,\tau}_N |$ from (\ref{S4S3E2}))
\begin{equation}\label{S4S3E7}
\begin{split}
& \mathbb{P}\left( \mathsf{High}_{m+1}^{+,N}(q^+) \cap \mathsf{High}_{m+1}^{-,N}(q^-)  \cap \mathsf{Side}^N(q^-, q^+) \cap \mathsf{Low}_{m+1}^N \right) = \\
&  \mathbb{E}\left[ \mathbb{E} \left[ {\bf 1}_{ \mathsf{High}_{m+1}^{+,N}(q^+)} \cdot  {\bf 1}_{ \mathsf{High}_{m+1}^{-,N}(q^-)}  \cdot   {\bf 1}_{\mathsf{Side}^N(q^-, q^+)}  \cdot  {\bf 1}_{ \mathsf{Low}_{m+1}^N }  \Big{|} \mathcal{F} \right] \right] = \\
&  \mathbb{E}\left[ {\bf 1}_{ \mathsf{High}_{m+1}^{+,N}(q^+)} \cdot  {\bf 1}_{ \mathsf{High}_{m+1}^{-,N}(q^-)}  \cdot   {\bf 1}_{\mathsf{Side}^N(q^-, q^+)}  \cdot \mathbb{E} \left[    {\bf 1}_{ \mathsf{Low}_{m+1}^N }  \Big{|} \mathcal{F} \right] \right] = \\
& \mathbb{E}\Bigg{[}{\bf 1}_{ \mathsf{High}_{m+1}^{+,N}(q^+)} \cdot  {\bf 1}_{ \mathsf{High}_{m+1}^{-,N}(q^-)}  \cdot   {\bf 1}_{\mathsf{Side}^N(q^-, q^+)}    \times \\
& \mathbb{P}_{H,H^{RW}}^{1,m+1,q^{-}, q^+,\vec{x},\vec{y},\infty,g}\left( \min_{x \in \llbracket u_N^-, u_N^+ \rrbracket }\left[\tilde{L}_{m+1}^N(x) - p x \right] \leq  - \Rb_{m+1} N^{\alpha/2}\right)\Bigg{]},
\end{split}
\end{equation}
where $\mathcal{F} =\mathcal{F}_{ext}(\llbracket 1, m + 1 \rrbracket \times \llbracket q^- +1, q^+ - 1 \rrbracket) $ as in (\ref{GibbsCond}), $\vec{x}= (L^N_1(q^-), \dots,L^N_{m+1}(q^-)) $, $\vec{y}= (L^N_1(q^+), \dots,L^N_{m+1}(q^+)) $ and $g = L^N_{m+2}\llbracket q^-, q^+ \rrbracket$. We mention that in the first equality we used the tower property for conditional expectations and in the second that $ \mathsf{High}_{m+1}^{+,N}(q^+)$, $\mathsf{High}_{m+1}^{-,N}(q^-)$, $\mathsf{Side}^N(q^-, q^+) \in \mathcal{F}$. Also $(\tilde{L}_1, \dots,\tilde{L}_{m+1}) $ is distributed according to $\mathbb{P}_{H,H^{RW}}^{1,m+1,q^{-}, q^+,\vec{x},\vec{y},\infty,g}$. 

Setting $\vec{x}'$, $\vec{y}'$ to be the same as $\vec{x}, \vec{y}$ except that $x'_{m+1} = pq^- -R_1 N^{\alpha/2}$ and $y'_{m+1} = pq^+ -R_1 N^{\alpha/2}$ we see from  Lemma \ref{MonCoup} that $\mathbb{P}$-almost surely
\begin{equation}\label{S4S3E8}
\begin{split}
&{\bf 1}_{ \mathsf{High}_{m+1}^{+,N}(q^+)}   {\bf 1}_{ \mathsf{High}_{m+1}^{-,N}(q^-)}  \mathbb{P}_{H,H^{RW}}^{1,m+1,q^{-}, q^+,\vec{x},\vec{y},\infty,g}\left( \min_{x \in \llbracket u_N^-, u_N^+ \rrbracket }\left[\tilde{L}_{m+1}^N(x) - p x \right] \leq  - \Rb_{m+1} N^{\alpha/2}\right) \leq \\
& {\bf 1}_{ \mathsf{High}_{m+1}^{+,N}(q^+)}   {\bf 1}_{ \mathsf{High}_{m+1}^{-,N}(q^-)}  \mathbb{P}_{H,H^{RW}}^{1,m+1,q^{-}, q^+,\vec{x}',\vec{y}',\infty,g}\left( \min_{x \in \llbracket u_N^-, u_N^+ \rrbracket }\left[\tilde{L}_{m+1}^N(x) - p x \right] \leq  - \Rb_{m+1} N^{\alpha/2}\right).
\end{split}
\end{equation}
From Lemma \ref{LNoDip} with the same choice of parameters as in Step 1, $R = N^{\alpha}$, $T_0 = q^-$, $T_1 = q^+$ we have $\mathbb{P}$-almost surely for $N \geq \Nb_{m+1}$  that 
\begin{equation}\label{S4S3E9}
\begin{split}
&   {\bf 1}_{\mathsf{Side}^N(q^-, q^+)} \cdot \mathbb{P}_{H,H^{RW}}^{1,m+1,q^{-}, q^+,\vec{x}',\vec{y}',\infty,g}\left( \min_{x \in \llbracket u_N^-, u_N^+ \rrbracket }\left[\tilde{L}_{m+1}^N(x) - p x \right] \leq  - \Rb_{m+1} N^{\alpha/2}\right) \leq \\
& {\bf 1}_{\mathsf{Side}^N(q^-, q^+)} \cdot (\epsilon/4).
\end{split}
\end{equation}
We mention that in deriving the last inequality we used that $\Rb_{m+1} \geq M^{\mathsf{dip}}$ as in Step 1. The fact that $R, T_0, T_1$ satisfy the conditions of Lemma \ref{LNoDip} follows from the inequalities in (\ref{S4S3E2}) and the fact that we took $r = (2\tau)^{-1}$ in Lemma \ref{LNoDip}. We also mention that we used that on $\mathsf{Side}^N(q^-, q^+)$ the random vectors $\vec{x}', \vec{y}'$ as in (\ref{S4S3E9}) satisfy the conditions of Lemma \ref{LNoDip} with $M^{\mathsf{side}} = R_1$.

Equations (\ref{S4S3E7}), (\ref{S4S3E8}) and (\ref{S4S3E9}) give (\ref{S4S3E6}). \\

{\bf \raggedleft Step 4.} In the next two steps we prove (\ref{S4S3E1}). In this step we specify the choice of $\tau, N_1$ and $R^{\mathsf{dip}}$. 

Let $A(m, p, \epsilon/64, H^{RW}, H)$ be as in Lemma \ref{LNoParDip} for $m, p ,\epsilon, H^{RW}$ as in the present lemma. In addition, let $S_1 > 0$ be sufficiently large so that 
\begin{equation}\label{S4S3E10}
S_1 \geq \phi_2(\epsilon/128) + 1,
\end{equation}
where $\phi_2$ is as in Definition \ref{Def1} of an $(\alpha, p, \lambda)$-good sequence. We then let $\tau \in \mathbb{N}$ and $V_1 \in \mathbb{N}$, be sufficiently large so that $\tau \geq 4r_0 + 2$ and for $N \geq V_1$
\begin{equation}\label{S4S3E11}
\begin{split}
&-  \lambda  \cdot \left(\frac{\lfloor \tau_{+} N^{\alpha} \rfloor}{N^{\alpha}}\right)^2 +  \frac{\lfloor \tau_{+} N^{\alpha} \rfloor - T^+_0}{T^+_1 - T^+_0}\cdot \lambda \cdot  \left(\frac{T^+_1}{N^{\alpha}} \right)^2    +  \frac{T^+_1 - \lfloor \tau_{+} N^{\alpha}\rfloor}{T^+_1 - T^+_0} \cdot \lambda \cdot  \left(\frac{T^+_0}{N^{\alpha}} \right)^2  > \\
& (A +2S_1)\left(\frac{T_1^+ - T_0^+}{N^{\alpha}} \right)^{1/2}   , \\
&-  \lambda  \cdot \left(\frac{\lfloor \tau_{-} N^{\alpha} \rfloor}{N^{\alpha}}\right)^2 +  \frac{\lfloor \tau_{-} N^{\alpha} \rfloor - T^-_0}{T^-_1 - T^-_0}\cdot \lambda \cdot  \left(\frac{T^-_1}{N^{\alpha}} \right)^2    +  \frac{T^-_1 - \lfloor \tau_{-} N^{\alpha}\rfloor}{T^-_1 - T^-_0} \cdot \lambda \cdot  \left(\frac{T^-_0}{N^{\alpha}} \right)^2  > \\
& (A +2S_1)\left(\frac{T_1^- - T_0^-}{N^{\alpha}} \right)^{1/2}   , \\
\end{split}
\end{equation}
where $T_0^+ =  t_N^+$, $T_1^+ = \lfloor \tau N^{\alpha} \rfloor$, $T_0^- = \lfloor - \tau N^{\alpha} \rfloor$, $T_1^- = t_N^-$, $\tau_+ = \lfloor N^{-\alpha} (T_0^+ + T_1^+)/2 \rfloor$,  and $\tau_- = \lfloor N^{-\alpha} (T_0^- + T_1^-)/2 \rfloor$. Note that such a choice of $\tau$ and $V_1$ is possible since the left sides in the two inequalities in (\ref{S4S3E11}) for all large $N$ are at least $(\lambda/8)\tau^2$, while the right sides are $O(\sqrt{\tau})$ for all $N$. This specifies our choice of $\tau$ and we set, as in Step 1, $v_N^{\pm,\tau} = \lfloor \pm \tau N^{\alpha} \rfloor$.

We next let $V_2 \in \mathbb{N}$ be sufficiently large so that $V_2 \geq V_1$ and for $N \geq V_2$ we have
\begin{equation}\label{S4S3E12}
\begin{split}
v_N^{+,\tau} - t_N^+ \geq N^{\alpha}, \hspace{2mm} t_N^- - v_N^{-,\tau}  \geq N^{\alpha}, \hspace{2mm} | v_N^{\pm, \tau} | \leq 2\tau N^{\alpha}, \hspace{2mm} T_N \geq  | v_N^{\pm, \tau} | \mbox{ and }\lfloor \tau_{\pm} N^{\alpha} \rfloor \in [T_0^{\pm}, T_1^{\pm}].
\end{split}
\end{equation}
Let $N_m^{\mathsf{sep}}(\tau, \epsilon/128)$, $N_m^{\mathsf{sep}}(4r_0, \epsilon/128)$ be as in Lemma \ref{keyL3}[$m$], which we have assumed to hold as part of our input in  \eqref{Imp2}. Finally, let $W_4$ and $M^{\mathsf{base}, \mathsf{top}}$ be as in Lemma \ref{LNoParDip} for $k = m$, $p, \lambda, H, H^{RW}$ as in the present lemma, $L = 2 \tau$, $\epsilon = \epsilon/128$ and $M^{\mathsf{top}} = S_1$. We let $V_3$ be sufficiently large so that $V_3 \geq V_2$, $V_3 \geq \max(N_m^{\mathsf{sep}}(\tau, \epsilon/128), N_m^{\mathsf{sep}}(4r_0, \epsilon/128))$ and for $N \geq V_3$ we have $N^{\alpha} \geq W_4$. 

Finally, from Definition \ref{Def1} of an $(\alpha, p, \lambda)$-good sequence, and (\ref{S4S3E10}), we have that there exists $V_4 \geq V_3$ such that for $N \geq V_4$ we have 
\begin{equation}\label{S4S3E13}
\begin{split} 
&\mathbb{P}  \left( \left| N^{-\alpha/2} \left( L_1^N(\lfloor \tau_{\pm} N^{\alpha} \rfloor ) - p \lfloor \tau_{\pm} N^{\alpha} \rfloor \right) + \lambda (N^{-\alpha}\lfloor \tau_{\pm} N^{\alpha} \rfloor)^2 \right| \geq  S_1 \right) < \epsilon/64,\\
&\mathbb{P}  \left( \left| N^{-\alpha/2} \left( L_1^N( v_N^{\pm, \tau} ) - p   v_N^{\pm, \tau}  \right) + \lambda (N^{-\alpha}v_N^{\pm, \tau})^2 \right|  \geq S_1  \right) < \epsilon/64,\\
&\mathbb{P}  \left( \left| N^{-\alpha/2} \left( L_1^N( t_N^{\pm} ) - p  t_N^{\pm}  \right) + \lambda (N^{-\alpha}t_N^{\pm})^2 \right|  \geq S_1  \right) < \epsilon/64.
\end{split}
\end{equation}
We now let $N_1 = V_4$ and set $R^{\mathsf{dip}} = M^{\mathsf{base}, \mathsf{top}}$ as in Lemma \ref{LNoParDip} with the choice of parameters we listed in the previous paragraph. This concludes our specification of $\tau, N_1$ and $R^{\mathsf{dip}} $.\\

{\bf \raggedleft Step 5.} In this final step we prove (\ref{S4S3E1}). As the proofs are quite similar we only establish the second line in (\ref{S4S3E1}).

Since $N_1 \geq N_m^{\mathsf{sep}}(\tau, \epsilon/128)$ we know that for $N \geq N_1$
\begin{equation*}
\mathbb{P}\left( \cup_{i = 1}^{m-1} \left\{ L_i^N(  v^{+,\tau}_N) - L_{i+1}^N(v^{+,\tau}_N)  \leq  \delta^{\mathsf{sep}}_m \right\} \right) < \epsilon/128. 
\end{equation*}
Combining the latter with the second line of (\ref{S4S3E13}) we conclude that for $N \geq N_1$
\begin{equation}\label{S4S3E14}
\begin{split}
&\mathbb{P}\left(\mathsf{Side}^N_+ \right) > 1- 3\epsilon/128, \mbox{ where } \mathsf{Side}^N_+ = \\
& \{ S_1 N^{\alpha/2} - N^{\alpha/2} \lambda(N^{-\alpha}v_N^{+,\tau})^2 \geq L_1^N(v^{+,\tau}_N) - pv^{+,\tau}_N \geq \cdots \geq L_{m}^N(v^{+,\tau}_N) - pv^{+,\tau}_N  \}. 
\end{split}
\end{equation}
Analogous arguments show that for $N \geq N_1$
\begin{equation}\label{S4S3E15}
\begin{split}
&\mathbb{P}\left(\mathsf{Side}^N_- \right) > 1- 3\epsilon/128, \mbox{ where } \mathsf{Side}^N_- = \\
& \{ S_1 N^{\alpha/2} - N^{\alpha/2} \lambda(N^{-\alpha}t_N^{+})^2 \geq L_1^N(t^{+}_N) - pt^{+}_N \geq \cdots \geq L_{m}^N(t^{+}_N) - pt^{+}_N \}. 
\end{split}
\end{equation}
We also define the events
$$\mathsf{Dip}^N_{m+1} = \{ L_{m+1}^N(x) - p x \leq -R^{\mathsf{dip}} N^{\alpha/2} \mbox{ for all } x \in \llbracket t_N^+, v_N^{+,\tau} \rrbracket \},$$
and 
$$\mathsf{LB}^N_{m} = \{ N^{-\alpha/2} \left(L_1^N(\lfloor \tau_{+} N^{\alpha} \rfloor ) - p \lfloor \tau_{+} N^{\alpha} \rfloor \right) + \lambda (N^{-\alpha}\lfloor \tau_{+} N^{\alpha} \rfloor)^2 \geq - S_1 \},$$
and note that from (\ref{S4S3E13}) we have for $N \geq N_1$
\begin{equation}\label{S4S3E16}
\begin{split}
\mathbb{P} \left( \mathsf{LB}^N_{m}\right) > 1 - \epsilon/64.
\end{split}
\end{equation}

From the $(H,H^{RW})$-Gibbs property of $\mathfrak{L}^N$ we have for $N \geq N_1$ (note $T_N \geq | v^{+, \tau}_N |$ from (\ref{S4S3E12}))
\begin{equation}\label{S4S3E17}
\begin{split}
&  \mathbb{P}\left( \mathsf{Dip}^N_{m+1} \cap \mathsf{Side}^N_-  \cap \mathsf{Side}^N_+  \cap \mathsf{LB}^N_{m}   \right)=  \mathbb{E}\left[ \mathbb{E} \left[ {\bf 1 }_{\mathsf{Dip}^N_{m+1}} \cdot {\bf 1 }_{\mathsf{Side}^N_- } \cdot {\bf 1 }_{ \mathsf{Side}^N_+} \cdot   {\bf 1 }_{\mathsf{LB}^N_{m}} \Big{|} \mathcal{F} \right] \right] = \\
&  \mathbb{E}\left[ {\bf 1 }_{\mathsf{Dip}^N_{m+1}} \cdot {\bf 1 }_{\mathsf{Side}^N_- } \cdot {\bf 1 }_{ \mathsf{Side}^N_+} \cdot  \mathbb{E} \left[   {\bf 1 }_{\mathsf{LB}^N_{m}}   \Big{|} \mathcal{F} \right] \right] =  \mathbb{E}\Bigg{[}{\bf 1 }_{\mathsf{Dip}^N_{m+1}} \cdot {\bf 1 }_{\mathsf{Side}^N_- } \cdot {\bf 1 }_{ \mathsf{Side}^N_+}    \times \\
& \mathbb{P}_{H,H^{RW}}^{1,m,t_N^{+}, v_N^{\tau, +},\vec{x},\vec{y},\infty,g}\left(N^{-\alpha/2} \left(\tilde{L}_1(\lfloor \tau_{+} N^{\alpha} \rfloor ) - p \lfloor \tau_{+} N^{\alpha} \rfloor \right) + \lambda (N^{-\alpha}\lfloor \tau_{+} N^{\alpha} \rfloor)^2 \geq - S_1 \right)\Bigg{]},
\end{split}
\end{equation}
where $\mathcal{F} =\mathcal{F}_{ext}(\llbracket 1, m  \rrbracket \times \llbracket t_N^++1,v_N^{\tau,+} - 1 \rrbracket) $ as in (\ref{GibbsCond}), also $\vec{x}= (L^N_1(t_N^+), \dots,L^N_{m}(t_N^+)) $, $\vec{y}= (L^N_1(v_N^{\tau,+}), \dots,L^N_{m}(v_N^{\tau,+})) $ and $g = L^N_{m+1}\llbracket t_N^+, v_N^{\tau, +} \rrbracket$. We mention that in the first equality we used the tower property for conditional expectations and in the second that $ \mathsf{Dip}^N_{m+1}$, $\mathsf{Side}^N_-$, $\mathsf{Side}^N_+ \in \mathcal{F}$. Also $(\tilde{L}_1, \dots,\tilde{L}_{m}) $ is distributed according to $\mathbb{P}_{H,H^{RW}}^{1,m,t_N^{+}, v_N^{\tau, +},\vec{x},\vec{y},\infty,g}$.

From Lemma \ref{LNoParDip} with the same choice of parameters as in Step 4, $T = N^{\alpha}$, $T_1 = v_N^{\tau,+}$, $T_0 = t_N^{+}$ we have $\mathbb{P}$-almost surely for $N \geq N_1$  that 
\begin{equation}\label{S4S3E18}
\begin{split}
& {\bf 1 }_{\mathsf{Dip}^N_{m+1}} \cdot {\bf 1 }_{\mathsf{Side}^N_- } \cdot {\bf 1 }_{ \mathsf{Side}^N_+}    \cdot  \mathbb{P}_{H,H^{RW}}^{1,m,t_N^{+}, v_N^{\tau, +},\vec{x},\vec{y},\infty,g}\Big(N^{-\alpha/2} \left(\tilde{L}_1(\lfloor \tau_{+} N^{\alpha} \rfloor ) - p \lfloor \tau_{+} N^{\alpha} \rfloor \right) + \\
& \lambda (N^{-\alpha}\lfloor \tau_{+} N^{\alpha} \rfloor)^2 \geq -S_1  \Big) \leq  {\bf 1 }_{\mathsf{Dip}^N_{m+1}} \cdot {\bf 1 }_{\mathsf{Side}^N_- } \cdot {\bf 1 }_{ \mathsf{Side}^N_+}  \cdot (\epsilon/64).
\end{split}
\end{equation}
We mention that in deriving the last inequality we used that $x = \lfloor \tau_{+} N^{\alpha} \rfloor \in [T_0, T_1]$, see (\ref{S4S3E12}), and the first inequality in (\ref{S4S3E11}). The fact that $T, T_0, T_1$ satisfy the conditions of Lemma \ref{LNoParDip} follows from the inequalities in (\ref{S4S3E12}). We also mention that we used that on $\mathsf{Side}_-^N \cap \mathsf{Side}_+^N$ the random vectors $\vec{x}, \vec{y}$ as in (\ref{S4S3E18}) satisfy the conditions of Lemma \ref{LNoParDip} with $M^{\mathsf{top}} = S_1$ as in Step 4, while on $\mathsf{Dip}^N_{m+1}$ the vector $g$ satisfies the conditions of Lemma \ref{LNoParDip} with $M^{\mathsf{base}, \mathsf{top}} = R^{\mathsf{dip}}$ as in Step 4.

Combining (\ref{S4S3E17}) and (\ref{S4S3E18}) we conclude that 
$$\mathbb{P}\left( \mathsf{Dip}^N_{m+1} \cap \mathsf{Side}^N_-  \cap \mathsf{Side}^N_+  \cap \mathsf{LB}^N_{m}   \right) \leq (\epsilon/64) \cdot \mathbb{P}\left( \mathsf{Dip}^N_{m+1} \cap \mathsf{Side}^N_-  \cap \mathsf{Side}^N_+   \right) \leq \epsilon/64.$$
The last inequality together with (\ref{S4S3E14}), (\ref{S4S3E15}) and (\ref{S4S3E16}) together give
\begin{equation*}
\begin{split}
&   \mathbb{P} \left( \mathsf{Dip}^N_{m+1} \right) \leq \mathbb{P}\left( \mathsf{Dip}^N_{m+1} \cap \mathsf{Side}^N_-  \cap \mathsf{Side}^N_+  \cap \mathsf{LB}^N_{m}   \right) + \mathbb{P}\left( (\mathsf{Side}^N_-)^c \right) + \\
& \mathbb{P}\left( (\mathsf{Side}^N_+)^c \right)  + \mathbb{P}\left( (\mathsf{LB}^N_{m})^c \right) \leq \epsilon/64 + 3\epsilon/128 + 3\epsilon/128 + \epsilon/64 = 5\epsilon/64 < \epsilon/8.
\end{split}
\end{equation*}
The last equation gives the second line in (\ref{S4S3E1}).
\end{proof}

%
\subsection{Proof of (\ref{Imp3})}\label{Section4.4} In this section we present the proof of (\ref{Imp3}), for which we require the following three results, whose proofs are postponed until Section \ref{Section7} -- see Lemmas \ref{S7LNoBigJump}, \ref{S7LNotClose} and \ref{S7LAccProb}.

\begin{lemma}\label{LNoBigJump} Let $\ell$ have distribution $\mathbb{P}^{T_0,T_1,x,y}_{H^{RW}}$ with $H^{RW}$ as in Definition \ref{AssHR}.  Fix $p \in \mathbb{R}$, $r, \epsilon \in (0,1)$, $M > 0$.  Then we can find $W_5 = W_5(p,r, \epsilon, M, H^{RW}) \in \mathbb{N}$ such that the following holds.
If $T_0, T_1 \in \mathbb{Z}$ are such that $T_1 - T_0 \geq  W_5$, $r (T_1 - T_0) \leq R \leq r^{-1}(T_1 - T_0) $ and $x,y \in \mathbb{R}$ satisfy
$$ \left|x  - pT_0\right|\leq M R^{1/2} \mbox{ and } \left|y  - pT_1 \right| \leq M R^{1/2} $$
then we have 
\begin{equation}\label{ENoBigJump}
\mathbb{P}^{T_0,T_1,x,y}_{H^{RW}} \left( \sup_{ s \in [T_0, T_1-1] }  \left| \ell(s) - \ell(s+1) \right| \leq R^{1/4} \right) \geq 1 - \epsilon.
\end{equation}
\end{lemma}
\begin{remark}
Lemma \ref{LNoBigJump} states that a random walk bridge between well-behaved endpoints $(T_0,x)$ and $(T_1,y)$ is unlikely to make large jumps in a single unit of time.
\end{remark}

\begin{lemma}\label{LNotClose}Fix $k \in \mathbb{N}$, $k \geq 2$ and let $\mathfrak{L} = (L_1, \dots, L_k)$ have law $\mathbb{P}_{H^{RW}}^{1, k, T_0 ,T_1, \vec{x}, \vec{y}}$ as in Definition \ref{Pfree}, where $H^{RW}$ is as in Definition \ref{AssHR}. Fix $p \in \mathbb{R}$, $r \in (0,1)$, $t \in (0,1/3)$, $M, \epsilon > 0$. Then we can find $W_6 = W_6(k,p,r,t,M,\epsilon, H^{RW}) \in \mathbb{N}$ and $\delta = \delta(k, p, r, t,  \epsilon, H^{RW})> 0$, such that the following holds. If $T_0, T_1 \in \mathbb{Z}$ are such that $T_1 - T_0 \geq  W_6$, $r (T_1 - T_0)\leq R \leq r^{-1} (T_1 - T_0)$, $t_0 \in \llbracket T_0, T_1\rrbracket $ is such that $\min( T_1 - t_0, t_0 -T_0 )\geq t (T_1 - T_0)$ and $\vec{x}, \vec{y} \in \mathbb{R}^k$ satisfy
$$ \left|x_i  - pT_0\right|\leq M R^{1/2} \mbox{ and } \left|y_i  - pT_1 \right| \leq M R^{1/2} \mbox{ for $i = 1, \dots, k$}$$
then we have 
\begin{equation}\label{NEotClose}
\mathbb{P}_{H^{RW}}^{1, k, T_0 ,T_1, \vec{x}, \vec{y}} \left( \min_{1 \leq i < j \leq k} \left|L_i(t_0) - L_j(t_0) \right| \leq \delta R^{1/2}  \right) \leq \epsilon.
\end{equation}
\end{lemma}
\begin{remark}
Lemma \ref{LNotClose} states that $k$ independent random walk bridges between well-behaved endpoints $(T_0,x_i)$ and $(T_1,y_i)$ are unlikely to be close to each other at some intermediate time $t_0$.
\end{remark}

\begin{lemma}\label{LAccProb}Fix $k \in \mathbb{N}$ and let $\mathfrak{L} = (L_1, \dots, L_k)$ have law $\mathbb{P}_{H,H^{RW}}^{1, k, T_0 ,T_1, \vec{x}, \vec{y},\infty,g}$ as in Definition \ref{Pfree}, where we assume that $H$ is as in Definition \ref{AssH}, while $H^{RW}$ satisfies the assumptions in Definition \ref{AssHR}. Let $M^{\mathsf{side}, \mathsf{top}},M^{\mathsf{side}, \mathsf{bot}}, M^{\mathsf{base}, \mathsf{top}}, p , t \in \mathbb{R}$ be given such that $M^{\mathsf{side}, \mathsf{top}} \geq M^{\mathsf{side}, \mathsf{bot}}$, $t \in (0, 1/3)$. Then we can find constants $W_7  \in \mathbb{N}$, $\delta > 0$, all depending on $ M^{\mathsf{side}, \mathsf{top}},M^{\mathsf{side}, \mathsf{bot}}, M^{\mathsf{base}, \mathsf{top}}, p , t , k, H$ and $H^{RW}$ so that the following holds.

Set $\Delta T = T_1 - T_0$ and suppose that $\vec{x}, \vec{y} \in \mathbb{R}^k$, $g \in Y^-( \llbracket T_0, T_1 \rrbracket)$ satisfy
\begin{itemize}
\item $ pT_0 +  M^{\mathsf{side},\mathsf{bot}} \Delta T^{1/2} \leq x_i \leq pT_0 + M^{\mathsf{side},\mathsf{top}} \Delta T^{1/2}$ for $i = 1, \dots, k$;
\item $ p T_1 + M^{\mathsf{side},\mathsf{bot}} \Delta T^{1/2}  \leq y_i \leq p T_1 +  M^{\mathsf{side},\mathsf{top}} \Delta T^{1/2}$ for $i = 1, \dots, k$;
\item  $g(j) \leq j p + M^{\mathsf{base},\mathsf{top}} \Delta T^{1/2}$ for $j \in \llbracket T_0, T_1\rrbracket$.
\end{itemize}
For any $\epsilon > 0$, $T_1, T_0 \in \mathbb{Z}$ with $\Delta T \geq W_7$, and $t_1, t_0 \in \llbracket T_0, T_1 \rrbracket$ with $\min(t_0 - T_0 , T_1 - t_1 , t_1 - t_0 ) \geq t \Delta T$
\begin{equation}\label{EAccProb}
\mathbb{P}_{H,H^{RW}}^{1, k, T_0 ,T_1, \vec{x}, \vec{y},\infty,g} \left(  Z_{H, H^{RW}}\left(t_0, t_1, \mathfrak{L}(t_0), \mathfrak{L}(t_1), \infty, g \llbracket t_0, t_1 \rrbracket  \right) \leq \delta \epsilon \right) \leq \epsilon,
\end{equation}
where $Z_{ H,H^{RW}}$ is the acceptance probability of Definition \ref{Pfree}, $g \llbracket t_0, t_1 \rrbracket $ is restriction of $g$ to $\llbracket t_0, t_1 \rrbracket$ and $\mathfrak{L}(a)  = (L_1(a), \dots, L_k(a))$ is the value of $\mathfrak{L}$ with law $\mathbb{P}_{H,H^{RW}}^{1, k, T_0 ,T_1, \vec{x}, \vec{y},\infty,g}$ at $a$.
\end{lemma}
\begin{remark}
Lemma \ref{LAccProb} states that a discrete line ensemble $\mathfrak{L}$ with law $\mathbb{P}_{H,H^{RW}}^{1, k, T_0 ,T_1, \vec{x}, \vec{y},\infty,g}$, whose entrance and exit data $\vec{x}, \vec{y}$ are well-behaved and the bottom bounding curve $g$ is not too high is likely to arrange itself in a configuration that  produces a bounded away from zero acceptance probability in a fraction smaller window $\llbracket t_0, t_1 \rrbracket$ that is well within $\llbracket T_0, T_1 \rrbracket$, where $\mathfrak{L}$ is defined.
\end{remark}

\begin{proof}[Proof of \eqref{Imp3}] We continue with the same notation as in the statement of the lemma and Section \ref{Section3.3}. All constants in the proof depend on the quantities in (\ref{Depend}) as well as $r \geq 2$ and $\epsilon > 0$ as in the statement of the lemma, which are fixed. For clarity we split the proof into four steps.\\

{\bf \raggedleft Step 1.} In this step we prove \eqref{AccP}. Let $r_0 = \lceil r \rceil +1$, $t^{\pm}_N = \lfloor \pm 2r_0 N^{\alpha} \rfloor$ and recall $s^{\pm}_N = \lfloor \pm rN^{\alpha} \rfloor$. Let $R_1 >0$ and $N_1 \in \mathbb{N}, N_1 \geq N_0$ be sufficiently large so that for $N \geq N_1$ we have 
\begin{equation}\label{S4S4E1}
\begin{split}
4r_0 N^{\alpha} \geq |t^{\pm}_N| \mbox{ and } T_N \geq |t^{\pm}_N|, \hspace{2mm} \min(t_N^+ - s_N^+,  s_N^+ - s_N^-, s_N^- - t_N^-) \geq (2r_0)^{-1}(t^{+}_N - t_N^-),
\end{split}
\end{equation}
and also 
\begin{equation}\label{S4S4E2}
\begin{split}
&R_1 \geq \max_{k \in \llbracket 1, m+2 \rrbracket}  \Rt_{k}(4r_0, \epsilon/(4m+8)), \hspace{2mm} R_1 \geq \max_{k \in \llbracket 1, m + 1\rrbracket}   \Rb_{k}(4r_0, \epsilon/(4m+8)), \mbox{ and }  \\
&N_1 \geq \max_{k \in  \llbracket 1, m+2 \rrbracket}  \Nt_{k}(4r_0, \epsilon/(4m+8)), \hspace{2mm} N_1 \geq \max_{k \in \llbracket 1, m + 1\rrbracket}   \Nb_{k}(4r_0, \epsilon/(4m+8)),
\end{split}
\end{equation}
where $\Rt_{k}, \Nt_{k}$ are as in Lemma \ref{keyL1}[$m+2$] and $\Rb_{k}, \Nb_{k}$ are as in Lemma \ref{keyL2}[$m+1$], which we have assumed to hold as part of our input in  \eqref{Imp3}.

Let $W_7, \delta$ be as in Lemma \ref{LAccProb} applied to $k = m+1$, $p, H, H^{RW}$ as in the present lemma, $t = (2r_0)^{-1}$, and $- M^{\mathsf{side}, \mathsf{bot}}  = M^{\mathsf{side}, \mathsf{top}} = M^{\mathsf{base}, \mathsf{top}} = R_1$. We finally let $\Na_{m+1} = \Na_{m+1}(r,\epsilon) \in \mathbb{N}$ be sufficiently large so that $\Na_{m+1} \geq N_1$ and for $N \geq \Na_{m+1}$ we have $t_N^+ - t_N^- \geq W_7$, and also set $\delta^{\mathsf{acc}}_{m+1}(r, \epsilon) = \delta/2$. This fixes $\Na_{m+1}, \delta^{\mathsf{acc}}_{m+1}$, and in the remainder of this step we prove \eqref{AccP}.\\

Define the event
\begin{equation}\label{S4S4E3}
\begin{split}
\mathsf{Bdd}^N_{m+1}=  &\left\{  \max_{ j \in \llbracket 1, m +1 \rrbracket} \big| L^N_j(t_N^{\pm})-pt_N^{\pm} \big|\leq R_1 N^{\alpha/2} \right\} \cap \\
& \left\{ \sup_{x\in [-4r_0 N^{\alpha}, 4r_0 N^{\alpha}] } [ L^N_{m+2}(x)-px ] \leq R_1N^{\alpha/2} \right\},
\end{split}
\end{equation}
and observe that from (\ref{S4S4E2}) and a union bound we have for $N \geq N_1$ (and thus $N \geq \Na_{m+1}$)
\begin{equation}\label{S4S4E4}
\begin{split}
1 - \mathbb{P} \left( \mathsf{Bdd}^N_{m+1} \right) < \epsilon/2.
\end{split}
\end{equation}
We also define the event  
\begin{equation}\label{S4S4E5}
\begin{split}
\mathsf{Z}^N_{m+1}= \{  Z_{H, H^{RW}}\left(s^-_N, s^+_N, \vec{x}, \vec{y}, \infty, L^N_{m+2}\llbracket s^-_N , s^+_N \rrbracket \right) \leq  \delta \epsilon/2 \},
\end{split}
\end{equation}
where $\vec{x}= (L^N_1(s_N^-), \dots,L^N_{m+1}(s_N^-)) $, $\vec{y}= (L^N_1(s_N^+), \dots,L^N_{m+1}(s_N^+)  ) $.

From the $(H,H^{RW})$-Gibbs property of $\mathfrak{L}^N$ we have for $N \geq \Na_{m+1}$ (note $T_N \geq \left| t^{\pm}_N \right|$ from (\ref{S4S4E1}))
\begin{equation}\label{S4S4E6}
\begin{split}
& \mathbb{P}\left( \mathsf{Bdd}^N_{m+1} \cap \mathsf{Z}^N_{m+1} \right) =  \mathbb{E}\left[ \mathbb{E} \left[ {\bf 1}_{ \mathsf{Bdd}^N_{m+1} } \cdot  {\bf 1}_{ \mathsf{Z}^N_{m+1}}  \Big{|} \mathcal{F} \right] \right] =   \mathbb{E}\left[{\bf 1}_{ \mathsf{Bdd}^N_{m+1} } \cdot  \mathbb{E} \left[   {\bf 1}_{ \mathsf{Z}^N_{m+1} } \Big{|} \mathcal{F} \right] \right]  = \\
&\mathbb{E}\left[{\bf 1}_{ \mathsf{Bdd}^N_{m+1} } \cdot  \mathbb{P}_{H,H^{RW}}^{1,m+1,t_N^-, t_N^+,\vec{x}',\vec{y}',\infty,g'}\left(  Z_{H, H^{RW}}\left(s^-_N, s^+_N, \vec{x}, \vec{y}, \infty, L^N_{m+2} \llbracket s^-_N , s^+_N \rrbracket \right) \leq  \delta \epsilon/2\right)\right],
\end{split}
\end{equation}
where $\mathcal{F} =\mathcal{F}_{ext}(\llbracket 1, m + 1 \rrbracket \times \llbracket t_N^-+1, t_N^+ - 1 \rrbracket) $ as in (\ref{GibbsCond}), $\vec{x}'= (L^N_1(t_N^-), \dots,L^N_{m+1}(t_N^-)) $, $\vec{y}'= (L^N_1(t_N^+), \dots,L^N_{m+1}(t_N^+)) $ and $g' = L^N_{m+2}\llbracket t_N^-, t_N^+ \rrbracket$. We mention that in the first equality we used the tower property for conditional expectations and in the second that $ \mathsf{Bdd}^N_{m+1}  \in \mathcal{F}$. Also $(\tilde{L}_1, \dots,\tilde{L}_{m+1}) $ is distributed according to $\mathbb{P}_{H,H^{RW}}^{1,m+1,t_N^-, t_N^+,\vec{x}',\vec{y}',\infty,g'}$ and we have also set $\vec{x}= (\tilde{L}_1(s_N^-), \dots, \tilde{L}_{m+1}(s_N^-)) $, $\vec{y}= (\tilde{L}_1(s_N^+), \dots, \tilde{L}_{m+1}(s_N^+)  ).$

Since $N \geq   \Na_{m+1}$ we have from Lemma \ref{LAccProb} (with the same parameters as in the beginning of the proof), $t_0 = s_N^-$, $t_1 = s_N^+$, $T_0 = t_N^-$, $T_1 = t_N^+$ that $\mathbb{P}$-almost surely
\begin{equation}\label{S4S4E7}
\begin{split}
{\bf 1}_{ \mathsf{Bdd}^N_{m+1} } \cdot  \mathbb{P}_{H,H^{RW}}^{1,m+1,t_N^-, t_N^+,\vec{x}',\vec{y}',\infty,g'}\left(  Z_{H, H^{RW}}\left(s^-_N, s^+_N, \vec{x}, \vec{y}, \infty, L^N_{m+2}\llbracket s^-_N , s^+_N \rrbracket \right) \leq  \delta \epsilon/2\right) \leq \epsilon/2,
\end{split}
\end{equation}
where we used that $\vec{x}', \vec{y}', g'$ satisfy the conditions of Lemma \ref{LAccProb} almost surely on $\mathsf{Bdd}^N_{m+1}$ from (\ref{S4S4E3}) and $t_0, t_1, T_0, T_1$ satisfy the conditions of Lemma \ref{LAccProb} from (\ref{S4S4E1}).

Combining (\ref{S4S4E4}), (\ref{S4S4E6}) and (\ref{S4S4E7}) we get 
\begin{equation*}
\begin{split}
\mathbb{P} \left( \mathsf{Z}^N_{m+1}\right) \leq \mathbb{P} \left( \mathsf{Z}^N_{m+1} \cap  \mathsf{Bdd}^N_{m+1} \right) + \mathbb{P} \left( ( \mathsf{Bdd}^N_{m+1})^c \right)  < \epsilon/2 + \epsilon/2 = \epsilon.
\end{split}
\end{equation*}
The last inequality implies \eqref{AccP}.\\

{\bf \raggedleft Step 2.} In the next three steps we establish \eqref{Sep}. The goal of this step is to specify our choice of $\delta^{\mathsf{sep}}_{m+1}$ and $\Ns_{m+1}$, and set up some useful notation.

Let $t_N^{\pm}$, $N_1, R_1, \mathsf{Bdd}^N_{m+1}$ be as in Step 1. Suppose further that $\Na_{m+1}(2r_0,\epsilon/16)$, $\delta^{\mathsf{acc}}_{m+1}(2r_0, \epsilon/16)$ are as in Step 1, and observe that from Step 1, we have for $N \geq \Na_{m+1}(2r_0,\epsilon/16)$ 
\begin{equation}\label{S4S4E8}
\begin{split}
&\mathbb{P} \left( \tilde{\mathsf{Z}}^N_{m+1} \right) < \epsilon/16, \mbox{ where } \\
&\tilde{\mathsf{Z}}^N_{m+1} = \{  Z_{H, H^{RW}}\left(t^-_N, t^+_N, \vec{x}, \vec{y}, \infty, L^N_{m+2}\llbracket t^-_N , t^+_N \rrbracket \right) \leq  \delta^{\mathsf{acc}}_{m+1}(2r_0, \epsilon/16)\},
\end{split}
\end{equation}
and $\vec{x}= (L^N_1(t_N^-), \dots,L^N_{m+1}(t_N^-)) $, $\vec{y}= (L^N_1(t_N^+), \dots,L^N_{m+1}(t_N^+)  ) $. 

Let $W_6, \delta$ be as in Lemma \ref{LNotClose} for $k = m+1$, $p, H^{RW}$ as in the present lemma, $t = (2r_0)^{-1}$, $r = (8r_0)^{-1}$, $M = R_1$ and $\epsilon = \delta^{\mathsf{acc}}_{m+1}(2r_0, \epsilon/16) \cdot (\epsilon/16)$. With this notation we set $\delta^{\mathsf{sep}}_{m+1} = \delta^{\mathsf{sep}}_{m+1}(r, \epsilon) = \delta/4$. Let $N_4 \in \mathbb{N}$ be sufficiently large so that $N_4 \geq N_1$, $N_4 \geq \Na_{m+1}(2r_0,\epsilon/16)$, and for $N \geq N_4$ 
\begin{equation}\label{S4S4E9}
\begin{split}
& r(t_N^+ - t_N^-) \leq N^{\alpha} \leq r^{-1}(t_N^+ - t_N^-), \hspace{2mm} \min( t_N^+ - s_N^{\pm}, s_N^{\pm} - t_N^-) \geq t (t_N^+ - t_N^-), \\
&  N^{-\alpha/4} < \delta^{\mathsf{sep}}_{m+1}, \mbox{ and }\exp\left( -H(3 N^{\alpha/4} ))\right) < \delta^{\mathsf{acc}}_{m+1}(2r_0, \epsilon/16) (\epsilon/16).
\end{split}
\end{equation}
We mention that the last inequality in (\ref{S4S4E9}) is achievable since by Definition \ref{AssH}  $\lim_{x \rightarrow \infty} H(x) = \infty$.

Finally, we pick $\Ns_{m+1} \in \mathbb{N}$ sufficiently large so that $\Ns_{m+1}  \geq N_4$ and for $N \geq \Ns_{m+1} $ we have $t_N^+ - t_N^- \geq W_5$, where $W_5$ is as in Lemma \ref{LNoBigJump} for $p, H^{RW}$ as in the present lemma, $r = (8r_0)^{-1}$, $\epsilon = (m+1)^{-1} \delta^{\mathsf{acc}}_{m+1}(2r_0, \epsilon/16) \cdot (\epsilon/16)$ and $M = R_1$. 

The above paragraph specifies our choice of $\delta^{\mathsf{sep}}_{m+1} $ and $\Ns_{m+1}$ and in the remainder of this step we fix some useful notation.

Define the events 
\begin{equation}\label{S4S4E10}
\begin{split}
&\mathsf{Sep}^{N}_{m+1} =\cap_{\zeta \in \{\pm 1\}} \left\{  \min_{1 \leq i < j \leq m+1} \left|L^N_i(s_N^{\zeta}) - L^N_j(s_N^{\zeta}) \right| > 4\delta^{\mathsf{sep}}_{m+1} N^{\alpha/2} \right\}, \\
&\mathsf{Osc}^{N}_{m+1} = \left\{  \max_{i \in \llbracket 1, m+ 1\rrbracket}  \max_{ x \in \llbracket t_N^-, t_N^+ -1 \rrbracket}  \left|L^N_i(x)  -L^N_i(x + 1) \right| \leq N^{\alpha/4} \right\}, \\
&\mathsf{Ord}^{N}_{m+1} =  \cup_{\zeta \in \{\pm 1\}} \cup_{i = 1}^{m} \left\{ L_i^N(  s^{\zeta}_N) - L_{i+1}^N(s^{\zeta}_N)  \leq  \delta^{\mathsf{sep}}_{m+1} N^{\alpha/2} \right\}.
\end{split}
\end{equation}

{\bf \raggedleft Step 3.} We claim that for $N \geq \Ns_{m+1}$ we have
\begin{equation}\label{S4S4E11}
\begin{split}
&\mathbb{P} \left( (\mathsf{Sep}^{N}_{m+1})^c \cap \mathsf{Bdd}^N_{m+1} \cap (\tilde{\mathsf{Z}}^N_{m+1})^c \right) \leq \epsilon/8,
\end{split}
\end{equation}
\begin{equation}\label{S4S4E12}
\begin{split}
&\mathbb{P} \left( (\mathsf{Osc}^{N}_{m+1})^c \cap \mathsf{Bdd}^N_{m+1} \cap  (\tilde{\mathsf{Z}}^N_{m+1})^c \right) \leq \epsilon/16,
\end{split}
\end{equation}
\begin{equation}\label{S4S4E13}
\begin{split}
&\mathbb{P} \left(\mathsf{Ord}^{N}_{m+1} \cap \mathsf{Osc}^{N}_{m+1} \cap \mathsf{Sep}^{N}_{m+1}  \cap  (\tilde{\mathsf{Z}}^N_{m+1})^c   \right) \leq \epsilon/16.
\end{split}
\end{equation}
We prove (\ref{S4S4E11}), (\ref{S4S4E12}) and (\ref{S4S4E13}) in the next step. Here we assume their validity and conclude the proof of (\ref{Sep}).

Observe that for $N \geq \Ns_{m+1}$ we have
\begin{equation*}
\begin{split}
&\mathbb{P} \left( \mathsf{Ord}^{N}_{m+1}  \right) \leq \mathbb{P} \left( (\mathsf{Bdd}^N_{m+1})^c \right) + \mathbb{P} \left( \tilde{\mathsf{Z}}^N_{m+1}\right) +  \mathbb{P} \left( \mathsf{Bdd}^N_{m+1} \cap  (\tilde{\mathsf{Z}}^N_{m+1})^c \cap \mathsf{Ord}^{N}_{m+1} \right) < \\
&9\epsilon/16 + \mathbb{P} \left( (\mathsf{Sep}^{N}_{m+1})^c \cap \mathsf{Bdd}^N_{m+1} \cap (\tilde{\mathsf{Z}}^N_{m+1})^c \right)  +  \mathbb{P} \left( (\mathsf{Osc}^{N}_{m+1})^c \cap \mathsf{Bdd}^N_{m+1} \cap  (\tilde{\mathsf{Z}}^N_{m+1})^c\right) + \\
& \mathbb{P} \left(\mathsf{Ord}^{N}_{m+1} \cap \mathsf{Osc}^{N}_{m+1} \cap \mathsf{Sep}^{N}_{m+1}  \cap  (\tilde{\mathsf{Z}}^N_{m+1})^c   \right) \leq 9 \epsilon/16 + 2 \epsilon/16 + \epsilon/16 + \epsilon/16 < \epsilon,
\end{split}
\end{equation*}
where in going from the first to the second line we used (\ref{S4S4E4}) and (\ref{S4S4E8}), and in the next to last inequality we used (\ref{S4S4E11}), (\ref{S4S4E12}) and (\ref{S4S4E13}). The last equation implies (\ref{Sep}).\\

{\bf \raggedleft Step 4.} In this final step we prove (\ref{S4S4E11}), (\ref{S4S4E12}) and (\ref{S4S4E13}).

 We set $\mathcal{F} =\mathcal{F}_{ext}(\llbracket 1, m + 1 \rrbracket \times \llbracket t_N^-+1, t_N^+ - 1 \rrbracket) $ as in (\ref{GibbsCond}), $\vec{x}'= (L^N_1(t_N^-), \dots,L^N_{m+1}(t_N^-)) $, $\vec{y}'= (L^N_1(t_N^+), \dots,L^N_{m+1}(t_N^+)) $ and $g' = L^N_{m+2}\llbracket t_N^-, t_N^+ \rrbracket$. We also write $\mathbb{P}_{H^{RW}}$ in place of $\mathbb{P}_{H^{RW}}^{1, m+1, t_N^- , t_N^+, \vec{x}', \vec{y}'}$, $\mathbb{E}_{H^{RW}}$ in place of $\mathbb{E}_{H^{RW}}^{1, m+1, t_N^- , t_N^+, \vec{x}', \vec{y}'}$ and $W_H(\tilde{\mathfrak{L}})$ in place of $W_{H}^{1, m+1, t_N^- ,t_N^+,\infty, g'} (\tilde{L}_{1}, \dots, \tilde{L}_{m+1})$ with $\tilde{\mathfrak{L}} = (\tilde{L}_1, \dots, \tilde{L}_{m+1})$ having distribution $\mathbb{P}_{H^{RW}}$.

From the $(H,H^{RW})$-Gibbs property of $\mathfrak{L}^N$ we have for $N \geq \Ns_{m+1}$ (note $T_N \geq \left| t^{\pm}_N \right|$ from (\ref{S4S4E1}))
\begin{equation}\label{S4S4E14}
\begin{split}
& \mathbb{P}\left(  \mathsf{Bdd}^N_{m+1} \cap (\tilde{\mathsf{Z}}^N_{m+1})^c  \cap  (\mathsf{Sep}^{N}_{m+1})^c \right) =  \mathbb{E}\left[ \mathbb{E} \left[ {\bf 1}_{ \mathsf{Bdd}^N_{m+1} } \cdot  {\bf 1}_{(\tilde{\mathsf{Z}}^N_{m+1})^c} \cdot  {\bf 1}_{ (\mathsf{Sep}^{N}_{m+1})^c }  \Big{|} \mathcal{F} \right] \right] = \\
&\mathbb{E}\left[ {\bf 1}_{ \mathsf{Bdd}^N_{m+1} } \cdot  {\bf 1}_{ (\tilde{\mathsf{Z}}^N_{m+1})^c} \cdot  \mathbb{E} \left[  {\bf 1}_{ (\mathsf{Sep}^{N}_{m+1})^c }  \Big{|} \mathcal{F} \right] \right]   = \mathbb{E}\Bigg{[} {\bf 1}_{ \mathsf{Bdd}^N_{m+1} } \cdot  {\bf 1}_{ (\tilde{\mathsf{Z}}^N_{m+1})^c} \cdot  \\
&\frac{\mathbb{E}_{H^{RW}} \left[W_H(\tilde{\mathfrak{L}}) \cdot {\bf 1}  \left\{  \min_{\zeta \in \{+, -\},1 \leq i < j \leq m+1} \left|\tilde{L}_i(s_N^{\zeta}) - \tilde{L}_j(s_N^{\zeta}) \right| \leq 4\delta^{\mathsf{sep}}_{m+1} N^{\alpha} \right\} \right] }{Z_{H, H^{RW}}\left(t^-_N, t^+_N, \vec{x}, \vec{y}, \infty, L^N_{m+2}\llbracket t^-_N , t^+_N \rrbracket \right)  } \Bigg{]} \leq \\
&  \mathbb{E}\left[ {\bf 1}_{ \mathsf{Bdd}^N_{m+1} } \cdot  {\bf 1}_{(\tilde{\mathsf{Z}}^N_{m+1})^c} \cdot  \frac{2 \cdot \delta^{\mathsf{acc}}_{m+1}(2r_0, \epsilon/16) \cdot (\epsilon/16)}{ \delta^{\mathsf{acc}}_{m+1}(2r_0, \epsilon/16) } \right] \leq \epsilon/8.
\end{split}
\end{equation}
We mention that in the first equality we used the tower property for conditional expectations and in the second that $ \mathsf{Bdd}^N_{m+1}, (\tilde{\mathsf{Z}}^N_{m+1})^c  \in \mathcal{F}$. The application of the $(H,H^{RW})$-Gibbs property was used in going from the second to the third line as well as (\ref{RND}). In going from the third to the fourth line we used that on $(\tilde{\mathsf{Z}}^N_{m+1})^c$ we have $Z_{H, H^{RW}}\left(t^-_N, t^+_N, \vec{x}, \vec{y}, \infty, L^N_{m+2}\llbracket t^-_N , t^+_N \rrbracket \right)   > \delta^{\mathsf{acc}}_{m+1}(2r_0, \epsilon/16)$ -- this lower bounds the denominator. We also used that $W_H(\tilde{\mathfrak{L}}) \in [0,1]$ and Lemma \ref{LNotClose} to upper bound the numerator, where Lemma \ref{LNotClose} is applied with the choice of parameters as in Step 2 twice --  once for $t_0 = s_N^-$ and once for $t_0 = s_N^+$, with $T_0 = t_N^-$, $T_1 = t_N^+$ and $R = N^{\alpha}$. We mention that the conditions of Lemma \ref{LNotClose} are satisfied $\mathbb{P}$-almost surely for $N \geq \Ns_{m+1}$ on $\mathsf{Bdd}^N_{m+1}$ in view of (\ref{S4S4E3}) and our choice of $\Ns_{m+1}$. Equation (\ref{S4S4E14}) implies (\ref{S4S4E11}).

By an analogous argument we have
\begin{equation}\label{S4S4E15}
\begin{split}
& \mathbb{P}\left(  \mathsf{Bdd}^N_{m+1} \cap (\tilde{\mathsf{Z}}^N_{m+1})^c  \cap  (\mathsf{Osc}^{N}_{m+1})^c \right) =  \mathbb{E}\left[ \mathbb{E} \left[ {\bf 1}_{ \mathsf{Bdd}^N_{m+1} } \cdot  {\bf 1}_{(\tilde{\mathsf{Z}}^N_{m+1})^c} \cdot  {\bf 1}_{ (\mathsf{Osc}^{N}_{m+1})^c }  \Big{|} \mathcal{F} \right] \right] = \\
&\mathbb{E}\left[ {\bf 1}_{ \mathsf{Bdd}^N_{m+1} } \cdot  {\bf 1}_{ (\tilde{\mathsf{Z}}^N_{m+1})^c} \cdot  \mathbb{E} \left[  {\bf 1}_{ (\mathsf{Osc}^{N}_{m+1})^c }  \Big{|} \mathcal{F} \right] \right]   =  \mathbb{E}\Bigg{[} {\bf 1}_{ \mathsf{Bdd}^N_{m+1} } \cdot  {\bf 1}_{ (\tilde{\mathsf{Z}}^N_{m+1})^c} \cdot \\
& \frac{\mathbb{E}_{H^{RW}} \left[W_H(\tilde{\mathfrak{L}}) \cdot {\bf 1} \left\{  \max_{i \in \llbracket 1, m+ 1\rrbracket}  \max_{ x \in \llbracket t_N^-, t_N^+ -1 \rrbracket}  \left|\tilde{L}_i(x)  - \tilde{L}_i(x + 1) \right| > N^{\alpha/4} \right\} \right] }{Z_{H, H^{RW}}\left(t^-_N, t^+_N, \vec{x}, \vec{y}, \infty, L^N_{m+2}\llbracket t^-_N , t^+_N \rrbracket \right)  } \Bigg{]} \leq \\
& \mathbb{E}\left[ {\bf 1}_{ \mathsf{Bdd}^N_{m+1} } \cdot  {\bf 1}_{(\tilde{\mathsf{Z}}^N_{m+1})^c} \cdot  \frac{(m+1) \cdot  (m+1)^{-1} \delta^{\mathsf{acc}}_{m+1}(2r_0, \epsilon/16) \cdot (\epsilon/16)}{ \delta^{\mathsf{acc}}_{m+1}(2r_0, \epsilon/16) } \right] \leq \epsilon/16.
\end{split}
\end{equation}
The above inequalities are derived in the same way as in the previous paragraph. The only exception is that in going from the third to the fourth line we applied Lemma \ref{LNoBigJump} instead of Lemma \ref{LNotClose}. We mention that we applied  Lemma \ref{LNoBigJump} $m+1$ times (once for each $\tilde{L}_i$) with the choice of parameters as in Step 2, $T_0 = t_N^-$, $T_1 = t_N^+$ and $R = N^{\alpha}$. The conditions of Lemma \ref{LNoBigJump} are satisfied $\mathbb{P}$-almost surely for $N \geq \Ns_{m+1}$ on $\mathsf{Bdd}^N_{m+1}$ in view of (\ref{S4S4E3}) and our choice of $\Ns_{m+1}$. This proves (\ref{S4S4E12}).\\

We finally establish (\ref{S4S4E13}). Let us denote by $\mathsf{Sep}_{m+1}$, $\mathsf{Osc}_{m+1}$ and $\mathsf{Ord}_{m+1}$ the events in (\ref{S4S4E10}) with $\mathfrak{L}^N = (L_1^N, \dots, L_{m+1}^N)$ replaced with $\tilde{\mathfrak{L}} = (\tilde{L}_1, \dots, \tilde{L}_{m+1})$, and suppose that $\tilde{\mathfrak{L}} \in Y(\llbracket 1, m+1 \rrbracket \times \llbracket t_N^-, t_N^+\rrbracket)$ satisfy the conditions in $\mathsf{Sep}_{m+1} \cap \mathsf{Osc}_{m+1} \cap \mathsf{Ord}_{m+1}$. We then have
$$ \tilde{L}_i(  s^{\zeta}_N) - \tilde{L}_{i+1}(s^{\zeta}_N)  \leq  \delta^{\mathsf{sep}}_{m+1}N^{\alpha/2} \mbox{ and } | \tilde{L}_i(  s^{\zeta}_N) - \tilde{L}_{i+1}(s^{\zeta}_N) | > 4\delta^{\mathsf{sep}}_{m+1}N^{\alpha/2}$$
for some $\zeta \in \{- , +\}$ and $i \in \llbracket 1, m \rrbracket$. We fix such a choice of $i$ and $\zeta$.

If $\zeta = +$ let $u^{\zeta}_N = s^{+}_N - 1$ and $v^{\zeta}_N = s^{+}_N$, while if $\zeta = -$ let $u^{\zeta}_N = s^{-}_N $ and $v^{\zeta}_N = s^{-}_N + 1$. Note that since $\tilde{\mathfrak{L}}$ satisfies the conditions in $\mathsf{Osc}_{m+1}$ we have
$$|\tilde{L}_i(  u^{\zeta}_N)  - \tilde{L}_i(  s^{\zeta}_N)| \leq N^{\alpha/4} \mbox{ and }|\tilde{L}_{i+1}(  v^{\zeta}_N)  - \tilde{L}_{i+1}(  s^{\zeta}_N)| \leq N^{\alpha/4}.$$
The above two sets of inequalities imply that
$$ \tilde{L}_i(  u^{\zeta}_N) - \tilde{L}_{i+1}(v^{\zeta}_N)   < -4\delta^{\mathsf{sep}}_{m+1} N^{\alpha/2} + N^{\alpha/2} \leq-3 N^{\alpha/4},$$
where in the last inequality we used the third inequality in (\ref{S4S4E9}). Consequently, we have 
\begin{equation}\label{S4S4E16}
\begin{split}
&{\bf 1}_{\mathsf{Sep}_{m+1}} \cdot {\bf 1}_{\mathsf{Osc}_{m+1}} \cdot {\bf 1}_{\mathsf{Ord}_{m+1}} \cdot W_H(\tilde{\mathfrak{L}}) \leq \\
& \exp\left( -H \left( \tilde{L}_{i+1}(v^{\zeta}_N) -  \tilde{L}_i(  u^{\zeta}_N) \right)\right) \leq \exp\left( -H(3 N^{\alpha/4} ))\right) < \delta^{\mathsf{acc}}_{m+1}(2r_0, \epsilon/16) \cdot (\epsilon/16).
\end{split}
\end{equation}
We mention that in the first inequality we used the definition of $W_H$ from (\ref{WH}) and the fact that $H(x) \in [0, \infty]$, in the second inequality we used that $H$ is increasing and in the last inequality we used the last inequality in (\ref{S4S4E9}).

Arguing as in the derivation of (\ref{S4S4E14}) and (\ref{S4S4E15}) we get
\begin{equation*}
\begin{split}
& \mathbb{P}\left(   (\tilde{\mathsf{Z}}^N_{m+1})^c  \cap \mathsf{Ord}^{N}_{m+1} \cap \mathsf{Osc}^{N}_{m+1} \cap \mathsf{Sep}^{N}_{m+1} \right) =  \mathbb{E}\left[ \mathbb{E} \left[ {\bf 1}_{(\tilde{\mathsf{Z}}^N_{m+1})^c} \cdot  {\bf 1}_{\mathsf{Ord}^{N}_{m+1} }  \cdot  {\bf 1}_{\mathsf{Osc}^{N}_{m+1}  }  \cdot  {\bf 1}_{\mathsf{Sep}^{N}_{m+1} }  \Big{|} \mathcal{F} \right] \right] = \\
&\mathbb{E}\left[ {\bf 1}_{(\tilde{\mathsf{Z}}^N_{m+1})^c} \cdot  \mathbb{E} \left[   {\bf 1}_{\mathsf{Ord}^{N}_{m+1} }  \cdot  {\bf 1}_{\mathsf{Osc}^{N}_{m+1}  }  \cdot  {\bf 1}_{\mathsf{Sep}^{N}_{m+1} }   \Big{|} \mathcal{F} \right] \right]   =   \\
&\mathbb{E}\left[ {\bf 1}_{(\tilde{\mathsf{Z}}^N_{m+1})^c} \cdot \frac{\mathbb{E}_{H^{RW}} \left[W_H(\tilde{\mathfrak{L}}) \cdot {\bf 1}_{\mathsf{Sep}_{m+1}} \cdot {\bf 1}_{\mathsf{Osc}_{m+1}} \cdot {\bf 1}_{\mathsf{Ord}_{m+1}} \right] }{Z_{H, H^{RW}}\left(t^-_N, t^+_N, \vec{x}, \vec{y}, \infty, L^N_{m+2}\llbracket t^-_N , t^+_N \rrbracket \right)  } \right] \leq \\
& \mathbb{E}\left[  {\bf 1}_{(\tilde{\mathsf{Z}}^N_{m+1})^c} \cdot  \frac{\delta^{\mathsf{acc}}_{m+1}(2r_0, \epsilon/16) (\epsilon/16)}{ \delta^{\mathsf{acc}}_{m+1}(2r_0, \epsilon/16) } \right] \leq \epsilon/16,
\end{split}
\end{equation*}
where in going from the third to the fourth line we used (\ref{S4S4E16}). The last inequality proves (\ref{S4S4E13}), which completes the proof of the lemma.
\end{proof}

%
\section{Continuous grand monotone coupling lemma}\label{Section6}
The goal of this section is to prove our continuous grand monotone coupling lemma, Lemma \ref{MonCoup}. We continue with the same notation as in Section \ref{Section2}.

%
\subsection{Preliminary results}\label{Section6.1}

We summarize some useful notation in the following definition.
\begin{definition}\label{FProj}Let $n,k \in \mathbb{N}$ and suppose that $H^{RW}$ and $\vec{H} = (H_1, \dots, H_{n})$ are as in Definition \ref{Pfree}. Let $\vec{x}, \vec{y} \in \mathbb{R}^{k}$, $\vec{z} \in Y^-(\llbracket 0, n+1\rrbracket)$, and fix subsets $A, B \subset \llbracket 1, k \rrbracket \times \llbracket 1, n \rrbracket$ and a point $P \in  \llbracket 1, k \rrbracket \times \llbracket 1, n \rrbracket$ such that $A \cap B = \emptyset$ and $A \cup B \cup \{ P \} = \llbracket 1, k \rrbracket \times \llbracket 1, n \rrbracket$. Finally, we fix $ x\in \mathbb{R}$ and $\vec{v} \in Y(A)$ (recall that $Y(A)$ was defined in Definition \ref{YVec}). With the above data we define the real-valued function $h_{k,n}^{\vec{x}, \vec{y}, \vec{z}}(x; \vec{v}; A, B, P)$ via 
\begin{equation}\label{S6DefProj}
h_{k,n}^{\vec{x}, \vec{y}, \vec{z}}(x; \vec{v}; A, B, P) = \int_{Y(B)} \prod_{i = 1}^{k}\prod_{j = 1}^{n+1} \g(u_{i,j} - u_{i, j -1})   \prod_{i = 1}^{k}\prod_{j = 1}^{n} e^{-H_j(u_{i+1, j+1} - u_{i,j})} \prod_{(i,j) \in B} du_{i,j},
\end{equation}
where $u_{k+1,i} = \vec{z}(i)$ for $i \in \llbracket 1, n+1 \rrbracket$, $u_{i,0}= x_i$ and $u_{i,n+1} = y_i$ for $i \in \llbracket 1, k \rrbracket$, $u_{i,j} = \vec{v}(i,j)$ for $(i,j) \in A$, and $u_{p_1, p_2} = x$ for $(p_1, p_2) = P$. When $B = \emptyset$ the function $h_{k,n}^{\vec{x}, \vec{y}, \vec{z}}(x; \vec{v}; A, B, P) $ stands for the integrand in (\ref{S6DefProj}) and there is no integral.

If $(a,b), (x,y) \in \mathbb{Z}^2$ we write $(a,b) <_{\mathbb{Z}^2} (x,y)$ if $a < x$ or $a = x$ and $b <y $ and note that $<_{\mathbb{Z}^2}$ defines a complete order on $\mathbb{Z}^2$. For any $P = (p_1, p_2) \in \llbracket 1, k \rrbracket \times \llbracket 1, n \rrbracket$ we let $A_P = \{ (x,y) \in \llbracket 1, k \rrbracket \times \llbracket 1, n \rrbracket: (p_1, p_2) <_{\mathbb{Z}^2} (x,y) \}$ and $B_P = \{ (x,y) \in \llbracket 1, k \rrbracket \times \llbracket 1, n \rrbracket: (x,y) <_{\mathbb{Z}^2} (p_1, p_2)\}$.
\end{definition}
\begin{remark}\label{ProjRem} In words, the function $h_{k,n}^{\vec{x}, \vec{y}, \vec{z}}(x; \vec{v}; A, B, P)$ in (\ref{S6DefProj}) for $P = (p_1, p_2)$ is the density of $L_{p_1}(p_2)$, where $\mathfrak{L} = (L_{1}, \dots, L_k)$ is a $\llbracket 1, k \rrbracket$-indexed line ensemble on $\llbracket 0, n + 1 \rrbracket$, whose law is $ \mathbb{P}_{\vec{H},H^{RW}}^{1, k, 0 , n +1 , \vec{x}, \vec{y},\infty, \vec{z}}$, conditioned on $\{ L_i(j): (i,j) \in A\}$. 
\end{remark}

\begin{lemma}\label{MCKF} Assume the same notation as in Definition \ref{FProj} and suppose that $H^{RW}, H_1, \dots, H_{n}$ are convex, and that $H_1, \dots, H_{n}$ are increasing. Suppose that $\vec{x}, \vec{y}, \vec{x}', \vec{y}' \in \mathbb{R}^k$ with $x_i \leq x_i'$ and $y_i \leq y_i'$ for $i =1, \dots, k$. In addition, suppose that $\vec{z}, \vec{z}' \in Y^-(\llbracket 0, n+1\rrbracket)$ with $\vec{z}(i) \leq \vec{z}'(i)$ for $i  \in \llbracket 0, n+1 \rrbracket$. Finally, suppose $P  = (p_1, p_2) \in \llbracket 1, k \rrbracket \times \llbracket 1, n \rrbracket$, $\vec{v}, \vec{v}' \in Y(A_P)$, and $s,s' \in \mathbb{R}$ with $\vec{v}(a) \leq \vec{v}'(a)$ for $a \in A_P$ and $s \leq s'$. Then we have
\begin{equation}\label{conv1}
\frac{\int_{-\infty}^s h_{k,n}^{\vec{x}, \vec{y}, \vec{z}}(w; \vec{v}; A_P, B_P, P) dw }{\int_{-\infty}^s h_{k,n}^{\vec{x}', \vec{y}', \vec{z}'}(w; \vec{v}'; A_P, B_P, P)   dw} \geq\frac{\int_{-\infty}^{s'} h_{k,n}^{\vec{x}, \vec{y}, \vec{z}}(w; \vec{v}; A_P, B_P, P) dw }{\int_{-\infty}^{s'} h_{k,n}^{\vec{x}', \vec{y}', \vec{z}'}(w; \vec{v}'; A_P, B_P, P)   dw}.
\end{equation}
In particular, we have
\begin{equation}\label{conv2}
\frac{\int_{-\infty}^s h_{k,n}^{\vec{x}, \vec{y}, \vec{z}}(w; \vec{v}; A_P, B_P, P) dw }{\int_{-\infty}^{\infty} h_{k,n}^{\vec{x}, \vec{y}, \vec{z}}(w; \vec{v}; A_P, B_P, P) dw } \geq\frac{\int_{-\infty}^s h_{k,n}^{\vec{x}', \vec{y}', \vec{z}'}(w; \vec{v}'; A_P, B_P, P)   dw}{\int_{-\infty}^{\infty} h_{k,n}^{\vec{x}', \vec{y}', \vec{z}'}(w; \vec{v}'; A_P, B_P, P)   dw}.
\end{equation}
\end{lemma}
\begin{proof}
Observe first that all the integrals are well-defined by our assumptions on $\vec{H},H^{RW}$. Also it is clear that we can obtain (\ref{conv2}) from (\ref{conv1}) upon taking the limit $s' \rightarrow \infty$. We thus focus on establishing (\ref{conv1}). For clarity we split the proof into three steps.\\

{\bf \raggedleft Step 1.} We begin by summarizing some notation and results that will be used throughout the proof. Note that if $F$ is convex, $x \leq y$ and $\Delta \geq 0$ then
\begin{equation}\label{convexFun}
F(x + \Delta) - F(x) \leq F(y+\Delta) -F(y),
\end{equation}
which can be deduced from \cite[Exercise 4.23]{Rudin}. In particular, using that $H^{RW}$ is convex we note that for $x,y,b,d \in \mathbb{R}$, $x \leq y$ and $b \leq d$ we have 
\begin{equation}\label{S6GConv}
\begin{split}
&\g(d - y) \g(b-x) \geq \g(d-x) \g(b-y),
\end{split}
\end{equation}
and using the convexity of $H_j$ and the fact that it is increasing we see for $j \in \llbracket 1, n\rrbracket$, $x, y \in \mathbb{R}$, $b,d \in [-\infty, \infty)$ with $x \leq y$ and $b \leq d$
\begin{equation}\label{S6HConv}
\begin{split}
&e^{-H_j(d - y)} e^{-H_j(b-x)} \geq e^{-H_j(d-x)} e^{-H_j(b-y)}.
\end{split}
\end{equation}

We define the function 
\begin{equation}\label{DefHHat}
\begin{split}
\hat{h}_{k,n}^{\vec{x}, \vec{y}, \vec{z}}(x; \vec{v}; P) = \int_{Y(B_P )} &\prod_{i = 1}^{p_1-1}\prod_{j = 1}^{n+1} \g(u_{i,j} - u_{i, j -1})   \prod_{i = 1}^{p_1 - 1}\prod_{j = 1}^{n} e^{-H_j(u_{i+1, j+1} - u_{i,j})} \times \\
&\prod_{j = 1}^{p_2} \g(u_{p_1,j} - u_{p_1, j -1})   e^{-H_j(u_{p_1+1, j+1} - u_{p_1,j})} \prod_{(i,j) \in B_P} du_{i,j},
\end{split}
\end{equation}
where the $u$ variables are as in Definition \ref{FProj} for the sets $A_P, B_P$. If $B_P= \emptyset$ the function $\hat{h}_{k,n}^{\vec{x}, \vec{y}, \vec{z}}(w; \vec{v}; P) $ denotes the integrand itself and there is no integral. We claim that for any $s,w \in \mathbb{R}$ with $w \leq s$ we have
\begin{equation}\label{IndIneq}
\begin{split}
\hat{h}_{k,n}^{\vec{x}, \vec{y}, \vec{z}}(s; \vec{v}; P) \hat{h}_{k,n}^{\vec{x}', \vec{y}', \vec{z}'}(w; \vec{v}'; P) \leq \hat{h}_{k,n}^{\vec{x}', \vec{y}', \vec{z}'}(s; \vec{v}'; P) \hat{h}_{k,n}^{\vec{x}, \vec{y}, \vec{z}}(w; \vec{v}; P) .
\end{split}
\end{equation}
We will prove (\ref{IndIneq}) in the steps below. Here we assume its validity and conclude the proof of (\ref{conv1}).\\

Define the function $R(s; \vec{x},\vec{y}, \vec{z}; \vec{v}; P) := \log \big( \int_{-\infty}^s h_{k,n}^{\vec{x}, \vec{y}, \vec{z}}(w; \vec{v}; A_P, B_P, P) dw  \big)$ and notice that for fixed $\vec{x},\vec{y}, \vec{z}, \vec{v},P$ the latter function is differentiable in $s$ and its derivative is given by
$$\frac{d R(s; \vec{x},\vec{y}, \vec{z}; \vec{v};P) }{ds} = \frac{ h_{k,n}^{\vec{x}, \vec{y}, \vec{z}}(s; \vec{v}; A_P, B_P, P) }{\int_{-\infty}^s h_{k,n}^{\vec{x}, \vec{y}, \vec{z}}(w; \vec{v}; A_P, B_P, P) dw}.$$
Upon taking logarithms on both sides of (\ref{conv1}) we see that our goal is to show that for $s' \geq s$
\begin{equation*}
\begin{split}
&R(s; \vec{x},\vec{y}, \vec{z}; \vec{v};P) - R(s; \vec{x}',\vec{y}', \vec{z}'; \vec{v}';P) \geq  R(s'; \vec{x},\vec{y}, \vec{z}; \vec{v}; P) - R(s'; \vec{x}',\vec{y}', \vec{z}'; \vec{v}'; P) ,
\end{split}
\end{equation*}
and so it suffices to show $\partial_s\big[R(s; \vec{x},\vec{y}, \vec{z}; \vec{v}; P) - R(s; \vec{x}',\vec{y}', \vec{z}'; \vec{v}'; P) \big] \leq 0$ or equivalently
$$ \frac{ h_{k,n}^{\vec{x}, \vec{y}, \vec{z}}(s; \vec{v}; A_P, B_P, P) }{\int_{-\infty}^s h_{k,n}^{\vec{x}, \vec{y}, \vec{z}}(w; \vec{v}; A_P, B_P, P) dw} \leq \frac{ h_{k,n}^{\vec{x}', \vec{y}', \vec{z}'}(s; \vec{v}'; A_P, B_P, P) }{\int_{-\infty}^s h_{k,n}^{\vec{x}', \vec{y}', \vec{z}'}(w; \vec{v}'; A_P, B_P, P) dw}$$
for all $s \in \mathbb{R}$. Cross-multiplying the above we see it suffices to show that for $w \leq s$ 
\begin{equation*}
\begin{split}
h_{k,n}^{\vec{x}, \vec{y}, \vec{z}}(s; \vec{v}; A_P, B_P, P) h_{k,n}^{\vec{x}', \vec{y}', \vec{z}'}(w; \vec{v}'; A_P, B_P, P) \leq h_{k,n}^{\vec{x}', \vec{y}', \vec{z}'}(s; \vec{v}'; A_P, B_P, P) h_{k,n}^{\vec{x}, \vec{y}, \vec{z}}(w; \vec{v}; A_P, B_P, P) .
\end{split}
\end{equation*}
We now observe from (\ref{S6DefProj}) and (\ref{DefHHat}) that 
$$h_{k,n}^{\vec{x}, \vec{y}, \vec{z}}(s; \vec{v}; A_P, B_P, P) = \hat{h}_{k,n}^{\vec{x}, \vec{y}, \vec{z}}(s; \vec{v}; P) G(u_{p_1, p_2 + 1} - s)\cdot C(\vec{x}, \vec{y}, \vec{z}, \vec{v}),$$
where $C(\vec{x}, \vec{y}, \vec{z}, \vec{v}) > 0$ is a positive constant that depends on $\vec{x}, \vec{y}, \vec{z}, \vec{v}$, $u_{p_1, p_2 + 1} = \vec{v}(p_1, p_2+1)$ if $p_2 + 1 \leq n$ and $u_{p_1, p_2 + 1} = y_{p_1}$ if $p_2 = n$. Consequently, the last inequality above is equivalent to 
\begin{equation}\label{S6R1}
\begin{split}
&\hat{h}_{k,n}^{\vec{x}, \vec{y}, \vec{z}}(s; \vec{v};  P) \hat{h}_{k,n}^{\vec{x}', \vec{y}', \vec{z}'}(w; \vec{v}'; P) G(u_{p_1, p_2 + 1} - s) G(u'_{p_1, p_2 + 1} - w) \leq \\
& \hat{h}_{k,n}^{\vec{x}', \vec{y}', \vec{z}'}(s; \vec{v}';  P) \hat{h}_{k,n}^{\vec{x}, \vec{y}, \vec{z}}(w; \vec{v}; P)G(u'_{p_1, p_2 + 1} - s) G(u_{p_1, p_2 + 1} - w),
\end{split}
\end{equation}
where $u'_{p_1, p_2 + 1} = \vec{v}'(p_1, p_2+1)$ if $p_2 + 1 \leq n$ and $u'_{p_1, p_2 + 1} = y'_{p_1}$ if $p_2 = n$.

Since $s \geq w$ and $u'_{p_1, p_2 + 1}  \geq u_{p_1, p_2 + 1} $ by assumption we see that (\ref{S6R1}) follows from (\ref{IndIneq}) and (\ref{S6GConv}) applied to $d = u'_{p_1, p_2 + 1} $, $y = s$, $b = u_{p_1, p_2 + 1}$ and $x = w$. This concludes the proof of the lemma, modulo (\ref{IndIneq}), which we establish in the next steps.\\

{\bf \raggedleft Step 2.} In the next two steps we prove (\ref{IndIneq}) by induction on $k$. In this step we prove the base case $k = 1$ and mention that our proof in this case is an adaptation of the one in \cite[Lemma 2.13]{BCDB}. \\

Since $k = 1$ we know that $p_1 = 1$ and $p_2 \in \llbracket 1, n \rrbracket$. We proceed to prove (\ref{IndIneq}) by induction on $p_2$. For the base case $p_2 = 1$ we have that (\ref{IndIneq}) is equivalent to
\begin{equation*}
 \g(s-x_1) e^{-H_1(\vec{z}(2) - s)} \g(w -x_1') e^{-H_1(\vec{z}'(2) - w)} \leq  \g(w-x_1) e^{-H_1(\vec{z}(2) - w)} \g(s -x_1') e^{-H_1(\vec{z}'(2) - s)},
\end{equation*}
which follows from (\ref{S6GConv}) applied to $d = s$, $y = x_1'$, $b = w$ and $x = x_1$, and (\ref{S6HConv}) applied to $j = 1$, $d = \vec{z}'(2)$, $y = s$, $b = \vec{z}(2)$ and $x = w$.

Suppose that we have proved (\ref{IndIneq}) for $p_2 = m$ and we wish to show it for $p_2 = m+1 \leq n$. Using (\ref{DefHHat}) we can rewrite (\ref{IndIneq}) for $p_2 = m+1$ as
\begin{equation}\label{IndIneqBC0}
\begin{split}
e^{-H_{m+1}(\vec{z}(m+2) - s )} e^{-H_{m+1}(\vec{z}'(m+2) - w)} &\int_{ \mathbb{R}}   \hat{h}_{k,n}^{\vec{x}, \vec{y}, \vec{z}}(x; \vec{v}_Q; Q) \g(s - x) dx  \times \\
&  \int_{ \mathbb{R}}   \hat{h}_{k,n}^{\vec{x}', \vec{y}', \vec{z}'}(y; \vec{v}'_Q; Q) \g(w - y)  dy\leq \\
e^{-H_{m+1}(\vec{z}'(m+2) - s)}e^{-H_{m+1}(\vec{z}(m+2) - w )} &\int_{ \mathbb{R}} \hat{h}_{k,n}^{\vec{x}', \vec{y}', \vec{z}'} (x; \vec{v}'_Q; Q) \g(s - x)  dx \times \\
&  \int_{ \mathbb{R}} \hat{h}_{k,n}^{\vec{x}, \vec{y}, \vec{z}}(y; \vec{v}_Q; Q) \g(w - y)   dy,
\end{split}
\end{equation}
where $Q = (1,m)$ and $\vec{v}_Q, \vec{v}'_Q \in Y(A_Q) = Y(A_P \cup \{P\})$ are such that $\vec{v}_Q (a) = \vec{v}(a)$, $\vec{v}'_Q(a) =\vec{v}'(a)$ for $a \in A_P$ and $\vec{v}_Q (P) = \vec{v}'_Q (P)  = 0$ (note that when $k = 1 = p_1$ the function $\hat{h}_{k,n}^{\vec{x}, \vec{y}, \vec{z}}(x; \vec{v}_Q; Q)$ does not depend on $\vec{v}_Q (P)$ and so the value we assign to $\vec{v}_Q (P)$ and $\vec{v}'_Q (P)$ is immaterial). Using (\ref{S6HConv}) with $j = m+1$, $d = \vec{z}'(m+2)$, $y = s$, $b = \vec{z}(m+2)$, $x = w$ we see that to prove (\ref{IndIneqBC0}) it suffices to show that for $s \geq w$ we have
\begin{equation}\label{IndIneqBC1}
\begin{split}
&\iint_{ \mathbb{R}^2} \hat{h}_{k,n}^{\vec{x}, \vec{y}, \vec{z}}(x; \vec{v}_Q; Q) \g(s - x)   \hat{h}_{k,n}^{\vec{x}', \vec{y}', \vec{z}'}(y; \vec{v}'_Q; Q) \g(w - y) dx dy \leq \\
&\iint_{ \mathbb{R}^2} \hat{h}_{k,n}^{\vec{x}', \vec{y}', \vec{z}'}(x; \vec{v}'_Q; Q) \g(s - x)  \hat{h}_{k,n}^{\vec{x}, \vec{y}, \vec{z}}(y; \vec{v}_Q; Q) \g(w - y) dx dy.
\end{split}
\end{equation}

Splitting the integrals in (\ref{IndIneqBC1}) over $\{x< y \}$ and $\{y <x \}$ and swapping the $x,y$ labels in the region $\{y <x \}$ we see that (\ref{IndIneqBC1}) is equivalent to
\begin{equation*}
\begin{split}
\iint_{ x < y} & \hat{h}_{k,n}^{\vec{x}, \vec{y}, \vec{z}}(x; \vec{v}_Q; Q) \g(s - x)   \hat{h}_{k,n}^{\vec{x}', \vec{y}', \vec{z}'}(y; \vec{v}'_Q; Q) \g(w - y) +  \\
& \hat{h}_{k,n}^{\vec{x}, \vec{y}, \vec{z}}(y; \vec{v}_Q; Q) \g(s - y)   \hat{h}_{k,n}^{\vec{x}', \vec{y}', \vec{z}'}(x; \vec{v}'_Q; Q) \g(w - x)  dx dy \leq \\
\iint_{ x < y} & \hat{h}_{k,n}^{\vec{x}', \vec{y}', \vec{z}'}(x; \vec{v}'_Q; Q) \g(s - x)  \hat{h}_{k,n}^{\vec{x}, \vec{y}, \vec{z}}(y; \vec{v}_Q; Q) \g(w - y) + \\
&\hat{h}_{k,n}^{\vec{x}', \vec{y}', \vec{z}'}(y; \vec{v}'_Q; Q) \g(s - y)  \hat{h}_{k,n}^{\vec{x}, \vec{y}, \vec{z}}(x; \vec{v}_Q; Q) \g(w - x) dx dy.
\end{split}
\end{equation*}
The latter inequality would follow if we can show that the first integrand is pointwise dominated by the second integrand, which is equivalent to proving that if $x < y$ we have
\begin{equation*}
\begin{split}
& A B XY + CD ZW \leq CD XY + ABZW, \mbox{ or equivalently } (AB - CD) (XY - ZW) \leq 0, \mbox{ where }\\
&A = \hat{h}_{k,n}^{\vec{x}, \vec{y}, \vec{z}}(x; \vec{v}_Q; Q) ,  B=   \hat{h}_{k,n}^{\vec{x}', \vec{y}', \vec{z}'}(y; \vec{v}'_Q; Q),  X = \g(s - x)  , Y =  \g(w - y) ,  \\
& C = \hat{h}_{k,n}^{\vec{x}, \vec{y}, \vec{z}}(y; \vec{v}_Q; Q) , D=   \hat{h}_{k,n}^{\vec{x}', \vec{y}', \vec{z}'}(x; \vec{v}'_Q; Q) , Z =   \g(s - y) , W =  \g(w - x) .
\end{split}
\end{equation*}
We observe that 
\begin{equation*}
XY - ZW \leq 0 \iff  \g(s - x)  \g(w - y)   \leq \g(s - y)   \g(w - x) , 
\end{equation*}
and the latter holds by (\ref{S6HConv}) with $d = s$, $b = w$ and $x,y$ as above (here we recall that $s \geq w$, $y > x$). In addition, by the induction hypothesis, (\ref{IndIneq}) applied to $k = 1$ and $Q$, we have
$$AB = \hat{h}_{k,n}^{\vec{x}, \vec{y}, \vec{z}}(x; \vec{v}_Q; Q) \hat{h}_{k,n}^{\vec{x}', \vec{y}', \vec{z}'}(y; \vec{v}'_Q; Q) \geq \hat{h}_{k,n}^{\vec{x}, \vec{y}, \vec{z}}(y; \vec{v}_Q; Q) \hat{h}_{k,n}^{\vec{x}', \vec{y}', \vec{z}'}(x; \vec{v}'_Q; Q) = CD,$$
and so $AB - CD \geq 0$. The last two inequalities give $(AB - CD) (XY - ZW) \leq 0$ and we conclude that $A B XY + CD ZW \leq CD XY + ABZW$. This proves (\ref{IndIneqBC1}), and so (\ref{IndIneq}) holds for $k = 1$ and $p_2 = m+1$. This concludes the induction step and we conclude (\ref{IndIneq}) for $k =1 $ and all $p_2 \in \llbracket 1, n \rrbracket$.\\

{\bf \raggedleft Step 3.} In this step we assume that (\ref{IndIneq}) holds for $k = r$ and proceed to prove it for $k = r +1 \geq 2$. \\

Suppose first that $p_1 \leq r$. Then from (\ref{DefHHat}) we have that 
$$\hat{h}_{r + 1,n}^{\vec{x}, \vec{y}, \vec{z}}(x; \vec{v}; P) = \hat{h}_{r,n}^{\vec{x}_r, \vec{y}_r, \vec{z}_r}(x; \vec{v}_r; P), \mbox{ and } \hat{h}_{r + 1,n}^{\vec{x}', \vec{y}', \vec{z}'}(x; \vec{v}'; P) = \hat{h}_{r,n}^{\vec{x}'_r, \vec{y}'_r, \vec{z}'_r}(x; \vec{v}'_r; P),$$
where 
\begin{equation}\label{ChangeCoord}
\begin{split}
&\vec{x}_r = (x_1, \dots, x_r), \hspace{2mm} \vec{y}_r = (y_1, \dots, y_r), \hspace{2mm} \vec{x}'_r = (x_1', \dots, x_r'), \hspace{2mm} \vec{y}'_r = (y_1', \dots, y_r') \\
&\vec{v}_r(a) = \vec{v}(a), \hspace{2mm} \vec{v}_r'(a) = \vec{v}'(a) \mbox{ for } a \in A_P \cap \llbracket 1, r  \rrbracket \times \llbracket 1, n \rrbracket, \vec{z}_r(0) = x_{r+1},\vec{z}'_r(0) = x'_{r+1}, \\
&\vec{z}_r(i) = \vec{v}(r+1, i), \hspace{2mm} \vec{z}'_r(i) =\vec{v}'(r+1, i) \mbox{ for $i \in \llbracket 1, n \rrbracket$},  \vec{z}_r(n+1) = y_{r+1}, \vec{z}'_r(n+1) = y'_{r+1} .
\end{split}
\end{equation}
In particular, we see that (\ref{IndIneq}) for $k = r+1$ and $p_1 \leq r$ is equivalent to
$$\hat{h}_{r,n}^{\vec{x}_r, \vec{y}_r, \vec{z}_r}(s; \vec{v}_r; P)  \hat{h}_{r,n}^{\vec{x}'_r, \vec{y}'_r, \vec{z}'_r}(w; \vec{v}'_r; P) \leq   \hat{h}_{r,n}^{\vec{x}'_r, \vec{y}'_r, \vec{z}'_r}(s; \vec{v}'_r; P) \hat{h}_{r,n}^{\vec{x}_r, \vec{y}_r, \vec{z}_r}(w; \vec{v}_r; P),$$
which holds by the induction hypothesis (\ref{IndIneq}) for $k = r$.  This proves (\ref{IndIneq}) when $k = r+1$ and $p_1 \leq k$ and in the sequel we focus on the remaining case $p_1 = r+1$.\\

To conclude the proof we need to show (\ref{IndIneq}) for $k = r+1 = p_1$ and $p_2 \in \llbracket 1, n \rrbracket$. As in Step 2 we proceed by induction on $p_2$. For the base case $p_2 = 1$ we use (\ref{DefHHat}) to rewrite (\ref{IndIneq}) as 
\begin{equation*}
\begin{split}
&G(s - x_{r+1})e^{-H_1(\vec{z}(2) - s )} G(w - x'_{r+1})e^{-H_1(\vec{z}'(2) - w )}  \times \\
& \int_{ \mathbb{R}} \hat{h}_{r,n}^{\vec{x}_r, \vec{y}_r, \vec{z}_r}(x;  \vec{v}_r ; Q) \g(y_r - x) dx   \int_{ \mathbb{R}} \hat{h}_{r,n}^{\vec{x}'_r, \vec{y}'_r, \vec{z}'_r}(y; \vec{v}_r' ; Q) \g(y_r' - y)  dy \leq \\
&G(w - x_{r+1})e^{-H_1(\vec{z}(2) - w )} G(s - x'_{r+1})e^{-H_1(\vec{z}'(2) - s )}  \times \\
& \int_{ \mathbb{R}} \hat{h}_{r,n}^{\vec{x}'_r, \vec{y}'_r, \vec{z}'_r}(x; \vec{v}_r' ; Q) \g(y_r' - x)  dx \int_{ \mathbb{R}} \hat{h}_{r,n}^{\vec{x}_r, \vec{y}_r, \vec{z}_r}(y;  \vec{v}_r ; Q) \g(y_r - y) dy ,
\end{split}
\end{equation*}
where $Q = (r,n)$ and $\vec{x}_r, \vec{y}_r, \vec{x}_r', \vec{y}_r', \vec{z}_r,\vec{z}_r', \vec{v}_r, \vec{v}'_r$ are as in (\ref{ChangeCoord}) with the exception of $\vec{z}_r(1), \vec{z}'_r(1)$, which we both set to $0$. We remark that $\vec{v}_r, \vec{v}_r' \in Y(\emptyset)$ and in deriving the last inequality we used that $ \hat{h}_{r,n}^{\vec{x}_r, \vec{y}_r, \vec{z}_r}(x;  \vec{v}_r ; Q)$ and $\hat{h}_{r,n}^{\vec{x}'_r, \vec{y}'_r, \vec{z}'_r}(x; \vec{v}_r' ; Q)$ are independent of $\vec{z}_r(1)$ and $\vec{z}'_r(1)$, respectively. Cancelling the two integrals on both sides we see that it suffices to show that for $s \geq w$ we have
\begin{equation*}
\begin{split}
&G(s - x_{r+1})e^{-H_1(\vec{z}(2) - s )} G(w - x'_{r+1})e^{-H_1(\vec{z}'(2) - w )}  \leq \\
&G(w - x_{r+1})e^{-H_1(\vec{z}(2) - w )} G(s - x'_{r+1}) e^{-H_1(\vec{z}'(2) - s )},
\end{split}
\end{equation*}
which follows from (\ref{S6GConv}) applied to $d = s$, $y = x_{r+1}'$, $b = w$ and $x = x_{r+1}$, and (\ref{S6HConv}) applied to $j = 1$, $d = \vec{z}'(2)$, $y = s$, $b = \vec{z}(2)$ and $x = w$. This proves (\ref{IndIneq}) for $k = r+1 = p_1$ and $p_2 = 1$.

Suppose that we have proved (\ref{IndIneq}) for $k = r+1 = p_1$, $p_2 = m$ and we wish to show it for $p_2 = m+1 \leq n$. Using (\ref{DefHHat}) we can rewrite (\ref{IndIneq}) for $k = r+1 = p_1$, $p_2 = m + 1$ as
\begin{equation}\label{IndIneqIH1}
\begin{split}
 e^{-H_{m+1}(\vec{z}(m+2) - s )} e^{-H_{m+1}(\vec{z}'(m+2) - w)} &\int_{ \mathbb{R}}   \hat{h}_{k,n}^{\vec{x}, \vec{y}, \vec{z}}(x; \vec{v}_Q; Q) \g(s - x) dx \times \\
&  \int_{ \mathbb{R}}  \hat{h}_{k,n}^{\vec{x}', \vec{y}', \vec{z}'}(y; \vec{v}'_Q; Q) \g(w - y)  dy \leq \\
e^{-H_{m+1}(\vec{z}'(m+2) - s)}e^{-H_{m+1}(\vec{z}(m+2) - w )}  & \int_{ \mathbb{R}}  \hat{h}_{k,n}^{\vec{x}', \vec{y}', \vec{z}'}(x; \vec{v}'_Q; Q) \g(s - x)  dx \times \\
& \int_{ \mathbb{R}}  \hat{h}_{k,n}^{\vec{x}, \vec{y}, \vec{z}}(y; \vec{v}_Q; Q) \g(w - y)  dy,
\end{split}
\end{equation}
where $Q = (r+1,m)$ and $\vec{v}_Q, \vec{v}'_Q \in Y(A_Q) = Y(A_P \cup \{P\})$ are such that $\vec{v}_Q (a) = \vec{v}(a)$, $\vec{v}'_Q(a) =\vec{v}'(a)$ for $a \in A_P$ and $\vec{v}_Q (P) = \vec{v}'_Q (P)  = 0$ (note that when $k = r+1 = p_1$ the function $\hat{h}_{k,n}^{\vec{x}, \vec{y}, \vec{z}}(x; \vec{v}_Q; Q)$ does not depend on $\vec{v}_Q (P)$ and so the value we assign to $\vec{v}_Q (P)$ and $\vec{v}'_Q (P)$ is immaterial). At this time we can repeat the proof of (\ref{IndIneqBC0}) from Step 2 verbatim to establish (\ref{IndIneqIH1}). This concludes the proof of (\ref{IndIneq}) for $k = r+1 = p_1$ and $p_2 = m + 1$, which shows by the induction on $m$ that (\ref{IndIneq}) holds for all $P \in \llbracket 1, k \rrbracket \times \llbracket 1, n\rrbracket$ if $k = r + 1$. This completes the induction step for $k$ and so we conclude (\ref{IndIneq}) for all $k, n \in \mathbb{N}$ and $P \in \llbracket 1, k \rrbracket \times \llbracket 1, n\rrbracket$. This concludes the proof of the lemma. 
\end{proof}

%
\subsection{Proof of Lemma  \ref{MonCoup}}\label{Section6.2}
For clarity we split the proof into five steps. In the first step we explain our construction of the probability space $( \Omega^{k,T}, \mathcal{F}^{k,T},\mathbb{P}^{k,T})$ and the random vectors $\ell^{k,T,\vec{x},\vec{y}, \vec{z}} \in Y( \llbracket 1, k \rrbracket \times \llbracket 0, T-1 \rrbracket)$ for all $\vec{x},\vec{y} \in \mathbb{R}^k$ and $\vec{z} \in Y^-(\llbracket 0, T-1 \rrbracket)$ on this space. In the second step we make two claims about the function $\Phi^{k,T}$ in the statement of the lemma and assuming the validity of these claims prove the parts I and II of the lemma. The two claims are proved in Steps 3 and 4, and in Step 5 we conclude the proof of part III of the lemma.\\

{\bf \raggedleft Step 1.} In this step we explain how to construct the probability space $( \Omega^{k,T}, \mathcal{F}^{k,T},\mathbb{P}^{k,T})$ and the random vectors $\ell^{k,T,\vec{x},\vec{y}, \vec{z}} \in Y( \llbracket 1, k \rrbracket \times \llbracket 0, T-1 \rrbracket)$. 

If $T = 2$ we take $\Omega^{k,T}$ to be a set with one point $\omega_0$, $\mathcal{F}^{k,T}$ to be the discrete $\sigma$-algebra and $\mathbb{P}^{k,T}$ to be the unit mass at $\omega_0$. The random vectors $\ell^{k,T,\vec{x},\vec{y}, \vec{z}}$ are then defined by $\ell^{k,T,\vec{x},\vec{y}, \vec{z}}(i, \cdot ) = (x_i,y_i)$ for $i \in \llbracket 1, k \rrbracket$ and clearly satisfy the conditions of the lemma. 

If $T \geq 3$ we let $( \Omega^{k,T}, \mathcal{F}^{k,T},\mathbb{P}^{k,T}) = \big((0,1)^{k(T-2)}, \mathcal{B}((0,1)^{k(T-2)}),\lambda \big),$ where $\mathcal{B}((0,1)^{k(T-2)})$ is the Borel $\sigma$-algebra on $(0,1)^{k(T-2)}$ and $\lambda$ is the Lebesgue measure on this set. This specifies the choice of $( \Omega^{k,T}, \mathcal{F}^{k,T},\mathbb{P}^{k,T})$. We next define the vector $\ell^{k,T,\vec{x},\vec{y}, \vec{z}}$. 

We first set $\ell^{k,T,\vec{x},\vec{y}, \vec{z}}(i, 0) = x_i$ and $\ell^{k,T,\vec{x},\vec{y}, \vec{z}}(i, T-1) = y_i$ for $i \in \llbracket 1, k \rrbracket$. We next let $P_m$ be an enumeration of the elements in $\llbracket 1, k \rrbracket \times \llbracket 1, T-2 \rrbracket$ so that $P_1 <_{\mathbb{Z}^2}  \cdots <_{\mathbb{Z}^2} P_{k(T-2) - 1}<_{\mathbb{Z}^2}  P_{k(T-2)}$, where we recall that $<_{\mathbb{Z}^2}$ was defined in Definition \ref{FProj}. We also define for $\vec{v} \in Y(A_{P_m})$ and $s \in \mathbb{R}$ the function
\begin{equation}\label{CDFRec}
F_{k,T-2}^{\vec{x}, \vec{y},\vec{z}} (s; \vec{v}; m): = \frac{\int_{-\infty}^s h_{k,T-2}^{\vec{x}, \vec{y}, \vec{z}}(w; \vec{v}; A_{P_m}, B_{P_m}, P_m) dw }{\int_{-\infty}^{\infty} h_{k,T-2}^{\vec{x}, \vec{y}, \vec{z}}(w; \vec{v}; A_{P_m}, B_{P_m}, P_m) dw },
\end{equation}
where $h_{k,T-2}^{\vec{x}, \vec{y}, \vec{z}}(w; \vec{v}; A_P, B_P, P) $, $A_{P}$ and $B_{P}$ are as in Definition \ref{FProj}. Notice that by assumption we know that $F_{k,T-2}^{\vec{x}, \vec{y},\vec{z}}(\cdot ; \vec{v}; m)$ is strictly increasing and defines a bijection between $\mathbb{R}$ and $(0,1)$. 

We denote the points $\omega \in (0,1)^{k(T-2)}$ by $\omega = (u_1, \dots, u_{k(T-2)})$ and specify $\ell^{k,T,\vec{x},\vec{y}, \vec{z}}(P_n)(\omega)$ recursively on $m$, with 
\begin{equation}\label{S6DefLRec}
\begin{split}
\ell^{k,T,\vec{x},\vec{y}, \vec{z}}(P_{m}) := [F_{k,T-2}^{\vec{x}, \vec{y},\vec{z}} (\cdot; \vec{v}_{m}; m )]^{-1}(u_{m}) \mbox{ for } m = k(T-2), k(T-2) - 1, \dots, 1,
\end{split}
\end{equation}
where $\vec{v}_{k(T-2)} \in Y(\emptyset)$ and if $\ell^{k,T,\vec{x},\vec{y}, \vec{z}}(P_{n})$ has been specified for $n = m +1$ we have $\vec{v}_m \in Y(A_{P_m})$ are defined via $\vec{v}_m(a) = \ell^{k,T,\vec{x},\vec{y}, \vec{z}}(a) $ for $a \in A_{P_m}$. This concludes the construction of $\ell^{k,T,\vec{x},\vec{y}, \vec{z}}$.\\

{\bf \raggedleft Step 2.} In this step we show that the construction of Step 1 satisfies parts I and II of the lemma. From our construction in Step 1 it is clear that $( \Omega^{k,T}, \mathcal{F}^{k,T},\mathbb{P}^{k,T})$ satisfies the conditions in part III of the lemma. We make the following two claims about $\Phi^{k,T}$ from part III. We claim that
\begin{equation}\label{S2Phi1}
\begin{split}
&\mbox{ the function $\Phi^{k,T}$ is a bijection between $\mathbb{R}^k \times \Omega^{k,T} \times \mathbb{R}^k \times  Y^-(\llbracket 0, T-1 \rrbracket) $ and }\\
& \mbox{ $ Y(\llbracket 1, k \rrbracket \times \llbracket 0, T-1 \rrbracket) \times   Y^-(\llbracket 0, T-1 \rrbracket)$ }
\end{split}
\end{equation}
and if $[\Phi^{k,T}]^{-1}$ denotes its inverse function then for any sequence $(\vec{w}^n, \vec{z}^n)$ converging to $(\vec{w}^\infty, \vec{z}^\infty)$ in $ Y(\llbracket 1, k \rrbracket \times \llbracket 0, T-1 \rrbracket) \times  Y^-(\llbracket 0, T-1 \rrbracket)$ we have that 
\begin{equation}\label{S2Phi2}
\lim_{n \rightarrow \infty} [\Phi^{k,T}]^{-1} (\vec{w}^n, \vec{z}^n) = [\Phi^{k,T}]^{-1}(\vec{w}^\infty, \vec{z}^\infty).
\end{equation}
The claims in (\ref{S2Phi1}) and (\ref{S2Phi2}) will be proved in Steps 3 and 4 below. Here we assume their validity and proceed to prove parts I and II of the lemma. \\

Notice that if we fix $\vec{z} \in  Y^-(\llbracket 0, T-1 \rrbracket)$ then by (\ref{S2Phi1}) and (\ref{S2Phi2}) we know that the function $[\Phi^{k,T}]^{-1} (\vec{w}, \vec{z})$ defines a continuous bijection between $Y(\llbracket 1, k \rrbracket \times \llbracket 0, T-1 \rrbracket)$ and $\mathbb{R}^k \times (0,1)^{k(T-2)} \times \mathbb{R}^k$ as a function of $\vec{w}$. By the Invariance of domain theorem \cite[Theorem 36.5]{Munk2} we see that $\Phi^{k,T}$ is also continuous and hence by restriction for fixed $\vec{x},\vec{y} \in \mathbb{R}^k$ we have that $\ell^{k,T,\vec{x},\vec{y}, \vec{z}}$ is a continuous function of $\omega$. In particular, the vector-valued functions $\ell^{k,T,\vec{x},\vec{y}, \vec{z}}$ in Step 1 are random vectors.

We next check that $\ell^{k,T,\vec{x},\vec{y}, \vec{z}}(\omega)$ has the law $\mathbb{P}_{\vec{H}, H^{RW}}^{ 1,k, 0, T - 1,\vec{x}, \vec{y}, \infty, \vec{z}}$ as in Definition \ref{Pfree}. We proceed to show that $(\ell^{k,T,\vec{x},\vec{y}, \vec{z}}(P_{k(T-2)}), \dots, \ell^{k,T,\vec{x},\vec{y}, \vec{z}}(P_{m}))$ has the same joint law as $(\mathfrak{L}(P_{k(T-2)}), \dots, \mathfrak{L}(P_{m}))$ for each $m = k(T-2), k(T-2) - 1, \dots, 1$, where $\mathfrak{L} = (L_{1}, \dots, L_k)$ is a $\llbracket 1, k \rrbracket$-indexed line ensemble on $\llbracket 0, T-1 \rrbracket$, whose law is $ \mathbb{P}_{\vec{H},H^{RW}}^{1, k, 0 , T-1 , \vec{x}, \vec{y},\infty, \vec{z}}$. If true then the equality of laws at $m = 1$ would imply the desired statement.

Observe that under $\mathbb{P}^{k,T}$ we have that $u_1, \dots, u_{k(T-2)}$ in Step 1 are i.i.d. uniform random variables on $(0,1)$. In addition, as explained in Remark \ref{ProjRem}, we have that $F_{k,T-2}^{\vec{x}, \vec{y},\vec{z}} (s; \vec{v}; m)$ from (\ref{CDFRec}) is nothing but the cumulative distribution function of $\mathfrak{L}(P_m)$, conditioned on $\{ \mathfrak{L}(P): P \in A_{P_m}\}$. The latter two observations and (\ref{S6DefLRec}) imply that $\ell^{k,T,\vec{x},\vec{y}, \vec{z}}(P_{k(T-2)})$ has the same law as $\mathfrak{L}(P_{k(T-2)})$. In addition, if $(\ell^{k,T,\vec{x},\vec{y}, \vec{z}}(P_{k(T-2)}), \dots, \ell^{k,T,\vec{x},\vec{y}, \vec{z}}(P_{m+1}))$ has the same joint law as $(\mathfrak{L}(P_{k(T-2)}), \dots, \mathfrak{L}(P_{m+1}))$ then (\ref{S6DefLRec}) implies that $(\ell^{k,T,\vec{x},\vec{y}, \vec{z}}(P_{k(T-2)}), \dots, \ell^{k,T,\vec{x},\vec{y}, \vec{z}}(P_{m}))$ has the same joint law as $(\mathfrak{L}(P_{k(T-2)}), \dots, \mathfrak{L}(P_{m}))$. This proves that $\ell^{k,T,\vec{x},\vec{y}, \vec{z}}(\omega)$ has the law $\mathbb{P}_{\vec{H}, H^{RW}}^{ 1,k, 0, T - 1,\vec{x}, \vec{y}, \infty, \vec{z}}$ and hence part I of the lemma holds.\\

In the remainder of this step we prove part II of the lemma, and assume that $\vec{x}, \vec{y}, \vec{z}, \vec{x}', \vec{y}', \vec{z}'$ are as in part II. From our construction we immediately get
$$\ell^{k,T,\vec{x},\vec{y}, \vec{z}}(i, 0) = x_i \leq x_i' = \ell^{k,T,\vec{x}',\vec{y}', \vec{z}'}(i, 0) \mbox{ and } \ell^{k,T,\vec{x},\vec{y}, \vec{z}}(i, T-1) = y_i \leq y_i' = \ell^{k,T,\vec{x}',\vec{y}', \vec{z}'}(i, T-1),$$ 
and so it suffices to show that 
\begin{equation}\label{P2LI}
\ell^{k,T,\vec{x},\vec{y}, \vec{z}}(P_m) \leq  \ell^{k,T,\vec{x}',\vec{y}', \vec{z}'}(P_m) \mbox{ for $m = k (T-2), \dots, 1$}.
\end{equation}
It follows from Lemma \ref{MCKF} (here we use that $H^{RW}, H_j$ are convex and $H_j$ are increasing for $j = 0, \dots, T-2$) that 
\begin{equation}\label{S6FIneq}
F_{k,T-2}^{\vec{x}, \vec{y},\vec{z}} (s; \vec{v}; m) \geq F_{k,T-2}^{\vec{x}', \vec{y}',\vec{z}'} (s; \vec{v}'; m),
\end{equation}
provided that $\vec{v}, \vec{v}' \in Y(A_{P_m})$ are such that $\vec{v}(a) \leq \vec{v}'(a)$ for $a \in A_{P_m}$. 

If $m =  k (T-2)$ we have from (\ref{S6DefLRec}) and (\ref{S6FIneq}), applied to $\vec{v}_m = \vec{v}'_m \in Y(\emptyset)$, that 
\begin{equation}\label{PointIneqS6}
\ell^{k,T,\vec{x},\vec{y}, \vec{z}}(P_m) = [F_{k,T-2}^{\vec{x}, \vec{y},\vec{z}} (\cdot ; \vec{v}_m; m)]^{-1}(u_m) \leq[F_{k,T-2}^{\vec{x}', \vec{y}',\vec{z}'} (\cdot ; \vec{v}'_m; m)]^{-1}(u_m) = \ell^{k,T,\vec{x}',\vec{y}', \vec{z}'}(P_m) , 
\end{equation}
which proves (\ref{P2LI}) when $m = k(T-2)$. Assuming we have established (\ref{P2LI}) for $m = k(T-2), k(T-2) - 1, \dots, n + 1 \geq 2$ we see from (\ref{S6DefLRec}) and (\ref{S6FIneq}), applied to $\vec{v}_n = \vec{v}'_n \in Y(A_{P_n})$, such that $\vec{v}_n(a) = \ell^{k,T,\vec{x},\vec{y}, \vec{z}}(a) \leq  \ell^{k,T,\vec{x}',\vec{y}', \vec{z}'}(a)  = \vec{v}_n'(a)$ for $a \in A_{P_n}$ that (\ref{PointIneqS6}) holds for $m = n$, which gives (\ref{P2LI}) for $m = n$. We conclude that (\ref{P2LI}) holds for all $m = k(T-2), k(T-2) - 1, \dots, 1$, which proves part II of the lemma.\\

{\bf \raggedleft Step 3.} In this step we prove (\ref{S2Phi1}). We define the function $\Psi^{k,T}:  Y(\llbracket 1, k \rrbracket \times \llbracket 0, T-1 \rrbracket) \times   Y^-(\llbracket 0, T-1 \rrbracket) \rightarrow \mathbb{R}^k \times (0,1)^{k(T-2)} \times \mathbb{R}^k \times  Y^-(\llbracket 0, T-1 \rrbracket)$ with $(0,1)^0$ denoting the set containing the single point $\omega_0$ as follows.

If $T = 2$ then the map is given by 
$$\Psi^{k,T}(\vec{w}, \vec{z}) = (\vec{x}, \omega_0, \vec{y}, \vec{z}) \mbox{ where } x_i = w(i,0) \mbox{ and }y_i = w(i,1) \mbox{ for $i \in \llbracket 1, k \rrbracket$}.$$
If $T \geq 3$ then the map is given by
\begin{equation}\label{InvFun}
\begin{split}
\Psi^{k,T}(\vec{w}, \vec{z}) = (\vec{x}, \vec{u}, \vec{y}, \vec{z}) \mbox{ where } &x_i = w(i,0) \mbox{ and }y_i = w(i,T-1) \mbox{ for $i \in \llbracket 1, k \rrbracket$ and }\\
&u_m =  F_{k,T-2}^{\vec{x}, \vec{y},\vec{z}} (\vec{w}(P_m); \vec{w}_{m}; m ),
\end{split}
\end{equation}
where $P_m$ and  $F_{k,T-2}^{\vec{x}, \vec{y},\vec{z}} (\cdot ; \vec{w}_{m}; m )$ are as in Step 1, and $\vec{w}_m  \in Y(A_{P_m})$ is such that $\vec{w}_m(a) = \vec{w}(a)$ for $a \in A_{P_m}$. We claim that for all $T \geq 2$, the function $\Psi^{k,T}$ is a left and right inverse to the function $\Phi^{k,T}$. If true, the latter will clearly imply (\ref{S2Phi1}).

 When $T = 2$ the latter is trivial, since the maps are basically the identity map, except that $\Psi^{k,2}$ inserts the coordinate $\omega_0$ between $\vec{w}(\cdot, 0)$ and $\vec{w}(\cdot, 1)$ in the vector $(\vec{w}, \vec{z})$, while $\Phi^{k,2}$ removes it. In the remainder we suppose that $T \geq 3$ and prove
\begin{equation}\label{ASDF1}
\begin{split}
&(\vec{\rho}, \vec{\zeta}) =  (\vec{w}, \vec{z}) \mbox{ and }  (\vec{\alpha}, \vec{\eta}, \vec{\beta}, \vec{\sigma}) = ( \vec{x}, \vec{u},  \vec{y}, \vec{z}), \mbox{ where } \\
&(\vec{\rho}, \vec{\zeta}) :=\Phi^{k,T} \left( \Psi^{k,T}(\vec{w}, \vec{z}) \right) \mbox{ and } (\vec{\alpha}, \vec{\eta}, \vec{\beta}, \vec{\sigma}) := \Psi^{k,T} \left( \Phi^{k,T}( \vec{x}, \vec{u}, \vec{y}, \vec{z})\right)
\end{split}
\end{equation}
for all $\vec{w} \in  Y(\llbracket 1, k \rrbracket \times \llbracket 0, T-1 \rrbracket) $, $\vec{z} \in   Y^-(\llbracket 0, T-1 \rrbracket)$, $\vec{x}, \vec{y} \in \mathbb{R}^{k}$ and $\vec{u} \in (0,1)^{k(T-2)}$. 

From the definition of $\Phi^{k,T}$ (see part III and Step 1) and $\Psi^{k,T}$ in (\ref{InvFun}) we see that $\vec{\zeta} = \vec{z}$, $\vec{\sigma} = \vec{z}$, $\vec{\alpha} = \vec{x}$, $\vec{\beta} = \vec{y}$ and $\vec{\rho}(i,j) = \vec{w}(i,j)$ for $j \in \{0, T-1\}$ and $i \in \llbracket 1, k \rrbracket$. We are thus left with proving 
\begin{equation}\label{ASDF2}
\begin{split}
\vec{\rho}(P_m) = \vec{w}(P_m) \mbox{ and } \eta_m = u_m \mbox{ for } m = k(T-2), \dots, 1. 
\end{split}
\end{equation}

From the definition of $\Phi^{k,T}$ and $\Psi^{k,T}$ we have for each $m = k(T-2), k(T-2) - 1, \dots, 1$ that
\begin{equation}\label{ASDF3}
\begin{split}
&\vec{\rho}(P_m) = [F_{k,T-2}^{\vec{x}, \vec{y},\vec{z}} (\cdot; \vec{\ell}_{m}; m )]^{-1}(F_{k,T-2}^{\vec{x}, \vec{y},\vec{z}} (\vec{w}(P_m); \vec{w}_{m}; m)),  \mbox{ and }\\
&\eta_m = F_{k,T-2}^{\vec{x}, \vec{y},\vec{z}} ([F_{k,T-2}^{\vec{x}, \vec{y},\vec{z}} (\cdot ; \vec{v}_{m}; m )]^{-1}(u_m) ; \vec{v}_{m}; m ) ,
\end{split}
\end{equation}
where $\vec{w}_m  \in Y(A_{P_m})$ is such that $\vec{w}_m(a) = \vec{w}(a)$ for $a \in A_{P_m}$ and $\vec{\ell}_m \in Y(A_{P_m})$ is such that $\vec{\ell}_m(a) = \ell^{k,T,\vec{x}, \vec{y}, \vec{z}}(a)$ for $\ell^{k,T,\vec{x}, \vec{y}, \vec{z}}$ as in Step 1 and the vector $ \vec{u}  \in (0,1)^{k(T-2)}$ as in (\ref{InvFun}). In addition, $\vec{v}_{m} \in Y(A_{P_m})$ is such that $\vec{v}_m(a) = \ell^{k,T,\vec{x}, \vec{y}, \vec{z}}(a)$ for $\ell^{k,T,\vec{x}, \vec{y}, \vec{z}}$ as in Step 1 and the vector $ \vec{u}  \in (0,1)^{k(T-2)}$ as in the second equality in (\ref{ASDF1}). The second equality in (\ref{ASDF3}) implies the second equality in (\ref{ASDF2}), and to conclude that the first equality in (\ref{ASDF2}) also holds it suffices to show that for $m = k(T-2), \dots, 1$
\begin{equation}\label{ASDF4}
\begin{split}
\vec{\ell}_{m} = \vec{w}_{m} .
\end{split}
\end{equation}

When $m = k(T-2)$ we have $\vec{\ell}_{m} = \vec{w}_{m}  \in Y(\emptyset)$ and (\ref{ASDF4}) holds. Assuming that (\ref{ASDF4}) holds for $m = k(T-2), \dots, n+1 \geq 2$ we have 
$$\vec{\ell}_n(a) = \vec{\ell}_{n+1}(a) = \vec{w}_{n+1}(a) = \vec{w}_n(a) \mbox{ for $a \in A_{P_{n+1}}$},$$
and so to prove (\ref{ASDF4}) for $m = n$ it suffices to show that $\vec{\ell}_{n}(P_{n+1}) = \vec{w}_n(P_{n+1}).$ Utilizing the definition of $\vec{\ell}_n$ we have that 
$$\vec{\ell}_n(P_{n+1}) = [F_{k,T-2}^{\vec{x}, \vec{y},\vec{z}} (\cdot; \vec{\ell}_{n+1}; n )]^{-1}(F_{k,T-2}^{\vec{x}, \vec{y},\vec{z}} (\vec{w}(P_{n+1}); \vec{w}_{n+1}; n+1 )) =\vec{w}(P_{n+1}) =  \vec{w}_n(P_{n+1}),  $$ 
where in the second equality we used that $\vec{w}_{n+1} =\vec{\ell}_{n+1}$ by the assumption that (\ref{ASDF4}) holds for $m = n +1$. This proves (\ref{ASDF4}) for $m = n$ and so we conclude (\ref{ASDF4}) for $m = k(T-2), \dots, 1$. This concludes the proof of (\ref{S2Phi1}).\\

{\bf \raggedleft Step 4.} In this step we prove (\ref{S2Phi2}). In view of our work in Step 3 we know that $[\Phi^{k,T}]^{-1}$ is nothing but the function $\Psi^{k,T}$ we constructed in that step. For $n \in \mathbb{N} \cup \{\infty\}$ we set $(\vec{x}^n, \vec{u}^n, \vec{y}^n, \vec{z}^n) : = \Psi^{k,T}(\vec{w}^n, \vec{z}^n)$ and then (\ref{S2Phi2}) becomes equivalent to
\begin{equation}\label{S2Phi3}
\lim_{n \rightarrow \infty} (\vec{x}^n, \vec{u}^n, \vec{y}^n, \vec{z}^n) =(\vec{x}^\infty, \vec{u}^\infty, \vec{y}^\infty, \vec{z}^\infty) \mbox{ in } \mathbb{R}^k \times (0,1)^{k(T-2)} \times \mathbb{R}^k \times  Y^-(\llbracket 0, T-1 \rrbracket),
\end{equation}
provided that $(\vec{w}^n, \vec{z}^n) \rightarrow (\vec{w}^\infty, \vec{z}^\infty)$ in $ Y(\llbracket 1, k \rrbracket \times \llbracket 0, T-1 \rrbracket) \times   Y^-(\llbracket 0, T-1 \rrbracket)$.

From the definition of $\Psi^{k,T}$ in (\ref{InvFun}) we immediately get $\lim_{n \rightarrow \infty} (\vec{x}^n, \vec{y}^n, \vec{z}^n) = (\vec{x}^\infty, \vec{y}^\infty, \vec{z}^\infty)$ in $\mathbb{R}^k  \times \mathbb{R}^k \times Y^-(\llbracket 0, T-1 \rrbracket)$ and so to prove (\ref{S2Phi3}) it suffices to show that for $m = k(T-2), k(T-2) - 1, \dots, 1$ we have $\lim_{n \rightarrow \infty} u^n_m = u^{\infty}_m$, or equivalently,
\begin{equation}\label{S2Phi3.5}
\lim_{n \rightarrow \infty} F_{k,T-2}^{\vec{x}^n, \vec{y}^n,\vec{z}^n} (\vec{w}^n(P_m); \vec{w}^n_{m}; m ) = F_{k,T-2}^{\vec{x}^\infty, \vec{y}^\infty,\vec{z}^\infty} (\vec{w}^\infty(P_m); \vec{w}^\infty_{m}; m ).
\end{equation}
From the definition of $F_{k,T-2}^{\vec{x}, \vec{y},\vec{z}} (s; \vec{v}; m )$ in (\ref{CDFRec}) we see that the last convergence statement is equivalent to proving for $m = k(T-2), k(T-2) - 1, \dots, 1$ that
\begin{equation}\label{S2Phi4}
\lim_{n \rightarrow \infty}  \frac{\int_{-\infty}^{s^n_m} h_{k,T-2}^{\vec{x}^n, \vec{y}^n, \vec{z}^n}(w; \vec{w}^n_{m}; A_{P_m}, B_{P_m}, P_m) dw }{\int_{-\infty}^{\infty} h_{k,T-2}^{\vec{x}^n, \vec{y}^n, \vec{z}^n}(w; \vec{w}^n_{m}; A_{P_m}, B_{P_m}, P_m) dw } = \frac{\int_{-\infty}^{s^{\infty}_m} h_{k,T-2}^{\vec{x}^\infty, \vec{y}^\infty, \vec{z}^\infty}(w; \vec{w}^\infty_{m}; A_{P_m}, B_{P_m}, P_m) dw }{\int_{-\infty}^{\infty} h_{k,T-2}^{\vec{x}^\infty, \vec{y}^\infty, \vec{z}^\infty}(w; \vec{w}^\infty_{m}; A_{P_m}, B_{P_m}, P_m) dw },
\end{equation}
where we have written $s_m^n = \vec{w}^n(P_m)$ for $n \in \mathbb{N}\cup \{\infty \}$. We show that the numerators and denominators on the left side of (\ref{S2Phi4}) converge to the numerator and the denominator on the right side, respectively. As the proofs are very similar, we only prove this statement for the numerators.\\

From the definition of $h_{k,n}^{\vec{x}, \vec{y}, \vec{z}}(x; \vec{v}; A, B, P)$ in (\ref{S6DefProj}) we see that the convergence of the numerators in (\ref{S2Phi4}) is equivalent to
\begin{equation*}
\begin{split}
&\lim_{n \rightarrow \infty}  \int_{-\infty}^{s_m^n}  \int_{Y(B_{P_m})} \prod_{i = 1}^{k}\prod_{j = 1}^{T-1} \g(u^n_{i,j} - u^n_{i, j -1})   \prod_{i = 1}^{k}\prod_{j = 1}^{T-2} e^{-H_j(u^n_{i+1, j+1} - u^n_{i,j})} \prod_{(i,j) \in B_{P_m}} du^n_{i,j} dw = \\
&  \int_{-\infty}^{s_m^\infty}  \int_{Y(B_{P_m})} \prod_{i = 1}^{k}\prod_{j = 1}^{T-1} \g(u^\infty_{i,j} - u^\infty_{i, j -1})   \prod_{i = 1}^{k}\prod_{j = 1}^{T-2} e^{-H_j(u^\infty_{i+1, j+1} - u^\infty_{i,j})} \prod_{(i,j) \in B_{P_m}} du^\infty_{i,j} dw,
\end{split}
\end{equation*}
where $u^n_{k+1,i} = \vec{z}^n(i)$ for $i \in \llbracket 0, T-1 \rrbracket$, $u^n_{i,0}= x^n_i$ and $u^n_{i,T-1} = y^n_i$ for $i \in \llbracket 1, k \rrbracket$, $u^n_{i,j} = \vec{w}^n_m(i,j)$ for $(i,j) \in A_{P_m}$, and $u^n_{p_1, p_2} = w$ for $(p_1, p_2) = P_m$. Applying the change of variables $y^n_{i,j} = u_{i,j}^n + s_m^n$ and $v = w +s_m^n$  above we see that it suffices to prove that 
\begin{equation}\label{S6QWE1}
\begin{split}
&\lim_{n \rightarrow \infty}  \int_{-\infty}^{0}  \int_{Y(B_{P_m})} \prod_{i = 1}^{k}\prod_{j = 1}^{T-1} \g(y^n_{i,j} - y^n_{i, j -1})   \prod_{i = 1}^{k}\prod_{j = 1}^{T-2} e^{-H_j(y^n_{i+1, j+1} - y^n_{i,j})} \prod_{(i,j) \in B_{P_m}} dy^n_{i,j} dv = \\
&  \int_{-\infty}^{0}  \int_{Y(B_{P_m})} \prod_{i = 1}^{k}\prod_{j = 1}^{T-1} \g(y^\infty_{i,j} - y^\infty_{i, j -1})   \prod_{i = 1}^{k}\prod_{j = 1}^{T-2} e^{-H_j(y^\infty_{i+1, j+1} - y^\infty_{i,j})} \prod_{(i,j) \in B_{P_m}} dy^\infty_{i,j} dv,
\end{split}
\end{equation}
where $y^n_{k+1,i} = \vec{z}^n(i) + s_{m}^n$ for $i \in \llbracket 0, T-1 \rrbracket$, $y^n_{i,0}= x^n_i + s_{m}^n$ and $y^n_{i,T-1} = y^n_i + s_{m}^n$ for $i \in \llbracket 1, k \rrbracket$ and $n \in \mathbb{N} \cup \{\infty\}$, $y^n_{i,j} = \vec{w}^n_m(i,j) + s_m^n$ for $(i,j) \in A_{P_m}$, and $y^n_{p_1, p_2} = v$ for $(p_1, p_2) = P_m$.  

Notice that by the continuity of $\g$ and $H_j$, for $j \in \llbracket 0, T-2 \rrbracket$, we know that the integrands on the top line of (\ref{S6QWE1}) converge pointwise to the integrand on the bottom. The fact that the integrals also converge then follows from the Generalized dominated convergence theorem (see \cite[Theorem 4.17]{Royden}) with dominating functions
$$f_n (\vec{y}, v ) =  \prod_{i = 1}^{k}\prod_{j = 1}^{T-1} \g(y_{i,j} - y_{i, j -1}), \mbox{ for } (\vec{y},v) \in Y(B_{P_m}) \times \mathbb{R}, $$
where $y_{i,0}= x^n_i + s_{m}^n$ and $y_{i,T-1} = y^n_i + s_{m}^n$ for $i \in \llbracket 1, k \rrbracket$,  $y_{i,j} = \vec{w}^n_m(i,j) + s_m^n$ for $(i,j) \in A_{P_m}$, and $y_{p_1, p_2} = v$ for $(p_1, p_2) = P_m$.  

Let us elaborate on the last argument. Since $H_j \geq 0$ for $j \in \llbracket 0, T-2 \rrbracket$ by assumption, we know that $f_n$ dominate the integrands on the top line of (\ref{S6QWE1}). Furthermore, by the continuity of $\g$ we conclude that $f_n$ converge pointwise to $f_\infty$, which has the same form as $f_n$ except that $y_{i,0}= x^\infty_i + s_{m}^\infty$ and $y_{i,T-1} = y^\infty_i + s_{m}^\infty$ for $i \in \llbracket 1, k \rrbracket$,  $y_{i,j} = \vec{w}^\infty_m(i,j) + s_m^\infty$ for $(i,j) \in A_{P_m}$. To conclude the application of the Generalized dominated convergence theorem we need to show
\begin{equation}\label{S6QWE2}
\begin{split}
\lim_{n \rightarrow \infty} &  \int_{-\infty}^{0} \int_{Y(B_{P_m})} \prod_{i = 1}^{k}\prod_{j = 1}^{T-1} \g(y^n_{i,j} - y^n_{i, j -1})    \prod_{(i,j) \in B_{P_m}} dy^n_{i,j} dv = \\
&   \int_{-\infty}^{0}  \int_{Y(B_{P_m})} \prod_{i = 1}^{k}\prod_{j = 1}^{T-1} \g(y^\infty_{i,j} - y^\infty_{i, j -1})    \prod_{(i,j) \in B_{P_m}} dy^\infty_{i,j} dv.
\end{split}
\end{equation}
Splitting the integral in (\ref{S6QWE2}) over $(i,j)$ such that $i \leq p_1 - 1$ and $i = p_1$ we see that (\ref{S6QWE2}) is equivalent to 
\begin{equation}\label{S6QWE3}
\begin{split}
\lim_{n \rightarrow \infty} & \prod_{i = 1}^{p_1- 1} \left[ \int_{\mathbb{R}^{T-2}} G(u_1 - x^n_i - s_m^n) \cdot \prod_{j = 2}^{T-2} G(u_j - u_{j-1}) \cdot G(y_i^n + s_m^n  - u_{T-2} ) \prod_{j = 1}^{T-2} du_j \right] \times \\
& \int_{\mathbb{R}^{p_2}} G(u_1 - x^n_{p_1} - s_m^n) \cdot \prod_{j = 2}^{p_2} G(u_j - u_{j-1}) \cdot G(t_m^n + s_m^n  - u_{p_2} ) {\bf 1 }\{ u_{p_2} \leq 0\} \prod_{j = 1}^{p_2} d u_j \times \\
& \prod_{(i,j) \in A_{P_m}} G(y_{i,j+1}^n  -y_{i,j}^n ) = \prod_{(i,j) \in A_{P_m}} G(y_{i,j+1}^\infty  -y_{i,j}^\infty ) \times \\
&   \prod_{i = 1}^{p_1- 1} \left[ \int_{\mathbb{R}^{T-2}} G(u_1 - x^\infty_i - s_m^n) \cdot \prod_{j = 2}^{T-2} G(u_j - u_{j-1}) \cdot G(y_i^\infty + s_m^\infty  - u_{T-2} ) \prod_{j = 1}^{T-2} du_j \right] \times \\
& \int_{\mathbb{R}^{p_2}} G(u_1 - x^\infty_{p_1} - s_m^\infty) \cdot \prod_{j = 2}^{p_2} G(u_j - u_{j-1}) \cdot G(t_m^\infty + s_m^\infty  - u_{p_2} ) {\bf 1 }\{ u_{p_2} \leq 0\} \prod_{j = 1}^{p_2} d u_j,
\end{split}
\end{equation}
where $t_m^n = y_{p_1}^n$ if $p_2 = T-2$ and $t_m^n = \vec{w}^n_m(p_1, p_2 + 1)$ if $p_2 \leq T-3$ for $n \in \mathbb{N} \cup \{ \infty\}$. We are left with establishing (\ref{S6QWE3}).\\

 By the continuity of $G$ and the fact that $\lim_{n \rightarrow \infty} y_{i,j}^n = y_{i,j}^\infty$ and $\lim_{n \rightarrow \infty} y_{i,j + 1}^n = y_{i,j + 1}^\infty$ for $(i,j) \in A_{P_m}$ we conclude that 
\begin{equation}\label{S6QWE4}
\lim_{n \rightarrow \infty} \prod_{(i,j) \in A_{P_m}} G(y_{i,j+1}^n  -y_{i,j}^n ) = \prod_{(i,j) \in A_{P_m}} G(y_{i,j+1}^\infty  -y_{i,j}^\infty ) .
\end{equation}
In addition, by performing a change of variables $\tilde{u}_1 = u_{1} - x_i^n - s_m^n$ and $\tilde{u}_j = u_j - u_{j-1}$ for $j \in \llbracket 2, T - 2 \rrbracket$ we see that for any $i \in \llbracket 1, p_1 - 1\rrbracket$ we have
\begin{equation}\label{S6QWE5}
\begin{split}
\lim_{n \rightarrow \infty} & \int_{\mathbb{R}^{T-2}} G(u_1 - x^n_i - s_m^n) \cdot \prod_{j = 2}^{T-2} G(u_j - u_{j-1}) \cdot G(y_i^n + s_m^n  - u_{T-2} ) \prod_{j = 1}^{T-2} du_j = \\
\lim_{n \rightarrow \infty} & \int_{\mathbb{R}^{T-2}}  \prod_{j = 1}^{T-2} G(\tilde{u}_j) \cdot G\left(y_i^n - x_i^n  - \sum_{j =1}^{T-2} \tilde{u}_j \right) \prod_{j = 1}^{T-2} d\tilde{u}_j = \\
&   \int_{\mathbb{R}^{T-2}}  \prod_{j = 1}^{T-2} G(\tilde{u}_j) \cdot G\left(y_i^\infty - x_i^\infty  - \sum_{j =1}^{T-2} \tilde{u}_j \right) \prod_{j = 1}^{T-2} d\tilde{u}_j = \\
& \int_{\mathbb{R}^{T-2}} G(u_1 - x^\infty_i - s_m^\infty) \cdot \prod_{j = 2}^{T-2} G(u_j - u_{j-1}) \cdot G(y_i^\infty + s_m^\infty  - u_{T-2} ) \prod_{j = 1}^{T-2} du_j,
\end{split}
\end{equation}
where the middle equality follows from the dominated convergence theorem (see \cite[Theorem 4.16]{Royden}) with dominating function $\|\g \|_\infty \cdot \prod_{i = 1}^{T-2}\g(\tilde{u}_i) $ (the pointwise convergence of the integrand follows from the continuity of $G$ and the fact that $\lim_{n \rightarrow \infty} x_i^n = x_i^{\infty}$ and $\lim_{n \rightarrow \infty} y_i^n = y_i^{\infty}$ by assumption). Finally, by performing a change of variables $\tilde{u}_1 = u_{1} - x_{p_1}^n - s_m^n$ and $\tilde{u}_j = u_j - u_{j-1}$ for $j \in \llbracket 2, p_2 \rrbracket$ we see that 
\begin{equation}\label{S6QWE6}
\begin{split}
\lim_{n \rightarrow \infty} & \int_{\mathbb{R}^{p_2}} G(u_1 - x^n_{p_1}- s_m^n) \cdot \prod_{j = 2}^{p_2} G(u_j - u_{j-1}) \cdot G(t_m^n + s_m^n  - u_{p_2} ){\bf 1 }\{ u_{p_2} \leq 0\} \prod_{j = 1}^{p_2} du_j = \\
\lim_{n \rightarrow \infty} & \int_{\mathbb{R}^{p_2}}  \prod_{j = 1}^{T-2} G(\tilde{u}_j) \cdot G\left(t_m^n - x_{p_1}^n  - \sum_{j =1}^{p_2} \tilde{u}_j \right){\bf 1 }\left\{ \sum_{i = 1}^{p_2} \tilde{u}_{i} + x_{p_1}^n + s_m^n \leq 0\right\} \prod_{j = 1}^{p_2} d\tilde{u}_j = \\
&   \int_{\mathbb{R}^{p_2}}  \prod_{j = 1}^{p_2} G(\tilde{u}_j) \cdot G\left(t_m^{\infty} - x_{p_1}^\infty  - \sum_{j =1}^{p_2} \tilde{u}_j \right) {\bf 1}\left\{ \sum_{i = 1}^{p_2} \tilde{u}_{i} + x_{p_1}^\infty + s_m^\infty \leq 0 \right\}  \prod_{j = 1}^{p_2} d\tilde{u}_j = \\
& \int_{\mathbb{R}^{p_2}} G(u_1 - x^\infty_{p_1} - s_m^\infty) \cdot \prod_{j = 2}^{p_2} G(u_j - u_{j-1}) \cdot G(y_i^\infty + s_m^\infty  - u_{T-2} ) {\bf 1 }\{ u_{p_2} \leq 0\}\prod_{j = 1}^{p_2} du_j,
\end{split}
\end{equation}
where the middle equality follows from the dominated convergence theorem (see \cite[Theorem 4.16]{Royden}) with dominating function $\|\g \|_\infty \cdot \prod_{i = 1}^{p_2}\g(\tilde{u}_i) $ (the pointwise convergence of the integrand follows from the continuity of $G$ and the fact that $\lim_{n \rightarrow \infty} x_{p_1}^n = x_{p_1}^{\infty}$ and $\lim_{n \rightarrow \infty} t_m^n = t_m^{\infty}$ by assumption). Combining (\ref{S6QWE4}), (\ref{S6QWE5}) and (\ref{S6QWE6}) we conclude (\ref{S6QWE3}) which completes our work in this step.\\

{\bf \raggedleft Step 5.} In this step we prove part III of the lemma. We already observed in Step 2 that $( \Omega^{k,T}, \mathcal{F}^{k,T},\mathbb{P}^{k,T})$ satisfies the conditions in part III. Furthermore, we showed in Step 3 that $\Phi^{k,T}$ is a bijection and in Step 4 that its inverse $\Psi^{k,T}$ is continuous. Thus we only need to prove that $\Phi^{k,T}$ is itself continuous. 

When $T = 2$ the continuity of $\Phi^{k,T}$ is immediate from its definition and so we assume that $T \geq 3$. Let $(\vec{x}^n, \vec{u}^n, \vec{y}^n, \vec{z}^n) \in \mathbb{R}^k \times (0,1)^{k(T-2)} \times \mathbb{R}^k \times  Y^-(\llbracket 0, T-1 \rrbracket)$ for $n \in \mathbb{N} \cup \{ \infty \}$ be such that 
$$\lim_{n \rightarrow \infty} (\vec{x}^n, \vec{u}^n, \vec{y}^n, \vec{z}^n)  = (\vec{x}^\infty, \vec{u}^\infty, \vec{y}^\infty, \vec{z}^\infty) ,$$
and put $(\vec{w}^n, \vec{z}^n) = \Phi^{k,T}(\vec{x}^n, \vec{u}^n, \vec{y}^n, \vec{z}^n) .$ We wish to show that 
\begin{equation}\label{S6Step5E1}
\lim_{n \rightarrow \infty} (\vec{w}^n, \vec{z}^n) = (\vec{w}^\infty, \vec{z}^\infty) .
\end{equation}
By the definition of $\Phi^{k,T}$ we know that 
$$\lim_{n \rightarrow \infty} \vec{z}^n = \vec{z}^\infty, \lim_{n \rightarrow \infty} \vec{w}^n(i, 0) = \lim_{n \rightarrow \infty} x_i^n = x_i^{\infty} = \vec{w}^\infty(i, 0) \mbox{ and }$$
$$ \lim_{n \rightarrow \infty} \vec{w}^n(i, T-1) = \lim_{n \rightarrow \infty} y_i^n = y_i^{\infty} = \vec{w}^\infty(i,T-1) \mbox{ for $i \in \llbracket 1, k \rrbracket$}.$$
Consequently, it suffices to show for $m = k(T-2), \dots, 1$ and $P_m$ as in Step 1 that 
\begin{equation}\label{S6Step5E2}
\lim_{n \rightarrow \infty} \vec{w}^n(P_m) =\vec{w}^\infty(P_m) .
\end{equation}

We claim that if $m \in \llbracket 1, k(T-2) \rrbracket$ and $\vec{v}^n_m \in Y(A_{P_m})$ for $n \in \mathbb{N} \cup \{ \infty \}$ are such that $\lim_{n \rightarrow \infty} \vec{v}^n_m = \vec{v}^{\infty}_m$ then
\begin{equation}\label{S6Step5E3}
\lim_{n \rightarrow \infty} [F_{k,T-2}^{\vec{x}^n, \vec{y}^n,\vec{z}^n} (\cdot; \vec{v}^n_{m}; m )]^{-1}(u^n_{m}) = [F_{k,T-2}^{\vec{x}^\infty, \vec{y}^\infty,\vec{z}^\infty} (\cdot; \vec{v}^\infty_{m}; m )]^{-1}(u^\infty_{m}).
\end{equation}
We will prove (\ref{S6Step5E3}) below. Here we assume its validity and conclude the proof of (\ref{S6Step5E2}).\\

We proceed to recursively prove (\ref{S6Step5E2}) for $m = k(T-2), \dots, 1$. From the definition of $\Phi^{k,T}$ in Step 1 we have that (\ref{S6Step5E2}) is equivalent to
\begin{equation}\label{S6Step5E4}
\lim_{n \rightarrow \infty} [F_{k,T-2}^{\vec{x}^n, \vec{y}^n,\vec{z}^n} (\cdot; \vec{v}^n_{m}; m )]^{-1}(u^n_{m}) =[F_{k,T-2}^{\vec{x}^\infty, \vec{y}^\infty,\vec{z}^\infty} (\cdot; \vec{v}^\infty_{m}; m )]^{-1}(u^\infty_{m}),
\end{equation}
$\vec{v}^n_c \in Y(A_{P_c})$ are such that $\vec{v}^n_c(a) = \vec{w}^n(a)$ for $a \in Y(A_{P_c})$ and $n \in \mathbb{N} \cup \{\infty\}$. When $m = k(T-2)$ we have that $\vec{v}^n_m \in Y(\emptyset)$ and so (\ref{S6Step5E4}) follows from (\ref{S6Step5E3}). This proves (\ref{S6Step5E2}) for $m = k(T-2)$. Assuming that we have proved (\ref{S6Step5E2}) for $m = k(T-2), \dots, c+1 \geq 2$ we have that $\lim_{n \rightarrow \infty} \vec{v}^n_c = \vec{v}^\infty_c$ and so (\ref{S6Step5E4}) with $m = c$ follows from (\ref{S6Step5E3}). This proves  (\ref{S6Step5E2}) for $m = k(T-2), \dots, c$ and so we conclude (\ref{S6Step5E2}) for all $m \in \llbracket 1, k(T-2) \rrbracket$.\\

To conclude the proof of part III we are left with establishing (\ref{S6Step5E3}) and we fix $m \in \llbracket 1, k(T-2) \rrbracket$ in the sequel. Setting $A_n = [F_{k,T-2}^{\vec{x}^n, \vec{y}^n,\vec{z}^n} (\cdot; \vec{v}^n_{m}; m )]^{-1}(u^n_{m})$ for $n \in \mathbb{N} \cup \{\infty\}$ we see that to prove (\ref{S6Step5E3}) it suffices to show that $\{ A_n\}_{n = 1}^\infty$ is bounded and all its subsequential limits are equal to $A_{\infty}$. 

Suppose first that $A_{n_r}$ converges to $\infty$ along some subsequence $n_r$. By the monotonicity of the function $F_{k,T-2}^{\vec{x}^n, \vec{y}^n,\vec{z}^n}$ we know for any $a \in \mathbb{R}$ 
\begin{equation*}
\begin{split}
&u_{m}^\infty = \lim_{r \rightarrow \infty} u_m^{n_r} = \lim_{r \rightarrow \infty} F_{k,T-2}^{\vec{x}^{n_r}, \vec{y}^{n_r},\vec{z}^{n_r}} (A_{n_r}; \vec{v}^{n_r}_{m}; m ) \geq \\ &\limsup_{r \rightarrow \infty} F_{k,T-2}^{\vec{x}^{n_r}, \vec{y}^{n_r},\vec{z}^{n_r}} (A_{n_r}; \vec{v}^{n_r}_{m}; m )= F_{k,T-2}^{\vec{x}^{\infty}, \vec{y}^{\infty},\vec{z}^{\infty}} (a; \vec{v}^{\infty}_{m}; m ),
\end{split}
\end{equation*}
where in the last equality we used (\ref{S2Phi3.5}). Letting $a \rightarrow \infty$ above we see $u_{m}^\infty \geq 1$, which is a contradiction as $u_{m}^\infty  \in (0,1)$.

Analogously, if $A_{n_r}$ converges to $-\infty$ along some subsequence $n_r$ then we have for any $a \in \mathbb{R}$
\begin{equation*}
\begin{split}
&u_{m}^\infty = \lim_{r \rightarrow \infty} u_{m}^{n_r} =\lim_{r \rightarrow \infty}  F_{k,T-2}^{\vec{x}^{n_r}, \vec{y}^{n_r},\vec{z}^{n_r}} (A_{n_r}; \vec{v}^{n_r}_{m}; m )\leq \\
& \liminf_{ n \rightarrow \infty} F_{k,T-2}^{\vec{x}^{n_r}, \vec{y}^{n_r},\vec{z}^{n_r}} (a; \vec{v}^{n_r}_{m}; m )= F_{k,T-2}^{\vec{x}^{\infty}, \vec{y}^{\infty},\vec{z}^{\infty}} (a; \vec{v}^{\infty}_{m}; m ),
\end{split}
\end{equation*}
where in the last equality we used (\ref{S2Phi3.5}). Letting $a \rightarrow -\infty$ above we see $u_{m}^\infty \leq 0 $, which is a contradiction as $u_{m}^\infty  \in (0,1)$. 

Finally, suppose that $A_{n_r}$ converges to $B_\infty$ along some subsequence $n_r$. Then
$$u^\infty_{m} = \lim_{r \rightarrow \infty} u^{n_r}_{m} =\lim_{r \rightarrow \infty}  F_{k,T-2}^{\vec{x}^{n_r}, \vec{y}^{n_r},\vec{z}^{n_r}} (A_{n_r}; \vec{v}^{n_r}_{m}; m ) = F_{k,T-2}^{\vec{x}^{\infty}, \vec{y}^{\infty},\vec{z}^{\infty}} (B_{\infty}; \vec{v}^{\infty}_{m}; m ),$$ 
where in the last equality we used (\ref{S2Phi3.5}). Applying $[F_{k,T-2}^{\vec{x}^\infty, \vec{y}^\infty,\vec{z}^\infty} (\cdot; \vec{v}^\infty_{m}; m )]^{-1}$ to both sides we see that $A_\infty = B_\infty$ as desired. The last three paragraphs show that $\{A_n\}_{n = 1}^\infty$ converges to $A_\infty$, concluding the proof of (\ref{S6Step5E3}) and hence part III of the lemma.

%
\section{Estimates on $(H,H^{RW})$-Gibbsian line ensembles}\label{Section7}  In this section we present the proofs of various statements throughout the paper. The main tools of our analysis will be the strong coupling of $H^{RW}$ random walk bridges and Brownian bridges, afforded by Proposition \ref{KMT}, the monotone coupling lemma, Lemma \ref{MonCoup}, and the partial $(H,H^{RW})$-Gibbs property, which is enjoyed by the measures $\mathbb{P}_{H,H^{RW}}^{1, k, T_0 ,T_1, \vec{x}, \vec{y},\infty,g}$ in view of Lemma \ref{HHRWSG}. We continue with the same notation as in Section \ref{Section2} and Definition \ref{AssH} and mention that essentially no other parts of the paper will be used in this section.

%
\subsection{Proof of Lemmas \ref{LStayInBand}, \ref{LNoBigJump} and \ref{LNotClose}}\label{Section7.1} Recall from Definition \ref{Pfree} that $\mathbb{P}_{H^{RW}}^{1,k, T_0,T_1, \vec{x},\vec{y}}$ is the law of $k$ independent $H^{RW}$ random walk bridge between the points $(T_0, x_i)$ and $(T_1,y_i)$ for $i \in \llbracket 1, k \rrbracket$, see (\ref{RWB}). We mention that these measures are a special case of $\mathbb{P}_{H,H^{RW}}^{1, k, T_0 ,T_1, \vec{x}, \vec{y},\infty,g}$ when $H \equiv 0$, and thus also satisfy the $(H,H^{RW})$-Gibbs property and Lemma \ref{MonCoup}. 

The goal of this section is to prove Lemmas \ref{LStayInBand}, \ref{LNoBigJump} and \ref{LNotClose} from Section \ref{Section4}, which are recalled here as Lemmas \ref{S7LStayInBand}, \ref{S7LNoBigJump} and \ref{S7LNotClose}, respectively, for the reader's convenience. 

In many of the proofs below we will require the following statement. In the notation of Proposition \ref{KMT} we have for any $A > 0$
\begin{equation}\label{S7Chebyshev}
\mathbb{P}\left(  \Delta(T,z) \geq A \right)\leq C  e^{-aA} e^{ \alpha (\log T)}e^{|z- p T|^2/T}, 
\end{equation}
where we recall that $\Delta(T,z)=  \sup_{0 \leq t \leq T} \big| \sqrt{T} B^\sigma_{t/T} + \frac{t}{T}z - \ell^{(T,z)}(t) \big|$. Equation (\ref{S7Chebyshev}) is an immediate consequence of (\ref{KMTeq}) and Chebyshev's inequality.

We also require the following simple result, which is an immediate consequence of \cite[Proposition 12.3.3]{Dudley}, see also \cite[Lemma 3.6]{DREU}.
\begin{lemma}\label{BBmax}
Let $B^\sigma$ be a Brownian bridge with diffusion parameter $\sigma > 0$ as in Section \ref{Section2.4}. Then for any $C,T> 0$ we have
\begin{equation}\label{sepBd}
\mathbb{P}\left(\max_{s\in[0,T]} \big|B^\sigma_{s/T}\big| \geq C\right) \leq 2\exp\left( - \frac{2C^2}{\sigma^2}\right).
\end{equation}
\end{lemma}

We now proceed with the proof of Lemma \ref{LStayInBand}.
\begin{lemma}\label{S7LStayInBand}[Lemma \ref{LStayInBand}] Let $\ell$ have distribution $\mathbb{P}_{H^{RW}}^{T_0,T_1, x,y}$(recall this was defined in Section \ref{Section2.4}) with $H^{RW}$ satisfying the assumptions in Definition \ref{AssHR}. For any $\epsilon \in (0,1)$, $p \in \mathbb{R}$ and $M > 0$ there exist $A = A(\epsilon,p,H^{RW})> 0$ and $W_1 = W_1(M, p, \epsilon,H^{RW}) \in \mathbb{N}$  such that the following holds. For $T_1 - T_0 \geq W_1$, $x, y \in \mathbb{R}$ with $|x - p T_0| \leq M (T_1-T_0)^{1/2}$ and $|y - p T_1| \leq M(T_1-T_0)^{1/2}$ we have
\begin{equation}\label{S7EStayInBand}
\mathbb{P}^{T_0,T_1,x,y}_{H^{RW}} \left( \sup_{ s \in [T_0, T_1] }  \left| \ell(s) - \frac{T_1 - s }{T_1 - T_0} \cdot x - \frac{s - T_0 }{T_1 - T_0} \cdot y  \right| \leq A (T_1 - T_0)^{1/2} \right) \geq 1 - \epsilon.
\end{equation}
\end{lemma} 
\begin{proof} Let $A = A(\epsilon,p,H^{RW}) \geq 2$ be sufficiently large so that 
\begin{equation}\label{S7S1L1E1}
2\exp\left( - \frac{2(A-1)^2}{\sigma_p^2}\right) < \epsilon/2,
\end{equation}
where $\sigma_p$ is as in Definition \ref{AssHR}. This specifies our choice of $A$.

From Proposition \ref{KMT} we know that we can find constants $0 < C, a, \alpha < \infty$ (depending on $p$ and $H^{RW}$) and a probability space with measure $\mathbb{P}$ on which are defined a Brownian bridge $B^{\sigma_p}$ with diffusion parameter $\sigma_p$ and a family of random curves $\ell^{(T,z)}$ on $[0, T]$, which is parametrized by $z \in \mathbb{R}$ such that $\ell^{(T,z)}$  has law $\mathbb{P}^{0,T,0,z}_{H^{RW}}$ and also (\ref{S7Chebyshev}) holds. Let $W_1 = W_1(M, p, \epsilon,H^{RW}) \in \mathbb{N}$ be sufficiently large so that for $T \geq W_1$ we have
\begin{equation}\label{S7S1L1E2}
C  e^{-a \sqrt{T} } e^{ \alpha (\log T)}e^{4M^2} < \epsilon/2.
\end{equation}
This specifies our choice of $W_1 $. We now proceed to prove (\ref{S7EStayInBand}).\\

Observe that by the triangle inequality we have for any $z\in \mathbb{R}$ and $T \in \mathbb{N}$
\begin{equation*}
\mathbb{P}\left( \sup_{ s \in [0, T] }  \left| \ell^{(T,z)}(s)  - \frac{s}{T} \cdot z  \right| > A T^{1/2} \right) \leq \mathbb{P}\left( \Delta(T,z) >  T^{1/2} \right) +  \mathbb{P}\left( \big|B^\sigma_{s/T}\big|  > A - 1\right).
\end{equation*}
Setting $z = y - x$ (note that $|z - pT| \leq 2M \sqrt{T}$) and $T \geq W_1$ we conclude that 
\begin{equation}\label{S7S1L1E3}
\begin{split}
&\mathbb{P}\left( \sup_{ s \in [0, T] }  \left| \ell^{(T,y-x)}(s)  - \frac{s}{T} \cdot (y-x)  \right| > A T^{1/2} \right)\leq \\
& C  e^{-a \sqrt{T} } e^{ \alpha (\log T)}e^{4M^2} +  2\exp\left( - \frac{2(A-1)^2}{\sigma_p^2}\right) < \epsilon.
\end{split}
\end{equation}
In first inequality we used (\ref{S7Chebyshev}) and (\ref{sepBd}), while in the second we used (\ref{S7S1L1E1}) and (\ref{S7S1L1E2}). 

Equation (\ref{S7S1L1E3}) is equivalent to (\ref{S7EStayInBand}) once we set $T = T_1 - T_0$, since under $\mathbb{P}$ the curve $\ell^{(T,y-x)}$ has distribution $\mathbb{P}^{0,T,0,y - x}_{H^{RW}}$, which is the same as the law of $\ell$ in (\ref{S7EStayInBand}) under a horizontal shift by $T_0$ and a vertical one by $x$.
\end{proof}


The following lemma can be found in \cite[Section 4.2]{MZ}. 
\begin{lemma}\label{LemmaI1}
We denote by $\Phi(x)$ and $\phi(x)$ the cumulative distribution function and density of a standard normal random variable. There is a constant $c_0 > 1$ such that for all $x \geq 0$ we have
\begin{equation}\label{LI2}
 \frac{1}{c_0(1+x)} \leq \frac{1 - \Phi(x)}{\phi(x)} \leq \frac{c_0}{1 +x},
\end{equation}
\end{lemma}

We now proceed with the proof of Lemma \ref{LNoBigJump}.
\begin{lemma}\label{S7LNoBigJump}[Lemma \ref{LNoBigJump}] Let $\ell$ have distribution $\mathbb{P}^{T_0,T_1,x,y}_{H^{RW}}$ with $H^{RW}$ as in Definition \ref{AssHR}.  Fix $p \in \mathbb{R}$, $r, \epsilon \in (0,1)$, $M > 0$.  Then we can find $W_5 = W_5(p,r, \epsilon, M, H^{RW}) \in \mathbb{N}$ such that the following holds.
If $T_0, T_1 \in \mathbb{Z}$ are such that $T_1 - T_0 \geq  W_5$, $r (T_1 - T_0) \leq R \leq r^{-1}(T_1 - T_0) $ and $x,y \in \mathbb{R}$ satisfy
$$ \left|x  - pT_0\right|\leq M R^{1/2} \mbox{ and } \left|y  - pT_1 \right| \leq M R^{1/2} $$
then we have 
\begin{equation}\label{S7ENoBigJump}
\mathbb{P}^{T_0,T_1,x,y}_{H^{RW}} \left( \sup_{ s \in [T_0, T_1-1] }  \left| \ell(s) - \ell(s+1) \right| \leq R^{1/4} \right) \geq 1 - \epsilon.
\end{equation}
\end{lemma}
\begin{proof} Let $\sigma_p$ be as in Definition \ref{AssHR}. Let $N_1 = N_1(p,r, \epsilon, M, H^{RW})  \in \mathbb{N}$ be sufficiently large so that for $T \geq N_1$ and $R \geq r T$ we have
\begin{equation}\label{S7S1L2E1}
T \geq 2, \hspace{2mm} R^{1/4}/4 > \frac{2r^{-1/2} M \sqrt{T} + |p|T }{T}  \mbox{ and } \frac{2 c_0 T}{\sqrt{2 \pi}} \cdot \exp \left( -\frac{R^{1/2}}{32 \sigma_p^2(1-T^{-1})} \right) < \epsilon/3,
\end{equation}
where $c_0$ is as in Lemma \ref{LemmaI1}.

From Proposition \ref{KMT} we know that we can find constants $0 < C, a, \alpha < \infty$ (depending on $p$ and $H^{RW}$) and a probability space with measure $\mathbb{P}$ on which are defined a Brownian bridge $B^{\sigma_p}$ with diffusion parameter $\sigma_p$ and a family of random curves $\ell^{(T,z)}$ on $[0, T]$, which is parametrized by $z \in \mathbb{R}$ such that $\ell^{(T,z)}$  has law $\mathbb{P}^{0,T,0,z}_{H^{RW}}$ and also (\ref{S7Chebyshev}) holds. Let $W_5 = W_5(p,r, \epsilon, M, H^{RW})  \in \mathbb{N}$ be sufficiently large so that $W_5 \geq N_1$ and for $T \geq W_5$ and $R \geq rT$ we have
\begin{equation}\label{S7S1L2E2}
 T \cdot C e^{-a R^{1/4}/4  } e^{ \alpha (\log T)}e^{4M^2/r} < \epsilon/3.
\end{equation}
This specifies our choice of $W_5 $. We now proceed to prove (\ref{S7ENoBigJump}).\\

Using the shift-invariance of $\mathbb{P}^{T_0,T_1,x,y}_{H^{RW}}$ and the fact that $\ell$ is a linear interpolation of its values on we see that (\ref{S7ENoBigJump}) is equivalent to 
\begin{equation*}
 \mathbb{P}^{0, T,0,y-x}_{H^{RW}} \left( \max_{ s \in \llbracket 0, T-1 \rrbracket }  \left| \ell(s) - \ell(s+1) \right| > R^{1/4} \right) < \epsilon,
\end{equation*}
where $T = T_1 - T_0$. Since under $\mathbb{P}$ the curve $\ell^{(T,z)}$  has law $\mathbb{P}^{0,T,0,z}_{H^{RW}}$ we see that it suffices to show that for $T \geq W_5$, $z \in \mathbb{R}$ with $|z - pT| \leq 2M r^{-1/2} \sqrt{T}$, $r T \leq R \leq r^{-1}T$ and $s \in \llbracket 0, T-1 \rrbracket $ we have
\begin{equation}\label{S7S1L2E3}
 \mathbb{P} \left(   \left| \ell^{(T,z)}(s) - \ell^{(T,z)}(s+1) \right| > R^{1/4} \right) < \epsilon/T.
\end{equation}
In the sequel we fix $z,T, s, R$ as above and prove (\ref{S7S1L2E3}).

By the triangle inequality and the second inequality in (\ref{S7S1L2E1}) we have 
\begin{equation}\label{S7S1L2E4}
\begin{split}
&\mathbb{P} \left(   \left| \ell^{(T,z)}(s) - \ell^{(T,z)}(s+1) \right| > R^{1/4} \right) \leq \mathbb{P} \left(   \left| \ell^{(T,z)}(s) - \frac{s}{T} \cdot z  - B^{\sigma_p}_{s/T} \right| > R^{1/4}/4 \right) + \\
& \mathbb{P} \left(   \left| \ell^{(T,z)}(s+1) - \frac{s + 1}{T} \cdot z  - B^{\sigma_p}_{(s+1)/T} \right| > R^{1/4}/4 \right) +  \mathbb{P} \left(   \left| B^{\sigma_p}_{s/T}- B^{\sigma_p}_{(s+1)/T} \right| > R^{1/4}/4 \right) \\
&  + \mathbb{P}( |z|/T > R^{1/4}/4 ) \leq 2  \mathbb{P} \left(  \Delta(T,z)  > R^{1/4}/4 \right) +  \mathbb{P} \left(   \left| B^{\sigma_p}_{s/T}- B^{\sigma_p}_{(s+1)/T} \right| > R^{1/4}/4 \right).
\end{split}
\end{equation}

We next have from (\ref{S7Chebyshev}) with $A = R^{1/4}/4$ and (\ref{S7S1L2E2}) that 
\begin{equation}\label{S7S1L2E5}
\begin{split}
& 2  \mathbb{P} \left(  \Delta(T,z)  > R^{1/4}/4 \right) < 2\epsilon/(3T).
\end{split}
\end{equation}

From \cite[Section 5.6.B]{KS} we have that $B^{\sigma_p}$ is a Gaussian process on $[0,1]$ with
\begin{equation}\label{BBCov}
\mathbb{E}[B^{\sigma_p}_t] = 0 \mbox{ and }\mathbb{E}[B^{\sigma_p}_r B^{\sigma_p}_s] = \sigma_p^2 ( \min(r,s) - rs).
\end{equation}  
This means that $X=  B^{\sigma_p}_{s/T}- B^{\sigma_p}_{(s+1)/T}$ is a normal random variable with mean $0$ and variance $\sigma_p^2 \cdot (1 - 1/T)$. Consequently,
\begin{equation}\label{S7S1L2E6}
\begin{split}
& \mathbb{P} \left(   \left| B^{\sigma_p}_{s/T}- B^{\sigma_p}_{(s+1)/T} \right| > R^{1/4}/4 \right) = \mathbb{P}(|X| >  R^{1/4}/4 ) = 2\left[1 - \Phi\left(\frac{ R^{1/4}}{4 \sigma_p \sqrt{1 - T^{-1}}} \right) \right] \leq \\
&2 c_0 \phi \left(\frac{ R^{1/4}}{4 \sigma_p \sqrt{1 - T^{-1}}} \right) < \epsilon/ (3T),
\end{split}
\end{equation}
where in the first inequality we used  Lemma \ref{LemmaI1} and in the second we used (\ref{S7S1L2E1}). 

Combining (\ref{S7S1L2E4}), (\ref{S7S1L2E5}) and (\ref{S7S1L2E6}) we get (\ref{S7S1L2E3}).
\end{proof}


\begin{lemma}\label{S7LNotClose}[Lemma \ref{LNotClose}] Fix $k \in \mathbb{N}$, $k \geq 2$ and let $\mathfrak{L} = (L_1, \dots, L_k)$ have law $\mathbb{P}_{H^{RW}}^{1, k, T_0 ,T_1, \vec{x}, \vec{y}}$ as in Definition \ref{Pfree}, where $H^{RW}$ is as in Definition \ref{AssHR}. Fix $p \in \mathbb{R}$, $r \in (0,1)$, $t \in (0,1/3)$, $M, \epsilon > 0$. Then we can find $W_6 = W_6(k,p,r,t,M,\epsilon, H^{RW}) \in \mathbb{N}$ and $\delta = \delta(k, p, r, t,  \epsilon, H^{RW})> 0$, such that the following holds. If $T_0, T_1 \in \mathbb{Z}$ are such that $T_1 - T_0 \geq  W_6$, $r (T_1 - T_0)\leq R \leq r^{-1} (T_1 - T_0)$, $t_0 \in \llbracket T_0, T_1\rrbracket $ is such that $\min( T_1 - t_0, t_0 -T_0 )\geq t (T_1 - T_0)$ and $\vec{x}, \vec{y} \in \mathbb{R}^k$ satisfy
$$ \left|x_i  - pT_0\right|\leq M R^{1/2} \mbox{ and } \left|y_i  - pT_1 \right| \leq M R^{1/2} \mbox{ for $i = 1, \dots, k$}$$
then we have 
\begin{equation}\label{S7NEotClose}
\mathbb{P}_{H^{RW}}^{1, k, T_0 ,T_1, \vec{x}, \vec{y}} \left( \min_{1 \leq i < j \leq k} \left|L_i(t_0) - L_j(t_0) \right| \leq \delta R^{1/2}  \right) \leq \epsilon.
\end{equation}
\end{lemma}
\begin{proof} Let $\sigma_p$ be as in Definition \ref{AssHR} and put $\lambda = \binom{k}{2}$. Let $\delta = \delta(k, p, r, t,  \epsilon, H^{RW})> 0$ be sufficiently small so for $s \in [t, 1- t]$ 
\begin{equation}\label{S7S1L3E1}
\frac{\lambda}{\sqrt{4\pi \sigma_p^2 \cdot s(1 - s) r }} \cdot 6 \delta < \epsilon/3.
\end{equation}
This specifies our choice of $\delta$.

Let $0 < C, a, \alpha < \infty$ be as in Proposition \ref{KMT} for $p$ and $H^{RW}$ as in the present lemma. Let $W_6 = W_6(k,p,r,t,M,\epsilon, H^{RW}) \in \mathbb{N}$ be sufficiently large so that if $T \geq W_6$ and $R \geq r T$ we have 
\begin{equation}\label{S7S1L3E2}
 \lambda \cdot  C e^{-a \delta R^{1/2}  } e^{ \alpha (\log T)}e^{4M^2/r} < \epsilon/3.
\end{equation}
This specifies our choice of $W_6$.  We now proceed to prove (\ref{S7NEotClose}).\\

Let us fix $1 \leq i < j \leq k$. We will prove that if $T = T_1 - T_0 \geq W_6$ we have 
\begin{equation}\label{S7S1L3E3}
\mathbb{P}_{H^{RW}}^{1, k, T_0 ,T_1, \vec{x}, \vec{y}} \left( \left|L_i(t_0) - L_j(t_0) \right| \leq \delta R^{1/2}  \right) \leq \epsilon \cdot \lambda^{-1},
\end{equation}
which upon taking a union bound would imply (\ref{S7NEotClose}). 

By taking the product of two copies of the probability space in Proposition \ref{KMT} we obtain a probability space with measure $\mathbb{P}$ on which are defined two independent Brownian bridges $B^{\sigma_p}$, $\tilde{B}^{\sigma_p}$ with diffusion parameter $\sigma_p$ and two independent families of random curves $\ell^{(T,z)}, \tilde{\ell}^{(T,z)}$ on $[0, T]$ whose laws are $\mathbb{P}^{0,T,0,z}_{H^{RW}}$ and also (\ref{S7Chebyshev}) holds for the pairs $\ell^{(T,z)}, B^{\sigma_p}$ and $\tilde{\ell}^{(T,z)}, \tilde{B}^{\sigma_p}$. Using the shift-invariance of $\mathbb{P}^{T_0,T_1,x,y}_{H^{RW}}$ we have that $x_i + \ell^{(T,y_i - x_i)}$ has law $\mathbb{P}^{0,T,x_i,y_i}_{H^{RW}}$ and $x_j + \tilde{\ell}^{(T,y_j - x_j)}$ has law $\mathbb{P}^{0,T,x_j,y_j}_{H^{RW}}$. Setting $z_i = y_i - x_i$, $z_j = y_i - x_i$ and $S = t_0 - T_0$ we conclude that (\ref{S7S1L3E3}) is equivalent to 
\begin{equation}\label{S7S1L3E4}
\mathbb{P} \left( \left|x_i - x_j + \ell^{(T, z_i)} (S) - \tilde{\ell}^{(T, z_j)} (S) \right| \leq \delta R^{1/2}  \right) \leq \epsilon \cdot \lambda^{-1}.
\end{equation}

Setting $\Delta(T,z)=  \sup_{0 \leq t \leq T} \big| \sqrt{T} B^{\sigma_p}_{t/T} + \frac{t}{T}z - \ell^{(T,z)}(t) \big|$ and $\tilde{\Delta}(T,z)=  \sup_{0 \leq t \leq T} \big| \sqrt{T} \tilde{B}^{\sigma_p}_{t/T} + \frac{t}{T}z - \tilde{\ell}^{(T,z)}(t) \big|$ we see by the triangle inequality that
\begin{equation}\label{S7S1L3E5}
\begin{split}
&\mathbb{P} \left( \left|x_i - x_j + \ell^{(T, z_i)} (S) - \tilde{\ell}^{(T, z_j)} (S) \right| \leq \delta R^{1/2}  \right) \leq  \mathbb{P} \left( \Delta(T,z_i) \geq \delta R^{1/2}\right)  + \\
&\mathbb{P} \left( \tilde{\Delta}(T,z_j) \geq \delta R^{1/2} \right) +  \mathbb{P} \left( \left| x_i - x_j + (S/T)(z_i - z_j) + \sqrt{T}\left(B^{\sigma_p}_{S/T} -     \tilde{B}^{\sigma_p}_{S/T} \right) \right| \leq 3\delta R^{1/2} \right).
\end{split}
\end{equation}

Using that $\max(|z_i|, |z_j|) \leq 2M r^{-1/2} T^{1/2}$ by assumption, (\ref{S7Chebyshev}) and (\ref{S7S1L3E2}) we get 
\begin{equation}\label{S7S1L3E6}
\begin{split}
\mathbb{P} \left( \Delta(T,z_i) \geq \delta R^{1/2}\right)  + \mathbb{P} \left( \tilde{\Delta}(T,z_j) \geq \delta R^{1/2} \right) \leq \lambda^{-1} \cdot (2\epsilon/3).
\end{split}
\end{equation}

We finally observe that $x_i - x_j + (S/T)(z_i - z_j) + \sqrt{T}\left(B^{\sigma_p}_{S/T} -     \tilde{B}^{\sigma_p}_{S/T} \right) $ is a normal random variable with mean
$x_i - x_j + (S/T)(z_i - z_j) $ and variance $2T \sigma_p^2 (S/T) (1 - S/T)$. In making the last observation we used that $ B^{\sigma_p}$ and $\tilde{B}^{\sigma_p}$ are independent Gaussian processes satisfying (\ref{BBCov}). In particular, if $X$ is a normal with mean $\mu$ and variance $\sigma^2$ such that 
$$\mu  =  R^{-1/2} \left( x_i - x_j + (S/T)(z_i - z_j) \right) \mbox{ and } \sigma^2  = 2(T/R) \sigma_p^2 s(1 - s), \mbox{ with $s = S/T$}, $$
we have that 
\begin{equation}\label{S7S1L3E7}
\begin{split}
&  \mathbb{P} \left( \left| x_i - x_j + (S/T)(z_i - z_j) + \sqrt{T}\left(B^{\sigma_p}_{S/T} -     \tilde{B}^{\sigma_p}_{S/T} \right) \right| \leq 3\delta R^{1/2} \right) =   \mathbb{P} \left( \left| X \right| \leq 3\delta \right) = \\
&\int_{-3\delta}^{3\delta} \frac{\exp\left( - (x- \mu)^2/(2 \sigma^2) \right)}{\sqrt{2 \pi \sigma^2} } dx \leq \frac{6 \delta}{\sqrt{2 \pi \sigma^2} }  \leq \frac{6 \delta}{\sqrt{4 \pi \sigma_p^2 s (1 -s)r } } < \frac{\epsilon}{3 \lambda}.
\end{split}
\end{equation}
In the above inequalities we used that $rT \leq R \leq r^{-1}T$, $s \in [t, 1- t]$ and (\ref{S7S1L3E1}). 

Combining (\ref{S7S1L3E5}), (\ref{S7S1L3E6}) and (\ref{S7S1L3E7}) we get (\ref{S7S1L3E4}), which completes the proof of the lemma.
\end{proof}

%
\subsection{Proof of Lemmas \ref{LNoDip}, \ref{LHighBottom} and \ref{LNoParDip}}\label{Section7.2} The goal of this section is to prove Lemmas \ref{LNoDip}, \ref{LHighBottom} and \ref{LNoParDip} from Section \ref{Section4}, which are recalled here as Lemmas  \ref{S7LNoDip}, \ref{S7LHighBottom} and \ref{S7LNoParDip}, respectively, for the reader's convenience.

We now proceed with the proof of Lemma \ref{LNoDip}.
\begin{lemma}\label{S7LNoDip}[Lemma \ref{LNoDip}] Fix $k \in \mathbb{N}$ and let $\mathfrak{L} = (L_1, \dots, L_k)$ have law $\mathbb{P}_{H,H^{RW}}^{1, k, T_0 ,T_1, \vec{x}, \vec{y},\infty,g}$ as in Definition \ref{Pfree}, where we assume that $H^{RW}$ is as in Definition \ref{AssHR} and $H$ is as in Definition \ref{AssH}. Let $ p\in\mathbb{R}$, $r, \epsilon \in (0,1)$ and $M^{\mathsf{side}} > 0$ be given. Then we can find constants $W_2 = W_2(k,p,r,\epsilon, M^{\mathsf{side}}, H, H^{RW}) \in \mathbb{N}$ and $M^{\mathsf{dip}} =  M^{\mathsf{dip}}(k,p,r,\epsilon, M^{\mathsf{side}}, H, H^{RW})> 0$ so that the following holds. 

 For any $T_0, T_1 \in \mathbb{Z}$ with $ T_1 - T_0 \geq W_2$, $r (T_1 - T_0) \leq R \leq  r^{-1} (T_1 - T_0)$, $\vec{x}, \vec{y} \in \mathbb{R}^k$  that satisfy
$$ \left| x_i - pT_0 \right| \leq M^{\mathsf{side}} R^{1/2}, \left| y_i - pT_1 \right| \leq M^{\mathsf{side}} R^{1/2} \mbox{ for $i \in \llbracket 1, k \rrbracket$, }$$ 
and $g \in Y^{-}(\llbracket T_0, T_1 \rrbracket)$ we have that 
\begin{equation}\label{S7ENoDip}
\mathbb{P}_{H,H^{RW}}^{1, k, T_0 ,T_1, \vec{x}, \vec{y},\infty,g} \left( L_k(x) - p x  \geq - M^{\mathsf{dip}} R^{1/2}  \mbox{ for all $x \in [T_0, T_1]$}\right) \geq 1 - \epsilon.
\end{equation}
\end{lemma}
\begin{proof} For clarity we split the proof into three steps. \\

{\bf \raggedleft Step 1.} In this step we specify the constants $W_2$ and $M^{\mathsf{dip}}$ as in the statement of the lemma.

Let $\delta = \delta(k, \epsilon) \in (0,1) $ be sufficiently small so that
\begin{equation}\label{S7S2L1E1}
 (1- \delta)^{-k - 1} \delta < \epsilon.
\end{equation}
Next, let $A = A(\delta, p, H^{RW}) > 0$ be as in Lemma \ref{S7LStayInBand} for $p, H^{RW}$ as in the present lemma and $\delta$ as above. In addition, let $B = B(k, \delta, H) > 0$ be sufficiently large so that for $T \in \mathbb{N}$  
\begin{equation}\label{S7S2L1E2}
\exp \left( - (k-1) T H \left( - B \sqrt{T}  \right) \right) \geq 1- \delta.
\end{equation}
Note that such a choice of $B$ is possible since by Definition \ref{AssH} we have that $\lim_{x \rightarrow \infty} x^2 H(-x) = 0$. 

With the above data we let $W_2 = W_1( A , p, \delta ,H^{RW}) $ as in Lemma \ref{S7LStayInBand} and 
$$M^{\mathsf{dip}} = r^{-1/2} ( r^{-1/2} M^{\mathsf{side}}  + (2k-1) \cdot A + (k-1) \cdot (B + |p|) ).$$
 This specifies our choice of $W_2$ and $ M^{\mathsf{dip}}$.\\

{\bf \raggedleft Step 2.} We put $T = T_1 - T_0$ and assume throughout that $T \geq W_2$ as in Step 1. Let $\vec{x}', \vec{y}' \in \mathbb{R}^k$ be
\begin{equation}\label{S7S2L1E3}
\begin{split}
&x_i' = pT_0 - z_i \mbox{ and } y_i' = pT_1 - z_i \mbox{ for $i \in \llbracket 1, k \rrbracket$, where }\\
&z_i = \sqrt{T} \left( r^{-1/2} M^{\mathsf{side}}+ 2(i-1) \cdot A + (i-1) \cdot (B + |p|) \right).
\end{split}
\end{equation}
and observe that $x_i' \leq x_i$, $y_i' \leq y_i$ for $i \in \llbracket 1, k \rrbracket$. In view of the latter inequalities and Lemma \ref{MonCoup} we conclude that 
\begin{equation}\label{S7S2L1E4}
\begin{split}
&\mathbb{P}_{H,H^{RW}}^{1, k, T_0 ,T_1, \vec{x}, \vec{y},\infty,g} \left( L_k(x) - p x  < - M^{\mathsf{dip}} R^{1/2}  \mbox{ for some $x \in [T_0, T_1]$}\right) \leq \\
&\mathbb{P}_{H,H^{RW}}^{1, k, T_0 ,T_1, \vec{x}', \vec{y}',\infty,-\infty} \left( L_k(x) - p x  < - M^{\mathsf{dip}} R^{1/2}  \mbox{ for some $x \in [T_0, T_1]$}\right).
\end{split}
\end{equation}

In addition, from (\ref{RND}) we have
\begin{equation}\label{S7S2L1E5}
\begin{split}
&\mathbb{P}_{H,H^{RW}}^{1, k, T_0 ,T_1, \vec{x}', \vec{y}',\infty,-\infty} \left( L_k(x) - p x  < - M^{\mathsf{dip}} R^{1/2}  \mbox{ for some $x \in [T_0, T_1]$}\right) = \\
& \frac{\mathbb{E}_{H^{RW}} \left[ W_H \cdot {\bf 1} \{ \ell_k(x) - p x  < - M^{\mathsf{dip}} R^{1/2}  \mbox{ for some $x \in [T_0, T_1]$} \} \right] }{\mathbb{E}_{H^{RW}} \left[ W_H \right]},
\end{split}
\end{equation}
where we write $\mathbb{P}_{H^{RW}}$ in place of $\mathbb{P}_{H^{RW}}^{1, k, T_0 ,T_1, \vec{x}', \vec{y}'}$, $\mathbb{E}_{H^{RW}}$ in place of $\mathbb{E}_{H^{RW}}^{1, k, T_0 ,T_1, \vec{x}', \vec{y}'}$ and $W_H$ in place of $W_{H}^{1, k, T_0 ,T_1,\infty, -\infty} ({\ell}_{1}, \dots, {\ell}_{k})$ to simplify the notation. In addition, $(\ell_1, \dots, \ell_k)$ has distribution $\mathbb{P}_{H^{RW}}$. 

We claim that 
\begin{equation}\label{S7S2L1E6}
\begin{split}
&\mathbb{E}_{H^{RW}} \left[ W_H \right] \geq (1 - \delta)^{k+1} \mbox{ and } \\
&\mathbb{E}_{H^{RW}} \left[ W_H \cdot {\bf 1} \{ \ell_k(x) - p x  < - M^{\mathsf{dip}} R^{1/2}  \mbox{ for some $x \in [T_0, T_1]$} \} \right]  \leq \delta.
\end{split}
\end{equation}
Observe that (\ref{S7S2L1E4}), (\ref{S7S2L1E5}) and (\ref{S7S2L1E6}) together imply (\ref{S7ENoDip}) in view of (\ref{S7S2L1E1}). Thus we have reduced the proof of the lemma to establishing the two inequalities in (\ref{S7S2L1E6}). We establish the second in this step, and postpone the first to the next step.\\

Using that $W_H \in [0,1]$, and that $\ell_1, \dots, \ell_k$ are independent under $\mathbb{P}_{H^{RW}}$ and have laws $\mathbb{P}_{H^{RW}}^{T_0 ,T_1, x_i', y_i'}$ for $i \in \llbracket 1, k \rrbracket$ we conclude that 
\begin{equation}\label{S7S2L1E7}
\begin{split}
&\mathbb{E}_{H^{RW}} \left[ W_H \cdot {\bf 1} \{ \ell_k(x) - p x  < - M^{\mathsf{dip}} R^{1/2}  \mbox{ for some $x \in [T_0, T_1]$} \} \right]  \leq \\
&\mathbb{P}_{H^{RW}}^{T_0, T_1, x_k', y_k'} (  \ell(x) - p x  < - M^{\mathsf{dip}} R^{1/2}   \mbox{ for some $x \in [T_0, T_1]$}  ) .
\end{split}
\end{equation}
Using the shift-invariance of $\mathbb{P}_{H^{RW}}^{T_0, T_1,x,y}$ and the definition of $x_k', y_k'$ from (\ref{S7S2L1E3}) we get 
\begin{equation}\label{S7S2L1E8}
\begin{split}
&\mathbb{P}_{H^{RW}}^{T_0, T_1, x_k', y_k'} (  \ell(x) - p x  < - M^{\mathsf{dip}} R^{1/2}   \mbox{ for some $x \in [T_0, T_1]$}  )  = \\
& \mathbb{P}_{H^{RW}}^{T_0, T_1, pT_0 , pT_1} \Big{(}  \ell(x) - p x  <  z_k  - M^{\mathsf{dip}} R^{1/2}  \mbox{ for some $x \in [T_0, T_1]$}  \Big{)} \leq \\
&  \mathbb{P}_{H^{RW}}^{T_0, T_1, pT_0 , pT_1}\left(  \ell(x) - p x < - A \sqrt{T}  \right) < \delta.
\end{split}
\end{equation}
We mention that in the first inequality we used the definition of $M^{\mathsf{dip}}$ from Step 1, equation (\ref{S7S2L1E3}) and that $R \geq r T$, while in the last inequality we used the definition of $A$ from Step 1 and the fact that $T_1 - T_0 = T \geq W_2$. Equations (\ref{S7S2L1E7}) and (\ref{S7S2L1E8}) imply the second line in (\ref{S7S2L1E6}).\\

{\bf \raggedleft Step 3.} In this step we prove the first line in (\ref{S7S2L1E6}). Let us define the events 
$$E_i = \left\{ \sup_{x \in [T_0, T_1] } \left| \ell_i(x) - p x +z_i  \right|  \leq A \sqrt{T} \right\} \mbox{ for $i \in \llbracket 1, k \rrbracket$}.$$
Using that $\ell_1, \dots, \ell_k$ are independent under $\mathbb{P}_{H^{RW}}$ and have laws $\mathbb{P}_{H^{RW}}^{T_0 ,T_1, x_i', y_i'}$ for $i \in \llbracket 1, k \rrbracket$ we conclude
\begin{equation}\label{S7S2L1E9}
\begin{split}
&\mathbb{P}_{H^{RW}} \left( \cap_{i = 1}^k E_i  \right) = \prod_{i = 1}^k\mathbb{P}_{H^{RW}} \left(  E_i  \right) =  \prod_{i = 1}^k\mathbb{P}_{H^{RW}}^{T_0 ,T_1, x_i', y_i'} \left(  \sup_{x \in [T_0, T_1] } \left| \ell(x) - p x + z_i \right|  \leq A \sqrt{T} \right) = \\
&\mathbb{P}_{H^{RW}}^{T_0 ,T_1, pT_0 , pT_1} \left(  \sup_{x \in [T_0, T_1] } \left| \ell(x) - p x \right|  \leq A \sqrt{T} \right)^k \geq (1- \delta)^{k},
\end{split}
\end{equation}
where in going from the first to the second line we used the shift invariance of $\mathbb{P}_{H^{RW}}^{T_0 ,T_1, x,y}$ and (\ref{S7S2L1E3}), while in the last inequality we used the definition of $A$ from Step 1 and the fact that $T \geq W_2$.

On the other hand, we have using (\ref{S7S2L1E3}) that on $\cap_{i = 1}^k E_i $ for $i \in \llbracket 1, k-1\rrbracket$
\begin{equation*}
\begin{split}
&\ell_{i}(x) \geq p x - z_i -A \sqrt{T} \geq p x - z_{i+1} + A \sqrt{T} + (B + |p|) \sqrt{T} \geq \\
& p (x+1) - z_{i+1} + A \sqrt{T} + B \sqrt{T} \geq \ell_{i+1}(x+1) + B \sqrt{T}.
\end{split}
\end{equation*}
The latter inequality, the monotonicity of $H$ (see Definition \ref{AssH}) and the definition of $W_H$ in (\ref{WH}) gives on $\cap_{i = 1}^k E_i $ 
\begin{equation}\label{S7S2L1E10}
W_H =  \exp \left( - \sum_{i = 1}^{k-1}  \sum_{ m = T_0}^{T_1-1} \hspace{-1mm}H (\ell_{i + 1}(m + 1) - \ell_{i}(m)) \right) \geq \exp \left( - (k-1) T H(-B \sqrt{T}) \right) \geq 1-\delta,
\end{equation}
where the last inequality used (\ref{S7S2L1E2}). Combining (\ref{S7S2L1E9}) and (\ref{S7S2L1E10}) proves the first line in (\ref{S7S2L1E6}).
\end{proof}

The following lemma can be found in \cite{CorHamA}. 
\begin{lemma}\label{Spread} \cite[Corollary 2.10]{CorHamA}. Let $U$ be an open subset of $C([0,1])$, which contains a function $f$ such that $f(0) = f(1) = 0$. If $B:[0,1] \rightarrow \mathbb{R}$ is a standard Brownian bridge then $\mathbb{P}(B[0,1] \subset U) > 0$.
\end{lemma}

We now proceed with the proof of Lemma \ref{LHighBottom}.
\begin{lemma}\label{S7LHighBottom}[Lemma \ref{LHighBottom}] Fix $k \in \mathbb{N}$ and let $\mathfrak{L} = (L_1, \dots, L_k)$ have law $\mathbb{P}_{H,H^{RW}}^{1, k, T_0 ,T_1, \vec{x}, \vec{y},\infty,g}$ as in Definition \ref{Pfree}, where we assume that $H$ is as in Definition \ref{AssH}, while $H^{RW}$ is as in Definition \ref{AssHR}. Let $ p\in \mathbb{R}$, $r,\epsilon \in (0,1)$, $M^{\mathsf{bot}}, M^{\mathsf{side}} > 0$ and $t \in (0,1/3)$ be given. Then we can find a constant $W_3 = W_3(k,p,r,M^{\mathsf{side}}, M^{\mathsf{bot}},t, \epsilon, H, H^{RW}) \in \mathbb{N}$ so that the following holds. 

For any $T_0, T_1 \in \mathbb{Z}$ with $  T_1 - T_0 \geq W_3$ and $t_0, t_1 \in \llbracket T_0, T_1 \rrbracket$ with $\min(t_0 - T_0,T_1 - t_1, t_1 - t_0)  \geq t (T_1 - T_0)$, $r (T_1 - T_0) \leq R \leq  r^{-1} (T_1 - T_0)$, $\vec{x}, \vec{y} \in \mathbb{R}^k, g \in Y^{-}(\llbracket T_0, T_1 \rrbracket)$  that satisfy
$$ \left| x_i - pT_0 \right| \leq M^{\mathsf{side}} R^{1/2}, \left| y_i - pT_1 \right| \leq M^{\mathsf{side}} R^{1/2} \mbox{ for $i \in \llbracket 1, k \rrbracket$, and }$$ 
$$g(j) - pj \geq M^{\mathsf{bot}} R^{1/2} \mbox{ for some $j \in \llbracket t_0, t_1 - 1 \rrbracket$,}$$
we have that 
\begin{equation}\label{S7EHighBottom}
\mathbb{P}_{H,H^{RW}}^{1, k, T_0 ,T_1, \vec{x}, \vec{y},\infty,g} \left(L_k(x) - p x  \geq M^{\mathsf{bot}} R^{1/2} - R^{1/4} \mbox{ for some $x \in [t_0, t_1]$}\right) \geq 1 - \epsilon.
\end{equation}
\end{lemma}
\begin{proof} For clarity we split the proof into three steps.\\

{\bf \raggedleft Step 1.} In this step we prove the lemma modulo a certain estimate, see (\ref{S7S2L2E5}), which will be established in the steps below.

Let $B = B(k, H) > 0$ be sufficiently large so that for $T \in \mathbb{N}$  
\begin{equation}\label{S7S2L2E1}
\exp \left( - k T H \left( - B \sqrt{T}  \right) \right) \geq 1/2.
\end{equation}
Note that such a choice of $B$ is possible since by Definition \ref{AssH} we have that $\lim_{x \rightarrow \infty} x^2 H(-x) = 0$.  We put $T = T_1 - T_0$ let $\vec{x}', \vec{y}' \in \mathbb{R}^k$ be defined via
\begin{equation}\label{S7S2L2E2}
\begin{split}
&x_i' = pT_0 - z_i \mbox{ and } y_i' = pT_1 - z_i \mbox{ for $i \in \llbracket 1, k \rrbracket$, where }\\
&z_i = \sqrt{T} \left( r^{-1/2}M^{\mathsf{side}}+ 6(i-1) + (i-1) \cdot (B + |p|) \right).
\end{split}
\end{equation}
By assumption we know there exists $j_0 \in \llbracket t_0, t_1 - 1\rrbracket$ such that $g(j) - pj \geq M^{\mathsf{bot}} R^{1/2}$ and we define $g' \in Y^-(\llbracket T_0, T_1 \rrbracket)$ so that 
\begin{equation}\label{S7S2L2E3}
\begin{split}
g'(j) = pj_0 + M^{\mathsf{bot}} R^{1/2} \mbox{ if } j = j_0 \mbox{ and } g'(j) = -\infty \mbox{ for $j \in \llbracket T_0 ,T_1 \rrbracket \setminus \{ j_0\}$.}
\end{split}
\end{equation}

From (\ref{S7S2L2E2}) $x_i' \leq x_i$, $y_i' \leq y_i$ for $i \in \llbracket 1, k \rrbracket$ and from (\ref{S7S2L2E3}) $g'(j) \leq g(j)$ for $j \in  \llbracket T_0 ,T_1 \rrbracket$, which by Lemma \ref{MonCoup} implies that 
\begin{equation}\label{S7S2L2E4}
\begin{split}
&\mathbb{P}_{H,H^{RW}}^{1, k, T_0 ,T_1, \vec{x}, \vec{y},\infty,g} \left(L_k(x) - p x  < M^{\mathsf{bot}} R^{1/2} - R^{1/4} \mbox{ for all $x \in [t_0, t_1]$}\right) \leq \\
&\mathbb{P}_{H,H^{RW}}^{1, k, T_0 ,T_1, \vec{x}', \vec{y}',\infty,g'} \left(L_k(x) - p x  < M^{\mathsf{bot}} R^{1/2} - R^{1/4} \mbox{ for all $x \in [t_0, t_1]$}\right) = \\
&\frac{\mathbb{E}_{H^{RW}} \left[ W_H \cdot {\bf 1} \{ \sup_{x \in [t_0, t_1]} [\ell_k(x)  - p x]  < M^{\mathsf{bot}} R^{1/2} - R^{1/4} \} \right] }{\mathbb{E}_{H^{RW}} \left[ W_H \right]},
\end{split}
\end{equation}
where we write $\mathbb{P}_{H^{RW}}$ in place of $\mathbb{P}_{H^{RW}}^{1, k, T_0 ,T_1, \vec{x}', \vec{y}'}$, $\mathbb{E}_{H^{RW}}$ in place of $\mathbb{E}_{H^{RW}}^{1, k, T_0 ,T_1, \vec{x}', \vec{y}'}$ and $W_H$ in place of $W_{H}^{1, k, T_0 ,T_1,\infty, g'} ({\ell}_{1}, \dots, {\ell}_{k})$ to simplify the notation. In addition, $(\ell_1, \dots, \ell_k)$ has distribution $\mathbb{P}_{H^{RW}}$. We mention that in deriving the last equality in (\ref{S7S2L2E4}) we used (\ref{RND}).

We claim that there exists $N_1 \in \mathbb{N}$ and $\delta > 0$, depending on $k,p,r,M^{\mathsf{side}}, M^{\mathsf{bot}},t, \epsilon, H, H^{RW}$, such that for $T = T_1 - T_0 \geq N_1$ we have 
\begin{equation}\label{S7S2L2E5}
\begin{split}
&\mathbb{E}_{H^{RW}} \left[ W_H \right] \geq \delta.
\end{split}
\end{equation}
We prove (\ref{S7S2L2E5}) in the steps below. Here we assume its validity and conclude the proof of the lemma.\\

Let $W_3 = W_3(k,p,r,M^{\mathsf{side}}, M^{\mathsf{bot}},t, \epsilon, H, H^{RW}) \in \mathbb{N}$ be sufficiently large so that $W_3 \geq N_1$ and for $T = T_1 - T_0 \geq W_3$ and $R \geq r T$ we have
\begin{equation}\label{S7S2L2E6}
\begin{split}
\exp \left(- H( R^{1/4} - |p| )  \right) \leq \delta \epsilon.
\end{split}
\end{equation}
Note that such a choice of $W_3$ is possible since by Definition \ref{AssH} we have that $\lim_{x \rightarrow \infty} H(x) = \infty$. We then observe by the definition of $W_H$ in (\ref{WH}) and the fact that $H(x) \in [0,\infty]$ 
\begin{equation}\label{S7S2L2E7}
\begin{split}
&\mathbb{E}_{H^{RW}} \left[ W_H \cdot {\bf 1} \Big{\{} \sup_{x \in [t_0, t_1]} [\ell_k(x)  - p x]  < M^{\mathsf{bot}} R^{1/2} - R^{1/4} \Big{\}} \right]  = \\
&\mathbb{E}_{H^{RW}} \Big{[} {\bf 1}  \Big{\{} \sup_{x \in [t_0, t_1]} [\ell_k(x)  - p x]  < M^{\mathsf{bot}} R^{1/2} - R^{1/4} \Big{\}} \cdot e^{- H(g'(j_0) - \ell_k(j_0-1)) } \times \\
& e^{- \sum_{i = 1}^{k-1}  \sum_{ m = T_0}^{T_1-1} \hspace{-1mm}H (\ell_{i + 1}(m + 1) - \ell_{i}(m))  }  \Big{]} \leq e^{- H(g'(j_0) - [M^{\mathsf{bot}} R^{1/2} - R^{1/4} + p(j_0-1)]) }  \leq e^{-H(R^{1/4} - |p|)} \leq \delta \epsilon. 
\end{split}
\end{equation}
We mention that in the next to last inequality we used (\ref{S7S2L2E3}) and the monotonicity of $H$, and in the last inequality we used (\ref{S7S2L2E6}).

Combining (\ref{S7S2L2E4}), (\ref{S7S2L2E5}) and (\ref{S7S2L2E7}) we get (\ref{S7EHighBottom}).\\

{\bf \raggedleft Step 2.} In this step we prove (\ref{S7S2L2E5}). Let $h: [0,1] \rightarrow \mathbb{R}$ be defined by
\begin{equation}\label{S7S2L2E8}
\begin{split}
h(x) = \begin{cases} (2U/t) x&\mbox{ if $x \in [0,t/2]$} \\ U &\mbox{ if $x \in [t/2, 1- t/2]$} \\  (2U/t)(1-x) &\mbox{ if $x \in [1-t/2,1]$},  \end{cases}
\end{split}
\end{equation}
where $U =r^{-1/2}M^{\mathsf{bot}} + r^{-1/2} M^{\mathsf{side}}+ 6k  + k \cdot (B + |p|)$. We claim that there exist $N_2 = N_2(p, H^{RW}, h)  \in \mathbb{N}$ and $\tilde{\delta} = \tilde{\delta}(p, H^{RW}, h)> 0$ such that for $T = T_1 - T_0 \geq N_2$ we have 
\begin{equation}\label{S7S2L2E9}
\begin{split}
&\mathbb{P}_{H^{RW}}^{T_0 ,T_1, pT_0, pT_1 }\left( \sup_{x \in [T_0, T_1] } \left| \ell(x) - p x - \sqrt{T} h((x - T_0)/T)  \right|  \leq 2\sqrt{T} \right) \geq \tilde{\delta}.
\end{split}
\end{equation}
We prove (\ref{S7S2L2E9}) in Step 3 below. Here we assume its validity and finish the proof of (\ref{S7S2L2E5}).\\

Let $N_1 \in \mathbb{N}$ be sufficiently large so that $N_1 \geq N_2$, $N_1 \geq (2U/t)$ and for $T \geq N_1$ we have $j_0 - T_0 \pm 1 \in [tT/2, (1-t/2)T]$. We also let $\delta = (1/2)\tilde{\delta}^k$ and proceed to prove (\ref{S7S2L2E5}) with this choice of $N_1$ and $\delta$.

Define the events
$$E_i = \left\{ \sup_{x \in [T_0, T_1] } \left| \ell_i(x) - p x +z_i - \sqrt{T} h((x - T_0)/T)  \right|  \leq 2\sqrt{T} \right\} \mbox{ for $i \in \llbracket 1, k \rrbracket$}.$$
Using that $\ell_1, \dots, \ell_k$ are independent under $\mathbb{P}_{H^{RW}}$ and have laws $\mathbb{P}_{H^{RW}}^{T_0 ,T_1, x_i', y_i'}$ for $i \in \llbracket 1, k \rrbracket$ we conclude for $T = T_1 - T_0 \geq N_1$
\begin{equation}\label{S7S2L2E11}
\begin{split}
&\mathbb{P}_{H^{RW}} \left( \cap_{i = 1}^k E_i  \right) =  \prod_{i = 1}^k\mathbb{P}_{H^{RW}}^{T_0 ,T_1, x_i', y_i'} \hspace{-1mm}\left(  \sup_{x \in [T_0, T_1] } \left| \ell(x) - p x + z_i - \sqrt{T} h((x - T_0)/T)  \right|  \leq 2\sqrt{T} \hspace{-1mm} \right) = \\
&\mathbb{P}_{H^{RW}}^{T_0 ,T_1, pT_0 , pT_1} \left(  \sup_{x \in [T_0, T_1] }\left| \ell(x) - p x - \sqrt{T} h((x - T_0)/T)  \right|  \leq 2\sqrt{T} \right)^k \geq \tilde{\delta}^{k},
\end{split}
\end{equation}
where in going from the first to the second line we used the shift invariance of $\mathbb{P}_{H^{RW}}^{T_0 ,T_1, x,y}$ and (\ref{S7S2L2E2}), while in the last inequality we used (\ref{S7S2L2E9}) and the fact that $T \geq N_1$.

On the other hand, we have using (\ref{S7S2L2E2}) that on $\cap_{i = 1}^k E_i $ for $i \in \llbracket 1, k-1\rrbracket$ and $T \geq N_1$
\begin{equation*}
\begin{split}
&\ell_{i}(x) \geq \sqrt{T} h((x-T_0)/T) +  p x - z_i - 2 \sqrt{T} \geq\sqrt{T} h((x-T_0)/T)  + p x - z_{i+1} + 4 \sqrt{T} + (B + |p|) \sqrt{T}  \\
&\geq h((x + 1 - T_0)/T)  + p (x+1) - z_{i+1} + 3 \sqrt{T} + B \sqrt{T} \geq \ell_{i+1}(x+1) + B \sqrt{T}.
\end{split}
\end{equation*}
In addition, we have on $\cap_{i = 1}^k E_i $ for $T \geq N_1$  
\begin{equation*}
\begin{split}
&\ell_{k}(j_0 - 1) \geq \sqrt{T} h((j_0 - 1-T_0)/T) +  p (j_0 - 1) - z_k - 2 \sqrt{T} =  \sqrt{T} U +  p (j_0 - 1) - z_k - 2 \sqrt{T} \geq \\
& \sqrt{T} r^{-1/2} M^{\mathsf{bot}} + 4 \sqrt{T} + B \sqrt{T} + pj_0 \geq g'(j_0) + B \sqrt{T},
\end{split}
\end{equation*}
where in the first equality we used that $(j_0 - 1-T_0)/T \in [t/2, (1-t/2)]$ (since $T \geq N_1$), in the first inequality we used (\ref{S7S2L2E2}) and the definition of $U$. In the last inequality we used (\ref{S7S2L2E3}) and the fact that $rT \leq R \leq r^{-1}T$. 

The above inequalities, the monotonicity of $H$ (see Definition \ref{AssH}) and the definition of $W_H$ in (\ref{WH}) gives on $\cap_{i = 1}^k E_i $ 
\begin{equation}\label{S7S2L2E12}
W_H =  \exp \left( - \sum_{i = 1}^{k}  \sum_{ m = T_0}^{T_1-1} \hspace{-1mm}H (\ell_{i + 1}(m + 1) - \ell_{i}(m)) \right) \geq \exp \left( - kT H(-B \sqrt{T}) \right) \geq 1/2,
\end{equation}
where $\ell_{k+1} = g'$, and the last inequality used (\ref{S7S2L2E1}). Combining (\ref{S7S2L2E11}) and (\ref{S7S2L2E12}) proves (\ref{S7S2L2E5}) with the above choice of $N_1$ and $\delta$.\\

{\bf \raggedleft Step 3.} In this step we prove (\ref{S7S2L2E9}). Let $\sigma_p$ be as in Definition \ref{AssHR} and $B^{\sigma_p}$ a Brownian bridge on $[0,1]$ with diffusion parameter $\sigma_p$. From Lemma \ref{Spread} we can find $\tilde{\delta} > 0$ such that  
\begin{equation}\label{S7S2L2E13}
\mathbb{P} \left( \sup_{x \in [0,1]} \left| B^{\sigma_p}_x - h(x) \right| < 1 \right) > 2 \tilde{\delta}.
\end{equation}
This specifies our choice of $\tilde{\delta}$. 

From Proposition \ref{KMT} we know that we can find constants $0 < C, a, \alpha < \infty$ (depending on $p$ and $H^{RW}$) and a probability space with measure $\mathbb{P}$ on which are defined a Brownian bridge $B^{\sigma_p}$ with diffusion parameter $\sigma_p$ and a family of random curves $\ell^{(T,z)}$ on $[0, T]$, which is parametrized by $z \in \mathbb{R}$ such that $\ell^{(T,z)}$  has law $\mathbb{P}^{0,T,0,z}_{H^{RW}}$ and also (\ref{S7Chebyshev}) holds. Let $N_2 = N_2(p,\tilde{\delta}, H^{RW})  \in \mathbb{N}$ be sufficiently large so that for $T \geq N_2$ we have
\begin{equation}\label{S7S2L2E14}
 C e^{-a T^{1/2}  } e^{ \alpha (\log T)} < \tilde{\delta}.
\end{equation}
This specifies our choice of $N_2$ and we proceed to prove (\ref{S7S2L2E9}). \\

By the shift-invariance of $\mathbb{P}^{T_0, T_1,x, y}_{H^{RW}}$ and the triangle inequality we have for $T = T_1 - T_0 \geq N_2$
\begin{equation*}
\begin{split}
&\mathbb{P}_{H^{RW}}^{T_0 ,T_1, pT_0, pT_1 }\left( \sup_{x \in [T_0, T_1] } \left| \ell(x) - p x - \sqrt{T} h((x - T_0)/T)  \right|  \leq 2\sqrt{T} \right)  = \\
&\mathbb{P}\left( \sup_{x \in [0, T] } \left| \ell^{(T, pT )}(x) - p x - \sqrt{T} h(x/T)  \right|  \leq 2\sqrt{T} \right) \geq   \mathbb{P} \left( \sup_{x \in [0,1]} \left| B^{\sigma_p}_x - h(x) \right| < 1 \right)  - \\
&   \mathbb{P}\left( \Delta(T, pT)  \geq  \sqrt{T} \right)  >  2\tilde{\delta} -  \mathbb{P}\left( \Delta(T, pT)  \geq  \sqrt{T} \right) > \tilde{\delta},
\end{split}
\end{equation*}
where in the next to last inequality we used (\ref{S7S2L2E13}) and in the last inequality we used (\ref{S7Chebyshev}) and (\ref{S7S2L2E14}). The last inequality gives (\ref{S7S2L2E9}).
\end{proof}


\begin{lemma}\label{S7LNoParDip}[Lemma \ref{LNoParDip}] Fix $k \in \mathbb{N}$ and let $\mathfrak{L} = (L_1, \dots, L_k)$ have law $\mathbb{P}_{H,H^{RW}}^{1, k, T_0 ,T_1, \vec{x}, \vec{y},\infty,g}$ as in Definition \ref{Pfree}, where we assume that $H$ is as in Definition \ref{AssH}, while $H^{RW}$ is as in Definition \ref{AssHR}. Let $ p \in \mathbb{R}$, $\lambda> 0$, $L \geq 1$, $ \epsilon \in (0,1)$, $M^{\mathsf{top}}  > 0 $ be given. Then we can find constants $A = A(k,p, \epsilon, H^{RW}, H) > 0$, $W_4 = W_4(k,p, \lambda, M^{\mathsf{top}}, \epsilon, H, H^{RW},L) \in \mathbb{N}$ and $M^{\mathsf{base, top}} = M^{\mathsf{base, top}} (k,p, \lambda, M^{\mathsf{top}}, \epsilon, H, H^{RW},L) > 0$ such that the following holds. 

If $T \geq W_4$, $T_0, T_1 \in \mathbb{Z}$ with $\max(|T_0|, |T_1|) \leq T \cdot L$, $\Delta T = T_1 - T_0 \geq T$, $\vec{x}, \vec{y} \in \mathbb{R}^k$ and $g \in Y^-(\llbracket T_0, T_1 \rrbracket)$ satisfy 
$$  x_i - pT_0 + \lambda T^{1/2} (T_0/T)^2 \leq M^{\mathsf{top}} T^{1/2}, y_i - pT_1 + \lambda T^{1/2}  (T_1/T)^2   \leq M^{\mathsf{top}} T^{1/2} \mbox{ for $i \in \llbracket 1, k \rrbracket$, and }$$ 
$$g(j) - pj \leq - M^{\mathsf{base, top}} T^{1/2} \mbox{ for all $j \in \llbracket T_0, T_1 - 1 \rrbracket$,}$$
then we have that 
\begin{equation}\label{S7ENoParDip}
\begin{split}
1-\epsilon \leq \mathbb{P}_{H,H^{RW}}^{1, k, T_0 ,T_1, \vec{x}, \vec{y},\infty,g} &\Bigg{(} L_1(x) - p x +  \frac{x- T_0}{T_1 - T_0}\cdot \lambda T^{1/2} (T_1/T)^2 +  \\
&   \frac{T_1 - x}{T_1 - T_0} \cdot \lambda T^{1/2} (T_0/T)^2 \leq (A +M^{\mathsf{top}}) \Delta T^{1/2}  \mbox{ for all $x \in [T_0, T_1]$}\Bigg{)}.
\end{split}
\end{equation}
\end{lemma}

\begin{proof} For clarity we split the proof into three steps. \\

{\bf \raggedleft Step 1.} In this step we specify the choice of $A, W_4$ and $M^{\mathsf{base, top}} $ as in the statement of the lemma.

Let $\delta = \delta(k, \epsilon) \in (0,1) $ be sufficiently small so that
\begin{equation}\label{S7S2L3E1}
 (1- \delta)^{-k - 1} \delta < \epsilon.
\end{equation}
Next, let $A_1 = A(\delta, p, H^{RW}) > 0$ be as in Lemma \ref{S7LStayInBand} for $p, H^{RW}$ as in the present lemma and $\delta$ as above. In addition, let $B = B(k, \delta, H) \geq 1$ be sufficiently large so that for $S \in \mathbb{N}$  
\begin{equation}\label{S7S2L3E2}
H(-B\sqrt{S}) \leq S^{-1} B^{-2} \mbox{ and } \exp \left( - k \cdot B^{-2}  \right) \geq 1 - \delta  .
\end{equation}
Note that such a choice of $B$ is possible since by Definition \ref{AssH} we have that $\lim_{x \rightarrow \infty} x^2 H(-x) = 0$. With the above data we set $A = 2k A_1 + k ( B + |p|)$ and this specifies our choice of $A$.

Let $W_2 = W_2(\lambda L^2, p, \delta, H^{RW})$ be as in Lemma \ref{S7LStayInBand} and we set $W_4 = W_2$. Finally, we choose $M^{\mathsf{base, top}} = M^{\mathsf{base, top}}(L, A_1, B,p) > 0$ sufficiently large so that 
\begin{equation}\label{S7S2L3E3}
M^{\mathsf{base, top}} T^{1/2} \geq |p| + A_1 \Delta T^{1/2} + B \Delta T^{1/2}.
\end{equation}
The above two paragraphs specify $A, W_4$ and $M^{\mathsf{base, top}} $.\\

{\bf \raggedleft Step 2.} In this step we prove (\ref{S7ENoParDip}) for $T \geq W_4$, $A$ and $M^{\mathsf{base, top}} $ as in Step 1.

Let $\vec{x}', \vec{y}' \in \mathbb{R}^k$ be given by
\begin{equation}\label{S7S2L3E6}
\begin{split}
&x_i' = pT_0 - \lambda T^{1/2} (T_0/T)^2+ z_i \mbox{ and } y_i' = pT_1   - \lambda T^{1/2} (T_1/T)^2 + z_i \mbox{ for $i \in \llbracket 1, k \rrbracket$, where }\\
&z_i =     \Delta T^{1/2} \left( M^{\mathsf{top}} +  2(k - i) \cdot A_1 + (k - i) \cdot (B + |p|) \right).
\end{split}
\end{equation}
In (\ref{S7S2L3E6}) $A_1$ and $B$ are as in Step 1.

Observe that $x_i' \leq x_i$, $y_i' \leq y_i$ for $i \in \llbracket 1, k \rrbracket$. From Lemma \ref{MonCoup} and (\ref{RND}) we conclude that 
\begin{equation}\label{S7S2L3E7}
\begin{split}
&\mathbb{P}_{H,H^{RW}}^{1, k, T_0 ,T_1, \vec{x}, \vec{y},\infty,g} \Bigg{(} L_1(x) - p x +  \frac{x- T_0}{T_1 - T_0}\cdot \lambda T^{1/2} (T_1/T)^2 +  \\
&   \frac{T_1 - x}{T_1 - T_0} \cdot \lambda T^{1/2} (T_0/T)^2 > (A +M^{\mathsf{top}}) \Delta T^{1/2}  \mbox{ for some $x \in [T_0, T_1]$}\Bigg{)} \leq \frac{1}{ \mathbb{E}_{H^{RW}} \left[ W_H \right]} \times  \\
&\mathbb{E}_{H^{RW}} \Big[ W_H \cdot {\bf 1} \Big\{ \sup_{x \in [T_0, T_1]} \Big[ \ell_1(x) - p x +  \frac{x- T_0}{T_1 - T_0}\cdot \lambda T^{1/2} (T_1/T)^2 +  \\
&  \frac{T_1 - x}{T_1 - T_0} \cdot \lambda T^{1/2} (T_0/T)^2 \Big]  > (A +M^{\mathsf{top}}) \Delta T^{1/2} \Big\} \Big{]},
\end{split}
\end{equation}
where we write $\mathbb{P}_{H^{RW}}$ in place of $\mathbb{P}_{H^{RW}}^{1, k, T_0 ,T_1, \vec{x}', \vec{y}'}$, $\mathbb{E}_{H^{RW}}$ in place of $\mathbb{E}_{H^{RW}}^{1, k, T_0 ,T_1, \vec{x}', \vec{y}'}$ and $W_H$ in place of $W_{H}^{1, k, T_0 ,T_1,\infty, g} ({\ell}_{1}, \dots, {\ell}_{k})$ to simplify the notation. In addition, $(\ell_1, \dots, \ell_k)$ has distribution $\mathbb{P}_{H^{RW}}$. 

We claim that for $T \geq W_4$ we have 
\begin{equation}\label{S7S2L3E8}
\begin{split}
& \mathbb{E}_{H^{RW}} \left[ W_H \right] \geq (1 - \delta)^{k+1}. 
\end{split}
\end{equation}
We prove (\ref{S7S2L3E8}) in the next step. Here we assume its validity and conclude the proof of (\ref{S7ENoParDip}).\\

Using that $W_H \in [0,1]$, and that $\ell_1, \dots, \ell_k$ are independent under $\mathbb{P}_{H^{RW}}$ and have laws $\mathbb{P}_{H^{RW}}^{T_0 ,T_1, x_i', y_i'}$ for $i \in \llbracket 1, k \rrbracket$ we conclude that 
\begin{equation}\label{S7S2L3E9}
\begin{split}
&\mathbb{E}_{H^{RW}} \Big[ W_H \cdot {\bf 1} \Big\{ \sup_{x \in [T_0, T_1]} \Big[ \ell_1(x) - p x +  \frac{x- T_0}{T_1 - T_0}\cdot \lambda T^{1/2} (T_1/T)^2 +  \\
&  \frac{T_1 - x}{T_1 - T_0} \cdot \lambda T^{1/2} (T_0/T)^2 \Big]  > (A +M^{\mathsf{top}}) \Delta T^{1/2} \Big\} \Big{]} \leq \\
&\mathbb{P}_{H^{RW}}^{T_0, T_1, x_1', y_1'} \Big( \sup_{x \in [T_0, T_1]} \Big[ \ell(x) - p x +  \frac{x- T_0}{T_1 - T_0}\cdot \lambda T^{1/2} (T_1/T)^2 +  \\
&  \frac{T_1 - x}{T_1 - T_0} \cdot \lambda T^{1/2} (T_0/T)^2\Big]  > (A +M^{\mathsf{top}}) \Delta T^{1/2}  \Big{)}. 
\end{split}
\end{equation}
Using the shift-invariance of $\mathbb{P}_{H^{RW}}^{T_0, T_1,x,y}$ and the definition of $x_1', y_1'$ from (\ref{S7S2L3E6}) we get 
\begin{equation}\label{S7S2L3E10}
\begin{split}
&\mathbb{P}_{H^{RW}}^{T_0, T_1, x_1', y_1'} \Big( \sup_{x \in [T_0, T_1]} \Big[ \ell(x) - p x +  \frac{x- T_0}{T_1 - T_0}\cdot \lambda T^{1/2} (T_1/T)^2 +  \\
&  \frac{T_1 - x}{T_1 - T_0} \cdot \lambda T^{1/2} (T_0/T)^2\Big]  > (A +M^{\mathsf{top}}) \Delta T^{1/2}  \Big{)}= \\
&\mathbb{P}_{H^{RW}}^{T_0, T_1, pT_0 - \lambda T^{1/2} (T_0/T)^2, pT_1- \lambda T^{1/2} (T_0/T)^2} \Big( \sup_{x \in [T_0, T_1]} \Big[ \ell(x) - p x +  \frac{x- T_0}{T_1 - T_0}\cdot \lambda T^{1/2} (T_1/T)^2 +  \\
&  \frac{T_1 - x}{T_1 - T_0} \cdot \lambda T^{1/2} (T_0/T)^2\Big]  > (A +M^{\mathsf{top}}) \Delta T^{1/2} - z_1  \Big{)} \leq \\
&\mathbb{P}_{H^{RW}}^{T_0, T_1, pT_0 - \lambda T^{1/2} (T_1/T)^2, pT_1- \lambda T^{1/2} (T_0/T)^2} \Big( \sup_{x \in [T_0, T_1]} \Big[ \ell(x) - p x +  \frac{x- T_0}{T_1 - T_0}\cdot \lambda T^{1/2} (T_0/T)^2 +  \\
&  \frac{T_1 - x}{T_1 - T_0} \cdot \lambda T^{1/2} (T_0/T)^2\Big]  > A_1 \Delta T^{1/2} \Big{)} < \delta.
\end{split}
\end{equation}
We mention that in the first inequality we used the definition of $A$ from Step 1 and equation (\ref{S7S2L3E6}), while in the last inequality we used the definition of $A_1$ from Step 1 and the fact that $T_1 - T_0 = \Delta  T \geq T\geq W_4$. Equations (\ref{S7S2L3E7}), (\ref{S7S2L3E8}), (\ref{S7S2L3E9}) and (\ref{S7S2L3E10}) together imply (\ref{S7ENoParDip}) once we also utilize (\ref{S7S2L3E1}).\\

{\bf \raggedleft Step 3.} In this step we prove (\ref{S7S2L3E8}). Let us define the events 
$$E_i = \left\{ \sup_{x \in [T_0, T_1] } \left| \ell_i(x) - p x - z_i  \right|  \leq A_1 \Delta T^{1/2} \right\} \mbox{ for $i \in \llbracket 1, k \rrbracket$}.$$
Using that $\ell_1, \dots, \ell_k$ are independent under $\mathbb{P}_{H^{RW}}$ and have laws $\mathbb{P}_{H^{RW}}^{T_0 ,T_1, x_i', y_i'}$ for $i \in \llbracket 1, k \rrbracket$ we conclude
\begin{equation}\label{S7S2L3E11}
\begin{split}
&\mathbb{P}_{H^{RW}} \left( \cap_{i = 1}^k E_i  \right) = \prod_{i = 1}^k\mathbb{P}_{H^{RW}} \left(  E_i  \right) =  \prod_{i = 1}^k\mathbb{P}_{H^{RW}}^{T_0 ,T_1, x_i', y_i'} \left(  \sup_{x \in [T_0, T_1] } \left| \ell(x) - p x - z_i \right|  \leq A_1 \Delta T^{1/2} \right) \\
& = \mathbb{P}_{H^{RW}}^{T_0 ,T_1, pT_0 - \lambda T^{1/2} (T_0/T)^2 , pT_1 - \lambda T^{1/2} (T_1/T)^2} \left(  \sup_{x \in [T_0, T_1] } \left| \ell(x) - p x \right|  \leq A_1\Delta T^{1/2} \right)^k \geq (1- \delta)^{k},
\end{split}
\end{equation}
where in going from the first to the second line we used the shift invariance of $\mathbb{P}_{H^{RW}}^{T_0 ,T_1, x,y}$ and (\ref{S7S2L3E6}), while in the last inequality we used the definition of $A_1$ from Step 1 and the fact that $T \geq W_4$.

On the other hand, we have using (\ref{S7S2L3E6}) that on $\cap_{i = 1}^k E_i $ for $i \in \llbracket 1, k-1\rrbracket$ and $T \geq W_4$
\begin{equation*}
\begin{split}
&\ell_{i}(x) \geq p x + z_i -A_1 \Delta T^{1/2} \geq p x + z_{i+1} + A_1 \Delta T^{1/2} + (B + |p|) \Delta T^{1/2} \geq \\
& p (x+1) + z_{i+1} + A_1 \Delta T^{1/2} + B \Delta T^{1/2} \geq \ell_{i+1}(x+1) + B \Delta T^{1/2}.
\end{split}
\end{equation*}
Using (\ref{S7S2L3E6}), the fact that $g(j) - pj \leq - M^{\mathsf{base, top}} T^{1/2} \mbox{ for all $j \in \llbracket T_0, T_1 - 1 \rrbracket$}$, and the definition of  $M^{\mathsf{base, top}} $ from Step 1, we conclude that on $\cap_{i = 1}^k E_i $ for $T \geq W_4$ we have
\begin{equation*}
\begin{split}
&\ell_{k}(x) \geq p x + z_k -A_1 \Delta T^{1/2} \geq p (x +1) - M^{\mathsf{base, top}} T^{1/2} + B \Delta T^{1/2}  \geq g(x+1) + B \Delta T^{1/2}.
\end{split}
\end{equation*}
The last two inequalities, the monotonicity of $H$ (see Definition \ref{AssH}) and the definition of $W_H$ in (\ref{WH}) gives on $\cap_{i = 1}^k E_i $ 
\begin{equation}\label{S7S2L3E12}
W_H =  \exp \left( - \sum_{i = 1}^{k}  \sum_{ m = T_0}^{T_1-1} \hspace{-1mm}H (\ell_{i + 1}(m + 1) - \ell_{i}(m)) \right) \geq \exp \left( - k \Delta T H(-B \Delta T^{1/2}) \right) \geq 1-\delta,
\end{equation}
where $\ell_{k+1} = g$ and the last inequality used (\ref{S7S2L3E2}). Combining (\ref{S7S2L3E11}) and (\ref{S7S2L3E12}) proves (\ref{S7S2L3E8}).
\end{proof}

%
\subsection{Proof of Lemma \ref{LAccProb}}\label{Section7.3} In this section we prove Lemma \ref{LAccProb}, whose statement is recalled here as Lemma \ref{S7LAccProb} for the reader's convenience. In order to establish the result we will require four lemmas, which we proceed to state and prove.


\begin{lemma}\label{LStayUp} Fix $k \in \mathbb{N}$ and let $\mathfrak{L} = (L_1, \dots, L_k)$ have law $\mathbb{P}_{H,H^{RW},S}^{1, k, 0 ,T, \vec{x}, \vec{y},\infty,g}$ as in Definition \ref{Pfree}, where we assume that $H$ is as in Definition \ref{AssH}, while $H^{RW}$ satisfies the assumptions in Definition \ref{AssHR}. Let $A, M, p \in \mathbb{R}$ and $t \in (0,1/2]$ be given. Then we can find constants $U_1 = U_1(A,M,p,t,k) \in \mathbb{N}$ and $\epsilon = \epsilon(A,M,p,t,k) > 0$ so that the following holds. If $T \geq U_1$, $g \in Y^-(\llbracket 0, T \rrbracket) $, $S \subset \llbracket 0, T-1 \rrbracket$, $x_i \geq  M T^{1/2}$ and $y_i \geq p T + MT^{1/2}$ for $i = 1, \dots, k$ then 
\begin{equation}\label{EStayUp}
\mathbb{P}_{H,H^{RW},S}^{1, k, 0 ,T, \vec{x}, \vec{y},\infty,g} \left(L_i(x) - p x  \geq A T^{1/2}  \mbox{ for all $x \in [t T, (1-t)T]$ and $i \in \llbracket 1, k \rrbracket$}\right) \geq \epsilon.
\end{equation}
\end{lemma}
\begin{proof} 

For clarity we split the proof into three steps.\\

{\bf \raggedleft Step 1.} In this step we introduce some notation and prove the lemma modulo a certain estimate, see (\ref{XOP5}), which will be established in the next steps.

Let $B_1 = B_1(k, H) > 0$ be sufficiently large so that for all $T \in \mathbb{N}$
\begin{equation}\label{XOP1}
\exp \left( - (k-1) T H \left( - B_1 \sqrt{T}  \right) \right) \geq 1/2.
\end{equation}
Note that such a choice of $B_1$ is possible, since by Definition \ref{AssH} we have that $\lim_{x \rightarrow \infty} x^2 H(-x) = 0$. 

Let $\vec{x}', \vec{y}' \in \mathbb{R}^k$ be defined via
\begin{equation}\label{XOP2}
\begin{split}
&x_i' = - z_i \mbox{ and } y_i' = pT - z_i \mbox{ for $i \in \llbracket 1, k \rrbracket$, where } z_i = \sqrt{T} \left( |M| + (i-1) \cdot (B_1 + 6  + |p|) \right).
\end{split}
\end{equation}
Since $x_i' \leq x_i$ and $y_i' \leq y_i$ for $i \in \llbracket 1, k \rrbracket$ we have from Lemma \ref{MonCoup} that 
\begin{equation}\label{XOP3}
\begin{split}
&\mathbb{P}_{H,H^{RW},S}^{1, k, 0 ,T, \vec{x}, \vec{y},\infty,g} \left(L_i(x) - p x  \geq A T^{1/2}  \mbox{ for all $x \in [t T, (1-t)T]$ and $i \in \llbracket 1, k \rrbracket$}\right) \geq \\
&\mathbb{P}_{H,H^{RW},S}^{1, k, 0 ,T, \vec{x}', \vec{y}',\infty, -\infty} \left(L_i(x) - p x  \geq A T^{1/2}  \mbox{ for all $x \in [t T, (1-t)T]$ and $i \in \llbracket 1, k \rrbracket$}\right)  = \\
& \frac{\mathbb{E}_{H^{RW}} \left[ W_{H,S} \cdot {\bf 1} \{ \ell_i(x) - p x  \geq A T^{1/2}  \mbox{ for all $x \in [t T, (1-t)T]$ and $i \in \llbracket 1, k \rrbracket$} \} \right] }{\mathbb{E}_{H^{RW}} \left[ W_{H,S} \right]},
\end{split}
\end{equation}
where we write $\mathbb{P}_{H^{RW}}$ in place of $\mathbb{P}_{H^{RW}}^{1, k, T_0 ,T_1, \vec{x}', \vec{y}'}$, $\mathbb{E}_{H^{RW}}$ in place of $\mathbb{E}_{H^{RW}}^{1, k, T_0 ,T_1, \vec{x}', \vec{y}'}$ and $W_{H,S}$ in place of $W_{H,S}^{1, k, T_0 ,T_1,\infty, -\infty} ({\ell}_{1}, \dots, {\ell}_{k})$ to simplify the notation. In addition, $(\ell_1, \dots, \ell_k)$ has distribution $\mathbb{P}_{H^{RW}}$. We mention that in deriving the last equality in (\ref{XOP4}) we used (\ref{RND}).

We claim that there exists $\epsilon  = \epsilon(A, M, p ,t, k) > 0$ and $U_1 = U_1(A,M,p,t,k) \in \mathbb{N}$ such that for $T \geq U_1$ we have
\begin{equation}\label{XOP4}
\begin{split}
&\mathbb{E}_{H^{RW}} \left[ W_{H,S} \cdot {\bf 1} \{ \ell_i(x) - p x  \geq A T^{1/2}  \mbox{ for all $x \in [t T, (1-t)T]$ and $i \in \llbracket 1, k \rrbracket$} \} \right]  \geq \epsilon.
\end{split}
\end{equation}
Note that (\ref{XOP3}), (\ref{XOP4}) and the fact that $W_{H,S} \in (0,1]$ together imply (\ref{EStayUp}). Thus we have reduced our problem to proving (\ref{XOP4}).\\

{\bf \raggedleft Step 2.} In this step we prove (\ref{XOP4}). Let $h: [0,1] \rightarrow \mathbb{R}$ be defined by
\begin{equation}\label{XOP5}
\begin{split}
h(x) = \begin{cases} (2U/t) x&\mbox{ if $x \in [0,t/2]$} \\ U &\mbox{ if $x \in [t/2, 1- t/2]$} \\  (2U/t)(1-x) &\mbox{ if $x \in [1-t/2,1]$},  \end{cases}
\end{split}
\end{equation}
where $U =|M| + |A|+ 6k  + k \cdot (B_1 + |p|)$. We claim that there exist $N_1 = N_1(p, H^{RW}, h)  \in \mathbb{N}$ and $\tilde{\delta} = \tilde{\delta}(p, H^{RW}, h)> 0$ such that for $T  \geq N_1$ we have 
\begin{equation}\label{XOP6}
\begin{split}
&\mathbb{P}_{H^{RW}}^{0 , T, 0, pT }\left( \sup_{x \in [0,T] } \left| \ell(x) - p x - \sqrt{T} h(x/T)  \right|  \leq 2\sqrt{T} \right) \geq \tilde{\delta}.
\end{split}
\end{equation}
We prove (\ref{XOP6}) in Step 3 below. Here we assume its validity and finish the proof of  (\ref{XOP4}).\\

We take $U_1  \in \mathbb{N}$ sufficiently large so that $U_1 \geq N_1$ and for $T \geq U_1$ we have $\sqrt{T} \geq 2U/t$. We also fix $\epsilon = (1/2)\tilde{\delta}^k$ and proceed to prove (\ref{XOP4}) with this choice of $U_1$ and $\epsilon$. Define the events
$$E_i = \left\{ \sup_{x \in [0, T] } \left| \ell_i(x) - p x +z_i - \sqrt{T} h(x/T)  \right|  \leq 2\sqrt{T} \right\} \mbox{ for $i \in \llbracket 1, k \rrbracket$}.$$
Using that $\ell_1, \dots, \ell_k$ are independent under $\mathbb{P}_{H^{RW}}$ and have laws $\mathbb{P}_{H^{RW}}^{0 , T, x_i', y_i'}$ for $i \in \llbracket 1, k \rrbracket$ we conclude
\begin{equation}\label{XOP7}
\begin{split}
&\mathbb{P}_{H^{RW}} \left( \cap_{i = 1}^k E_i  \right) =  \prod_{i = 1}^k\mathbb{P}_{H^{RW}}^{0 ,T, x_i', y_i'} \hspace{-1mm}\left(  \sup_{x \in [0, T] } \left| \ell(x) - p x + z_i - \sqrt{T} h(x/T)  \right|  \leq 2\sqrt{T}  \right) = \\
&\mathbb{P}_{H^{RW}}^{0, T, 0, pT} \left(  \sup_{x \in [0, T] }\left| \ell(x) - p x - \sqrt{T} h(x/T)  \right|  \leq 2\sqrt{T} \right)^k \geq \tilde{\delta}^{k},
\end{split}
\end{equation}
where in going from the first to the second line we used the shift invariance of $\mathbb{P}_{H^{RW}}^{T_0 ,T_1, x,y}$ and (\ref{XOP2}), while in the last inequality we used (\ref{XOP6}) and the fact that $T \geq U_1$.

On the other hand, we have using (\ref{XOP2}) that on $\cap_{i = 1}^k E_i $ for $i \in \llbracket 1, k-1\rrbracket$ and $T \geq U_1$
\begin{equation*}
\begin{split}
&\ell_{i}(x) \geq \sqrt{T} h(x/T) +  p x - z_i - 2 \sqrt{T} \geq\sqrt{T} h(x/T)  + p x - z_{i+1} + 4 \sqrt{T} + (B + |p|) \sqrt{T}  \\
&\geq h((x + 1)/T)  + p (x+1) - z_{i+1} + 3 \sqrt{T} + B \sqrt{T} \geq \ell_{i+1}(x+1) + B \sqrt{T}.
\end{split}
\end{equation*}
The above inequalities, the monotonicity of $H$, the fact that $H \geq 0$ (see Definition \ref{AssH}) and the definition of $W_{H,S}$ in (\ref{WH}) gives on $\cap_{i = 1}^k E_i $ 
\begin{equation}\label{XOP8}
W_{H,S} =  \exp \left( - \sum_{i = 1}^{k-1}  \sum_{ m \in S} \hspace{-1mm}H (\ell_{i + 1}(m + 1) - \ell_{i}(m)) \right) \geq \exp \left( - (k-1)T H(-B \sqrt{T}) \right) \geq 1/2,
\end{equation}
where in the last inequality used (\ref{XOP1}). Combining (\ref{XOP7}) and (\ref{XOP8}) gives for $T \geq U_1$
\begin{equation}\label{XOP9}
\begin{split}
&\mathbb{E}_{H^{RW}} \left[ W_{H,S} \cdot {\bf 1} \{ \ell_i(x) - p x  \geq A T^{1/2}  \mbox{ for all $x \in [t T, (1-t)T]$ and $i \in \llbracket 1, k \rrbracket$} \} \right] \geq \\
&\mathbb{E}_{H^{RW}} \left[ W_{H,S} \cdot {\bf 1}_{\cap_{i = 1}^k E_i } \right] \geq \tilde{\delta}^{k}/2,
\end{split}
\end{equation}
which implies (\ref{XOP4}) with the above choice of $\epsilon$ and $U_1$. We mention that the first inequality in (\ref{XOP9}) used that on $\cap_{i = 1}^k E_i$ we have for $x \in [t, (1-t)]$ and $i \in \llbracket 1, k \rrbracket$ that 
$$\ell_i(x) - px \geq \sqrt{T} U - z_i \geq \sqrt{T} \left( |A| +  B_1 + |p| + 6  \right) \geq \sqrt{T}A.$$

{\bf \raggedleft Step 3.} In this step we prove (\ref{XOP6}). Let $\sigma_p$ be as in Definition \ref{AssHR} and $B^{\sigma_p}$ a Brownian bridge on $[0,1]$ with diffusion parameter $\sigma_p$. From Lemma \ref{Spread} we can find $\tilde{\delta} > 0$ such that  
\begin{equation}\label{XOP10}
\mathbb{P} \left( \sup_{x \in [0,1]} \left| B^{\sigma_p}_x - h(x) \right| < 1 \right) > 2 \tilde{\delta}.
\end{equation}
This specifies our choice of $\tilde{\delta}$. 

From Proposition \ref{KMT} we know that we can find constants $0 < C, a, \alpha < \infty$ (depending on $p$ and $H^{RW}$) and a probability space with measure $\mathbb{P}$ on which are defined a Brownian bridge $B^{\sigma_p}$ with diffusion parameter $\sigma_p$ and a family of random curves $\ell^{(T,z)}$ on $[0, T]$, which is parametrized by $z \in \mathbb{R}$ such that $\ell^{(T,z)}$  has law $\mathbb{P}^{0,T,0,z}_{H^{RW}}$ and also (\ref{S7Chebyshev}) holds. Let $N_1 = N_1(p,\tilde{\delta}, H^{RW})  \in \mathbb{N}$ be sufficiently large so that for $T \geq N_1$ we have
\begin{equation}\label{XOP11}
 C e^{-a T^{1/2}  } e^{ \alpha (\log T)} < \tilde{\delta}.
\end{equation}
This specifies our choice of $N_1$ and we proceed to prove (\ref{XOP6}). \\

By the triangle inequality we have for $T \geq N_1$
\begin{equation*}
\begin{split}
&\mathbb{P}_{H^{RW}}^{0 , T, 0, pT }\left( \sup_{x \in [0, T] } \left| \ell(x) - p x - \sqrt{T} h(x/T)  \right|  \leq 2\sqrt{T} \right)  = \\
&\mathbb{P}\left( \sup_{x \in [0, T] } \left| \ell^{(T, pT )}(x) - p x - \sqrt{T} h(x/T)  \right|  \leq 2\sqrt{T} \right) \geq   \mathbb{P} \left( \sup_{x \in [0,1]} \left| B^{\sigma_p}_x - h(x) \right| < 1 \right)  - \\
&   \mathbb{P}\left( \Delta(T, pT)  \geq  \sqrt{T} \right)  >  2\tilde{\delta} -  \mathbb{P}\left( \Delta(T, pT)  \geq  \sqrt{T} \right) > \tilde{\delta},
\end{split}
\end{equation*}
where in the next to last inequality we used (\ref{XOP10}) and in the last inequality we used (\ref{S7Chebyshev}) and (\ref{XOP11}). The last inequality gives (\ref{XOP6}).
\end{proof}


\begin{lemma}\label{LStayPut} Fix $k \in \mathbb{N}$ and let $\mathfrak{L} = (L_1, \dots, L_k)$ have law $\mathbb{P}_{H,H^{RW},S}^{1, k, 0 ,T, \vec{x}, \vec{y},\infty,g}$ as in Definition \ref{Pfree}, where we assume that $H$ is as in Definition \ref{AssH}, while $H^{RW}$ satisfies the assumptions in Definition \ref{AssHR}. Let $ \epsilon > 0$, $M, p \in \mathbb{R}$ be given. Then we can find constants $U_2 = U_2(M, \epsilon,p ,k) \in \mathbb{N}$ and $B = B(M, \epsilon, p ,k) \in \mathbb{R}$ so that the following holds. If $T \geq U_2$, $g \in Y^-(\llbracket 0, T \rrbracket) $, $S \subset \llbracket 0, T-1 \rrbracket$, $x_i \leq  M T^{1/2}$, $y_i \leq p T + MT^{1/2}$ for $i = 1, \dots, k$ and $g(r) \leq r p + M T^{1/2}$ for $r \in \llbracket 0, T\rrbracket$ then 
\begin{equation}\label{EStayPut}
\mathbb{P}_{H,H^{RW},S}^{1, k, 0 ,T, \vec{x}, \vec{y},\infty,g} \left(  L_i(x) - p x  \geq B T^{1/2} \mbox{ for some $x \in [0,T]$ and some $i \in \llbracket 1, k \rrbracket$}\right) \leq \epsilon.
\end{equation}
\end{lemma}
\begin{proof} For clarity we split the proof into three steps. \\

{\bf \raggedleft Step 1.} In this step we specify the constants $U_2$ and $B$ as in the statement of the lemma.

Let $\delta = \delta(k, \epsilon) \in (0,1) $ be sufficiently small so that
\begin{equation}\label{YOP1}
 k (1- \delta)^{-k - 1} \delta < \epsilon.
\end{equation}
Next, let $A_1 = A(\delta, p, H^{RW}) > 0$ be as in Lemma \ref{S7LStayInBand} for $p, H^{RW}$ as in the present lemma and $\delta$ as above. In addition, let $B_1 = B_1(k, \delta, H) > 0$ be sufficiently large so that for $T \in \mathbb{N}$  
\begin{equation}\label{YOP2}
\exp \left( - k T H \left( - B_1 T^{1/2}  \right) \right) \geq 1- \delta.
\end{equation}
Note that such a choice of $B_1$ is possible since by Definition \ref{AssH} we have that $\lim_{x \rightarrow \infty} x^2 H(-x) = 0$. 

With the above data we let $B = |M| + (2k + 1) \cdot A_1 + (k + 1) \cdot (B_1 + |p|)$ and $U_2 = W_1( B , p, \delta ,H^{RW}) $ as in Lemma \ref{S7LStayInBand}. This specifies our choice of $U_2$ and $B$.\\

{\bf \raggedleft Step 2.} We assume throughout that $T \geq U_2$ as in Step 1. Let $\vec{x}', \vec{y}' \in \mathbb{R}^k$ be given by
\begin{equation}\label{YOP3}
\begin{split}
&x_i' =  z_i \mbox{ and } y_i' = pT + z_i \mbox{ for $i \in \llbracket 1, k \rrbracket$, where }\\
&z_i = T^{1/2} \left(|M|+ 2(k-i + 1) \cdot A_1 + (k-i + 1) \cdot (B_1 + |p|) \right).
\end{split}
\end{equation}
and observe that $x_i \leq x_i'$, $y_i \leq y_i'$ for $i \in \llbracket 1, k \rrbracket$. In view of the latter inequalities and Lemma \ref{MonCoup} we conclude that 
\begin{equation}\label{YOP4}
\begin{split}
&\mathbb{P}_{H,H^{RW},S}^{1, k, 0 ,T, \vec{x}, \vec{y},\infty,g} \left(  L_i(x) - p x  \geq B T^{1/2} \mbox{ for some $x \in [0,T]$ and some $i \in \llbracket 1, k \rrbracket$}\right) \leq \\
&\mathbb{P}_{H,H^{RW},S}^{1, k, 0 ,T, \vec{x}', \vec{y}',\infty,g} \left(  L_i(x) - p x  \geq B T^{1/2} \mbox{ for some $x \in [0,T]$ and some $i \in \llbracket 1, k \rrbracket$}\right).
\end{split}
\end{equation}

In addition, from (\ref{RND}) we have
\begin{equation}\label{YOP5}
\begin{split}
&\mathbb{P}_{H,H^{RW},S}^{1, k, 0 ,T, \vec{x}', \vec{y}',\infty,g} \left(  L_i(x) - p x  \geq B T^{1/2} \mbox{ for some $x \in [0,T]$ and some $i \in \llbracket 1, k \rrbracket$}\right) = \\
& \frac{\mathbb{E}_{H^{RW}} \left[ W_{H,S} \cdot {\bf 1} \{ \ell_i(x) - p x  \geq B T^{1/2} \mbox{ for some $x \in [0,T]$ and some $i \in \llbracket 1, k \rrbracket$} \} \right] }{\mathbb{E}_{H^{RW}} \left[ W_{H,S} \right]},
\end{split}
\end{equation}
where we write $\mathbb{P}_{H^{RW}}$ in place of $\mathbb{P}_{H^{RW}}^{1, k, T_0 ,T_1, \vec{x}', \vec{y}'}$, $\mathbb{E}_{H^{RW}}$ in place of $\mathbb{E}_{H^{RW}}^{1, k, T_0 ,T_1, \vec{x}', \vec{y}'}$ and $W_{H,S}$ in place of $W_{H,S}^{1, k, T_0 ,T_1,\infty, g} ({\ell}_{1}, \dots, {\ell}_{k})$ to simplify the notation. In addition, $(\ell_1, \dots, \ell_k)$ has distribution $\mathbb{P}_{H^{RW}}$. 

We claim that 
\begin{equation}\label{YOP6}
\begin{split}
&\mathbb{E}_{H^{RW}} \left[ W_{H,S} \right] \geq (1 - \delta)^{k+1} \mbox{ and } \\
&\mathbb{E}_{H^{RW}} \left[ W_{H,S} \cdot {\bf 1} \{ \ell_i(x) - p x  \geq B T^{1/2} \mbox{ for some $x \in [0,T]$ and some $i \in \llbracket 1, k \rrbracket$} \} \right]  \leq k\delta.
\end{split}
\end{equation}
Observe that (\ref{YOP4}), (\ref{YOP5}) and (\ref{YOP6}) together imply (\ref{EStayPut}) in view of (\ref{YOP1}). Thus we have reduced the proof of the lemma to establishing the two inequalities in (\ref{YOP6}). We establish the second in this step, and postpone the first to the next step.\\

Using that $W_{H,S} \in [0,1]$, and that $\ell_1, \dots, \ell_k$ are independent under $\mathbb{P}_{H^{RW}}$ and have laws $\mathbb{P}_{H^{RW}}^{0 ,T, x_i', y_i'}$ for $i \in \llbracket 1, k \rrbracket$ we conclude that 
\begin{equation}\label{YOP7}
\begin{split}
&\mathbb{E}_{H^{RW}} \left[ W_{H,S} \cdot {\bf 1} \{ \ell_i(x) - p x  \geq B T^{1/2} \mbox{ for some $x \in [0,T]$ and some $i \in \llbracket 1, k \rrbracket$} \} \right] \leq \\
&\sum_{i = 1}^k\mathbb{P}^{0,T, x_i', y_i'}_{H^{RW}} \left( {\bf 1} \{ \ell(x) - p x  \geq B T^{1/2} \mbox{ for some $x \in [0,T]$} \} \right) \leq \\
&\sum_{i = 1}^k\mathbb{P}^{0,T, 0, pT}_{H^{RW}} \left( {\bf 1} \{ \ell(x) - p x  \geq B T^{1/2} - z_i \mbox{ for some $x \in [0,T]$} \} \right) \leq \\
&\sum_{i = 1}^k\mathbb{P}^{0,T, 0, pT}_{H^{RW}} \left( {\bf 1} \{ \ell(x) - p x  \geq T^{1/2}(A_1 + B_1 + |p|)  \mbox{ for some $x \in [0,T]$} \} \right) \leq k \delta,
\end{split}
\end{equation}
where in going from the second to the third line we used the shift invariance of $\mathbb{P}_{H^{RW}}^{T_0 ,T_1, x, y}$ and (\ref{YOP3}), and in going from the third to the fourth line we used the definition of $B$ in Step 1 and (\ref{YOP3}). The last inequality in (\ref{YOP7}) used the definition of $A_1$ and $\delta$ from Step 1. Equation (\ref{YOP7})  implies the second line in (\ref{YOP6}).\\

{\bf \raggedleft Step 3.} In this step we prove the first line in (\ref{YOP6}). Let us define the events 
$$E_i = \left\{ \sup_{x \in [T_0, T_1] } \left| \ell_i(x) - p x - z_i  \right|  \leq A_1 T^{1/2} \right\} \mbox{ for $i \in \llbracket 1, k \rrbracket$}.$$
Using that $\ell_1, \dots, \ell_k$ are independent under $\mathbb{P}_{H^{RW}}$ and have laws $\mathbb{P}_{H^{RW}}^{0 ,T, x_i', y_i'}$ for $i \in \llbracket 1, k \rrbracket$ we conclude
\begin{equation}\label{YOP8}
\begin{split}
&\mathbb{P}_{H^{RW}} \left( \cap_{i = 1}^k E_i  \right) = \prod_{i = 1}^k\mathbb{P}_{H^{RW}} \left(  E_i  \right) =  \prod_{i = 1}^k\mathbb{P}_{H^{RW}}^{0 ,T, x_i', y_i'} \left(  \sup_{x \in [0, T] } \left| \ell(x) - p x - z_i \right|  \leq A_1 T^{1/2} \right) = \\
&\mathbb{P}_{H^{RW}}^{0 ,T, 0 , pT} \left(  \sup_{x \in [0, T] } \left| \ell(x) - p x \right|  \leq A_1 T^{1/2} \right)^k \geq (1- \delta)^{k},
\end{split}
\end{equation}
where in going from the first to the second line we used the shift invariance of $\mathbb{P}_{H^{RW}}^{T_0 ,T_1, x,y}$ and (\ref{YOP3}), while in the last inequality we used the definition of $A_1$ from Step 1 and the fact that $T \geq U_2$.

On the other hand, we have using (\ref{YOP3}) that on $\cap_{i = 1}^k E_i $ for $i \in \llbracket 1, k-1\rrbracket$
\begin{equation*}
\begin{split}
&\ell_{i}(x) \geq p x + z_i -A_1T^{1/2} \geq p x + z_{i+1} + A_1 T^{1/2} + (B_1 + |p|) T^{1/2} \geq \\
& p (x+1) + z_{i+1} + A_1 T^{1/2} + B_1 T^{1/2} \geq \ell_{i+1}(x+1) + B_1 T^{1/2}.
\end{split}
\end{equation*}
Using (\ref{YOP3}), the fact that $g(j) - pj \leq M T^{1/2} \mbox{ for all $j \in \llbracket 0, T \rrbracket$}$, and the definition of  $B $ from Step 1, we conclude that on $\cap_{i = 1}^k E_i $ for $T \geq U_2$ we have
\begin{equation*}
\begin{split}
&\ell_{k}(x) \geq p x + z_k -A_1 T^{1/2} \geq p (x +1) + |M| T^{1/2} +  A_1 T^{1/2} + B_1  T^{1/2} \geq g(x+1) + B_1 T^{1/2}.
\end{split}
\end{equation*}
The last two inequalities, the monotonicity of $H$ (see Definition \ref{AssH}) and the definition of $W_{H,S}$ in (\ref{WH}) gives on $\cap_{i = 1}^k E_i $ 
\begin{equation}\label{YOP9}
W_{H,S} =  \exp \left( - \sum_{i = 1}^{k}  \sum_{ m  \in S} H (\ell_{i + 1}(m + 1) - \ell_{i}(m)) \right) \geq \exp \left( - k T H(-B_1 T^{1/2}) \right) \geq 1-\delta,
\end{equation}
where $\ell_{k+1} = g$ and the last inequality used (\ref{YOP2}). Combining (\ref{YOP8}) and (\ref{YOP9}) proves the first line in (\ref{YOP6}).
\end{proof}


\begin{lemma}\label{LSep} Fix $k \in \mathbb{N}$ and let $\mathfrak{L} = (L_1, \dots, L_k)$ have law $\mathbb{P}_{H,H^{RW},S}^{1, k, 0 ,T, \vec{x}, \vec{y},\infty,g}$ as in Definition \ref{Pfree}, where we assume that $H$ is as in Definition \ref{AssH}, while $H^{RW}$ satisfies the assumptions in Definition \ref{AssHR}. Let $ M^{\mathsf{side}, \mathsf{top}}_{k},M^{\mathsf{side}, \mathsf{bot}}_{k}, M_{k}^{\mathsf{base}, \mathsf{top}}, M^{\mathsf{sep}}_k, p , t \in \mathbb{R}$ be given such that $M^{\mathsf{side}, \mathsf{top}}_{k} \geq M^{\mathsf{side}, \mathsf{bot}}_{k}$, $M^{\mathsf{sep}}_k > 0$, $t \in (0, 1/2]$. Then we can find constants $V_k  \in \mathbb{N}$, $\epsilon_k > 0$, $A^i_{k}, B_k^i \in \mathbb{R}$ for $i = 1,\dots, k$ all depending on $ M^{\mathsf{side}, \mathsf{top}}_{k},M^{\mathsf{side}, \mathsf{bot}}_{k}, M_{k}^{\mathsf{base}, \mathsf{top}}, M^{\mathsf{sep}}_k, p , t , k, H, H^{RW}$ so that the following all hold.
\begin{enumerate}
\item $B^1_k \geq A^1_k \geq B^2_k \geq A^2_k \geq \cdots \geq B^k_k \geq A^k_k \geq M_{k}^{\mathsf{base}, \mathsf{top}} $;
\item $A^i_k - B^{i+1}_{k} \geq M^{\mathsf{sep}}_k$ for $i = 1, \dots, k-1$ and $A^k_k - M_{k}^{\mathsf{base}, \mathsf{top}}  \geq M_k^{\mathsf{sep}}$.
\item  If $T \geq V_k$, $g \in Y^-(\llbracket 0, T \rrbracket) $, $\vec{x}, \vec{y} \in \mathbb{R}^k$ satisfy
\begin{itemize}[leftmargin=*]
\item $ M^{\mathsf{side},\mathsf{bot}}_{k} T^{1/2} \leq x_i \leq  M^{\mathsf{side},\mathsf{top}}_{k} T^{1/2}$ for $i = 1, \dots, k$;
\item  $ p T + M^{\mathsf{side},\mathsf{bot}}_{k} T^{1/2}  \leq y_i \leq p T +  M^{\mathsf{side},\mathsf{top}}_{k} T^{1/2}$ for $i = 1, \dots, k$;
\item  $g(j) \leq j p + M_{k}^{\mathsf{base},\mathsf{top}} T^{1/2}$ for $j \in \llbracket 0, T\rrbracket$,  
\end{itemize}
then for any $S \subset \llbracket 0, T-1 \rrbracket$ we have
\begin{equation}\label{ESep}
\mathbb{P}_{H,H^{RW},S}^{1, k, 0 ,T, \vec{x}, \vec{y},\infty,g} \left(  B^i_k T^{1/2} \geq L_i(x) - p x  \geq A^i_k T^{1/2} \mbox{ for all $x \in[t T, (1-t)T]$, $i \in \llbracket 1, k \rrbracket$}    \right) \geq \epsilon_k.
\end{equation}
\end{enumerate}
\end{lemma}
\begin{proof} In the proof below, unless otherwise specified, all constants will depend on $ M^{\mathsf{side}, \mathsf{top}}_{k}$, $M^{\mathsf{side}, \mathsf{bot}}_{k}$, $M_{k}^{\mathsf{base}, \mathsf{top}}$, $M^{\mathsf{sep}}_k$, $p , t , k, H, H^{RW}$. We prove the lemma by induction on $k \in \mathbb{N}$. For clarity we split the proof into five steps.\\

{\bf \raggedleft Step 1.} In this step we prove the base case $k = 1$. From Lemma \ref{LStayUp} applied to $k = 1$, $M = M^{\mathsf{side}, \mathsf{bot}}_{1}$, $A = M^{\mathsf{base}, \mathsf{top}}_{1} + M^{\mathsf{sep}}_{1} $ and $p,t$ as in the statement of this lemma we can find $N_1 \in \mathbb{N}$ and $\delta_1 > 0$ such that for $T \geq N_1$, 
 $$\mathbb{P}_{H,H^{RW},S}^{1, 1, 0 ,T, \vec{x}, \vec{y},\infty,g} \left(  L_1(x) - p x  \geq A T^{1/2} \mbox{ for all $x \in[t T, (1-t)T]$}  \right) \geq 2\delta_1.$$
On the other hand, from Lemma \ref{LStayPut} applied to $k = 1$, $\epsilon = \delta_1$, $M = \max( M^{\mathsf{side}, \mathsf{top}}_{1}, M_{1}^{\mathsf{base},\mathsf{top}})$ and $p$ as in the statement of this lemma we can find $N_2 \in \mathbb{N}$ and $B \in \mathbb{R}$ such that 
$$\mathbb{P}_{H,H^{RW},S}^{1, 1, 0 ,T, \vec{x}, \vec{y},\infty,g} \left(  L_1(x) - p x  \geq B T^{1/2} \mbox{ for some $x \in[0, T]$}  \right) \leq \delta_1.$$
The last two inequalities imply $B \geq A$ and together yield the statement of the lemma with $V _1 := \max(N_1, N_2)$, $\epsilon_1 := \delta_1$, $A_1 := A$ and $B_1 := B$. \\

{\bf \raggedleft Step 2.} We suppose that we have established the lemma for $k = m$, where $m \in \mathbb{N}$, and in the remaining steps proceed to prove it for $k = m+1$. In this step we specify $V_{m+1}  \in \mathbb{N}$, $\epsilon_{m+1} > 0$, $A^i_{m+1}, B_{m+1}^i \in \mathbb{R}$ for $i = 1,\dots, m+1$.

Let us put $s = t/10 \in (0, 1/20]$, $T_0 = \lceil s T \rceil$, $T_1 = \lfloor (1 -s)T \rfloor$ and $\Delta T = T_1 - T_0$. From Lemma \ref{LStayUp} applied to $k = m+1$, $M = M^{\mathsf{side}, \mathsf{bot}}_{m+1}$, $A = 2\left|M^{\mathsf{base}, \mathsf{top}}_{m+1}\right| + 2 M^{\mathsf{sep}}_{m+1}  + 1$, $t = s$ and $p, H, H^{RW}$ as in the statement of this lemma we can find $N_1 \in \mathbb{N}$ and $\delta_1 > 0$ such that for $T \geq N_1$, 
 $$\mathbb{P}_{H,H^{RW},S}^{1, m+1, 0 ,T, \vec{x}, \vec{y},\infty,g} \left(  L_i(x) - p x \geq A \Delta T^{1/2} \mbox{ for all $x \in[s T, (1-s)T]$ and $i \in \llbracket 1, m+1 \rrbracket$}  \right) \geq 2\delta_1.$$
On the other hand, from Lemma \ref{LStayPut} applied to $k = m+1$, $\epsilon = \delta_1$, $M = \max( M^{\mathsf{side}, \mathsf{top}}_{m+1}, M_{m+1}^{\mathsf{base},\mathsf{top}})$ and $p, H, H^{RW}$ as in the statement of this lemma we can find $N_2 \in \mathbb{N}$ and $B \in \mathbb{R}$ such that 
$$\mathbb{P}_{H,H^{RW},S}^{1, m+1, 0 ,T, \vec{x}, \vec{y},\infty,g} \left(  L_i(x) - p x  \geq B \Delta T^{1/2} \mbox{ for some $x \in [0,T]$ and some $i \in \llbracket 1, m+1 \rrbracket$} \right) \leq \delta_1.$$
The last two inequalities imply $B \geq A \geq 0$ and for $T \geq \max(N_1, N_2)$ 
\begin{equation}\label{YR1}
\begin{split}
&\mathbb{P}_{H,H^{RW},S}^{1, m+1, 0 ,T, \vec{x}, \vec{y},\infty,g} \left( E_{m+1}  \right) \geq \delta_1, \mbox{ where $E_{m+1}$ denotes the event } \\
& \left\{  B \Delta T^{1/2} \geq L_i(x) - p x  \geq A \Delta T^{1/2} \mbox{ for all $x \in[s T, (1-s)T]$ and $i \in \llbracket 1, m+1 \rrbracket$}  \right\}.
\end{split}
\end{equation}

Let $V_m$, $\epsilon_m$, $A_m^i, B_m^i$ for $i = 1, \dots, m$  be the available by the induction hypothesis parameters for 
$$\mbox{ $ M^{\mathsf{side}, \mathsf{top}}_{m} = M_{m}^{\mathsf{base}, \mathsf{top}} = B$, $M^{\mathsf{side}, \mathsf{bot}}_{m} = A$, $ M^{\mathsf{sep}}_m =  2M^{\mathsf{sep}}_{m+1} + 2$ }, $$
$p, H ,H^{RW}$ as in the statement of this lemma and $t = s$.  We let $N_3 \in \mathbb{N}$ be large so that for $T \geq N_3$
\begin{equation}\label{BigT}
\begin{split}
&\mbox{$\Delta T \geq V_m$, \hspace{2mm} $T_0 + s  \Delta T \leq tT$, and $(1-t)T \leq T_0 + (1-s) \Delta T$},\\
& (1 -2s)^{1/2}\cdot (B_m^i +1) T^{1/2} \geq  B_m^i \Delta T^{1/2} \mbox{ and }  (1 -2s)^{1/2} \cdot (A_m^i - 1) T^{1/2} \leq  A_m^i \Delta T^{1/2} \\
& (1 -2s)^{1/2} \cdot (B +1) T^{1/2} \geq  B \Delta T^{1/2} \mbox{ and }  (1 -2s)^{1/2} \cdot (A - 1) T^{1/2} \leq  A \Delta T^{1/2}.
\end{split}
\end{equation}
We set  $A_{m+1}^i  := (1 - 2s)^{1/2}(A_{m}^i - 1) $ and $B_{m+1}^i  := (1 - 2s)^{1/2}(B_{m}^i + 1) $ $i = 1, \dots, m$ and also $A^{m+1}_{m+1} := (1- 2s)^{1/2} (A - 1)$ and $B^{m+1}_{m+1} := (1- 2s)^{1/2} (B + 1)$. Finally, we put $V_{m+1} := \max(N_1, N_2, N_3)$ and $\epsilon_{m+1} := \delta_1 \cdot \epsilon_{m}$.

The above two paragraphs specify our choice of $V_{m+1}  \in \mathbb{N}$, $\epsilon_{m+1} > 0$, $A^i_{m+1}, B_{m+1}^i \in \mathbb{R}$ for $i = 1,\dots, m+1$ and below we proceed to show that they satisfy the conditions of the lemma.\\

{\bf \raggedleft Step 3.} In this step we show that $A^i_{m+1}, B_{m+1}^i \in \mathbb{R}$ for $i = 1,\dots, m+1$ from the previous step satisfy conditions (1) and (2) in the lemma.

We first observe that by the induction hypothesis for $i = 1, \dots, m$ we have  
$$B_{m+1}^i - A^i_{m+1} =  (1 - 2s)^{1/2}(B_{m}^i  - A^i_{m} +2) \geq 2 (1 - 2s)^{1/2} \geq 0,$$
and also for $i = 1, \dots, m-1$
\begin{equation*}
\begin{split}
A_{m+1}^i - B^{i+1}_{m+1} = \hspace{2mm} & (1 - 2s)^{1/2} \cdot (A_{m}^i - B^{i}_{m+1} -2) \geq  (1 - 2s)^{1/2} \cdot \left(M^{\mathsf{sep}}_{m}  -2  \right) =  \\
&  (1 - 2s)^{1/2}\cdot 2M^{\mathsf{sep}}_{m+1}  \geq M^{\mathsf{sep}}_{m+1}. 
\end{split}
\end{equation*}
where we used $(1-2s)^{1/2} \geq 1/2$ as $s \in (0, 1/20]$. In addition, by the induction hypothesis we have
\begin{equation*}
\begin{split}
A^{m}_{m+1} - B^{m+1}_{m+1} =  \hspace{2mm} &(1 - 2s)^{1/2} \cdot (A_{m}^m -  B - 2) \geq (1 - 2s)^{1/2} \cdot (M^{\mathsf{sep}}_m  - 2)=  \\
&(1 - 2s)^{1/2}\cdot 2M^{\mathsf{sep}}_{m+1}   \geq M^{\mathsf{sep}}_{m+1}.
\end{split}
\end{equation*}
Finally, by the definition of $A_{m+1}^{m+1}$ and $B_{m+1}^{m+1}$ we have
\begin{equation*}
\begin{split}
&A^{m+1}_{m+1} - M_{m+1}^{\mathsf{base}, \mathsf{top}}  = (1 - 2s)^{1/2} \cdot\left(2\left|M^{\mathsf{base}, \mathsf{top}}_{m+1}\right| + 2M^{\mathsf{sep}}_{m+1}  \right)  - M_{m+1}^{\mathsf{base}, \mathsf{top}} \geq M^{\mathsf{sep}}_{m+1} \mbox{ and }\\
& B^{m+1}_{m+1} - A^{m+1}_{m+1} = (1 - 2s)^{1/2} \cdot (B - A + 2) \geq 0.
\end{split}
\end{equation*}
The above three equations show that conditions (1) and (2) are satisfied for our $A^i_{m+1}, B_{m+1}^i \in \mathbb{R}$ for $i = 1,\dots, m+1$. We next proceed to verify condition (3).  \\

{\bf \raggedleft Step 4.} In this step we fix $S \subset \llbracket 0, T- 1 \rrbracket$, $\vec{x}, \vec{y}, g$ as in condition (3) and prove (\ref{ESep}). Let us define the event 
\begin{equation}\label{YR2}
\begin{split}
F_{m+1} = \hspace{2mm} & \left\{  B \Delta T^{1/2} \geq L_{m+1}(x) - p x  \geq A \Delta T^{1/2} \mbox{ for all $x \in[T_0, T_1]$} \right\}  \cap \\
 &\left\{  B \Delta T^{1/2} \geq L_{i}(T_0) - p T_0 , L_{i}(T_1) - p T_1  \geq A \Delta T^{1/2} \mbox{ for all $i \in \llbracket 1, m \rrbracket$} \right\},
\end{split}
\end{equation}
and observe that $E_{m+1} \subset  F_{m+1}$ and so from (\ref{YR1}) we get
\begin{equation}\label{YR3}
\begin{split}
&\mathbb{P}_{H,H^{RW},S}^{1, m+1, 0 ,T, \vec{x}, \vec{y},\infty,g} \left( F_{m+1}  \right) \geq \delta_1.
\end{split}
\end{equation}
In addition, we define the event $G_{m+1}$ 
\begin{equation}\label{YR4}
\begin{split}
G_{m+1} = \left\{ B^i_{m+1} T^{1/2} \geq L_i(x) - p x  \geq A^i_{m+1} T^{1/2} \mbox{ for all $x \in[t T, (1-t)T]$ and $i \in \llbracket 1, m \rrbracket$}   \right\}.
\end{split}
\end{equation}
We set
$$\mathcal{F}_{m+1} = \sigma \left( L_i(s): (i,s) \in \llbracket 1, m+1 \rrbracket \times \llbracket 0, T \rrbracket \setminus \llbracket 1, m \rrbracket \times  \llbracket T_0+ 1, T_1 - 1 \rrbracket \right) ,$$
and recall that this was the external sigma algebra as in (\ref{GibbsCond}). 

We claim that for $T \geq V_{m+1}$ we have $\mathbb{P}_{H,H^{RW},S}^{1, m+1, 0 ,T, \vec{x}, \vec{y},\infty,g} $ almost surely
\begin{equation}\label{YR5}
{\bf 1}_{F_{m+1}} \cdot \mathbb{E}_{H,H^{RW},S}^{1, m+1, 0 ,T, \vec{x}, \vec{y},\infty,g}  \left[ {\bf 1}_{G_{m+1}} \vert \mathcal{F}_{m+1} \right] \geq {\bf 1}_{F_{m+1}}  \cdot \epsilon_m.
\end{equation}
We will prove (\ref{YR5}) in the next step. Here we assume its validity and conclude the proof of (\ref{ESep}).

Writing $\mathbb{E}_{m+1}$ for $\mathbb{E}_{H,H^{RW},S}^{1, m+1, 0 ,T, \vec{x}, \vec{y},\infty,g}$ we have for $T \geq V_{m+1}$ 
\begin{equation*}
\begin{split}
&\mathbb{P}_{H,H^{RW},S}^{1, m+1, 0 ,T, \vec{x}, \vec{y},\infty,g} \left( F_{m+1} \cap G_{m+1} \right) = \mathbb{E}_{m+1}  \left[ {\bf 1}_{F_{m+1}} \cdot {\bf 1}_{G_{m+1}} \right] = \mathbb{E}_{m+1}  \left[ \mathbb{E}_{m+1}  \left[ {\bf 1}_{F_{m+1}} \cdot {\bf 1}_{G_{m+1}} \vert \mathcal{F}_{m+1} \right] \right]   \\
& =  \mathbb{E}_{m+1}  \left[{\bf 1}_{F_{m+1}}  \mathbb{E}_{m+1}  \left[  {\bf 1}_{G_{m+1}} \vert \mathcal{F}_{m+1} \right] \right] \geq  \epsilon_{m}\mathbb{E}_{m+1}  \left[{\bf 1}_{F_{m+1}}  \right] \geq \epsilon_{m} \cdot \delta_1 = \epsilon_{m+1},
\end{split}
\end{equation*}
where the second equality follows from the tower property of conditional expectations, the third equality used the fact that $F_{m+1} \in \mathcal{F}_{m+1}$, the first inequality on the second line used (\ref{YR5}), the second inequality on the second line used (\ref{YR3}) and the last equality used the definition of $\epsilon_{m+1}$. Since the event in (\ref{ESep}) contains $F_{m+1} \cap G_{m+1} $ for $T \geq V_{m+1}$, here we used (\ref{BigT}), we see that the last inequality implies (\ref{ESep}).\\

{\bf \raggedleft Step 5.} In this final step we prove (\ref{YR5}). Put $S' = S \cap \llbracket T_0, T_1 - 1\rrbracket$ and observe that by the partial $(\vec{H},H^{RW})$-Gibbs property, see (\ref{GibbsEq}), applied to the function 
$$F'( \mathfrak{L}') = {\bf 1}\left\{ B^i_{m+1} T^{1/2} \geq L'_i(x) - p x  \geq A^i_{m+1} T^{1/2} \mbox{ for all $x \in[t T, (1-t)T]$ and $i \in \llbracket 1, m \rrbracket$}   \right\},$$
we have for $T \geq V_{m+1}$ that $\mathbb{P}_{H,H^{RW},S}^{1, m+1, 0 ,T, \vec{x}, \vec{y},\infty,g}$ almost surely
\begin{equation}\label{YR6}
\mathbb{E}_{H,H^{RW},S}^{1, m+1, 0 ,T, \vec{x}, \vec{y},\infty,g}  \left[ {\bf 1}_{G_{m+1}} \vert \mathcal{F}_{m+1} \right] = \mathbb{E}_{H,H^{RW},S'}^{1, m, T_0 ,T_1, \vec{x}', \vec{y}',\infty, g'}  \left[ F'( \mathfrak{L}')  \right],
\end{equation}
where $x_i' = L_i(T_0)$, $y_i' = L_{i}(T_1)$ for $i \in \llbracket 1, m \rrbracket$ and $g'(j) = L_{m+1}(j)$ for $j \in \llbracket T_0, T_1 \rrbracket$. The assumption $T \geq V_{m+1}$ was used to ensure that $tT \geq T_0$ and $(1-t)T \leq T_1$, as follows from (\ref{BigT}), and so $F'$ is a well-defined function on $Y(\llbracket 1, m \rrbracket \times \llbracket T_0, T_1 \rrbracket)$. 

We note that if $\mathfrak{L}'$ and $\mathfrak{L}''$ are respectively distributed according to $\mathbb{P}_{H,H^{RW},S'}^{1, m, T_0 ,T_1, \vec{x}', \vec{y}',\infty, g'}$ and $\mathbb{P}_{H,H^{RW},S''}^{1, m, 0 ,\Delta T, \vec{x}'', \vec{y}'',\infty, g''}$, where 
$$S'' = \{s \in \llbracket 0, \Delta T - 1 \rrbracket : s + T_0 \in S' \}, \hspace{2mm} x_i'' = x'_i -  pT_0, \hspace{2mm} y_i'' = y_i' - pT_0 \mbox{ for $i \in \llbracket 1, m \rrbracket$}, $$
$$ g''(j) = g'(j + T_0) - pT_0 \mbox{ for $j \in \llbracket 0, \Delta T\rrbracket$},$$
then $\mathfrak{L}'$ and $\mathfrak{L}''$ have the same distribution except for a reindexing and a vertical shift by $-pT_0$. The latter observation implies that 
\begin{equation}\label{YR7}
\begin{split}
\mathbb{E}_{H,H^{RW},S'}^{1, m, T_0 ,T_1, \vec{x}', \vec{y}',\infty, g'}  \left[ F'( \mathfrak{L}')  \right] = \hspace{2mm} &\mathbb{P}_{H,H^{RW},S''}^{1, m, 0 , \Delta T, \vec{x}'', \vec{y}'',\infty, g''}   \Big( B^i_{m+1} T^{1/2} \geq L''_i(x) - px  \geq A^i_{m+1} T^{1/2}  \\
&  \mbox{ for all $x \in[t T - T_0, (1-t)T - T_0]$ and $i \in \llbracket 1, m \rrbracket$} \Big).
\end{split}
\end{equation}

Next we observe by the definition of $A_{m+1}^i$ and $B_{m+1}^i$ for $i = 1, \dots, m$ and (\ref{BigT}) that
\begin{equation*}
\begin{split}
&B^i_{m+1} T^{1/2} = (1 -2s)^{1/2}\cdot (B_m^i +1) T^{1/2} \geq  B_m^i \Delta T^{1/2} \mbox{ and } \\
&A^i_{m+1} T^{1/2} = (1 -2s)^{1/2}\cdot (A_m^i - 1) T^{1/2} \leq  A_m^i \Delta T^{1/2},
\end{split}
\end{equation*}
which combined with (\ref{YR6}) and (\ref{YR7}) imply for $T \geq V_{m+1}$ that $\mathbb{P}_{H,H^{RW},S}^{1, m+1, 0 ,T, \vec{x}, \vec{y},\infty,g}$ almost surely
\begin{equation*}
\begin{split}
&\mathbb{E}_{H,H^{RW},S}^{1, m+1, 0 ,T, \vec{x}, \vec{y},\infty,g}  \left[ {\bf 1}_{G_{m+1}} \vert \mathcal{F}_{m+1} \right] \geq \mathbb{P}_{H,H^{RW},S''}^{1, m, 0 , \Delta T, \vec{x}'', \vec{y}'',\infty, g''}   \Big( B^i_{m} \Delta T^{1/2} \geq L''_i(x) - px  \geq   \\
& A^i_{m} \Delta T^{1/2} \mbox{ for all $x \in[t T - T_0, (1-t)T - T_0]$ and $i \in \llbracket 1, m \rrbracket$} \Big).
\end{split}
\end{equation*}

From (\ref{BigT}) we have $tT - T_0 \geq s  \Delta T$ and $(1-t)T - T_0 \leq (1-s) \Delta T$ and so the last inequality implies for $T \geq V_{m+1}$ that $\mathbb{P}_{H,H^{RW},S}^{1, m+1, 0 ,T, \vec{x}, \vec{y},\infty,g}$ almost surely
\begin{equation}\label{YR8}
\begin{split}
&\mathbb{E}_{H,H^{RW},S}^{1, m+1, 0 ,T, \vec{x}, \vec{y},\infty,g}  \left[ {\bf 1}_{G_{m+1}} \vert \mathcal{F}_{m+1} \right] \geq \mathbb{P}_{H,H^{RW},S''}^{1, m, 0 , \Delta T, \vec{x}'', \vec{y}'',\infty, g''}   \Big( B^i_{m} \Delta T^{1/2} \geq L''_i(x) - px  \geq   \\
& A^i_{m} \Delta T^{1/2} \mbox{ for all $x \in[ s  \Delta T, (1-s)\Delta T]$ and $i \in \llbracket 1, m \rrbracket$} \Big).
\end{split}
\end{equation}
On the event $F_{m+1}$ we have that 
$$B \Delta T^{1/2} x_i'' \geq A \Delta T^{1/2},  \hspace{2mm} B \Delta T^{1/2} \geq y_i'' - p \Delta T \geq A \Delta T^{1/2}  \mbox{ for $i \in \llbracket 1, m \rrbracket$ and }, $$
$$g''(j) - j p \leq B \Delta T^{1/2} \mbox{ for } j \in \llbracket 0, \Delta T \rrbracket. $$ 
The latter observation and the induction hypothesis imply that for $T \geq V_{m+1}$ on $F_{m+1}$ we have $\mathbb{P}_{H,H^{RW},S}^{1, m+1, 0 ,T, \vec{x}, \vec{y},\infty,g}$ almost surely
\begin{equation*}
\begin{split}
\mathbb{P}_{H,H^{RW},S''}^{1, m, 0 , \Delta T, \vec{x}'', \vec{y}'',\infty, g''}   \Big( &B^i_{m} \Delta T^{1/2} \geq L''_i(x) - px  \geq A^i_{m} \Delta T^{1/2}   \\
& \mbox{ for all $x \in[ s  \Delta T, (1-s)\Delta T]$ and $i \in \llbracket 1, m \rrbracket$} \Big) \geq \epsilon_m,
\end{split}
\end{equation*}
which together with (\ref{YR8}) imply (\ref{YR5}). This suffices for the proof.
\end{proof}


\begin{lemma}\label{LAccProbP}Fix $k \in \mathbb{N}$ and let $H^{RW}$ be as in Definition \ref{AssHR}, and $H$ as in Definition \ref{AssH}. For any $p \in \mathbb{R}$, $r \in (0,1)$ we can find $A > 0$, depending on $k, p,r ,  H^{RW}, H$, so that the following holds. 

Let $A^i, B^i \in \mathbb{R}$ for $i = 1,\dots, k$ be given such that 
$$ \mbox{ $B^1\geq A^1 \geq B^2 \geq A^2 \geq \cdots \geq B^k \geq A^k $ and $A^i - B^{i+1} \geq A$ for $i = 1, \dots, k-1$ }.$$
We can find $U_3 \in \mathbb{N}$, depending on $k, p,r, H, H^{RW}, A, A^i, B^i$ (for $i = 1, \dots, k$) such that if $T_0, T_1 \in \mathbb{Z}$, $\Delta T = T_1 - T_0 \geq U_3$, $r \Delta T \leq R \leq r^{-1} \Delta T$, $\vec{x}, \vec{y} \in \mathbb{R}^k$, $g \in Y^{-}(\llbracket T_0, T_1 \rrbracket) $ satisfy
$$ A^i R^{1/2} \leq x_i  - pT_0\leq B^i R^{1/2} \mbox{ and } A^i R^{1/2} \leq y_i  - pT_1 \leq B^i  R^{1/2} \mbox{ for $i = 1, \dots, k$, and }$$
$$ g(j) - pj \leq (A^k - A) R^{1/2} \mbox{ for $j \in \llbracket T_0, T_1 \rrbracket$, we have }$$
\begin{equation}\label{EAccProbP}
  Z_{ H,H^{RW}}\left(T_0, T_1, \vec{x}, \vec{y}, \infty, g  \right) \geq  1/2.
\end{equation}
where $Z_{ H,H^{RW}}$ is the acceptance probability of Definition \ref{Pfree}.
\end{lemma}
\begin{proof} For clarity we split the proof into two steps.\\

{\bf \raggedleft Step 1.} Let $\delta \in (0,1)$ be sufficiently small, depending on $k$, so that 
\begin{equation}\label{JOP1}
(1- \delta)^{k+1} \geq 1/2.
\end{equation}
Let $A_1 = A(\delta, p, H^{RW})$ be as in Lemma \ref{S7LStayInBand} and $B_1 > 0$ be sufficiently large so that for all $T \in \mathbb{N}$
\begin{equation}\label{JOP2}
\exp \left( - k T H \left( - B_1 \sqrt{T}  \right) \right) \geq 1- \delta.
\end{equation}
Note that such a choice of $B_1$ is possible, since by Definition \ref{AssH} we have that $\lim_{x \rightarrow \infty} x^2 H(-x) = 0$. We put $A = 2r^{-1/2} (A_1 + B_1 + |p|)$, which specifies our choice of $A$ in the lemma.

Let $\mathfrak{L} = (L_1, \dots, L_k)$ have law $\mathbb{P}_{H^{RW}}^{1, k, T_0 ,T_1, \vec{x}, \vec{y}}$, as in Definition \ref{Pfree}, where $T_0, T_1, \vec{x}, \vec{y}$ are as in the statement of the lemma. We write $\mathbb{P}_{H^{RW}}$ in place of $\mathbb{P}_{H^{RW}}^{1, k, T_0 ,T_1, \vec{x}, \vec{y}}$, $\mathbb{E}_{H^{RW}}$ in place of $\mathbb{E}_{H^{RW}}^{1, k, T_0 ,T_1, \vec{x}, \vec{y}}$, and $W_H$ in place of $W_H^{1, k, T_0, T_1, \infty, g}(L_1, \dots, L_k)$ to simplify the notation. 

Let us define the events 
$$E_i = \left\{  \sup_{ s \in [T_0, T_1] }  \left| L_i(s) - \frac{T_1 - s }{T_1 - T_0} \cdot x_i - \frac{s - T_0 }{T_1 - T_0} \cdot y_i  \right| \leq A_1 \Delta T^{1/2}  \right\} \mbox{ for $i \in \llbracket 1, k \rrbracket$}.$$
We claim that we can find $U_3 \in \mathbb{N}$, depending on $p, H, H^{RW}, A, A^i, B^i$ (for $i = 1, \dots, k$), such that for $\Delta T \geq U_3$ we have 
\begin{equation}\label{JOP3}
 \mathbb{P}_{H^{RW}} \left( \cap_{i = 1}^k E_i  \right) \geq (1 -\delta)^k,
\end{equation}
and $\mathbb{P}_{H^{RW}}$-almost surely
\begin{equation}\label{JOP4}
\prod_{i = 1}^k {\bf 1}_{E_i} \cdot W_H \geq (1 - \delta) \cdot \prod_{i = 1}^k {\bf 1}_{E_i}.
\end{equation}
We prove (\ref{JOP3}) and (\ref{JOP4}) in the second step. Here we assume their validity and conclude the proof of the lemma.\\

Using Definition \ref{Pfree} we have
\begin{equation*}
  Z_{ H,H^{RW}}\left(T_0, T_1, \vec{x}, \vec{y}, \infty, g  \right) = \mathbb{E}_{H^{RW}} \left[W_H \right] \geq (1- \delta)^{k+1} \geq 1/2,
\end{equation*}
where the first inequality used (\ref{JOP3}) and (\ref{JOP4}), while the second one used (\ref{JOP1}). The last inequality implies (\ref{EAccProbP}).\\

{\bf \raggedleft Step 2.} In this step we prove (\ref{JOP3}) and (\ref{JOP4}). Let $M = r^{-1/2} \cdot (|A^m| + |B^1|)$ and note that by assumption we have 
$$| x_i  - pT_0| \leq M \Delta T^{1/2} \mbox{ and }| y_i  - pT_1| \leq M \Delta T^{1/2} \mbox{ for $i \in \llbracket 1, k \rrbracket$},$$
where we used that $r \Delta T \leq R \leq r^{-1} \Delta T$. We set $U_3 = W_1 (M, p, \epsilon, H^{RW})$ as in Lemma \ref{S7LStayInBand}, which specifies our choice of $U_3$.

Using that $L_1, \dots, L_k$ are independent under $\mathbb{P}_{H^{RW}}$ and have laws $\mathbb{P}_{H^{RW}}^{T_0 ,T_1, x_i, y_i}$ for $i \in \llbracket 1, k \rrbracket$ we conclude for $\Delta T \geq U_3$
\begin{equation*}
\begin{split}
&\mathbb{P}_{H^{RW}} \left( \cap_{i = 1}^k E_i  \right) = \prod_{i = 1}^k\mathbb{P}_{H^{RW}} \left(  E_i  \right) = \\
&  \prod_{i = 1}^k\mathbb{P}_{H^{RW}}^{T_0 ,T_1, x_i, y_i} \left( \sup_{ s \in [T_0, T_1] }  \left| \ell(s) - \frac{T_1 - s }{T_1 - T_0} \cdot x_i - \frac{s - T_0 }{T_1 - T_0} \cdot y_i  \right| \leq A_1 \Delta T^{1/2}  \right)  \geq (1- \delta)^{k},
\end{split}
\end{equation*}
where in the last inequality we used the definition of $A_1$ from Step 1 and the fact that $\Delta T \geq U_3$. This proves (\ref{JOP3}).

We next observe that on $\cap_{i = 1}^k E_i $ for $i \in \llbracket 1, k-1\rrbracket$ we have
\begin{equation*}
\begin{split}
&L_{i}(s) \geq \frac{T_1 - s }{T_1 - T_0} \cdot x_i + \frac{s - T_0 }{T_1 - T_0} \cdot y_i - A_1\Delta T^{1/2} \geq ps + A^i R^{1/2}  - A_1\Delta T^{1/2} \geq \\
& ps + B^{i+1} R^{1/2}  - A_1\Delta T^{1/2} + A R^{1/2} \geq p (s + 1) + B^{i+1} R^{1/2}  + A_1\Delta T^{1/2} +B_1 \Delta T^{1/2} \geq \\
& L_{i+1}(s+1)+B_1 \Delta T^{1/2} ,
\end{split}
\end{equation*}
where we used the definition of $x_i, y_i$, $A$, $A_1$ and $r \Delta T \leq R \leq r^{-1} \Delta T$. In addition, on $\cap_{i = 1}^k E_i $ 
\begin{equation*}
\begin{split}
&L_{k}(s) \geq ps + A^k R^{1/2}  - A_1\Delta T^{1/2}\geq g(s+1) +B_1 \Delta T^{1/2},
\end{split}
\end{equation*}
where we used the definition of $g$, $A$, $A_1$ and $r \Delta T \leq R \leq r^{-1} \Delta T$.

The above inequalities, the monotonicity of $H$ (see Definition \ref{AssH}) and the definition of $W_H$ in (\ref{WH}) gives on $\cap_{i = 1}^k E_i $ 
\begin{equation*}
W_H =  \exp \left( - \sum_{i = 1}^{k}  \sum_{ m = T_0}^{T_1-1} \hspace{-1mm}H (L_{i + 1}(m + 1) - L_{i}(m)) \right) \geq \exp \left( - k \Delta T  H(-B_1\Delta T^{1/2}) \right) \geq 1/2,
\end{equation*}
where $L_{k+1} = g$, and the last inequality used (\ref{JOP2}). The last inequality gives (\ref{JOP4}).
\end{proof}


\begin{lemma}\label{S7LAccProb}[Lemma \ref{LAccProb}] Fix $k \in \mathbb{N}$ and let $\mathfrak{L} = (L_1, \dots, L_k)$ have law $\mathbb{P}_{H,H^{RW}}^{1, k, T_0 ,T_1, \vec{x}, \vec{y},\infty,g}$ as in Definition \ref{Pfree}, where we assume that $H$ is as in Definition \ref{AssH}, while $H^{RW}$ satisfies the assumptions in Definition \ref{AssHR}. Let $M^{\mathsf{side}, \mathsf{top}},M^{\mathsf{side}, \mathsf{bot}}, M^{\mathsf{base}, \mathsf{top}}, p , t \in \mathbb{R}$ be given such that $M^{\mathsf{side}, \mathsf{top}} \geq M^{\mathsf{side}, \mathsf{bot}}$, $t \in (0, 1/3)$. Then we can find constants $W_7  \in \mathbb{N}$, $\delta > 0$, all depending on $ M^{\mathsf{side}, \mathsf{top}},M^{\mathsf{side}, \mathsf{bot}}, $ $M^{\mathsf{base}, \mathsf{top}}, p , t , k, H$ and $H^{RW}$ so that the following holds.

Set $\Delta T = T_1 - T_0$ and suppose that $\vec{x}, \vec{y} \in \mathbb{R}^k$, $g \in Y^-( \llbracket T_0, T_1 \rrbracket)$ satisfy
\begin{itemize}
\item $ pT_0 +  M^{\mathsf{side},\mathsf{bot}} \Delta T^{1/2} \leq x_i \leq pT_0 + M^{\mathsf{side},\mathsf{top}} \Delta T^{1/2}$ for $i = 1, \dots, k$;
\item $ p T_1 + M^{\mathsf{side},\mathsf{bot}} \Delta T^{1/2}  \leq y_i \leq p T_1 +  M^{\mathsf{side},\mathsf{top}} \Delta T^{1/2}$ for $i = 1, \dots, k$;
\item  $g(j) \leq j p + M^{\mathsf{base},\mathsf{top}} \Delta T^{1/2}$ for $j \in \llbracket T_0, T_1\rrbracket$.
\end{itemize}
For any $\epsilon > 0$, $T_1, T_0 \in \mathbb{Z}$ with $\Delta T \geq W_7$, and $t_1, t_0 \in \llbracket T_0, T_1 \rrbracket$ with $\min(t_0 - T_0 , T_1 - t_1 , t_1 - t_0 ) \geq t \Delta T$
\begin{equation}\label{S7EAccProb}
\mathbb{P}_{H,H^{RW}}^{1, k, T_0 ,T_1, \vec{x}, \vec{y},\infty,g} \left(  Z_{H, H^{RW}}\left(t_0, t_1, \mathfrak{L}(t_0), \mathfrak{L}(t_1), \infty, g \llbracket t_0, t_1 \rrbracket  \right) \leq \delta \epsilon \right) \leq \epsilon,
\end{equation}
where $Z_{ H,H^{RW}}$ is the acceptance probability of Definition \ref{Pfree}, $g \llbracket t_0, t_1 \rrbracket $ is restriction of $g$ to $\llbracket t_0, t_1 \rrbracket$ and $\mathfrak{L}(a)  = (L_1(a), \dots, L_k(a))$ is the value of $\mathfrak{L}$ with law $\mathbb{P}_{H,H^{RW}}^{1, k, T_0 ,T_1, \vec{x}, \vec{y},\infty,g}$ at $a$.
\end{lemma}
\begin{proof}
For clarity we split the proof into two steps.\\

{\bf \raggedleft Step 1.} In this step we prove (\ref{S7EAccProb}), modulo a certain claim, see (\ref{OP1}), which is established in the second step below.

Let us set $S = \llbracket T_0, t_0 \rrbracket \cup \llbracket t_1 , T_1-1 \rrbracket$. Let $\tilde{\mathfrak{L}} = (\tilde{L}_1, \dots, \tilde{L}_k)$ be a $\llbracket 1, k\rrbracket$-indexed discrete line ensemble on $\llbracket T_0, T_1 \rrbracket$ that has law $\mathbb{P}_{H,H^{RW},S}^{1, k, T_0 ,T_1, \vec{x}, \vec{y},\infty,g}$. For brevity, we write $\mathbb{P}_{\mathfrak{L}}$ for $\mathbb{P}_{H,H^{RW}}^{1, k, T_0 ,T_1, \vec{x}, \vec{y},\infty,g}$ and $\mathbb{P}_{\tilde{\mathfrak{L}}}$ for $\mathbb{P}_{H,H^{RW},S}^{1, k, T_0 ,T_1, \vec{x}, \vec{y},\infty,g}$.

We claim that there exist constants $h_1, h_2 > 0$ and $N_1 \in \mathbb{N}$, all depending on $ M^{\mathsf{side}, \mathsf{top}},M^{\mathsf{side}, \mathsf{bot}}, $ $M^{\mathsf{base}, \mathsf{top}}, p , t , k, H$ and $H^{RW}$, such that for $\Delta T \geq N_1$ we have 
\begin{equation}\label{OP1}
\mathbb{P}_{\tilde{\mathfrak{L}}} \left(   Z_{H, H^{RW}}\left(t_0, t_1, \tilde{\mathfrak{L}}(t_0), \tilde{\mathfrak{L}}(t_1), \infty, g \llbracket t_0, t_1 \rrbracket  \right)  \geq h_1 \right)  \geq h_2.
\end{equation}
We prove (\ref{OP1}) in the second step. Here we assume its validity and conclude the proof of (\ref{S7EAccProb}).\\

Let $S' =\llbracket T_0, t_0 \rrbracket \cup \llbracket t_1 , T_1 \rrbracket$ and denote by $\mathbb{P}_{\mathfrak{L}'}$ and $\mathbb{P}_{\tilde{\mathfrak{L}}'}$ the laws of the $Y(\llbracket 1, k \rrbracket \times S')$-valued random variables obtained from restricting $\mathfrak{L}$ and $\tilde{\mathfrak{L}}$ (distributed according to $\mathbb{P}_{\mathfrak{L}}$  and  $\mathbb{P}_{\tilde{\mathfrak{L}}}$ respectively) to the set $\llbracket 1, k \rrbracket \times S'$. We observe from Definition \ref{Pfree} that $\mathbb{P}_{{\mathfrak{L}}'}$ is absolutely continuous with respect to $\mathbb{P}_{\tilde{\mathfrak{L}}'}$ and its Radon-Nikodym derivative is given by
\begin{equation}\label{OP2}
\frac{d\mathbb{P}_{\mathfrak{L}'}}{d\mathbb{P}_{\tilde{\mathfrak{L}}'}}\left(\mathfrak B\right) = (Z')^{-1} Z_{H,H^{RW}}\left(t_0,t_1,\mathfrak{B}\left(t_0\right),\mathfrak{B}\left(t_1\right), \infty,g\llbracket t_0,t_1\rrbracket\right),
\end{equation}
with $Z' = \mathbb{E}_{\tilde{\mathfrak{L}}'}\left[Z_{ H,H^{RW}}\left(t_0, t_1, \mathfrak B(t_0), \mathfrak B(t_1), \infty, g \llbracket t_0, t_1\rrbracket \right)\right]$ (here $\mathfrak B$ is $\mathbb{P}_{\tilde{\mathfrak{L}}'}$-distributed).

Now note that $Z_{H,H^{RW}}\left(t_0, t_1, \mathfrak B(t_0), \mathfrak B(t_1), \infty, g\llbracket t_0, t_1\rrbracket\right)$ is a bounded, measurable function of $\left((\mathfrak B(t_0), \mathfrak B(t_1)\right)$ (in view of Lemma \ref{ContinuousGibbsCond}). In addition, the law of $\left((\mathfrak B(t_0), \mathfrak B(t_1)\right)$ under $\mathbb{P}_{\tilde{\mathfrak{L}}'}$ is the same as that of $\big(\tilde{\mathfrak{L}}(t_0), \tilde{\mathfrak{L}}(t_1)\big)$ by way of the restriction. It follows from (\ref{OP1})
\begin{equation}\label{OP3}
\begin{split}
Z' = \hspace{2mm} &\mathbb{E}_{\tilde{\mathfrak{L}}'}\left[Z_{H,H^{RW}}\left(t_0, t_1, \mathfrak B(t_0), \mathfrak B(t_1), \infty ,g\llbracket t_0, t_1\rrbracket\right)\right] = \\
& \mathbb{E}_{\tilde{\mathfrak{L}}}\left[Z_{H,H^{RW}}\left(t_0, t_1, \tilde{\mathfrak{L}}(t_0), \tilde{\mathfrak{L}}(t_1),  \infty ,g\llbracket t_0, t_1\rrbracket \right)\right]\geq h_1 h_2.
\end{split}
\end{equation}	
Similarly, the law of $\left(\mathfrak B(t_0), \mathfrak{B}(t_1)\right)$ under $\mathbb{P}_{\mathfrak{L}'}$ is the same as that of $\left(\mathfrak{L}(t_0), \mathfrak{L}(t_1)\right)$ under $\mathbb{P}_{\mathfrak{L}}$. Hence
\begin{equation}\label{OP4}
\begin{split}
&\mathbb{P}_{\mathfrak{L}}\left(Z_{H,H^{RW}}(t_0, t_1, \mathfrak{L}(t_0), \mathfrak{L}(t_1), \infty, g\llbracket t_0, t_1\rrbracket)\leq  h_1 h_2 \epsilon\right)=\\
&\mathbb{P}_{\mathfrak{L}'}\left(Z_{H,H^{RW}} \left(t_0, t_1, \mathfrak B(t_0), \mathfrak B(t_1), \infty, g\llbracket t_0, t_1\rrbracket \right)\leq h_1 h_2 \epsilon \right).
\end{split}
\end{equation}
Writing $E=\{\mathfrak{B} \in  Y(\llbracket 1, k \rrbracket \times S'): Z_{H,H^{RW}}\left(t_0, t_1, \mathfrak B(t_0), \mathfrak B(t_1), \infty, g\llbracket t_0, t_1\rrbracket \right)\leq h_1 h_2 \epsilon \}$ we get
\begin{equation*}
\begin{split}
\mathbb{P}_{\mathfrak{L}'}(E)= &\int_{Y(\llbracket 1, k \rrbracket \times S')} \hspace{-10mm}{\bf 1}_E  d\mathbb{P}_{\mathfrak{L}'}(\mathfrak{B}) = \frac{1}{Z'}\int_{Y(\llbracket 1, k \rrbracket \times S')}\hspace{-10mm}{\bf 1}_E  \cdot  Z_{H,H^{RW}}\left(t_0, t_1, \mathfrak B(t_0), \mathfrak B(t_1), \infty, g\llbracket t_0, t_1\rrbracket\right)  d\mathbb{P}_{\tilde{\mathfrak{L}}'}(\mathfrak{B}) \leq \\
& \frac{1}{Z'}\int_{Y(\llbracket 1, k \rrbracket \times S')}\hspace{-10mm}{\bf 1}_E  \cdot  h_1 h_2 \epsilon  d\mathbb{P}_{\tilde{\mathfrak{L}}'}(\mathfrak{B})  \leq \frac{1}{Z'} \cdot h_1 h_2 \epsilon  \leq \epsilon,
\end{split}
\end{equation*}
where in the second equality we used (\ref{OP2}), in going from the first to the second line we used the definition of $E$, and in the last inequality we used (\ref{OP3}). The last inequality and (\ref{OP4}) together imply (\ref{S7EAccProb}) with $\delta = h_1 h_2$ and $W_7 = N_1$.\\

{\bf \raggedleft Step 2.} In this step we prove (\ref{OP1}). Let $A > 0$ be as in Lemma \ref{LAccProbP} for $p,k, H^{RW}, H$ as in the statement of this lemma and $r = t$.  Suppose further that $\epsilon_{k}$, $V_{k}$ and $A_{k}^{i}, B_{k}^i$ are as in Lemma \ref{LSep} for $k,p,t, H, H^{RW}$ as in the statement of this lemma
$$ M^{\mathsf{side}, \mathsf{top}}_{k} = M^{\mathsf{side}, \mathsf{top}}, \hspace{2mm}    M_{k}^{\mathsf{base}, \mathsf{top}} = M^{\mathsf{base}, \mathsf{top}} , \hspace{2mm} M^{\mathsf{side}, \mathsf{bot}}_{k} = M^{\mathsf{side}, \mathsf{bot}}, \hspace{2mm} M^{\mathsf{sep}}_k = A .$$
With the above notation we define $h_1 = 1/2$ and $h_2 = \epsilon_k$. We also let $N_1 \in \mathbb{N}$ be sufficiently large so that $N_1 \geq V_k$, and $t N_1 \geq U_3$, where $U_3$ is as in Lemma \ref{LAccProbP} for $k,p,H, H^{RW}$ as in the present lemma, $r = t$ and $A^i = A_k^i$, $B^i = B^i_k$ for $i = 1, \dots, k$.  
This specifies the choice of $h_1, h_2$ and $N_1$ for which we prove (\ref{OP1}).\\

Let us set $\hat{S} = \llbracket 0, t_0 -T_0 \rrbracket \cup \llbracket t_1 - T_0 , \Delta T - 1 \rrbracket$, and define $\vec{x}', \vec{y}' \in \mathbb{R}^k$, $g' \in Y^-( \llbracket 0, \Delta T \rrbracket)$ through
$$x_i' = x_i - pT_0, \hspace{2mm} y_i' = y_i - pT_0 \mbox{ for $i  \in \llbracket 1,k \rrbracket$ and } g'(j) = g(j + T_0) - pT_0 \mbox{ for $j \in \llbracket 0, \Delta T \rrbracket$}.$$
With the above data we let $\hat{\mathfrak{L}} = (\hat{L}_1, \dots, \hat{L}_k)$ be a $\llbracket 1, k\rrbracket$-indexed discrete line ensemble on $\llbracket 0, \Delta T \rrbracket$ that has law $\mathbb{P}_{H,H^{RW},\hat{S}}^{1, k, 0 ,\Delta T, \vec{x}', \vec{y}',\infty,g'}$. For simplicity we write $\mathbb{P}_{\hat{\mathfrak{L}}}$ in place of $\mathbb{P}_{H,H^{RW},\hat{S}}^{1, k, 0 ,\Delta T, \vec{x}', \vec{y}',\infty,g'}$.

Observe that the law of $\hat{\mathfrak{L}}$ under $\mathbb{P}_{\hat{\mathfrak{L}}}$ is the same as that of $\tilde{\mathfrak{L}}$ under $\mathbb{P}_{\tilde{\mathfrak{L}}}$ except for a reindexing and a vertical shift by $-pT_0$. The latter implies that 
\begin{equation}\label{OP5}
\begin{split}
&\mathbb{P}_{\tilde{\mathfrak{L}}} \left(   Z_{H, H^{RW}}\left(t_0, t_1, \tilde{\mathfrak{L}}(t_0), \tilde{\mathfrak{L}}(t_1), \infty, g \llbracket t_0, t_1 \rrbracket  \right)  \geq 1/2 \right)  = \\
&\mathbb{P}_{\hat{\mathfrak{L}}} \left(   Z_{H, H^{RW}}\left(s_0, s_1, \hat{\mathfrak{L}}(s_0), \hat{\mathfrak{L}}(s_1), \infty, g' \llbracket s_0, s_1 \rrbracket  \right)  \geq 1/2 \right),
\end{split}
\end{equation}
where $s_0 = t_0 - T_0$ and $s_1 = t_1 - T_0$. In addition, from (\ref{ESep}), with $S = \hat{S}$, we have for $\Delta T \geq N_1$ 
\begin{equation*}
\begin{split}
&\mathbb{P}_{\hat{\mathfrak{L}}}\left(  B^i_k \Delta T^{1/2} \geq \hat{L}_i(x) - p x  \geq A^i_k \Delta T^{1/2} \mbox{ for all $x \in[t \Delta T, (1-t)\Delta T]$ and $i \in \llbracket 1, k \rrbracket$}    \right) \geq \epsilon_k,
\end{split}
\end{equation*}
where we used that $\vec{x}'$, $\vec{y}'$, $g'$ satisfy the conditions of Lemma \ref{LSep} for the above $ M^{\mathsf{side},\mathsf{top}}_{k}$, $M^{\mathsf{side}, \mathsf{bot}}_{k}$, $M_{k}^{\mathsf{base}, \mathsf{top}}$, $M^{\mathsf{sep}}_k$, $p , t \in \mathbb{R}$ by construction.

The last inequality together with the fact that $s_0 \geq t \Delta T$ and $s_1 \leq (1-t) \Delta T$ (this follows from our assumptions that $t_0 - T_0 \geq t \Delta T$, $T_1 - t_1 \geq t \Delta T$) give for $\Delta T \geq N_1$
\begin{equation}\label{OP6}
\begin{split}
&\mathbb{P}_{\hat{\mathfrak{L}}}\left(  B^i_k \Delta T^{1/2} \geq \hat{L}_i(x) - p x  \geq A^i_k \Delta T^{1/2} \mbox{ for $x \in \{s_0, s_1\}$ and $i \in \llbracket 1, k \rrbracket$}    \right) \geq \epsilon_k.
\end{split}
\end{equation}

We further note that  by our choice of $A$ above, and Lemma \ref{LAccProbP} (with the parameters we specified above) we have for $\Delta T \geq N_1$ that
\begin{equation}\label{OP7}
\begin{split}
&\left\{  B^i_k \Delta T^{1/2} \geq \hat{L}_i(x) - p x  \geq A^i_k \Delta T^{1/2} \mbox{ for $x \in \{s_0, s_1\}$ and $i \in \llbracket 1, k \rrbracket$}    \right\} \subset \\
& \left\{Z_{H, H^{RW}}\left(s_0, s_1, \hat{\mathfrak{L}}(s_0), \hat{\mathfrak{L}}(s_1), \infty, g' \llbracket s_0, s_1 \rrbracket  \right)  \geq 1/2  \right\}.
\end{split}
\end{equation}
We mention that in (\ref{OP7}), $(\hat{L}_1(s_0), \dots, \hat{L}_k(s_0)) $ and $(\hat{L}_1(s_1), \dots, \hat{L}_k(s_1)) $ play the roles of $\vec{x}, \vec{y} $ in (\ref{EAccProbP});  in (\ref{OP7}), $s_0, s_1$ play the roles of $T_0, T_1$ in (\ref{EAccProbP}); in (\ref{OP7}), $g' \llbracket s_0, s_1 \rrbracket$ plays the role of $g$ in (\ref{EAccProbP});  in (\ref{OP7}), $ \Delta T$ plays the role of $R$ in (\ref{EAccProbP}). We also mention that $1 \geq (s_1 - s_0) / \Delta T  \geq t$ since $s_1 - s_0 = t_1 - t_0 \geq t \Delta T$ by definition.

Equations (\ref{OP5}), (\ref{OP6}) and (\ref{OP7}) together imply (\ref{OP1}). This suffices for the proof.

\end{proof}

%
\subsection{Proof of Lemma \ref{scaledavoidBB}}\label{Section7.4} In this section we prove Lemma \ref{scaledavoidBB}, whose statement is recalled here as Lemma \ref{S7scaledavoidBB} for the reader's convenience. In order to establish the result we will require the following result, which can be found in \cite{CorHamA}.
\begin{lemma}\label{NoTouch} \cite[Corollary 2.9]{CorHamA}. Fix a continuous function $f: [0,1] \rightarrow \mathbb{R}$ such that $f(0) > 0$ and $f(1) > 0$. Let $B$ be a standard Brownian bridge and let $C = \{ B(t) > f(t) \mbox{ for some $t \in [0,1]$}\}$ (crossing) and $T = \{ B(t) = f(t) \mbox{ for some } t\in [0,1]\}$ (touching). Then $\mathbb{P}(T \cap C^c) = 0.$
\end{lemma}

\begin{lemma}\label{S7scaledavoidBB}[Lemma \ref{scaledavoidBB}]
Fix $k\in\mathbb{N}$, $p \in \mathbb{R}$, $\alpha > 0$, and $a,b\in\mathbb{R}$ with $a<b$. Let $f_{\infty}:[a-1,b+1]\to(-\infty,\infty]$, $g_{\infty}:[a-1,b+1]\to[-\infty,\infty)$ be continuous functions such that $f_{\infty}(t) > g_{\infty}(t)$ for all $t\in[a-1,b+1]$. Let $\vec{x},\vec{y}\in W_k^\circ$ be such that $f_{\infty}(a) > x_1$, $f_{\infty}(b) > y_1$, $g_{\infty}(a) < x_k$, and $g_{\infty}(b) < y_k$. Let $a_N = \lfloor aN^\alpha\rfloor N^{-\alpha}$ and $b_N = \lceil bN^\alpha\rceil N^{-\alpha}$, and let $f_N : [a-1,b+1]\to(-\infty,\infty]$ and $g_N : [a-1,b+1]\to[-\infty,\infty)$ be continuous functions such that $f_N\to f_{\infty}$ and $g_N\to g_{\infty}$ uniformly on $[a-1,b+1]$. If $f_{\infty} \equiv \infty$ the last statement means that $f_N \equiv \infty$ for all large enough $N$ and if $g_{\infty} \equiv - \infty$ the latter means that $g_N \equiv -\infty$ for all large enough $N$. 

Lastly, let $\vec{x}\,^N, \vec{y}\,^N \in \mathbb{R}^k$, write $\tilde{x}^N_i =  N^{-\alpha/2} \sigma_p^{-1} (x_i^N - pa_N N^{\alpha})$, $\tilde{y}^N_i = N^{-\alpha/2} \sigma_p^{-1} (y_i^N - pb_N N^{\alpha})$, and suppose that $\tilde{x}^N_i \to x_i$ and $\tilde{y}^N_i \to y_i$ as $N\to\infty$ for each $i\in\llbracket 1,k\rrbracket$. 

Let $\mathcal{Y}^N$ have laws $\mathbb{P}^{\mathsf{scaled}}_N$ as in Definition \ref{scaledRW} for $\vec{x} = \vec{x}^N, \vec{y} = \vec{y}^N$, $a = a_N$, $b = b_N$, $f,g$ given by
\begin{equation*}
\begin{split}
&N^{-\alpha/2} \sigma_p^{-1}( f(xN^{\alpha})  - p x N^{\alpha}) = f_N(x) \mbox{ and } N^{-\alpha/2} \sigma_p^{-1}( g(xN^{\alpha})  - p x N^{\alpha}) = g_N(x) \mbox{ for $x \in [a_N,  b_N]$}.
\end{split}
\end{equation*}
Let $\mathcal{Z}^N = \mathcal{Y}^N|_{\llbracket 1, k\rrbracket \times[a,b]}$, i.e. $\mathcal{Z}^N$ is a sequence of random variables on $C(\llbracket 1, k \rrbracket \times [a,b])$ obtained by projecting $\mathcal{Y}^N$ to $\llbracket 1, k\rrbracket \times[a,b]$, then the laws of $\mathcal{Z}^N$ converge weakly to $\mathbb{P}^{a,b,\vec{x},\vec{y},f_{\infty},g_{\infty}}_{avoid}$ as $N\to\infty$.
\end{lemma}
\begin{proof} For clarity we split the proof into three steps.\\

{\bf \raggedleft Step 1.} Let $x,y \in \mathbb{R}$ and $x^N, y^N \in \mathbb{R}$ be such that 
\begin{equation}\label{ZOP1}
\begin{split}
&\lim_{N \rightarrow \infty} \tilde{x}^N = x \mbox{ and }\lim_{N \rightarrow \infty}  \tilde{y}^N = y, \mbox{ where } \\
&\mbox{ $\tilde{x}^N = N^{-\alpha/2} \sigma_p^{-1} (x^N - p a_N N^{\alpha})$ and $\tilde{y}^N = N^{-\alpha/2} \sigma_p^{-1} (y^N - p b_N N^{\alpha})$.}
\end{split}
\end{equation}
Suppose that $\ell^N$ has law $\mathbb{P}_{H^{RW}}^{\lfloor a N^{\alpha} \rfloor , \lceil b N^{\alpha} \rceil, x^N, y^N}$, put  $a_N = \lfloor aN^\alpha\rfloor N^{-\alpha}$, $b_N = \lceil bN^\alpha\rceil N^{-\alpha}$
$$h^N(x) = N^{-\alpha/2} \sigma_p^{-1} \cdot ( \ell^N(xN^{\alpha}) - p x N^{\alpha}) \mbox{ for $x \in [a_N, b_N]$},$$
and let $\tilde{h}^N \in C([a,b])$ be the restriction of $h^N$ to $[a,b]$. In this step we prove that $\tilde{h}^N$ converge weakly to $\mathbb{P}^{a,b, x,y}_{free}$ -- the law of a Brownian bridge (from $B(a) = x$ to $B(b) = y$) with diffusion parameter $1$ as in Section \ref{Section2.1}. 

Let us write $z^N = y^N-x^N$ and $T_N = (b_N-a_N)N^\alpha$. Let $\tilde{B}$ be a standard Brownian bridge on $[0,1]$, and define random variables $B^N$, $B$ taking values in $C([a_N,b_N])$, $C([a,b])$, respectively, via
\begin{align*}
B^N(t) &= \sqrt{b_N-a_N}\cdot\tilde{B}\left(\frac{t-a_N}{b_N-a_N}\right) + \frac{t-a_N}{b_N-a_N}\cdot \tilde{y}^N + \frac{b_N-t}{b_N-a_N}\cdot \tilde{x}^N,\\
B(t) &= \sqrt{b-a}\cdot \tilde{B}\left(\frac{t-a}{b-a}\right) + \frac{t-a}{b-a}\cdot y + \frac{b-t}{b-a}\cdot x.
\end{align*}
We observe that $B$ has law $\mathbb{P}^{a,b,x,y}_{free}$ and $B^N|_{[a,b]}\implies B$ as $N\to\infty$. By \cite[Theorem 3.1]{Bill}, to show that $\tilde{h}^N\implies B$, it suffices to find a sequence of probability spaces supporting ${h}^N,B^N$ so that
\begin{equation}\label{scaledRWweak}
\rho(B^N, {h}^N) := \sup_{t\in[a_N,b_N]} |B^N(t) - {h}^N(t)| \implies 0 \quad \mathrm{as} \quad N\to\infty.
\end{equation}
From Theorem \ref{KMT} for each $N\in\mathbb{N}$ there is a probability space supporting $B^N$ and ${h}^N$, as well as constants $C',a',\alpha' > 0$, depending on $H^{RW}$ and $p$, such that
\begin{equation}\label{scaledRWcheb}
\mathbb{E}\left[e^{a'\Delta_N}\right] \leq C' e^{\alpha' \log N} e^{|z^N-pT_N|^2/N^\alpha},
\end{equation}
where $\Delta_N = N^{\alpha/2} \sigma_p \cdot \rho(B^N, {h}^N) $. Since 
$$\lim_{N \rightarrow \infty}N^{-\alpha/2} (z^N - pT_N) = \sigma_p (y - x)$$
by assumption, we can find $N_0\in\mathbb{N}$ and $A>0$ so that $|z-pT_N| \leq AN^{\alpha/2}$ for $N\geq N_0$. Then for any fixed $\epsilon > 0$ and $N\geq N_0$, Chebyshev's inequality and \eqref{scaledRWcheb} give
$$\mathbb{P}(\rho(B^N,h^N) > \epsilon) \leq C' e^{-a' \epsilon \sigma_p N^{\alpha/2}}e^{\alpha' \log N} e^{A^2}.$$
The right hand side tends to 0 as $N\to\infty$, implying \eqref{scaledRWweak}.\\

{\bf \raggedleft Step 2.} Let $\mathfrak{L}^N = (L_1^N, \dots, L_k^N)$ have law $\mathbb{P}_{H^{RW}}^{1,k , \lfloor a N^{\alpha} \rfloor, \lceil b N^{\alpha} \rceil, \vec{x}^N, \vec{y}^N}$, and let $\mathcal{L}^N = (\mathcal{L}^N_1, \dots, \mathcal{L}^N_k)$ be the $\llbracket 1, k\rrbracket$-indexed line ensemble on $[a,b]$, defined via
\begin{equation}\label{ZOP2}
\mathcal{L}^N_i(t) = N^{-\alpha/2} \sigma_p^{-1} \cdot \left( L^N_i(tN^{\alpha}) - p s N^{\alpha} \right), \quad t\in [a,b].
\end{equation}
We further write $W_H(\mathfrak{L}^N) $ in place of $W_{H}^{1, k,  \lfloor a N^{\alpha} \rfloor ,\lceil b N^{\alpha} \rceil ,F_N,G_N} (L^N_{1}, \dots, L^N_{k})$ as in (\ref{WH}), where 
\begin{equation*}
\begin{split}
&N^{-\alpha/2} \sigma_p^{-1}( F_N(xN^{\alpha})  - p x N^{\alpha}) = f_N(x) \mbox{, } N^{-\alpha/2} \sigma_p^{-1}( G_N(xN^{\alpha})  - p x N^{\alpha}) = g_N(x) \mbox{ for $x \in [a_N,  b_N]$}.
\end{split}
\end{equation*}
Let $\mathcal{B} = (B_1, \dots, B_k)$ have law $\mathbb{P}_{free}^{1,k , a,  b, \vec{x}, \vec{y}}$ as in Section \ref{Section2.1} and define the event 
\begin{equation}\label{ZOP3}
E_{\mathsf{avoid}} = \{ g_{\infty}(x) < B_k(x) < \cdots < B_1(x) < f_{\infty}(x) \mbox{ for all $x \in [a, b]$} \}.
\end{equation}
We claim that for any bounded continuous $H: C(\llbracket 1, k \rrbracket \times [a,b]) \rightarrow \mathbb{R}$ we have 
\begin{equation}\label{ZOP4}
\lim_{N \rightarrow \infty} \mathbb{E} \left[ H(\mathcal{L}^N) \cdot W_H(\mathfrak{L}^N)  \right] =  \mathbb{E} \left[ H(\mathcal{B}) \cdot {\bf 1}_{E_{\mathsf{avoid}}} \right].
\end{equation}
We prove (\ref{ZOP4}) in the next step. Here we assume its validity and conclude the proof of the lemma.\\

Note that by our assumptions on $f_{\infty},g_{\infty}, \vec{x}, \vec{y}$ we can find $\epsilon > 0$ and real-valued $h_1,\dots,h_k \in C([a,b])$, depending on $a,b,\vec{x},\vec{y},f_{\infty},g_{\infty}$, with $h_i(a) = x_i$, $h_i(b)=y_i$ for $i\in\llbracket 1,k\rrbracket$, such that if $u_i:[a,b]\to\mathbb{R}$ are continuous functions with $\rho(u_i,h_i) = \sup_{x\in[a,b]} |u_i(x)-h_i(x)| < \epsilon$, then
\begin{equation*}
f(x) - \epsilon > u_1(x) + \epsilon > u_1(x) - \epsilon > \cdots > u_k(x) + \epsilon > u_k(x) - \epsilon > g(x) + \epsilon
\end{equation*}
for all $x\in[a,b]$. By Lemma \ref{Spread}, we have
\begin{equation}\label{BBwindow}
\mathbb{P}^{a,b,\vec{x},\vec{y}}_{free}(E_{\mathsf{avoid}}) \geq \mathbb{P}^{a,b,\vec{x},\vec{y}}_{free}\left(\sup_{x \in [a,b]} |B_i (x) - h_i(x)|  < \epsilon \mbox{ for } i\in\llbracket 1,k\rrbracket \right) > 0.
\end{equation}
Fixing any bounded continuous $H: C(\llbracket 1, k \rrbracket \times [a,b]) \rightarrow \mathbb{R}$ we conclude that 
\begin{equation}\label{ZOP5}
\lim_{N \rightarrow \infty} \mathbb{E} \left[ H(\mathcal{Z}^N) \right]  \hspace{-1mm} = \hspace{-1mm} \lim_{N \rightarrow \infty}  \frac{\mathbb{E} \left[ H(\mathcal{L}^N) W_H(\mathfrak{L}^N)  \right]}{\mathbb{E} \left[  W_H(\mathfrak{L}^N)  \right]}  \hspace{-1mm} =  \hspace{-1mm} \frac{\mathbb{E} \left[ H(\mathcal{B}) \cdot {\bf 1}_{E_{\mathsf{avoid}}} \right]}{\mathbb{E} \left[ {\bf 1}_{E_{\mathsf{avoid}}} \right]}  \hspace{-1mm} =   \hspace{-1mm} \mathbb{E}^{a,b,\vec{x},\vec{y},f_{\infty},g_{\infty}}_{avoid} \left[H(\mathcal{Q}) \right],
\end{equation}
where $\mathcal{Q} = (\mathcal{Q}_1, \dots, \mathcal{Q}_k)$ has law $\mathbb{E}^{a,b,\vec{x},\vec{y},f_{\infty},g_{\infty}}_{avoid}$. Let us elaborate on (\ref{ZOP5}) briefly. The first equality follows from (\ref{RND}). The second equality follows from the convergence of both numerator and denominator on the left to the ones on the right, by (\ref{ZOP4}), and the fact that the denominator on the right is positive by (\ref{BBwindow}). The last equality follows from Definition \ref{DefAvoidingLaw}. Equation (\ref{ZOP5}) now implies the statement of the lemma.\\

{\bf \raggedleft Step 3.} From our work in Step 1, we have that $\mathcal{L}^N$ converges weakly to $\mathcal{B}$ as $N \rightarrow \infty$. By the Skorokhod representation theorem \cite[Theorem 6.7]{Bill}, we can also assume that $\mathcal{L}^N$ and $\mathcal{B}$ are all defined on the same probability space with measure $\mathbb{P}$ and the convergence is $\mathbb{P}$-almost sure. Here we are implicitly using Lemma \ref{Polish}  from which we know that the random variables $\mathcal{L}^N$ and $\mathcal{B}$ take value in a Polish space so that the Skorokhod representation theorem is applicable. 

If $\omega \in E_{\mathsf{avoid}}$, we can find $\delta(\omega) > 0$ such that 
\begin{equation*}
\begin{split}
&g_{\infty}(x)  + \delta(\omega) < B_k(x) - \delta(\omega) < B_k(x) + \delta(\omega) < B_{k-1} (x) - \delta(\omega) < \\
&B_{k-1} (x) + \delta(\omega)  < \cdots < B_1(x) - \delta(\omega)  <  B_1(x) + \delta(\omega)  < f_{\infty}(x)  - \delta(\omega) \mbox{ for all $x \in [a, b]$}. 
\end{split}
\end{equation*}
Combining the last inequalities with the definition of $W_H(\mathfrak{L}^N)$ from (\ref{WH}), the fact that $H(x) \geq 0$ is increasing, cf. Definition \ref{AssH}, we conclude that for all large $N$, depending on $\omega \in E_{\mathsf{avoid}}$, we have
\begin{equation}\label{ZOP6}
1 \geq W_H(\mathfrak{L}^N) \geq \exp \left( - (k+1) \cdot \left( \lceil bN^{\alpha} \rceil - \lfloor a N^{\alpha} \rfloor \right) \cdot H(- \delta (\omega) \sigma_p  N^{\alpha/2}) \right) \rightarrow 1 \mbox{ as } N \rightarrow \infty,
\end{equation}
where in the last convergence we used that $\lim_{x\rightarrow \infty} x^2 H(-x) = 0$, as follows from Definition \ref{AssH}. On the other hand, on the event 
$$E_{\mathsf{cross}} = \cup_{i = 0}^{k} \{ B_i(x) < B_{i+1}(x) \mbox{ for some $x \in [a, b]$} \},$$
where $B_0 = f_\infty$ and $B_{k+1} = g_{\infty}$ we know that there exists $\delta(\omega) > 0$, $x \in [a,b]$, $i \in \llbracket 0, k \rrbracket$ such that 
$$B_i(x) < B_{i+1}(x) - \delta (\omega).$$
Consequently, for all large $N$, depending on $\omega \in E_{\mathsf{cross}}$ we have
\begin{equation}\label{ZOP7}
0 \leq W_H(\mathfrak{L}^N) \leq \exp \left( - H( \delta (\omega) \sigma_p  N^{\alpha/2}) \right) \rightarrow 0 \mbox{ as } N \rightarrow \infty,
\end{equation}
where in the last convergence we used that $\lim_{x\rightarrow \infty}  H(x) = \infty$, as follows from Definition \ref{AssH}. 

We finally have from Lemma \ref{NoTouch} and (\ref{BBDef}) that $\mathbb{P} ( E_{\mathsf{avoid}} \cup E_{\mathsf{cross}} ) = 1$ (here we used the fact that $B_1, \dots, B_k$ are independent Brownian bridges and so $B_i - B_{i+1}$ is also a Brownian bridge with $\sqrt{2}$ times the diffusion parameter). Combining the last observation with (\ref{ZOP6}) and (\ref{ZOP7}) we conclude that $\mathbb{P}$-almost surely we have 
$$\lim_{N \rightarrow \infty}  H(\mathcal{L}^N) \cdot W_H(\mathfrak{L}^N) = H(\mathcal{B}) \cdot {\bf 1}_{E_{\mathsf{avoid}}},$$
which implies (\ref{ZOP4}) by the bounded convergence theorem.
\end{proof}

\bibliographystyle{amsalpha}
\bibliography{PD}

\end{document}